\IfFileExists{pdfmanagement.sty}{%
  \RequirePackage{pdfmanagement}
  \SetKeys[document/metadata]{uncompress}
}{%
  \ifdefined\pdfcompresslevel
    \pdfcompresslevel=0\relax
  \fi
  \ifdefined\pdfobjcompresslevel
    \pdfobjcompresslevel=0\relax
  \fi
}
\documentclass[11pt]{amsart}

\usepackage{etoolbox}

\makeatletter
\patchcmd{\@setauthors}
  {\MakeUppercase{\authors}}
  {\authors}
  {}
  {\PackageWarning{authorcase}{Author patch failed}}
\makeatother

 \usepackage{amsfonts,graphics,amsmath,amsthm,amsfonts,amscd,amssymb,amsmath,latexsym,multicol,
 mathrsfs}
\usepackage{epsfig,url}
\usepackage{flafter}
\usepackage{fancyhdr}
\usepackage{hyperref}
\usepackage[dvipsnames]{xcolor}
\usepackage{hyperref}
\hypersetup{
	colorlinks=true,
	linkcolor=purple,
	citecolor=ForestGreen,
	filecolor=cyan,
	urlcolor=purple
}

\usepackage{graphicx}
\usepackage{xcolor}
\usepackage{float}
\usepackage{bm}
\usepackage{enumitem}
\usepackage{multirow}
\usepackage{scalerel,stackengine}
\usepackage{stmaryrd}
\usepackage{datetime}
\usepackage[title]{appendix}
\usepackage{url}
\usepackage{tabulary}
\usepackage{booktabs}
\usepackage{tikz}
\usetikzlibrary{positioning}
\usepackage{verbatim}

\usepackage[utf8]{inputenc}
\usepackage[english]{babel}
\usepackage{csquotes}
\usepackage{comment}
\usepackage[
backend=biber,
style=alphabetic,
sorting=anyt,
doi=false,isbn=false,url=false,
maxalphanames=99,maxbibnames=99,backref=true,
]{biblatex}
\DefineBibliographyStrings{english}{
  backrefpage  = {\hspace{-0.35em}},
  backrefpages = {\hspace{-0.35em}},
}

\DeclareFieldFormat{bracketswithperiod}{\mkbibbrackets{#1}}

\renewbibmacro*{pageref}{%
  \iflistundef{pageref}
    {}
    {\printtext[bracketswithperiod]{%
       \ifnumgreater{\value{pageref}}{1}
         {\bibstring{backrefpages}\ppspace}
         {\bibstring{backrefpage}\ppspace}%
       \printlist[pageref][-\value{listtotal}]{pageref}}%
     }}

 \theoremstyle{plain}
 \newtheorem{theorem}[subsection]{Theorem}
 \newtheorem{lemma}[subsection]{Lemma}
 \newtheorem{corollary}[subsection]{Corollary}
 \newtheorem{proposition}[subsection]{Proposition}

 \theoremstyle{definition}
 \newtheorem{definition}[subsection]{Definition}
 \newtheorem{definitionnotation}[subsection]{Notation}

 \newtheorem{remark}[subsection]{Remark}

 \theoremstyle{remark}

 \numberwithin{equation}{subsection}

 \theoremstyle{plain} 
\newcommand{\thistheoremname}{}
\newtheorem{genericthm}[subsection]{\thistheoremname}

\newcommand{\abs}[1]{\left\vert#1\right\vert}
\newcommand{\set}[1]{\left\{#1\right\}}

\newcommand{\norm}[1]{\left\Vert#1\right\Vert}

\newcommand{\spec}{\operatorname{Spec}}

\newcommand{\codim}{\operatorname{codim}}
\newcommand{\Div}{\operatorname{Div}}
\newcommand{\hg}{\operatorname{H}}
\newcommand{\chara}{\operatorname{char}}

\newcommand{\hilb}{\operatorname{Hilb}}

\newcommand{\Supp}{\operatorname{Supp}}

\newcommand{\sdeg}{\operatorname{sdeg}}

\newcommand{\centre}{\operatorname{centre}}

\newcommand{\Ex}{\operatorname{Ex}}
\newcommand{\red}{\operatorname{red}}

\newcommand{\coeff}{\operatorname{coeff}}
\newcommand{\nor}{\operatorname{nor}}
\newcommand{\vol}{\operatorname{vol}}

\newcommand{\tr}{\operatorname{tr}}

\newcommand{\bb}{\operatorname{b}}

\newcommand{\pic}{\operatorname{Pic}}

\newcommand{\ZE}{\operatorname{ZE}}
\newcommand{\Z}{\operatorname{Z}}
\newcommand{\N}{\operatorname{N}}
\newcommand{\Gal}{\operatorname{Gal}}

\newcommand{\NE}{\operatorname{\overline{NE}}}
\newcommand{\Nef}{\operatorname{Nef}}
\newcommand{\Amp}{\operatorname{Amp}}
\newcommand{\ADiv}{\operatorname{ADiv}}

\newcommand{\an}{\operatorname{an}}
\newcommand{\mld}{\operatorname{mld}}
\newcommand{\univ}{\operatorname{univ}}
\newcommand{\interior}{\operatorname{int}}
\newcommand{\eff}{\operatorname{eff}}
\newcommand{\ord}{\operatorname{ord}}

\newcommand{\mni}{\medskip\noindent}

\newcommand{\mbb}{\mathbb}
\newcommand{\QQ}{\mbb{Q}}
\newcommand{\NN}{\mbb{N}}
\newcommand{\ZZ}{\mbb{Z}}
\newcommand{\CC}{\mbb{C}}
\newcommand{\RR}{\mbb{R}}
\newcommand{\AAA}{\mbb{A}}
\newcommand{\PP}{\mbb{P}}

\newcommand{\mc}{\mathcal}
\newcommand{\mcA}{\mc{A}}

\newcommand{\mcC}{\mc{C}}
\newcommand{\mcD}{\mc{D}}
\newcommand{\mcE}{\mc{E}}
\newcommand{\mcF}{\mc{F}}
\newcommand{\mcG}{\mc{G}}
\newcommand{\mcH}{\mc{H}}
\newcommand{\mcI}{\mc{I}}

\newcommand{\mcL}{\mc{L}}
\newcommand{\mcM}{\mc{M}}

\newcommand{\mcO}{\mc{O}}
\newcommand{\mcP}{\mc{P}}
\newcommand{\mcQ}{\mc{Q}}
\newcommand{\mcR}{\mc{R}}

\newcommand{\mcU}{\mc{U}}
\newcommand{\mcV}{\mc{V}}
\newcommand{\mcW}{\mc{W}}
\newcommand{\mcX}{\mc{X}}
\newcommand{\mcY}{\mc{Y}}

\newcommand{\mf}{\mathfrak}

\newcommand{\mb}{\mathbf}
\newcommand{\OO}{\mc{O}}

\newcommand{\wh}{\widehat}
\newcommand{\wt}{\widetilde}
\newcommand{\ol}{\overline}

\newcommand{\bir}{\dashrightarrow}

\usepackage{graphicx}
\usepackage{xcolor}
\usepackage{float}
\usepackage{amssymb}
\usepackage{amsmath}
\usepackage{mathrsfs}
\usepackage{bm}
\usepackage[all,cmtip]{xy}
\usepackage{enumitem}
\usepackage[utf8]{inputenc}
\usepackage[english]{babel}
\usepackage{multirow}
\usepackage{mathtools}
\usepackage{scalerel,stackengine}
\usepackage{stmaryrd}
\usepackage{datetime}
\usepackage[title]{appendix}
\usepackage{url}
\usepackage{tabulary}
\usepackage{booktabs}
\usepackage{tikz}
\usetikzlibrary{positioning}
\usepackage{verbatim}
\usepackage{braket}

\DeclarePairedDelimiter\ceil{\lceil}{\rceil}
\DeclarePairedDelimiter\rddown{\lfloor}{\rfloor}
\DeclareFieldFormat{postnote}{#1}
\DeclareFieldFormat{multipostnote}{#1}

\makeatletter
\def\moverlay{\mathpalette\mov@rlay}
\def\mov@rlay#1#2{\leavevmode\vtop{%
   \baselineskip\z@skip \lineskiplimit-\maxdimen
   \ialign{\hfil$\m@th#1##$\hfil\cr#2\crcr}}}
\newcommand{\charfusion}[3][\mathord]{
    #1{\ifx#1\mathop\vphantom{#2}\fi
        \mathpalette\mov@rlay{#2\cr#3}
      }
    \ifx#1\mathop\expandafter\displaylimits\fi}
\makeatother

\newcommand{\bigcupdot}{\charfusion[\mathop]{\bigcup}{\cdot}}

\makeatletter
\newcommand*{\da@rightarrow}{\mathchar"0\hexnumber@\symAMSa 4B }
\newcommand*{\da@leftarrow}{\mathchar"0\hexnumber@\symAMSa 4C }
\newcommand*{\xdashrightarrow}[2][]{%
  \mathrel{%
    \mathpalette{\da@xarrow{#1}{#2}{}\da@rightarrow{\,}{}}{}%
  }%
}
\newcommand{\xdashleftarrow}[2][]{%
  \mathrel{%
    \mathpalette{\da@xarrow{#1}{#2}\da@leftarrow{}{}{\,}}{}%
  }%
}
\newcommand*{\da@xarrow}[7]{%
  \sbox0{$\ifx#7\scriptstyle\scriptscriptstyle\else\scriptstyle\fi#5#1#6\m@th$}%
  \sbox2{$\ifx#7\scriptstyle\scriptscriptstyle\else\scriptstyle\fi#5#2#6\m@th$}%
  \sbox4{$#7\dabar@\m@th$}%
  \dimen@=\wd0 %
  \ifdim\wd2 >\dimen@
    \dimen@=\wd2 %
  \fi
  \count@=2 %
  \def\da@bars{\dabar@\dabar@}%
  \@whiledim\count@\wd4<\dimen@\do{%
    \advance\count@\@ne
    \expandafter\def\expandafter\da@bars\expandafter{%
      \da@bars
      \dabar@ 
    }%
  }%
  \mathrel{#3}%
  \mathrel{%
    \mathop{\da@bars}\limits
    \ifx\\#1\\%
    \else
      _{\copy0}%
    \fi
    \ifx\\#2\\%
    \else
      ^{\copy2}%
    \fi
  }%
  \mathrel{#4}%
}
\makeatother

\title{\large S\MakeLowercase{tein degree over fields of characteristic zero}}
\author[Caucher Birkar and Santai Qu]{%
  \large Caucher Birkar and Santai Qu
}
\date{\today}

\begin{document}

\begin{abstract}
    We show that the Stein degree of horizontal divisors on generalised log Calabi-Yau fibrations 
    or certain vertical divisors of Fano type fibrations over arbitrary fields of characteristic zero
    is bounded by constants that are independent on the ground field.
\end{abstract}

\maketitle

\tableofcontents

\addtocontents{toc}{\protect\setcounter{tocdepth}{1}}


\section{\bf Introduction}


\mni
{\textbf{\sffamily{Boundedness of Stein degree of divisors.}}}
Let $S\to Z$ be a projective surjective morphism of integral schemes,
and let $S\to V\to Z$ be the Stein factorisation.  The \emph{Stein degree} of $S\to Z$
is defined to be 
\[ \sdeg(S/Z) \coloneqq \deg (V/Z). \]
The main result of this paper is the boundedness of Stein degree of divisors 
on log Calabi-Yau fibrations over arbitrary fields of characteristic zero.
This is established in \cite{birkar2025steindegreelogcalabiyau} for 
log Calabi-Yau fibrations over algebraically closed field of characteristic zero.
In this paper, we show the stronger boundedness over non-closed fields of characteristic zero.
The crucial point is that the upper bound on Stein degree is independent on the ground field.

Let $k$ be a field of characteristic zero.
A \emph{generalised log Calabi-Yau fibration} $(X, B + \mb{M}) \to Z$ over $k$ consists of a generalised log canonical 
generalised pair $(X, B + \mb{M})$ over $k$ and a contraction between $k$-varieties $X\to Z$ 
such that $K_X + B + \mb{M}_X \sim_{\RR}0/Z$.

\begin{theorem}[$\mc{H}_{\univ}(d,t)$]\label{dim-d-t}
    Let $d\in \NN$ and let $t\in (0,1]$.
    There exists an $N_h(d,t) \in \NN$ depending only on $d,t$ such that
    the following holds.

    Let $k$ be a field of characteristic zero.
	Let $(X, B + \mb{M})\to Z$ be a generalised log Calabi-Yau fibration over $k$
	of relative dimension $d$, that is, a general fibre of $X\to Z$ has dimension $d$.
    Let $S$ be a horizontal$/Z$ irreducible component of $B$ 
	whose coefficient in $B$ is $\ge t$.  Then 
    \[ \sdeg (S^{\nor}/Z) \le N_h(d,t), \]
    where $S^{\nor}$ is the normalisation of $S$.
\end{theorem}

When $S$ is vertical$/Z$, we show the boundedness of Stein degree
by assuming that the image of $S$ in $Z$ is a divisor and that
$(X, B + \mb{M})\to Z$ is a \emph{Fano type generalised log Calabi-Yau fibration} over $k$
(or a \emph{Fano type fibration} over $k$ for short), that is, $(X, B + \mb{M})\to Z$
is a generalised log Calabi-Yau fibration over $k$ and $X$ is of Fano type over $Z$.
Recall that a contraction of $k$-varieties $X\to Z$ is called \emph{of Fano type} 
if there exists a big$/Z$ $\RR$-divisor $\Gamma$ such that $(X, \Gamma)$ is Kawamata log terminal (klt)
and $K_X + \Gamma \sim_{\RR}0/Z$.

\begin{theorem}[$\mc{V}_{\univ}(d, t)$]\label{vertical-bnd-intro}
    Let $d\in \NN$ and let $t \in (0,1]$.  
    There exists an $N_v(d,t) \in \NN$ depending only on $d,t$ such that
    the following holds.

    Let $k$ be a field of characteristic zero.
    Let 
    $(X, B + \mb{M})\to Z$ 
    be a Fano type generalised log Calabi-Yau fibration over $k$
    of relative dimension $d$.  
    Let $S$ be a vertical$/Z$ irreducible component of $B$ such that
    \begin{itemize} 
        \item [\emph{(1)}] the coefficient of $S$ in $B$ is $\ge t$, and
        \item [\emph{(2)}] the image $T$ of $S$ in $Z$ has codimension one in $Z$.
    \end{itemize}
    Then
    \[ \sdeg (S^{\nor}/T) \le N_v(d,t), \]
    where $S^{\nor}$ is the normalisation of $S$.
\end{theorem}

The proofs of $\mc{H}_{\univ}(d, t)$ and $\mc{V}_{\univ}(d, t)$ are proceeded by interweaving inductions
on the relative dimension $d$ in $\mc{H}_{\univ}(d, t)$ and $\mc{V}_{\univ}(d, t)$.
In particular, Theorem~\ref{vertical-bnd-intro} is one of the crucial ingredients
in the proof of Theorem~\ref{dim-d-t}, and the proof of Theorem~\ref{vertical-bnd-intro} 
is also the most involved part of this paper.
The proofs in this paper do not depend on toroidal geometry,
which is one of the main techniques in \cite{birkar2025steindegreelogcalabiyau}.
Moreover, by taking $\dim Z = 0$, i.e. $Z$ is the spectrum of the ground field,
Theorem~\ref{dim-d-t} proves \cite[Conjecture~1.17]{B-moduli} for log Calabi-Yau pairs 
over fields of characteristic zero.


\mni
{\textbf{\sffamily{Residue degree on log Calabi-Yau pairs.}}}
Let $X$ be a scheme of finite type over a field $k$.
For a closed point $x\in X$, the \emph{residue degree of $x$} is the degree
of the field extension $k(x)/k$, where $k(x)$ denotes the residue field of $X$ at $x$.

Let $k$ be a field of characteristic zero.
Let $(X, B)$ be a log canonical pair over $k$.
For a closed point $x\in X$, we define
\[ \mld_x(X, B) \coloneqq \inf\Set{ a(E, X, B)\mid E~\mbox{is a prime divisor over}~X~\mbox{with}~\centre_X E =x }, \]
which is called the \emph{minimal log discrepancy} (or \emph{mld} for short) of $(X, B)$ at $x$.
The mld measures the severity of singularities of the pair at $x$.
A corollary of Theorem~\ref{dim-d-t} is that the residue degree of closed points 
on log Calabi-Yau pairs is controlled by the singularity.

\begin{theorem}\label{arithmetic-thm}
    Let $d\in \NN$ and $t\in [0,1)$.
    Then there exists an $R(d,t)\in \NN$ depending only on $d,t$ such that the following holds.
    
    Let $(X, B)$ be a log Calabi-Yau pair of dimension $d$
    over a field $k$ of characteristic zero, i.e.
    $X$ is geometrically integral projective scheme over $k$ with $\dim X = d$, 
    $B$ is an $\RR$-divisor on $X$
    such that $(X, B)$ is log canonical, and 
    \[ K_X+B\sim_{\RR}0. \]
    Assume $\mld_x(X, B)\le t$.  Then the residue degree of $x$ is $\le R(d,t)$.
\end{theorem}


\mni
{\textbf{\sffamily{Acknowledgements.}}}
This project began in October 2025 as a continuation of \cite{birkar2025steindegreelogcalabiyau},
where we developed the main birational strategy
and established Stein degree boundedness over algebraically
closed fields.  The present paper extends this theory to
arbitrary fields of characteristic zero.  An arithmetic
bound for ordinary log Calabi-Yau pairs with rational
boundaries was independently announced in 
\cite{OAI:Arithmetic-Stein-degree-bounds-for-log-Calabi-Yau-pairs-September-25-2026} on
October 6, 2026.  The birational strategy common to \cite{OAI:Arithmetic-Stein-degree-bounds-for-log-Calabi-Yau-pairs-September-25-2026}
and the present paper originates in \cite{birkar2025steindegreelogcalabiyau}, while the
additional arithmetic arguments are different.  Our proofs
are independent of \cite{OAI:Arithmetic-Stein-degree-bounds-for-log-Calabi-Yau-pairs-September-25-2026}. 
With further work, Theorems~\ref{dim-d-t} and \ref{arithmetic-thm} can also be derived 
from \cite{OAI:Arithmetic-Stein-degree-bounds-for-log-Calabi-Yau-pairs-September-25-2026},
but Theorem~\ref{vertical-bnd-intro} and our results on MMP and complements
over non-closed fields are not covered by that work.

The first author was supported by a grant from Tsinghua University and 
a grant of the National Program of Overseas High Level Talent. 
The main ideas and overall direction of this
project were developed by the authors. 
AI tools (ChatGPT models) were used during
the preparation of the paper as technical assistants to help with checking arguments, 
locating relevant references, and improving the exposition. All mathematical statements, proofs,
and conclusions were verified by the authors.


\section{\bf Preliminaries}

The set of natural numbers $\NN$ is the set of all positive integers, 
so it does not contain 0; cf. \cite[\S 2]{B-Fano}.  Denote by $\QQ^{>0}$, respectively, by $\RR^{>0}$,
the set of all positive rational numbers, respectively, the set of all positive real numbers.


\subsection{Schemes and varieties}\label{ss-sch-var-defn}

(1)
A \emph{scheme} is defined in the sense of \cite[\S II.2]{Hart};
all the schemes in this paper are Noetherian.  
For a scheme $X$, denote by $\red X$ the maximal reduced closed subscheme of $X$, 
called the \emph{reduction} of $X$.
For an integral scheme $X$, we denote by $X^{\nor}$ the \emph{normalisation} of $X$.
A scheme $X$ is called \emph{of characteristic zero} if the residue field of every scheme-theoretic point of $X$
has characteristic zero.

\medskip
(2)
Let $k$ be an arbitrary field.
By an \emph{algebraic scheme over $k$} (or a \emph{scheme over $k$}, or a \emph{$k$-scheme}), 
we mean a scheme that is of finite type and separated over $\spec k$.
A \emph{variety} is a quasi-projective, integral $k$-scheme
that is geometrically integral;
when emphasising the quasi-projectivity, we also call it a \emph{quasi-projective variety}.

\medskip
(3)
Let $k$ be a field, and let $K/k$ be an arbitrary field extension.
Let $X$ be a $k$-scheme.  To save space, we also denote
by $X\otimes_k K$ fibre product $X\times_{\spec k}\spec K$,
which should not lead to confusion from the context; cf. \cite[\S 1.1]{BLR}.

\medskip
(4)
Let $k$ be a field of characteristic zero and $K/k$ an arbitrary field extension.
Assume that $X$ is a normal $k$-scheme.
Then by \cite[Corollaire~(6.14.5)]{EGA-IV-2}, $X\otimes_k K$ is also normal.

\medskip
(5)
Let $\phi\colon X\bir Y$ be a rational map between integral schemes.
We say that $\phi$ is \emph{of finite type} if there exists an open subset $U\subseteq X$
such that $\phi$ is well-defined on $U$, and that $U\to Y$ is of finite type.


\subsection{Contractions}
 
A \emph{contraction} is a proper morphism of schemes $f\colon X\to Y$
such that $f_*\mc{O}_X = \mc{O}_Y$; $f$ is not necessarily birational (see \cite[\S 2.1]{B-Fano}).
In particular, $f$ has \emph{geometrically connected} fibres;
see \cite[Chap.~III, Corollaire~(4.3.2), Remarque~(4.3.4)]{EGA-III-1}.
Moreover, by \cite[Proposition~4.1.5]{Liu06}, if $X$ is normal, then $Y$ is also normal.

We say that a birational map $\phi\colon X\dashrightarrow Y$ between normal schemes is a 
\emph{birational contraction} if $\phi$ is proper
and $\phi^{-1}$ does not contract any divisors; see \cite[page~424]{BCHM}.


\subsection{General fibres}

Let $X\to Y$ be a morphism between schemes, where $Y$ is integral.
Let $K(Y)$ be the function field of $Y$.
By \emph{generic fibre}, we mean the $K(Y)$-scheme $X\times_Y \spec K(Y)$;
by \emph{geometric generic fibre}, we mean the $\ol{K(Y)}$-scheme $X\times_Y \spec \ol{K(Y)}$.
A \emph{general fibre} is a fibre over a scheme-theoretic point in some nonempty open subset of $Y$.
A \emph{closed general fibre} is a general fibre over a closed point of $Y$.


\subsection{Stein factorisation}

Let $f\colon X\to S$ be a proper morphism between schemes.
By \cite[Chap.~III, Corollaire~(4.3.3)]{EGA-III-1},
we have the Stein factorisation 
\[ X\xrightarrow{g} V \xrightarrow{h} S, \]
where $g$ is proper with geometrically connected fibres, and $h\colon V\to S$ is finite.
If $f$ is surjective and proper, we define the \emph{Stein degree} of $f$ as
\[ \sdeg (X/S) \coloneqq \deg(V/S). \]


\begin{lemma}\label{stein-generic-fibre}
    Let $X\to S$ be a proper surjective morphism between integral scheme of characteristic zero.
    Then there exists a nonempty open subset $U\subseteq S$ such that
    for each scheme-theoretic point $s\in U$, 
    the number of connected components of $X_{\bar{s}}\coloneqq X\times_S \spec\ol{k(s)}$
    is equal to $\sdeg(X/S)$, where $k(s)$ denotes the residue field of $S$ at $s$.
\end{lemma}

\begin{proof}
    This follows from \cite[Remarque~(4.3.4)]{EGA-III-1} and \cite[Proposition~(9.7.8)]{EGA-IV-3}.
\end{proof}


\begin{definitionnotation}
	Let $X$ be a scheme.  Denote by $\mf{n}(X)$ 
	the number of irreducible components of $X$.  
    Let $X\to Y$ be a dominant morphism of finite type between integral schemes.
	We denote by $\mf{n}(X/Y)$ the number of irreducible components of a general \emph{geometric} fibre of $X\to Y$,
    which is well-defined by \cite[Proposition~(9.7.8)]{EGA-IV-3}.
    
	Let $X\to Y$ be a dominant morphism of finite type between schemes with $Y$ integral.
    Denote by $\dim (X/Y)$ the maximal dimension of irreducible components 
    of a general fibre of $X\to Y$.
    By \cite[Proposition~(9.5.5)]{EGA-IV-3}, $\dim (X/Y)$ is well-defined.
    In particular, if $X\to Y$ is a dominant $k$-morphism between integral schemes of finite type
    over a field $k$,
    every irreducible component of a general fibre of $X\to Y$ is of dimension $\dim X - \dim Y$.
\end{definitionnotation}


\begin{lemma}\label{num-g-fibre}
    Let $f\colon X\to Y$ be a dominant morphism of finite type
    between integral schemes of characteristic zero.  Let $K(X)$ and $K(Y)$
    be the function fields of $X$ and $Y$ respectively.
    Then for an arbitrary nonempty open subset $U\subset X$, we have
    \begin{equation}
    \mf{n} (X/Y) = \mf{n}(U/Y) = \mf{n} \big( \spec \big( K(X)\otimes_{K(Y)} \ol{K(Y)} \big) \big). 
    \end{equation}
    In particular, $K(X)\otimes_{K(Y)} \ol{K(Y)}$ is the direct product of $\mf{n}(X/Y)$ fields.
\end{lemma}

\begin{proof}
    The proof follows verbatim from \cite[Lemma~2.6]{birkar2025steindegreelogcalabiyau};
    see also \cite[Corollaire~(4.5.10)]{EGA-IV-2}.
\end{proof}


\begin{definition}\label{defn-rational-num-g-fib}
    Let $f\colon X\dashrightarrow Y$ be a dominant rational map of finite type
    between integral schemes of characteristic zero.  Let $K(X)$ and $K(Y)$ be the 
    function fields of $X$ and $Y$ respectively.  
    Let $U\subseteq X$ be an arbitrary open subset of $X$ on which 
    $f$ is well-defined.  Define 
    \[ \mf{n}(X/Y) \coloneqq \mf{n}(U/Y). \]
    Note that $\mf{n}(X/Y)$ does not depend on the choice of $U$, hence is well-defined by Lemma~\ref{num-g-fibre}.
\end{definition}


\begin{lemma}\label{stein-multiplicity}
    Let $f\colon X\dashrightarrow Y$ and $g\colon Y\dashrightarrow Z$
    be dominant rational maps of finite type between integral schemes of characteristic zero.  Then
    \[ \mf{n}(X/Z) \le \mf{n}(X/Y)\cdot \mf{n}(Y/Z). \]
    Moreover, if $\mf{n}(X/Y) = 1$, i.e. a general fibre of $f$ is geometrically irreducible, then $\mf{n}(X/Z) = \mf{n}(Y/Z)$.
\end{lemma}

\begin{proof}
    By Lemma~\ref{num-g-fibre}, $r \coloneqq \mf{n}(Y/Z)$ is equal to the cardinality 
    of $\spec K$, where $K \coloneqq K(Y)\otimes_{K(Z)}\ol{K(Z)}$.  Again by Lemma~\ref{num-g-fibre}, we can write
    $K = K_1\oplus \cdots \oplus K_r$, where every $K_i$ is a field.
    Let $U\subseteq X$ be an open subset on which $f$ and $g\circ f$ are well-defined.
    Since the absolute Galois group of $K(Z)$ acts transitively on the field extensions $K(Y)\to K_i$,
    the cardinality $\mf{n}(U\times_Y \spec K_i)$ is constant for $i = 1, \dots, r$.  Moreover, it is evident that
    \[ \mf{n}(U\times_Y \spec K_i) \le \mf{n}(U\times_Y \spec \ol{K_i}) = \mf{n}(U\times_Y\spec \ol{K(Y)}) 
    = \mf{n}(X/Y). \]
    Thus, we can conclude that
    \begin{align*}
        \mf{n}(X/Z) &\coloneqq \mf{n}(U/Z) = \mf{n}\big(U\times_Z \spec\ol{K(Z)}\big) \\
                    &= \mf{n}\bigr(U\times_{Y} \spec\big(K(Y)\otimes_{K(Z)} \ol{K(Z)} \big)\bigr) \\
                    &= \mf{n} \bigr( \bigcupdot_{1\le i\le r} U\times_Y\spec K_i \bigr) \\
                    &= \sum_{1\le i\le r} \mf{n}(U\times_Y\spec K_i) \\
                    &\le \mf{n}(X/Y)\cdot \mf{n}(Y/Z).
    \end{align*}
    Moreover, if $\mf{n}(X/Y) = 1$, then $\mf{n}(U\times_Y\spec K_i) = 1$ for every $1\le i\le r$.
    Hence we have the equality $\mf{n}(X/Z) = \mf{n}(Y/Z)$ as desired.
\end{proof}


\begin{lemma}\label{finite-fibre}
	Let $\pi\colon X\to Y$ be a proper, generically finite morphism between integral schemes
    of characteristic zero,	where $Y$ is normal.  
    Then there exists an open subset $U\subseteq Y$
	satisfying 
    \[ \codim (Y\setminus U, Y) \ge 2 \] 
    such that $\dim X_y = 0$ for every scheme-theoretic point $y\in U$.
	Moreover, 
	\[ \mf{n}(X_{\ol{y}}) \le \mf{n}(X/Y) \]
	for every scheme-theoretic point $y\in U$, where $X_{\bar{y}}$ is the geometric fibre of $X\to Y$ over $y$.
\end{lemma}

\begin{proof}
	The proof follows verbatim from \cite[Lemma~2.11]{birkar2025steindegreelogcalabiyau}.
\end{proof}


\subsection{Log resolutions}\label{log-resolution}

Let $X$ be an integral scheme and $D = \sum d_i D_i$ an $\RR$-divisor on $X$.
A \emph{log resolution} of $(X, D)$ is a proper birational morphism
$g\colon Y\to X$ such that $Y$ is regular, $\Ex(g)$ is a divisor and $\Ex(g)\cup g^{-1}(\Supp D)$
is a simple normal crossing divisor; cf. \cite[page~5]{km98}.
Log resolutions exist for excellent integral schemes of characteristic zero;
see \cite[Theorem~2.3.6, Lemma~4.2.4]{temkin-desing-chara-0}.
If $X$ is an integral scheme of finite type over a field of characteristic zero,
log resolutions exist by \cite[Theorem~10.45]{Kol_singularities_of_MMP}.

Let $X$ be an excellent, integral, normal scheme of characteristic zero 
and $D$ an $\RR$-divisor no $X$.  Let $\pi\colon Y\to X$ be a log resolution of $(X, D)$.
Let $Y\xrightarrow{\pi'} Y'\to X$ be the Stein factorisation of $\pi$; see \cite[Chap~III, \S 4.3]{EGA-III-1}.
Then $Y'\to X$ is a finite birational morphism of integral normal schemes.
By \cite[Lemme~(8.12.10.1)]{EGA-IV-3}, $Y'\to X$ is an isomorphism.
Thus, $\pi\colon Y\to X$ is a contraction.


\subsection{Divisors}\label{divisors-pre}

Let $X$ be a scheme.  A \emph{prime divisor} on $X$ is an integral closed subscheme of $X$
of codimension one.  By a \emph{divisor}, we mean a \emph{Weil divisor} on $X$, that is,
a finite $\ZZ$-linear combination of prime divisors on $X$, 
so a divisor is not necessarily Cartier.
By a \emph{$\QQ$-divisor} (respectively, an \emph{$\RR$-divisor}), we mean a finite linear combination
$\sum_i b_i B_i$, where every $B_i$ is a prime divisor on $X$ and $b_i\in \QQ$ 
(respectively, $b_i\in \RR$). 
Let $D$ be an $\RR$-divisor on $X$.
For any prime divisor $S$ on $X$, we denote by $\coeff_S D$ the coefficient of $S$ in $D$. 
Moreover, we denote by $\Supp D$ the support of $D$, which is a closed subset of $X$.
If there is no confusion in the context, we view $\Supp D$ as the reduced divisor whose 
set of irreducible components is equal to that of $D$.

A $\QQ$-divisor (respectively, an $\RR$-divisor) is called \emph{$\QQ$-Cartier}
(respectively, \emph{$\RR$-Cartier})
if it is a $\QQ$-linear (respectively, an $\RR$-linear) combination of Cartier divisors.
Let $B_1$ and $B_2$ be two $\RR$-divisors on $X$.  We say that $B_1\sim B_2$ (respectively, 
$B_1\sim_{\QQ} B_2$, respectively, $B_1\sim_{\RR} B_2$) if $B_1 - B_2$ is a $\ZZ$-linear 
(respectively, a $\QQ$-linear, respectively, an $\RR$-linear) combination of principal divisors;
see \cite[\S 2.3]{B-Fano}.

Let $f\colon X\to Z$ be a morphism between schemes, and let $L$ and $M$ be $\RR$-divisors
on $X$.  We say that \emph{$L\sim M$ over $Z$} 
(respectively, \emph{$L\sim_{\QQ} M$ over $Z$}, respectively, \emph{$L\sim_{\RR} M$ over $Z$})
if there is a Cartier (respectively, a $\QQ$-Cartier, respectively, an $\RR$-Cartier) divisor $N$ on $Z$ 
such that $L-M\sim f^*N$ (respectively, $L-M \sim_{\QQ} f^*N$, respectively, $L-M \sim_{\RR} f^*N$); cf. \cite[\S 2.3]{B-Fano}.

Let $f\colon X\to Z$ be a morphism of schemes, and let $D$ be a nonzero $\RR$-divisor on $X$.
We say that $D$ is \emph{vertical$/Z$} if $f(\Supp D)$ does not contain
any generic point of $Z$.
If $D$ does not have any vertical$/Z$ irreducible components,
we say that $D$ is \emph{horizontal$/Z$}.


\subsection{Flat pullback of divisors}\label{ss-flat-pullback}

Let $f\colon X\to Y$ be a flat morphism of schemes (which is not necessarily of finite type).
Let $D = \sum a_i D_i$ be an $\RR$-divisor on $Y$, where $a_i\in \RR$ and each $D_i$ is a prime divisor.
Denote by $f^* D_i$ the flat pullback of $D_i$ in the sense of \cite[\S 1.7]{Fulton}.
We call the $\RR$-divisor
\[ f^*D \coloneqq \sum a_i f^*D_i \]
the \emph{flat pullback of $D$ to $X$} or just the \emph{pullback of $D$ to $X$}.

Let $k$ be a field of characteristic zero, and let $X$ be a $k$-variety.
Let $B = \sum b_i B_i$ be an $\RR$-divisor on $X$, where each $B_i$ is prime divisor.
Denote by $B_{\bar k}$ the flat pullback of $B$ from $X$ to $X_{\bar k}$.
By \cite[Corollaire~(4.6.4)]{EGA-IV-2}, $B_{i, \bar k}\coloneqq B_i\times_{\spec k}\spec \bar k$ is a reduced scheme.
Thus, for every irreducible component $D$ of $B_{i, \bar k}$, we have
$\coeff_D B_{\bar k} = \coeff_{B_i} B$.


\begin{lemma}\label{descend-R-Cartier}
    Let $X$ be a variety over a field $k$ of characteristic zero,
    and let $K/k$ be a field extension.
    Set $X_K \coloneqq X\otimes_k K$.
    Let $D$ be an $\RR$-divisor on $X$, and let $D_K$
    be its flat pullback to $X_K$.
    Then $D$ is $\RR$-Cartier if and only if
    $D_K$ is $\RR$-Cartier.
\end{lemma}

\begin{proof}
    If $D$ is $\RR$-Cartier, then it is evident that $D_K$ is $\RR$-Cartier; 
    cf. \cite[Example~6.2.9]{Fulton}.
    Conversely, assume that
    \[ D_K = \sum_{i=1}^r a_i D_i, \]
    where $a_i\in \RR$, and each $D_i$ is a Cartier divisor on $X_K$.
    Then there exist an intermediate field extension $k\subset k'\subset K$,
    where $k'/k$ is finitely generated,
    and Cartier divisors $D_i'$ on $X'\coloneqq X\otimes_k k'$
    such that each $D_i$ is the pullback of $D_i'$ to $X_K$.

    By \cite[Th\'eor\`eme~(8.5.2), Proposition~(8.5.5)]{EGA-IV-3}, there exist an integral affine scheme $P$
    of finite type over $k$ whose function field is $k'$, and a $P$-scheme $X_P$ such that
    $X'\to \spec k'$ is the generic fibre of $X_P\to P$, where $X_P\coloneqq X\times_{\spec k} P$.
    Moreover, we can assume that there are Cartier divisors $C_i$ on $X_P$
    so that for each $i$, the restriction $C_i|_{X'}$ is equal to $D_i'$.
    \[\xymatrix{
    X\ar[d] & X_P\ar[d]\ar[l] & X'\ar[d]\ar[l] & X_K\ar[d]\ar[l] \\
    \spec k & P\ar[l] & \spec k'\ar[l] & \spec K\ar[l]
    }\]
    Let $D_P$ be the flat pullback of $D$ to $X_P$, and let
    \[ E_P\coloneqq D_P - \sum_{i=1}^r a_i C_i. \]
    Then $E_P$ is zero on the generic fibre $X'$.
    Thus, by \cite[Corollaire~(1.8.5)]{EGA-IV-1}, up to shrinking $P$, we can assume $E_P= 0 $.
    Hence $D_P$ is $\RR$-Cartier.
    
    Let $p\in P$ be a general closed point, and let $L$ be the residue field of $P$ at $p$, 
    which is a finite extension of $k$.
    Let $X_L \coloneqq X\otimes_k L$, and let $C_{i,L}$ be the restriction of $C_i$ to $X_L$.
    Denote by $D_L$ the flat pullback of $D$ to $X_L$, which is equal to the restriction of $D_P$ to $X_L$.  Then
    \[ D_L = \sum_{i=1}^r a_i C_{i, L}. \]
    By \cite[Claim~(2.40.1)]{Kol_singularities_of_MMP}, the pushdown of each $C_{i, L}$ to $X$
    is Cartier.  Moreover, since the pushdown of $D_L$ to $X$ is equal to $\deg(L/k)\cdot D$,
    we can conclude that $D$ is $\RR$-Cartier.    
\end{proof}


\subsection{Birational pullback of divisors}

Let $\phi\colon X\bir Y$ be a birational map between excellent integral schemes of characteristic zero, 
which may not be a birational contraction,
i.e. $\phi^{-1}$ can contract some divisors on $Y$.  Let $D$ be an $\RR$-Cartier divisor on $Y$.
Let $p\colon W\to X$ be a resolution so that the induced map $q\colon W\bir Y$ is a morphism.
Denote by $\phi^* D$ the $\RR$-divisor $p_* q^*D$.
It is easy to see that $\phi^*D$ does not depend on the choice of $W$.
We call $\phi^*D$ the \emph{birational pullback} of $D$.
Note that $\phi^*D$ may not be $\RR$-Cartier.



\subsection{Ample, semiample, and big divisors}
Let $X\to S$ be a proper morphism of schemes.
Two Cartier divisors $D_1, D_2$ on $X$ are called \emph{numerically equivalent over $S$},
denoted by $D_1\equiv D_2/S$, if for every 1-cycle $C\in \Z_1(X/S)$, 
the intersection numbers $D_1\cdot C$ and $D_2\cdot C$ are equal.
We denote by $\N^1(X/S)$ the Abelian group of all equivalent classes 
of numerically equivalent over $S$ Cartier divisors on $X$; see \cite[(4.2.3)]{Kol-rc}.
Moreover, for $k=\QQ, \RR$, we define
\[ \N^1(X/S)_k \coloneqq \N^1(X/S)\otimes_{\ZZ}k. \]
Let $f\colon X\to S$ be a morphism between schemes.
An $\RR$-divisor $A$ on $X$ is called \emph{ample over $S$} (or \emph{ample$/S$})
if it can be expressed as a finite sum $\sum a_i A_i$,
where $a_i\in \RR^{>0}$ and $A_i$ an ample$/S$ Cartier divisor on $X$.
For \emph{relatively semiample $\RR$-Cartier $\RR$-divisors}, see \cite[Definition~4.1]{lyu-takumi-2026relativeminimalmodelprogram}.
For \emph{relatively big $\RR$-Cartier $\RR$-divisors} on integral schemes, 
see \cite[Definition~5.8]{lyu-takumi-2026relativeminimalmodelprogram}.
For \emph{relatively big (Weil) $\RR$-divisors} on integral schemes, 
see \cite[Definition~5.20, Lemma~5.21]{lyu-takumi-2026relativeminimalmodelprogram}.


\begin{lemma}\label{fppf-ample}
    Let $S$ be a scheme, and let $f\colon X\to Y$ be a projective $S$-morphism between $S$-schemes.
    Let $g\colon S'\to S$ be a faithfully flat morphism between schemes.
    Denote by
    \[ X'\coloneqq X\times_S S'~~\text{and}~~
    Y'\coloneqq Y\times_S S'.\]
    Let $A$ be an $\RR$-Cartier $\RR$-divisor on $X$ and $A'$ its pullback to $X'$.
    Then $A$ is ample$/Y$ (respectively, very ample$/Y$) if and only if
    $A'$ is ample$/Y'$ (respectively, very ample$/Y'$).
\end{lemma}

\begin{proof}
    The equivalence for very ampleness follows from \cite[Corollaire~(2.7.2)]{EGA-IV-2}.
    If $A$ is an ample$/Y$ $\RR$-divisor on $X$, again by \cite[Corollaire~(2.7.2)]{EGA-IV-2},
    $A'$ is ample$/Y'$.  Conversely, assume that $A'$ is ample$/Y'$.
    In the $\RR$-linear space of $\RR$-divisors on $X$,
    choose a finite-dimensional $\QQ$-linear subspace $V$
    of $\QQ$-Cartier divisors on $X$ such that $A\in V\otimes_{\QQ}{\RR}$.  
    By \cite[Theorem~4.16]{lyu-takumi-2026relativeminimalmodelprogram},
    the relative ample cone $\Amp(X'/Y')$ is open.  
    Hence we can pick $\QQ$-Cartier $\QQ$-divisors $H_1,\dots,H_r$
    in $V$ that are sufficiently close to $A$ so that each pullback
    $H_i'$ of $H_i$ to $X'$ is ample$/Y'$, and that we can write
    \[ A = \sum_{i=1}^r  a_i H_i, \]
    where $a_i>0$, and $\sum_{i=1}^r a_i = 1$.
    For each $i$, take $m_i\in\NN$ such that $m_i H_i$
    is Cartier.  By \cite[Corollaire~(2.7.2)]{EGA-IV-2}, $m_i H_i$ is ample$/Y$.
    Thus, $A$ is ample$/Y$.
\end{proof}


\begin{lemma}\label{semi-ample-codim-1}
    Let $f\colon X\to Z$ be a contraction between normal irreducible schemes.
    Let $D$ be a semiample$/Z$ $\RR$-divisor so that 
    $D\sim_{\RR}0/V$ for a nonempty open subset $V\subset Z$.
    Then there is an open subset $U\subseteq Z$ with $\codim(Z\setminus U, Z) \ge 2$
    such that $D\sim_{\RR}0/U$.
\end{lemma}

\begin{proof}
    By \cite[Definition~4.1]{lyu-takumi-2026relativeminimalmodelprogram},
    we can write $D = \sum_{i=1}^r a_i D_i$, where $a_i\in \RR^{>0}$,
    and each $D_i$ is a semiample$/Z$ Cartier divisor on $X$.
    Let $f_i\colon X\to U_i$ be the contraction induced by $D_i$, i.e.
    there is an ample$/Z$ $\QQ$-divisor $H_i$ on $U_i$ so that $D_i \sim_{\QQ}f_i^*H_i/Z$.
    Let
    \[ U\coloneqq U_1\times_Z \cdots\times_Z U_r.  \]
    It is evident that $U\to Z$ is a contraction.
    Take the Stein factorisation of $X\to U$, and let
    $g\colon X\to Y$ be the part with geometrically connected fibres.
    Then the induced $Y\to Z$ is also a contraction.
    Denote by $A_i$ the pullback of $H_i$ to $Y$ via 
    the $i$-th projection $U\to U_i$ and the finite morphism $Y\to U$.
    Then $A_i$ is a semiample $\QQ$-Cartier $\QQ$-divisor on $Y$.
    Moreover, $A = \sum_{i=1}^r a_i A_i$ is ample$/Z$.  
    It is evident that
    \[ D\sim_{\RR} g^*A/Z. \]
    As $D\sim_{\RR} 0/V$, $Y\to Z$ is a birational contraction.
    Then $Y\to Z$ is an isomorphism over every codimension one point of $Z$
    \cite[Corollary~4.4.3]{Liu06}.
\end{proof}


\begin{lemma}\label{descent-semiample}
    Let $f\colon X\to Y$ be a morphism scheme over a field $k$, and let $K/k$ be a field extension.
    Denote by $f_K\colon X_K\to Y_K$ the base change of $f$, 
    where $X_K\coloneqq X\otimes_k K$ and $Y_K\coloneqq Y\otimes_k K$.
    Let $A$ be a Cartier divisor on $X$, and let $A_K$ be its pullback to $X_K$.
    Then $A$ is semiample over $Y$ if and only if $A_K$ is semiample over $Y_K$.
\end{lemma}

\begin{proof}
    It is evident that $A_K$ is semiample if $A$ is so.
    Now assume that $A_K$ is semiample over $Y_K$.
    Pick $m\in \NN$ so that $\mcL_K \coloneqq \mcO_{X_K}(m A_K)$ is globally generated over $Y_K$.
    Then $\mcL_K$ is the pullback of $\mcL\coloneqq \mcO_X(mA)$ to $X_K$.
    Since $\mcL_K$ is globally generated over $Y_K$, we have a canonical surjective homomorphism of sheaves
    \[ f_K^*\big( (f_K)_* \mcL_K\big) \to \mcL_K. \]
    By \cite[Lemme~(2.3.1)]{EGA-IV-2}, we have
    \[ f_K^*\big((f_K)_*\mcL_K\big) = f_K^* \big((f_*\mcL)\otimes_{\mcO_Y} \mcO_{Y_K} \big) = (f^*f_*\mcL)\otimes_{\mcO_X}\mcO_{X_K}. \]
    Then by \cite[Proposition~(2.2.7)]{EGA-IV-2}, we see that
    \[ f^*f_* \mcL\to \mcL \]
    is also surjective.  Thus, $\mcL$ is globally generated over $Y$, i.e. 
    $A$ is semiample over $Y$.
\end{proof}


\begin{lemma}\label{no-van-big-divisor}
	Let $\phi\colon X\dasharrow Y/Z$ be a birational contraction between irreducible 
	  normal schemes that are projective over an integral scheme $Z$ of characteristic zero, 
    and let $D$ be a big$/Z$ $\RR$-Cartier $\RR$-divisor
	on $X$.  Then $\phi_*D$ is also a big$/Z$ $\RR$-divisor on $Y$.
	In particular, if $D$ is an $\RR$-Cartier big$/Z$ divisor on $X$
    and if $\dim (X/Z)>0$, then $\phi_* D \not=0$.
\end{lemma}

\begin{proof}
    Let $p\colon W\to X, q\colon W\to Y$ be a common resolution of $\phi$.
    Then $p^*D$ is a big$/Z$ $\RR$-Cartier $\RR$-divisor on $W$ and $\phi_*D=q_*p^*D$.
    Thus, we can assume that $\phi$ is a morphism.
    Let $H$ be an ample$/Z$ divisor on $Y$.
    By \cite[Lemma~5.11]{lyu-takumi-2026relativeminimalmodelprogram},
    for a sufficiently small $\epsilon\in \RR^{>0}$,
    we can write $D - \epsilon \phi^*H\sim_{\RR}G/Z$ 
    for some effective $\RR$-divisor $G$ on $X$.
    Thus, $\phi_* D\sim_{\RR} \epsilon H + \phi_* G/Z$.
    By \cite[Definition~5.20]{lyu-takumi-2026relativeminimalmodelprogram},
    $\phi_* D$ is big$/Z$.
    Now, assume $D$ is an $\RR$-Cartier big$/Z$ divisor on $X$ and that $\dim(X/Z)>0$.
    If $\phi_*D=0$, then the zero divisor is big$/Z$.
    This contradicts $\dim(X/Z)>0$ by \cite[Definition~5.20]{lyu-takumi-2026relativeminimalmodelprogram}.
    Hence $\phi_*D\neq 0$.
\end{proof}


\subsection{Nef and pseudo-effective divisors}
Let $f\colon X\to S$ be a proper morphism between schemes.
Let $\Z_m(X/S)$ be the free Abelian group generated by irreducible and reduced $m$-dimensional 
subschemes $Y\subseteq X$ such that $f(Y)$ is a point.
The effective cycles form a monoid which is denoted by $\ZE_m(X/S)$; see \cite[(4.1.1)]{Kol-rc}.
An $\RR$-Cartier $\RR$-divisor $M$ on $X$ is called \emph{$f$-nef} (or \emph{nef over $S$}, 
or \emph{nef$/S$}) if $M\cdot C\ge 0$ for every $C\in \ZE_1(X/S)$; see \cite[(4.2.7)]{Kol-rc}.

Let $X\to Y$ be a proper morphism between schemes,
and let $L$ be an $\RR$-Cartier $\RR$-divisor on $X$.
Let $\eta\in Y$ be a generic point of $Y$.  Denote by $X_{\eta}$ fibre of $X\to Y$ over $\eta$,
and by $D_{\eta}$ the restriction of $D$ to $X_{\eta}$.
We say that $D$ is \emph{pseudo-effective$/Y$} if
for each generic point $\eta\in Y$, $[D_{\eta}]\in \N^1(X_{\eta})_{\RR}$
is the limit of classes of effective $\QQ$-Cartier $\QQ$-divisors on $X_{\eta}$;
cf. \cite[Definition~5.10]{lyu-takumi-2026relativeminimalmodelprogram}.
When $Y=\spec k$ for a field $k$, we just say that $D$ is \emph{pseudo-effective}.

\begin{lemma}\label{fppf-peff-K-k}
    Let $K/k$ be an extension of fields.
    Let $V$ be a projective $k$-scheme (respectively, a geometrically integral projective $k$-scheme), 
    and let $D$ be an $\RR$-Cartier $\RR$-divisor on $V$.
    Denote by $V_K\coloneqq V\otimes_k K$, and by $D_K$ the pullback of $D$ to $V_K$.
    Then $D$ is nef (respectively, pseudo-effective) if and only if $D_K$
    is nef (respectively, pseudo-effective).
\end{lemma}

\begin{proof}
    We first prove the result for nefness. 
    Let $H$ be an ample Cartier divisor on $V$, and let $H_K$ be its pullback to $V_K$.
    If $D$ is nef, then by \cite[Theorem~4.16]{lyu-takumi-2026relativeminimalmodelprogram}, 
    $D + \epsilon H$ is ample for every $\epsilon>0$. 
    By Lemma~\ref{fppf-ample}, $D_K + \epsilon H_K$ is ample for every $\epsilon>0$.
    Again by \cite[Theorem~4.16]{lyu-takumi-2026relativeminimalmodelprogram}, we see that $D_K$ is nef.
    Conversely, assume that $D_K$ is nef.  For any integral curve $C\subset V$, 
    let $C_K\coloneqq C\otimes_k K$, and let
    \[ [C_K] = \sum_i m_i [C_i] \]
    be the fundamental cycle of $C_K$, where $m_i\in \ZZ^{>0}$, and each $C_i$
    is an integral curve on $V_K$.  Then by \cite[VI.2.7, VI.2.14]{Kol-rc}, 
    \[ D\cdot C = D_K\cdot [C_K] = \sum_i m_i (D_K\cdot C_i) \ge 0. \]
    Thus, we can conclude that $D$ is nef.

    Now treat pseudo-effectivity.  Let $H$ be an ample Cartier divisor on $V$ and $H_K$
    the pullback to $V_K$, which is ample by Lemma~\ref{fppf-ample}.
    If $D$ is pseudo-effective, 
    by \cite[Lemma~5.14]{lyu-takumi-2026relativeminimalmodelprogram}, $D + \epsilon H$ is big for every $\epsilon>0$.
    Then by Lemma~\ref{fppf-ample} and \cite[Corollary~5.12]{lyu-takumi-2026relativeminimalmodelprogram}, 
    $D_K + \epsilon H_K$ is big for every $\epsilon>0$.
    Again by \cite[Corollary~5.12]{lyu-takumi-2026relativeminimalmodelprogram}, $D_K$ is pseudo-effective.
    Conversely, assume that $D_K$ is pseudo-effective.
    Then $D_K +\epsilon H_K$ is big for every $\epsilon>0$. 
    Pick a finite-dimensional $\QQ$-linear space $W$ of $\QQ$-Cartier $\QQ$-divisors
    so that $D+\epsilon H \in W\otimes_{\QQ}\RR$.
    By openness of the relative big cone \cite[Corollary~5.12]{lyu-takumi-2026relativeminimalmodelprogram},
    we can pick $F_1,\dots, F_r\in W$ whose pullbacks to $V_K$ are big so that we can write
    \[ D+\epsilon H = \sum_{i=1}^r a_i F_i, \]
    where $a_i\in \RR^{>0}$, and $\sum_{i=1}^r a_i = 1$.
    By \cite[Proposition~III.9.3]{Hart}, each $F_i$ is big.
    Then $D+\epsilon H$ is big for each $\epsilon>0$.
    As the relative pseudo-effective cone is closed 
    (cf. \cite[Definition~5.10]{lyu-takumi-2026relativeminimalmodelprogram}), $D$ is pseudo-effective.
\end{proof}


\begin{lemma}\label{fppf-peff}
    Let $S$ be a scheme, and let $f\colon X\to Y$
    be a projective $S$-morphism between $S$-schemes.  Let
    $g\colon S'\to S$ be a faithfully flat morphism of schemes.  Denote by
    \[ X'\coloneqq X\times_S S'~~\mbox{and}~~ Y'\coloneqq Y\times_S S'. \]
    Let $L$ be an
    $\RR$-Cartier $\RR$-divisor on $X$, and let $L'$ be its pullback to $X'$.
    Then $L$ is nef$/Y$ if and only if $L'$ is nef$/Y'$.
\end{lemma}

\begin{proof}
    For nefness, let $y'\in Y'$, and let $y$ be the image of $y'$ in $Y$.  Then 
    $X_{y'}' = X_y\otimes_{k(y)}k(y')$, and $L'|_{X'_{y'}}$
    is equal to the pullback of $L|_{X_y}$ to $X'_{y'}$.
    By Lemma~\ref{fppf-peff-K-k}, $L|_{X_y}$ is nef if and only if $L'|_{X'_{y'}}$ is nef.
    As $Y'\to Y$ is surjective, we can conclude that $L$ is nef$/Y$
    if and only if $L'$ is nef$/Y'$.
\end{proof}


\begin{lemma}[Negativity]\label{neg-lemma}
    Let $k$ be a field of characteristic zero.
    Let $\pi\colon X\to Y$ be a proper birational morphism between normal $k$-varieties.
    Let $B$ be an $\RR$-Cartier $\RR$-divisor on $X$ such that $-B$ is nef$/Y$.
    Then $B\ge 0$ if and only if $\pi_* B\ge 0$.
\end{lemma}

\begin{proof}
    Denote by $\bar{k}$ the algebraic closure of $k$.
    Let $\bar{X}\coloneqq X\otimes_k\bar{k}$, $\bar{Y}\coloneqq Y\otimes_k\bar{k}$,
    and $\bar{\pi}\colon \bar{X}\to \bar{Y}$ the base change of $\pi$.
    Let $\bar{B}$ be the flat pullback of $B$ to $\bar{X}$.
    By Lemma~\ref{fppf-peff}, $-B$ is nef$/Y$ if and only if $-\bar{B}$
    is nef$/\bar{Y}$.  It is evident that $B\ge 0$ if and only if $\bar{B}\ge 0$.
    Hence the result follows immediately from the negativity lemma over algebraically closed field
    of characteristic zero; see \cite[Lemma~3.39]{km98} and \cite[Lemma~3.3]{birkar12-special-LMMP}.
\end{proof}


\subsection{Volume of divisors}
(1)
Let $X$ be a projective scheme of dimension $d$ over a field $k$, which may not 
be algebraically closed, and let $D$ be a Cartier divisor on $X$.
The \emph{volume of $D$ over $k$} is defined to be
\begin{equation}\label{vol-limit}
    \vol_{/k}(D) \coloneqq \lim_{m\to \infty} \dfrac{\dim_k \hg^0(X, \mcO_X(mD))}{m^{d}/d!}.
\end{equation}
When $X$ is integral, the limit \eqref{vol-limit} is well-defined;
see \cite{vol-non-closed-fields} and
\cite[\S5.2]{lyu-takumi-2026relativeminimalmodelprogram}.
When $D$ is nef over $\spec k$, but $X$ may not be integral, 
by \cite[Theorem~VI.2.15]{Kol-rc}, the limit \eqref{vol-limit} exists and
\[ \vol_{/k}(D) = (D)^d_k, \]
where $(\cdot)^d_k$ denotes the intersection number over $k$.

\medskip
(2)
Let $X$ be a projective integral scheme of dimension $d$ over a field $k$.
Let $L$ be a $\QQ$-Cartier $\QQ$-divisor on $X$.
Pick $m\in \NN$ such that $m L$ is a Cartier divisor.  Then we set 
\begin{equation}\label{vol-Q-divisor}
    \vol_{/k}(L) \coloneqq \frac{\vol_{/k}(mL)}{m^{d}}.
\end{equation}
We claim that $\vol_{/k}(L)$ is well-defined.
Indeed, let $n\in \NN$ be another integral such that $nL$ is Cartier;
then by \cite[Proposition~5.7]{lyu-takumi-2026relativeminimalmodelprogram}, we have
\[ \dfrac{\vol_{/k}(mL)}{m^{d}} = \dfrac{ n^{d}\vol_{/k}(mL) }{(mn)^{d}}
= \dfrac{\vol_{/k}(mnL)}{(mn)^{d}} 
= \dfrac{ m^{d}\vol_{/k}(nL) }{(mn)^{d}}
= \dfrac{\vol_{/k}(nL)}{n^{d}}, \]
which proves the claim as desired.  Hence $\vol_{/k}(L)$ is well-defined.

\medskip
(3)
Let $X$ be a projective scheme of dimension $d$ over a field $k$,
and let $L$ be a \emph{nef} $\QQ$-Cartier $\QQ$-divisor on $X$.  
We also define $\vol_{/k}(L)$ as in \eqref{vol-Q-divisor}.
In particular, if $mL$ is Cartier for some $m\in \NN$, then by 
\cite[VI.2.7.4]{Kol-rc}, we have
\[ \frac{\vol_{/k}(mL)}{m^{d}} = 
\frac{(mL)^d_{k}}{m^{d}} = (L)^d_k. \]
Hence $\vol_{/k}(L)$ is also well-defined in this case, and $\vol_{/k}(L) = (L)^d_k$.

\medskip
(4)
We show that the volume of a nef divisor is constant for a flat family.

\begin{lemma}\label{volume-fibre-constant}
    Let $X\to Y$ be a projective flat morphism of schemes, where $Y$ is integral.
    Let $L$ be a nef$/Y$ $\QQ$-divisor on $X$.  
    Assume $mL$ is Cartier for some $m\in \NN$.
    Then for any scheme-theoretic points $y,y'\in Y$, we have
    \[ \vol_{/k(y)} (\mcO_X(mL)|_{X_y}) = \vol_{/k(y')}(\mcO_X(mL)|_{X_{y'}}), \]
    where $X_y, X_{y'}$ are the fibres of $X\to Y$ over $y,y'$ respectively.
    In particular, if the pullbacks $L_y, L_{y'}$ of $L$ to $X_y, X_{y'}$ are well-defined, we get
    \[ \vol_{/k(y)}(L_{y}) = \vol_{/k(y')}(L_{y'}). \]
\end{lemma}

\begin{proof}
    This follows from (3) as above and \cite[Proposition~VI.2.9]{Kol-rc}.
\end{proof}


\subsection{Pairs and singularities}
We will use standard notions from \cite{km98, BCHM, Kol_singularities_of_MMP}.
A \emph{sub-pair} $(X, B)$ consists of a normal, excellent, integral
scheme of characteristic zero $X$ and an
$\RR$-divisor $B$ such that $K_X + B$ is $\RR$-Cartier.
We call a sub-pair $(X, B)$ a \emph{pair} if $B\ge 0$; in this case, 
we call $B$ the \emph{boundary divisor} (or \emph{boundary}) of $(X,B)$.

Let $W\to X$ be a log resolution of a sub-pair $(X,B)$, and let 
$K_W + B_W$ be the pullback of $K_X + B$.
Denote by $\coeff_D B_W$ the coefficient of $B_W$ at a prime divisor $D$ on $W$.
The \emph{log discrepancy} of $D$ with respect to $(X, B)$
is defined as
\[ a(D, X, B) \coloneqq 1-\coeff_D B_W. \]
We say that $(X, B)$ is \emph{sub-lc} (respectively, \emph{sub-klt}, respectively, \emph{sub-$\epsilon$-lc})
if $a(D, X, B)$ is $\ge 0$ (respectively, $>0$, respectively, $\ge \epsilon$) 
for every divisor $D$ on an arbitrary log resolution $W\to X$ of $(X, B)$.
If $(X, B)$ is a pair, we say that $(X, B)$ is \emph{lc} 
(respectively, \emph{klt}, respectively, \emph{$\epsilon$-lc}).

\begin{lemma}[cf. \protect{\cite[Lemma~2.3]{B-BAB}}]\label{link-log-discrepancies}
    Let $X$ be a normal, excellent, integral scheme of characteristic zero.
    Let $(X, B)$ and $(X, C)$ be sub-pairs, and let 
    \[ \Delta = t B + (1-t)C \]
    for some real number $t\in [0,1]$.  Then
    \[ a(D, X, \Delta) = t a(D, X, B) + (1-t) a(D, X, C) \]
    for any prime divisor over $X$.  In particular,
    if $(X, B)$ is $\epsilon$-lc, and if $(X, C)$ is $\delta$-lc,
    then $(X, \Delta)$ is sub-$(t\epsilon + (1-t)\delta)$-lc.
\end{lemma}

\begin{proof}
    For any prime divisor $D$ over $X$, by \S\ref{log-resolution}, we can take a 
    common log resolution $W\to X$ of $(X, B)$ and $(X, C)$ so that $D$
    is a prime divisor on $W$.  Then the proof follows immediately from the computation of $a(D, X, \Delta)$
    by a linear combination of the pullbacks of $K_X+B$ and $K_X+C$ to $W$.
\end{proof}


\begin{lemma}\label{mld-attainment}
    Let $k$ be a field of characteristic zero, and let $(X,B)$ be an
    lc pair over $k$.  Then for every closed point $x\in X$, 
    there exists a prime divisor $E$ over $X$ such that
    \[ \centre_X E = x~~\mbox{and}~~\mld_x(X,B) = a(E,X,B). \]
\end{lemma}

\begin{proof}
    Let $f\colon Y\to X$ be a log resolution of $(X, B)$ such that $f^{-1}(x)$is an snc divisor on $Y$.
    Write
    \[ K_Y + B_Y = f^*(K_X+B). \]
    Let $E_1, \dots, E_r$ be irreducible components of $f^{-1}(x)$.  Set
    \[ \lambda\coloneqq \min_{1\le i\le r} a(E_i, X, B). \]
    Let $G$ be the reduction of $f^{-1}(x)$.
    It is evident that $(Y, B_Y + \lambda G)$ is sub-lc.
    Let $F$ be an arbitrary prime divisor over $X$ with $\centre_X F = x$.
    Then $\centre_Y F$ is contained in $G$.  As $G$ is a Cartier divisor,
    by \cite[Lemma~2.5]{Kol_singularities_of_MMP}, we have
    \[ a(F, X, B) = a(F, Y, B_Y) = a(F, Y, B_Y + \lambda G) + \lambda \coeff_F (f'^*G) \ge \lambda, \]
    where $f'\colon W\to Y$ is a higher resolution such that $F$ is a divisor on $W$.
    Hence $\mld_x(X, B)\ge \lambda$.
    Thus, any $E_i$ attaining the minimum $\lambda$ satisfies that 
    $a(E_i, X, B) = \mld_x(X, B)$.
\end{proof}


\subsection{$\text{B}$-divisors}
We take the conventions of $\bb$-divisors and generalised pairs from
\cite[\S 2.3]{Corti-3-folds}, \cite{B-Zhang} and \cite[\S2.7, \S2.13]{B-Fano}.

Let $X$ be an excellent, normal, irreducible scheme of characteristic zero.
A \emph{$\bb$-$\RR$-Cartier $\bb$-divisor over $X$} is the choice of a projective birational morphism
$Y\to X$ from a normal irreducible scheme 
and an $\RR$-Cartier divisor $M$ on $Y$ up to the following equivalence:
another projective birational morphism $Y'\to X$ from
a normal, irreducible scheme and an $\RR$-Cartier divisor $M'$
defines the same $\bb$-$\RR$-Cartier $\bb$-divisor if there is a common resolution $W\to Y$
and $W\to Y'$ on which the pullbacks of $M$ and $M'$ coincide.
Given an equivalence class of some $Y\to X$ and $M$ that represents a $\bb$-$\RR$-Cartier $\bb$-divisor over $X$,
we denote by $\mb{M}$ the equivalence class. 
We also say that $\mb{M}$ is a $\bb$-$\RR$-Cartier $\bb$-divisor over $X$; cf. \cite[2.3.8]{Corti-3-folds}.

Let $X$ be an excellent, normal, irreducible scheme, 
and let $\mb{M}$ be a $\bb$-$\RR$-Cartier $\bb$-divisor over $X$.
If $\mb{M}$ is represented by $Y\to X$ and an $\RR$-Cartier divisor $M$ on $Y$,
we say that \emph{$\mb{M}$ descends to $Y$}.  Moreover,
let $X'\to X$ be an arbitrary projective birational morphism from a normal irreducible scheme $X'$.
Let $\pi\colon Y\to X'$ be a sufficiently high resolution so that $\mb{M}$ descends to $Y$, that is,
$\mb{M}$ is represented by $Y\to X$ and an $\RR$-Cartier divisor $M$ on $Y$.
The \emph{trace} of $\mb{M}$ on $X'$ is the $\RR$-divisor (cf. \cite[2.3.10]{Corti-3-folds})
\[ \mb{M}_{X'} \coloneqq \tr_{X'} \mb{M} \coloneqq \pi_* M.  \]
Note that $\mb{M}_{X'}$ may not be $\RR$-Cartier.

On the other hand, let $X\dashrightarrow X''$ be an arbitrary birational map,
where $X''$ is an excellent, normal, irreducible scheme.
Let $Y\to X''$ and $Y\to X$ be a sufficiently high resolution of $X\dashrightarrow X''$
such that the $\bb$-$\RR$-Cartier $\bb$-divisor $\mb{M}$ over $X$ is represented by an $\RR$-Cartier
divisor $M$ on $Y$.  Then the equivalence class of $Y\to X''$ and $M$
also defines a $\bb$-$\RR$-Cartier $\bb$-divisor over $X''$,
which we also denote by $\mb{M}$.  
This should not lead to confusion in the context as a $\bb$-$\RR$-Cartier $\bb$-divisor
over $X$ can be viewed as an $\RR$-valued function on the set of all geometric valuations
of the function field $K(X)$, having finite support on some (hence any) birational model of $X$;
cf. \cite{Sh-3-fold-log-models} and \cite[\S 1.2]{Ambro04-Sh-boundary}.


\subsection{Generalised pairs}

A \emph{generalised sub-pair} (or a \emph{g-sub-pair} for short) 
$(X/S, B + \mb{M})$ consists of an excellent, normal, irreducible scheme $X$ 
equipped with a projective morphism $X\to S$
to an integral scheme $S$, an $\RR$-divisor $B$ on $X$, and a $\bb$-$\RR$-Cartier $\bb$-divisor $\mb{M}$
over $X$ such that $\mb{M}$ descends to a nef$/S$ $\RR$-divisor on some normal birational model of $X$
and that $K_X + B + \mb{M}_X$ is $\RR$-Cartier.  
When $B\ge 0$, we call $(X/S, B + \mb{M})$ a \emph{generalised pair} (or a \emph{g-pair} for simplicity).
When $S$ is not relevant, we usually drop it and do not mention it.

Let $(X/S, B + \mb{M})$ be a generalised sub-pair.  Let $\phi\colon Y\to X$ be a sufficiently
high log resolution of $(X, B + \mb{M})$ so that $\mb{M}$ descends to $Y$.  We can write
\[ K_Y + B_Y + \mb{M}_Y = \phi^*(K_X + B + \mb{M}_X) \]
for some uniquely determined $\RR$-divisor $B_Y$ on $Y$.  For a prime divisor $D$ on $Y$,
the \emph{generalised log discrepancy} $a(D, X, B + \mb{M})$ is defined to be
\[ a(D, X, B + \mb{M}) \coloneqq 1 - \coeff_D B_Y. \]
We say that $(X/S, B + \mb{M})$ is \emph{generalised sub-lc} 
(respectively, \emph{generalised sub-klt}, respectively,
\emph{generalised sub-$\epsilon$-lc}) if for each prime divisor $D$ over $X$, the 
generalised log discrepancy $a(D, X, B + \mb{M})$ is $\ge 0$ (respectively, $>0$, respectively, $\ge \epsilon$).
For simplicity, we also say that $(X, B + \mb{M})$ is \emph{g-sub-lc}
(respectively, \emph{g-sub-klt}, respectively, \emph{g-sub-$\epsilon$-lc}).
When $B\ge 0$, we say that g-pair is \emph{g-lc}
(respectively, \emph{g-klt}, respectively, \emph{g-$\epsilon$-lc}).
Note that $S$ does not play any role in the definition of singularities.


\begin{lemma}\label{CY-log-discre-no-change}
    Let $k$ be a field of characteristic zero.
	Let $(X, B + \mb{M})\to Z$ be a generalised log Calabi-Yau fibration over $k$, 
    and let $\phi\colon X\dashrightarrow Y$
	be a birational contraction$/Z$.  Denote by $B_Y$ the pushdown of $B$ to $Y$.
	Then for any divisor $E$ over $Y$, we have
	\[ a(E, X, B + \mb{M}) = a(E, Y, B_Y + \mb{M}). \]
	In particular, for any real number $\epsilon\ge 0$, if $(X, B + \mb{M})$ is g-$\epsilon$-lc,
	then $(Y, B_Y + \mb{M})$ is also g-$\epsilon$-lc.
\end{lemma}

\begin{proof}
	The proof if identical to \cite[Lemma~2.17]{birkar2025steindegreelogcalabiyau}.
\end{proof}


\begin{lemma}[cf. \protect{\cite[Lemma~3.38]{km98}}]\label{discrep-mmp}
    Let $k$ be a field of characteristic zero.
    Consider a commutative diagram
    \[\xymatrix{
    X\ar[rd]_f\ar@{-->}[rr]^-{\phi} & & X'\ar[ld]^{f'} \\
     & Y &
    }\]
    where $X, X', Y$ are normal $k$-varieties, and $f, f'$ are proper and birational.
    Let $B, B'$ be $\RR$-divisors on $X, X'$ respectively,
    and let $\mb{M}$ be a $\bb$-$\RR$-Cartier $\bb$-divisor over $Y$ such that
    $(X, B + \mb{M})$ and $(X', B'+\mb{M})$ are g-pairs.  Assume that
    \begin{itemize}
        \item $f_*B = f'_*B'$,
        \item $-(K_X+B+\mb{M}_X)$ is nef$/Y$, and
        \item $K_{X'} + B' + \mb{M}_{X'}$ is nef$/Y$.
    \end{itemize}
    Then for an arbitrary exceptional prime divisor $E$ over $Y$, we have
    \[ a(E, X, B+\mb{M})\le a(E, X', B'+\mb{M}). \]
    Strict inequality holds if either
    \begin{itemize}
        \item [\emph{(i)}] $-(K_X+B+\mb{M}_X)$ is ample$/Y$, and $f$ is not an isomorphism over the generic point of $\centre_YE$, or
        \item [\emph{(ii)}] $K_{X'}+B'+\mb{M}_{X'}$ is ample$/Y$, and $f'$ is not an isomorphism over the generic point of $\centre_Y E$.
    \end{itemize}
\end{lemma}

\begin{proof}
    Let $W$ be a normal $k$-variety admitting birational projective morphisms
    $p\colon W\to X$ and $q\colon W\to X'$ such that $\centre_W E$ is a prime divisor on $W$,
    $\mb{M}$ descends to $W$, and $f\circ p = f'\circ q$.  Write
    \begin{align*}
        K_W + \sum_i (1-a(E_i, X, B+\mb{M}))E_i + \mb{M}_W &= p^*(K_X + B + \mb{M}_X),~\mbox{and} \\
        K_W + \sum_i (1-a(E_i, X', B'+\mb{M}))E_i + \mb{M}_W &= q^*(K_{X'} + B' + \mb{M}_{X'}).
    \end{align*}
    Then by assumption,
    \[ H\coloneqq \sum_i \big( a(E_i, X, B+\mb{M}) - a(E_i, X', B'+\mb{M}) \big)E_i \]
    is nef$/Y$.  As $f_*B = f'_*B'$, $H$ is exceptional over $Y$.
    By Lemma~\ref{neg-lemma},
    we see that $H\le 0$, hence the inequality of log discrepancies holds.
    If $H$ is not numerically trivial over the generic point $\eta$ of $\centre_Y E$,
    then $H$ is nonzero over $\eta$.  
    Hence if either (i) or (ii) is satisfied, then the strict inequality holds.
\end{proof}


\subsection{Flat pullback of \texorpdfstring{$\bb$-$\RR$-Cartier $\bb$-divisors}{TEXT}}\label{ss-pullback-R-Cartier-pairs}
Let $S'\to S$ be a faithfully flat morphism of schemes over a field of characteristic zero.
Let $X$ be a normal scheme over $S$.

\medskip
(1)
Let $Y\to X$ be a projective birational morphism from a normal scheme, and
let $M_Y$ be an $\RR$-Cartier $\RR$-divisor.  Then $M_Y$ induces a $\bb$-$\RR$-Cartier $\bb$-divisor $\mb{M}$ over $X$.
To this end, let $V\to X$ be an arbitrary projective birational morphism from a normal scheme $V$,
and pick a common resolution $W\to Y, W\to V$ over $X$.
As $M_Y$ is $\RR$-Cartier, we take its pullback to $W$, which is denoted by $\mb{M}_W$.
Then we set the trace $\mb{M}_V$ of $\mb{M}$ on $V$ as the pushdown of $\mb{M}_W$ to $V$.
It is evident that $\mb{M}_V$ does not depend on the choice of $W$.

\medskip
(2)
Assume that the fibre product $X'\coloneqq X\times_S S'$ is also a normal scheme.
Let $\mb{M}$ be an $\bb$-$\RR$-Cartier $\bb$-divisor over $X$.
We show that $\mb{M}$ induces a $\bb$-$\RR$-Cartier $\bb$-divisor over $X'$ canonically.
Let $Y\to X$ be a resolution so that $\mb{M}$ descends to an $\RR$-Cartier $\RR$-divisor $\mb{M}_Y$ on $Y$.
Let $Y'$ be the normalisation of the main component of $X'\times_X Y$.
The flat pullback $M_{Y'}'$ of $\mb{M}_Y$ to $Y'$ is an $\RR$-Cartier $\RR$-divisor
on $Y'$.  Then by (1), there exists a canonically induced 
$\bb$-$\RR$-Cartier $\bb$-divisor $\mb{M}'$ over $X'$.

\medskip
(3)
Keep the notation and assumptions in (2).
Assume further that $\mb{M}$ is $\bb$-nef over $S$;
in particular, we can assume $\mb{M}_Y$ is nef$/S$.
Then by Lemma~\ref{fppf-peff}, $\mb{M}'_{Y'}$
is also nef$/S'$.  Hence $\mb{M}'$ is $\bb$-nef over $S'$.

\medskip
(4)
Let $K/k$ be a field extension between fields of characteristic zero.  
Assume $S$ is a $k$-variety and $S' = S\otimes_k K$.
Assume that $(X, B + \mb{M})$ is a g-sub-pair over $S$.  Let
\[ X' \coloneqq X\otimes_k K \cong X\otimes_S S'. \]
Let $B'$ be the flat pullback of $B$ to $X'$.
By (2) and (3) as above, there is a canonically induced $\bb$-$\RR$-Cartier $\bb$-divisor
$\mb{M}'$ over $X'$, which is $\bb$-nef over $S'$.
Since $\mb{M}'_{X'}$ is the flat pullback of $\mb{M}_X$ to $X'$,
we see that $K_{X'} + B' + \mb{M}'_{X'}$ is also $\RR$-Cartier.
Thus, $(X', B' + \mb{M}')$ is a g-sub-pair.

\medskip
(5)
Keep the notation and assumptions in (4).
Let $\epsilon\in \RR^{\ge 0}$.
Then $(X, B + \mb{M})$ is g-sub-$\epsilon$-lc if and only if
$(X', B' + \mb{M}')$ is g-sub-$\epsilon$-lc.
To this end, let $W\to X$ be a log resolution of $(X, B)$ such that 
$\mb{M}$ descends to a nef$/S$ $\RR$-Cartier $\RR$-divisor $\mb{M}_W$ on $W$.
Let $W'\coloneqq W\otimes_k K$, which is also a log resolution of $(X', B')$.
By construction, $\mb{M}'$ descends to the nef$/S'$ $\RR$-Cartier $\RR$-divisor $\mb{M}'_{W'}$.  
Let $K_W+B_W + \mb{M}_W$ be the pullback of $K_X+B+\mb{M}_X$ to $W$, and let
$K_{W'} + B'_{W'} + \mb{M}'_{W'}$ be the pullback of $K_{X'} + B' + \mb{M}'_{X'}$ to $W'$.
Then the flat pullback of $K_W+B_W+\mb{M}_W$ to $W'$ is equal to
$K_{W'} + B'_{W'} + \mb{M}'_{W'}$.
It is evident that the set of coefficients of $B_W$ is equal to that of $B'_{W'}$.
Thus, we get the result as desired.


\subsection{Descent of $\bb$-nef divisors}\label{ss-descent-b-nef}

Let $k$ be a field of characteristic zero and 
$\bar{k}$ the algebraic closure of $k$.
Let $X$ be a normal $k$-variety and $\bar{X} \coloneqq X\otimes_k \bar{k}$.
Let $\mb{M}$ be a $\bb$-divisor over $X$.

\begin{lemma}\label{dom-by-k-varieties}
    Let $\bar{\pi}\colon \bar{Y}\to \bar{X}$ be an arbitrary birational projective morphism
    from a normal $\bar{k}$-variety $\bar{Y}$.
    Then there exists a normal $k$-variety $P$ admitting a birational projective morphism $P\to X$
    of $k$-varieties
    such that the induced $\bar{P}\to \bar{X}$ factors through $\bar{Y}\to \bar{X}$.
\end{lemma}

\begin{proof}
    By \cite[Theorem~13.102]{GW-agbook}, $\bar{Y}\to \bar{X}$ is the blow-up
    of a closed subscheme $\bar{Z}\subset \bar{X}$.
    Denote by $\bar{\mc I}$ the ideal sheaf on $\bar{X}$ defining $\bar{Z}$.
    Let $k'/k$ be a finite Galois field extension such that $\bar{\mc I} = \mc{I}'\otimes_{k'}\bar{k}$
    for an ideal sheaf $\mc{I}'$ on $X'\coloneqq X\otimes_k k'$.
    Let $Y'$ be the blow-up of $X'$ along $\mc{I}'$.
    Then by \cite[Proposition~13.91]{GW-agbook}, we have $\bar{Y} = Y'\otimes_{k'}\bar{k}$.
    Denote by $G\coloneqq \Gal (k'/k)$ the Galois group of $k'/k$.  
    Then each $g\in G$ induces an automorphism of $X'$;
    for simplicity, we also denote it by $g\colon X'\to X'$.
    Let
    \[ \mc{J}' \coloneqq \prod_{g\in G} g^*\mc{I}', \]
    which is a $G$-invariant ideal sheaf on $X'$.  
    By Galois descent (see \cite[Theorem~14.86]{GW-agbook}),
    $\mc{J}'$ descends to an ideal sheaf $\mc{J}\subset \mcO_X$.
    Denote by $W'$ (respectively, by $W$) the blow-up of $X'$ along $\mc{J}'$
    (respectively, of $X$ along $\mc{J}$).
    Again by \cite[Proposition~13.91]{GW-agbook}, $W'\cong W\times_X X'$.
    By \cite[(5.1.2)~(v), page~35]{Raynaud-Gruson-plat-1971},
    there is an induced birational projective morphism $W'\to Y'$.
    Let $P$ be the normalisation of $W$, and let $\bar{P}\coloneqq P\otimes_k \bar{k}$.
    Then $\bar{P}$ is a normal $\bar{k}$-variety such that the birational projective morphism
    $\bar{P}\to \bar{X}$ factors through $\bar{Y}\to \bar{X}$.
\end{proof}

\medskip

(1)
We show that there is an induced $\bb$-divisor $\ol{\mb{M}}$ over $\bar{X}$.

\medskip

(1.1)
Let $\bar{\pi}\colon \bar{Y}\to \bar{X}$ be an arbitrary birational projective morphism
from a normal $\bar{k}$-variety $\bar{Y}$.
By Lemma~\ref{dom-by-k-varieties}, there exist a normal $k$-variety $P$
and a birational projective morphism $\phi\colon P\to X$ of $k$-varieties
so that the induced $\bar{\phi}\colon \bar{P}\to \bar{X}$ factors through $\bar{Y}\to \bar{X}$, i.e.
there exists a morphism $\bar{\psi}\colon \bar{P}\to \bar{Y}$ satisfying
$\bar{\phi} = \bar{\pi}\circ \bar{\psi}$.
Denote by $(\mb{M}_P)_{\bar{k}}$ the flat pullback of $\mb{M}_P$ to $\bar{P}$.
Then we set
\begin{equation}\label{induced-b-defn}
    \ol{\mb M}_{\bar{Y}}\coloneqq \bar{\psi}_* ((\mb{M}_P)_{\bar{k}}).
\end{equation}
We claim that $\ol{\mb M}_{\bar{Y}}$ does not depend on the choice of $P\to X$.
Let $\theta\colon Q\to X$ be another morphism of $k$-varieties such that
$\bar{\theta}\colon \bar{Q}\to \bar{X}$ factors through a morphism $\bar{\mu}\colon \bar{Q}\to \bar{Y}$.
Let $\bar{p}\colon \bar{V}\to \bar{P}, \bar{q}\colon \bar{V}\to \bar{Q}$
be a common resolution such that $\bar{\psi}\circ\bar{p} = \bar{\mu}\circ \bar{q}$.
Again by Lemma~\ref{dom-by-k-varieties}, there exists a morphism of 
$k$-varieties $W\to X$ so that the induced $\bar{W}\to \bar{X}$
factors through the morphism $\bar{\psi}\circ\bar{p} = \bar{\mu}\circ \bar{q}$
from $\bar{V}\to \bar{Y}$.
Replacing $W$ by higher resolutions, we can assume that $W\to X$ factors through
$P\to X$ and $Q\to X$.  Then $\mb{M}_P, \mb{M}_Q$ are pushdowns of $\mb{M}_W$
to $P, Q$, respectively.  Taking the flat pullbacks, we see that 
$(\mb{M}_P)_{\bar{k}}, (\mb{M}_Q)_{\bar{k}}$ are the pushdowns 
of $(\mb{M}_W)_{\bar{k}}$ to $\bar{P}, \bar{Q}$, respectively.
Hence the $\ol{\mb{M}}_{\bar{Y}}$ in \eqref{induced-b-defn} is the same as
the pushdown of $(\mb{M}_W)_{\bar{k}}$ to $\bar{Y}$, 
which is also equal to the pushdown of $(\mb{M}_Q)_{\bar{k}}$ to $\bar{Y}$.
Thus, $\ol{\mb{M}}_{\bar{Y}}$ does not depend on the choice of $P\to X$.

\medskip

(1.2)
Let $\bar{U}\to \bar{X}$ be another birational projective morphism 
from a normal $\bar{k}$-variety $\bar{U}$ such that $\bar{U}\to \bar{X}$
factors through a morphism $\bar{\delta}\colon \bar{U}\to \bar{Y}$.
By taking a common higher model of $\bar{Y}, \bar{U}$, and applying Lemma~\ref{dom-by-k-varieties} again,
we can assume that $\bar{\phi}\colon \bar{P}\to \bar{X}$ also factors through
a morphism $\bar{\varphi}\colon \bar{P}\to \bar{U}$.
Then by the construction in Step~1, 
\[ \ol{\mb M}_{\bar{Y}}\coloneqq \bar{\psi}_* ((\mb{M}_P)_{\bar{k}})~~\mbox{and}
~~\ol{\mb M}_{\bar{U}}\coloneqq \bar{\varphi}_* ((\mb{M}_P)_{\bar{k}}). \]
Hence we can conclude that 
\[ \ol{\mb M}_{\bar{Y}} = \bar{\delta}_*(\ol{\mb M}_{\bar{U}}). \]
Therefore, we get a well-defined $\bb$-divisor $\ol{\mb M}$ over $\bar{X}$.


\medskip

(2)
We show that $\ol{\mb M}$ is $\bb$-$\RR$-Cartier if and only if
$\mb M$ is $\bb$-$\RR$-Cartier.

\medskip

(2.1)
First, assume $\ol{\mb M}$ is $\bb$-$\RR$-Cartier, 
and $\ol{\mb M}$ descends to a projective birational model $\bar{Y}\to \bar{X}$,
where $\bar{Y}$ is a normal $\bar{k}$-variety.
By Lemma~\ref{dom-by-k-varieties}, there exists a normal $k$-variety $P$
admitting a birational projective morphism $P\to X$
such that $\bar{P}\to \bar{X}$ factors through $\bar{Y}\to \bar{X}$,
where $\bar{P}\coloneqq P\otimes_k \bar{k}$.
Let $\phi\colon Q\to P$ be an arbitrary birational projective morphism 
from a normal $k$-variety and $\bar{Q} \coloneqq Q\otimes_k \bar{k}$.
Let $\bar{\phi}\colon \bar{Q}\to \bar{P}$ be the induced morphism.
As $\ol{\mb M}$ descends to $\bar{Y}$, 
$\ol{\mb M}_{\bar{P}}$ is $\RR$-Cartier, and
$\ol{\mb M}_{\bar{Q}} = \bar{\phi}^* \ol{\mb M}_{\bar{P}}$.
By construction, $\ol{\mb M}_{\bar{Q}}, \ol{\mb M}_{\bar{P}}$
are the flat pullbacks of $\mb{M}_Q, \mb{M}_P$ respectively.
Then by Lemma~\ref{descend-R-Cartier}, $\mb{M}_Q, \mb{M}_P$
are also $\RR$-Cartier.
Thus, we can conclude that $\mb{M}_Q = \phi^* \mb{M}_P$,
i.e. $\mb{M}$ descends to $P$.


\medskip

(2.2)
Next, assume $\mb{M}$ is a $\bb$-$\RR$-Cartier $\bb$-divisor.
Let $W\to X$ be a birational projective morphism from a normal $k$-variety $W$
such that $\mb{M}$ descends to $W$.
We show that $\ol{\mb{M}}$ is also $\bb$-$\RR$-Cartier,
and that $\ol{\mb{M}}$ descends to $\bar{W}\coloneqq W\otimes_k \bar{k}$.
Let $\bar{\phi}\colon \bar{Y}\to \bar{W}$ be an arbitrary birational projective morphism
from a normal $\bar{k}$-variety.  
By Lemma~\ref{dom-by-k-varieties}, there exists a normal $k$-variety $P$ 
admitting a birational projective morphism $\psi\colon P\to W$
such that the induced $\bar{\psi}\colon \bar{P}\to \bar{W}$ factors through $\bar{Y}\to \bar{W}$, i.e.
there exists a morphism $\bar{\theta}\colon \bar{P}\to \bar{Y}$ satisfying
$\bar{\phi}\circ \bar{\theta} = \bar{\psi}$.
As $\mb{M}$ descends to $W$, $\mb{M}_P = \psi^* \mb{M}_W$.
By (1), $\ol{\mb{M}}_{\bar{P}}, \ol{\mb{M}}_{\bar{W}}$
are the flat pullbacks of $\mb{M}_P, \mb{M}_W$, respectively.
Thus, we get $\ol{\mb{M}}_{\bar{P}} = \bar{\psi}^* (\ol{\mb{M}}_{\bar{W}})$.
Again by (1), we have $\ol{\mb M}_{\bar{Y}} = \bar{\theta}_* \ol{\mb M}_{\bar{P}}$.  Hence
\[ \ol{\mb M}_{\bar{Y}} = \bar{\theta}_* \big(\bar{\psi}^* (\ol{\mb{M}}_{\bar{W}})\big)
= \bar{\theta}_*\bar{\theta}^* (\bar{\phi}^*(\ol{\mb M}_{\bar{W}}))
= \bar{\phi}^* (\ol{\mb{M}}_{\bar{W}}). \]
This gives the desired result.


\medskip

(3)
Assume that $X$ admits a projective $k$-morphism $X\to S$ to a $k$-variety $S$.                         
Let $\bar{S} \coloneqq S\otimes_k \bar{k}$.
We show that if $\ol{\mb M}$ is $\bb$-nef over $\bar{S}$ 
if and only if $\mb M$ is $\bb$-nef over $S$.

\medskip

(3.1) 
First, assume that $M$ is $\bb$-nef over $S$.
Then there exists a $k$-variety $Y$ admitting a birational projective morphism $Y\to X$
such that $\mb{M}$ descends to an $\RR$-divisor $\mb{M}_Y$ that is nef over $S$.
Let $\bar{Y}\coloneqq Y\otimes_k \bar{k}$.  By (1), $\ol{\mb M}_{\bar Y}$ is the 
flat pullback of $\mb{M}_Y$.  Then by Lemma~\ref{fppf-peff}, $\ol{\mb M}_{\bar Y}$
is nef over $\bar{S}$.  Moreover, by (2), $\ol{\mb M}$ 
descends to $\bar{Y}$.  Hence $\ol{\mb M}$ is $\bb$-nef over $\bar{S}$.

\medskip

(3.2)
Conversely, assume that $\ol{\mb M}$ descends to a nef$/\bar{S}$ $\RR$-divisor $\ol{\mb M}_{\bar{Y}}$ on
a normal $\bar{k}$-variety $\bar{Y}$ admitting a birational projective morphism $\bar{Y}\to \bar{X}$.
By Lemma~\ref{dom-by-k-varieties}, there exist a normal $k$-variety $P$ and a birational projective morphism
$P\to X$ such that the induced $\bar{P}\to \bar{X}$ factors through $\bar{Y}\to \bar{X}$.
Moreover, by (2.1), $\mb{M}$ descends to $P$.  
Since $\ol{\mb M}_{\bar{P}}$ is the pullback of $\ol{\mb M}_{\bar{Y}}$,
and since $\ol{\mb M}_{\bar{P}}$ is the flat pullback of $\mb{M}_P$, 
by Lemma~\ref{fppf-peff}, $\mb{M}_P$ is nef over $S$.
Thus, $\mb{M}$ is $\bb$-nef over $S$.


\medskip

(4)
Let $\ol{\mb{N}}$ be a $\bb$-$\RR$-divisor on $\bar{X}$.
Assume for each projective birational morphism $W\to X$ from
a normal $k$-variety, the traces $\ol{\mb{N}}_{\bar{W}}$ and $\ol{\mb{M}}_{\bar{W}}$
are equal.  We show that $\ol{\mb N} = \ol{\mb M}$ as $\bb$-$\RR$-divisors over $\bar{X}$.
To this end, let $\bar{Y}\to \bar{X}$ be an arbitrary birational projective morphism
from a normal $\bar{k}$-variety $\bar{Y}$.
By Lemma~\ref{dom-by-k-varieties}, there exists a normal $k$-variety $P$ 
admitting a birational projective morphism $P\to X$
such that the induced $\bar{P}\to \bar{X}$ factors through $\bar{Y}\to \bar{X}$.
By assumption, we have $\ol{\mb{N}}_{\bar{P}} = \ol{\mb{M}}_{\bar{P}}$.
By pushing down the $\RR$-divisors to $\bar{Y}$, we get that
$\ol{\mb{N}}_{\bar{Y}} = \ol{\mb{M}}_{\bar{Y}}$,
hence $\ol{\mb N} = \ol{\mb M}$ as $\bb$-$\RR$-divisors.


\medskip

(5)
Let $\epsilon\in \RR^{\ge 0}$.
By Lemma~\ref{descend-R-Cartier}, it is evident that $(X, B + \mb{M})$
is a g-sub-pair (respectively, a g-pair) if and only if 
$(\bar{X}, \bar{B} + \ol{\mb M})$ is a g-sub-pair (respectively, a g-pair),
where $\bar{B}$ is the flat pullback of $B$, and $\ol{\mb M}$ is the $\bb$-$\RR$-divisor 
induced by $\mb M$.
We claim that $(\bar{X}, \bar{B} + \ol{\mb M})$ is a g-sub-$\epsilon$-lc g-sub-pair
(respectively, a g-$\epsilon$-lc g-pair) if and only if
$(X, B + \mb{M})$ is a g-sub-$\epsilon$-lc g-sub-pair (respectively, a g-$\epsilon$-lc g-pair).

\medskip

(5.1)
First, assume that $(X, B + \mb{M})$ is g-sub-$\epsilon$-lc, but
$(\bar{X}, \bar{B} + \ol{\mb M})$ is not g-sub-$\epsilon$-lc.
Then there exist a log resolution $\bar{\phi}\colon \bar{Y}\to \bar{X}$ of $(\bar{X}, \bar{B} + \ol{\mb M})$
and a prime divisor $\bar{D}$ on $\bar{Y}$ such that 
$a(\bar{D},\bar{X}, \bar{B} + \ol{\mb M}) < \epsilon$.
We can assume that $\ol{\mb M}$ descends to $\bar{Y}$.
By Lemma~\ref{dom-by-k-varieties}, there exists a normal $k$-variety $P$ 
admitting a birational projective morphism $P\to X$
such that the induced $\bar{P}\to \bar{X}$ factors through $\bar{Y}\to \bar{X}$.
Moreover, by (2.1), $\mb M$ descends to $P$.
We can assume $P\to X$ is a log resolution of $(X, B + \mb{M})$.
Let $K_P + B_P + \mb{M}_P$ be the pullback of $K_X + B + \mb{M}_X$,
and let $\bar{B}_{\bar{P}}$ be the flat pullback of $B_P$ to $\bar{P}$.
Then $K_{\bar{P}} + \bar{B}_{\bar{P}} + \ol{\mb{M}}_{\bar{P}}$
is the pullback of $K_{\bar{X}} + \bar{B} + \ol{\mb{M}}_{\bar{X}}$,
where $\bar{B}_{\bar{P}}$ has an irreducible component with coefficient $>1-\epsilon$.
Thus, $B_P$ also has an irreducible component whose coefficient is $>1-\epsilon$,
a contradiction.  Hence $(\bar{X}, \bar{B} + \ol{\mb M})$ is also g-sub-$\epsilon$-lc.

\medskip

(5.2)
Conversely, assume that $(\bar{X}, \bar{B} + \ol{\mb M})$ is g-sub-$\epsilon$-lc.
If $(X, B + \mb{M})$ is not g-sub-$\epsilon$-lc, there exist a log resolution $W\to X$
of $(X, B + \mb{M})$ and a prime divisor $D$ on $W$ such that
$\mb{M}$ descends to $W$ and $a(D, X, B + \mb{M}) < \epsilon$.
The base change $\bar{W}\to \bar{X}$ gives a log resolution
of $(\bar{X}, \bar{B} + \ol{\mb M})$.  Moreover, by (2.2), $\ol{\mb M}$ descends to $\bar{W}$.  Let
$K_{\bar W} + \bar{B}_{\bar W} + \ol{\mb M}_{\bar{W}}$ be the pullback 
of $K_{\bar X} + \bar{B} + \ol{\mb M}_{\bar{X}}$.
Let $\bar{D}$ be the flat pullback of $D$ to $\bar{X}$,
and let $\bar{E}$ be an irreducible component of $\bar{D}$.
Then the coefficient of $\bar{E}$ in $\bar{B}_{\bar W}$ is $>1-\epsilon$, contradicting the assumption.
Thus, $(\bar{X}, \bar{B} + \ol{\mb M})$ is g-sub-$\epsilon$-lc
if and only if $(X, B + \mb{M})$ is g-sub-$\epsilon$-lc.

\medskip

(5.3)
As $(\bar{X}, \bar{B} + \ol{\mb M})$ is a g-pair if and only if $(X, B + \mb{M})$ is a g-pair,
by (5.1) and (5.2), we deduce that $(\bar{X}, \bar{B} + \ol{\mb M})$ is g-$\epsilon$-lc if and only if
$(X, B + \mb{M})$ is g-$\epsilon$-lc.


\subsection{Adjunction for fibre spaces}\label{ss-adj-fib-spaces}

(1)
Let $(X/S,B + \mb{M})$ be a g-sub-pair, where $X$ is an excellent, normal, irreducible scheme of characteristic zero.
Let $f\colon X\to Z$ be a contraction$/S$ with $\dim Z>0$ 
such that $(X,B+\mb{M})$ is g-sub-lc near the 
generic fibre of $f$, and $K_X+B+\mb{M}_X\sim_\RR 0/Z$. 
(Note that as $f$ is a contraction, $Z$ is also a normal irreducible scheme of characteristic zero.)
We recall the \emph{adjunction for fibre spaces} as in \cite[\S 3.4]{B-Fano}
and \cite[\S 11.4]{chen-han-liu-xie-2023}.

\medskip

(2)
For each prime divisor $D$ on $Z$, let 
$t_D$ be the lc threshold of $f^*D$ with respect to $(X,B+\mb{M})$ over the generic point of $D$, that is, 
\[ t_D\coloneqq \sup\Set{t\in \RR \mid (X, B + t f^* D + \mb{M})
~\mbox{is g-sub-lc over the generic point of}~D }. \]
Note that $f^*D$ is well-defined over the generic point of $D$ as $Z$ is normal.

We claim that there are only finitely many prime divisors $D$ on $Z$ with $t_D \not= 1$.
Let $\pi\colon Y\to X$ be a log resolution of $(X, B + \mb{M})$ on which
$\mb{M}$ descends to a nef$/S$ divisor $\mb{M}_Y$; see \S\ref{log-resolution}.
Denote by $g\colon Y\to Z$ be the induced contraction.
Write
\[ K_Y + B_Y + \mb{M}_Y = \pi^*(K_X + B + \mb{M}_X). \]
Denote by $B_Y^h$ (respectively, by $B_Y^v$) the horizontal$/Z$ (respectively, vertical$/Z$)
part of $B$.
By generic smoothness (see \cite[Exercise~10.40]{GW-agbook}),
there is an open dense subset $V\subseteq Z$ such that 
\begin{itemize}
    \item $V$ is a regular scheme,
    \item $g(\Supp B_Y^v)\cap V = \emptyset$,
    \item the image of every vertical$/Z$ stratum of $\Supp B_Y^h$ does not intersect $V$, and
    \item for every horizontal$/Z$ stratum $W$ of $\Supp B_Y^h$, $W|_{g^{-1}(V)}\to V$ is a smooth morphism.
\end{itemize}
Then by \cite[Chap.~IV, \S D.2, Proposition~22]{Serre-local-algebra},
for each prime divisor $D$ with $\eta_D\in V$, where $\eta_D$ is the generic point of $D$,
$(Y, B_Y + g^*D)$ is an snc pair over $V$, hence $t_D = 1$.

\medskip

(3)
Let
\[ b_D\coloneqq 1 - t_D~~\mbox{and}~~B_Z\coloneqq \sum b_D D, \]
where the sum runs over all the prime divisors on $Z$. 
By (2), there are only finitely many nonzero $b_D$.
By assumption, $K_X+B+\mb{M}_X\sim_{\RR} f^*L_Z$ for some $\RR$-Cartier 
$\RR$-divisor $L_Z$ on $Z$.  
Letting 
\[ N_Z=L_Z-(K_Z+B_Z), \] 
we get 
\[ K_X+B+\mb{M}_X\sim_\RR f^*(K_Z+B_Z+N_Z). \]
We call $B_Z$ the \emph{discriminant part} and ${N_Z}$ the \emph{moduli part} of adjunction. 
Note that the $\RR$-divisor $B_Z$ is uniquely determined, but $M_Z$ depends on the choice of $L_Z$.

\medskip

(4)
Keep the notation in (1).
Let $\phi\colon X'\to X$ and $\psi\colon Z'\to Z$ be projective birational morphisms from excellent, normal,
irreducible schemes of characteristic zero such that the induced map
$f'\colon X'\bir Z'$ is a morphism.
\[\xymatrix{
X'\ar[r]^{\phi}\ar[d]_{f'} & X\ar[d]^f \\
Z'\ar[r]^-{\psi} & Z
}\]
Write
\[ K_{X'} + B' +\mb{M}_{X'} = \phi^*(K_X+B+\mb{M}_X). \]
As in (3), we define a discriminant divisor $B_{Z'}$, and 
taking $L_{Z'}=\psi^*L_Z$ gives a moduli divisor $N_{Z'}$ so that 
\[ K_{X'}+B' + \mb{M}_{X'}\sim_\RR f'^*(K_{Z'}+B_{Z'}+N_{Z'}). \]
Let $D$ be a prime divisor on $Z$ and $D'$ its birational transform on $Z'$.
It is evident that $t_{D'} = t_D$, hence $B_Z=\psi_*B_{Z'}$.
As $L_{Z'} = \psi^* L_Z$ and $\psi_*K_{Z'} = K_Z$, we see that $N_Z=\psi_*N_{Z'}$.
Thus, we get a well-defined b-$\RR$-divisor $\mathbf{N}$ over $Z$ such that $\mathbf{N}_Z = N_Z$;
see \cite[\S 2.3.2]{Corti-3-folds}.

\medskip


(5) Next we relate the singularities of $(X, B + \mb{M})$ and $(Z, B_Z + \mb{N})$ in a rough sense.

\begin{lemma}\label{l-adj-fib-eps-lc}
Let $\epsilon\in \RR^{\ge 0}$. 
Under the notation and assumptions of (1),
suppose there is a prime divisor $D$ on 
some birational model of $X$ such that $a(D, X, B + \mb{M})\le \epsilon$ 
and that $D$ is vertical$/Z$.  Then there is a resolution $Z'\to Z$ and 
an irreducible component $T$ of 
$B_{Z'}$ with coefficient $\ge 1-\epsilon$. 
\end{lemma}

\begin{proof}
Pick resolutions $X'\to X$ and $Z'\to Z$ so that 
\begin{itemize}
    \item the induced map $f'\colon X'\bir Z'$ is a morphism,
    \item $D$ is a prime divisor on $X'$,
    \item the image of $D$ on $Z'$ is a prime divisor $T$, and 
    \item $\mb{M}$ descends to $X'$.
\end{itemize}
Let $K_{X'} + B' + \mb{M}_{X'}$ be the pullback of $K_X + B + \mb{M}_X$. 
Since $a(D, X, B + \mb{M})\le \epsilon$, $\coeff_D B' \ge 1-\epsilon$. 
Then the generalised lc threshold of $f'^*T$ with 
respect to $(X', B' + \mb{M})$ over the generic point of $T$ is $\le \epsilon$. Therefore, 
the coefficient of $T$ in $B_{Z'}$ is $\ge 1-\epsilon$.
\end{proof}

\medskip


(6)
Let $k$ be a field of characteristic zero.
Let $(X, B + \mb{M})$ be a g-sub-pair over $k$ admitting a contraction
$X\to Z$ to a $k$-variety $Z$.
When $(X, B+\mb{M})$ is g-lc over the generic point of $Z$,
the adjunction for fibre spaces gives a g-pair on $Z$.

\begin{theorem}\label{t-adj-fib-spaces}
    Let $k$ be a field of characteristic zero, 
    and let $X\to S$ be a projective morphism of $k$-varieties.
    Let $(X/S, B + \mb{M})$ be a g-pair and $f\colon X\to Z$ a contraction$/S$
    of $k$-varieties such that $K_X + B + \mb{M}_X\sim_{\RR}0/Z$.
    Suppose that $(X,B+\mb{M})$ is g-lc.  Then we can write
    \[ K_X + B + \mb{M}_X \sim_{\RR} f^*(K_Z + B_Z + \mb{N}_Z), \]
    where $(Z, B_Z + \mb{N})$ is a g-lc g-pair.
\end{theorem}

\begin{proof}
    \emph{Step~1}.
    By adjunction for fibre spaces, we can write
    \[ K_X + B + \mb{M}_X \sim_{\RR} f^*(K_Z + B_Z + \mb{N}_Z), \]
    where $B_Z$ is the discriminant part of adjunction, 
    the $\bb$-$\RR$-divisor $\mb N$ is the moduli part of adjunction, 
    and $\mb{N}_Z$ is the trace of $\mb N$ on $Z$. 
    We assume that $\mb{N}$ is defined by a choice of an 
    $\RR$-Cartier $\RR$-divisor $L_Z$ on $Z$; see (4).
    
    Denote by $\bar{k}$ the algebraic closure of $k$.
    Let $\bar{X} \coloneqq X\otimes_k \bar{k}$, and let
    $\ol{\mb M}$ be the $\bb$-$\RR$-Cartier divisor over $\bar{X}$ induced by $\mb{M}$.
    Let $\bar{f}\colon \bar{X}\to \bar{Z}$ be the induced contraction over $\bar{S}$.
    Let $\bar{B}$ be the flat pullback of $B$ to $\bar{X}$.
    Again by adjunction of fibre spaces, we can write
    \[ K_{\bar{X}} + \bar{B} + \mb{M}_{\bar{X}} \sim_{\RR} \bar{f}^*(K_{\bar{Z}} + \bar{C} + \ol{\mb{Q}}_{\bar{Z}}), \]
    where $\bar{C}$ is the discriminant part, the $\bb$-$\RR$-divisor $\ol{\mb{Q}}$
    is the moduli part, and $\ol{\mb{Q}}_{\bar{Z}}$ is the trace of $\ol{\mb{Q}}$ on $\bar{Z}$.
    Denote by $L_{\bar{Z}}$ the flat pullback of $L_Z$ to $\bar{Z}$.
    By Lemma~\ref{descend-R-Cartier}, $L_{\bar{Z}}$ is an $\RR$-Cartier $\RR$-divisor on $\bar{Z}$.
    We assume that $\ol{\mb{Q}}$ is defined by $L_{\bar{Z}}$ as in (4).

    \medskip
    \emph{Step~2}.
    Let $\bar{E}$ be a prime divisor on $\bar{Z}$.
    Denote by $D$ the image of $\bar{E}$ in $Z$ equipped with the reduced structure.
    It is evident that $D$ is a prime divisor on $Z$.
    As $\chara k =0$, $D\otimes_k \bar{k}$ is a reduced subscheme of $\bar{Z}$
    consisting of the orbits of $\bar{E}$ under the action of $\Gal (\bar{k}/k)$.
    By \cite[\S 2.42]{Kol_singularities_of_MMP}, for every $a\in \RR$,
    $(X, B + a f^* D + \mb{M})$ is g-lc over the generic point of $D$
    if and only if $(\bar{X}, \bar{B} + a \bar{f}^*\bar{E} + \ol{\mb M})$
    is g-lc over the generic point of $\bar{E}$.    
    Thus, we see that $\bar{C}$ is the flat pullback of $B_Z$ from $Z$ to $\bar{Z}$. 
    By \cite[Lemma~11.3.5]{chen-han-liu-xie-2023}, $(\bar{X}, \bar{B} + \ol{\mb M})\to \bar{Z}$
    is an lc trivial fibration.
    Then by \cite[Theorem~11.4.4]{chen-han-liu-xie-2023}, $\ol{\mb N}$ is $\bb$-$\RR$-Cartier
    and $\bb$-nef over $\bar{S}$, and $(\bar{Z}, \bar{C} + \ol{\mb N})$ is a g-lc g-pair.
    By \S\ref{ss-descent-b-nef} (4), $\ol{\mb Q}$ is the $b$-divisor induced by $\mb N$.
    Then by \S\ref{ss-descent-b-nef} (1), (2), and (3), $\mb{N}$ is a $\bb$-$\RR$-Cartier
    $\bb$-divisor that is $\bb$-nef over $S$.
    Moreover, by \S\ref{ss-descent-b-nef} (5), $(Z, B_Z + \mb{N})$ is g-lc.
\end{proof}


\subsection{Galois descent of divisors}

\begin{definition}
    Let $k$ be a field of characteristic zero and $k'/k$ a finite Galois extension.
    Denote by $G\coloneqq \Gal(k'/k)$ the Galois group of $k'/k$, 
    and by $\ell$ the cardinality of $G$.
    Let $X$ be a $k$-variety and $X' \coloneqq X\otimes_k {k'}$.
    Let $D'$ be an $\RR$-divisor on $X'$.  
    The \emph{Galois average of $D'$ with respect to $k'/k$} is the $\RR$-divisor
    \[ G_a(D') \coloneqq \frac{1}{\ell}\sum_{i=1}^{\ell} g_i^* D', \]
    where $G = \Set{g_1, \dots, g_{\ell}}$, and $g_i^* D'$ is the flat pullback of $D'$
    via the induced automorphism $g_i\colon X'\to X'$.
    As $G_a(D')$ is $G$-invariant, it is evident that there exists a unique $\RR$-divisor $D$ on $X$
    such that $D'$ is the flat pullback of $D$ to $X'$.  
    We call $D$ the \emph{Galois descent of $D'$ to $X$}.
    It is evident that if $B$ is an $\RR$-divisor on $X$, and if $B'$ is the flat pullback 
    of $B$ to $X'$, then the Galois descent of $B'$ to $X$ is just $B$.
\end{definition}


\begin{lemma}\label{Galois-descent-R-Cartier}
    Let $k$ be a field of characteristic zero and $k'/k$ a finite Galois extension.
    Let $X$ be a $k$-variety and $X' \coloneqq X\otimes_k{k'}$.
    Let $D'$ be an $\RR$-Cartier $\RR$-divisor (respectively, a $\QQ$-Cartier $\QQ$-divisor)
    on $X'$ and $D$ the Galois descent of $D'$ to $X$.
    Then $D$ is also $\RR$-Cartier (respectively, $\QQ$-Cartier).
\end{lemma}

\begin{proof}
    Assume that $D'$ is $\RR$-Cartier.
    Denote by $G$ the Galois group $\Gal(k'/k)$.
    Then we can write $D' = \sum_{i=1}^r a_i D_i'$
    for some $a_i\in \RR$ and Cartier divisors $D_1', \dots, D_r'$.
    Denote by $D_i$ the Galois descent of $D_i'$.
    It is evident that $G_a(D') = \sum_{i=1}^r a_i G_a(D_i')$, hence $D = \sum_{i=1}^r a_i D_i$.
    Thus, we can assume $D'$ is a Cartier divisor.
    Then the flat pullback of $\ell D$ to $X'$ is equal to $\ell\cdot G_a(D')$, 
    which is a Cartier divisor on $X'$.
    By \cite[Proposition~14.48]{GW-agbook}, $\ell D$ is Cartier, hence $D$ is $\QQ$-Cartier.
\end{proof}


\begin{lemma}\label{descent-g-pair}
    Let $k$ be a field of characteristic zero and $k'/k$ a finite Galois extension.
    Let $X\to S$ be a projective morphism of $k$-varieties.
    Let $\mb{M}$ be a $\bb$-divisor over $X$, and let
    $\mb{M}'$ be the induced $\bb$-divisor over $X'\coloneqq X\otimes_k{k'}$.
    Assume that $(X', B' + \mb{M}')$ is a g-pair over $S'\coloneqq S\otimes_k {k'}$.
    Denote by $B$ the Galois descent of $B'$ to $X$.
    Then $(X, B + \mb{M})$ is also a g-pair.
    Moreover, for any $\epsilon\in \RR^{\ge 0}$, if $(X', B' + \mb{M}')$ is g-$\epsilon$-lc,
    then $(X, B + \mb{M})$ is also g-$\epsilon$-lc.
\end{lemma}

\begin{proof}
    As $(X'/S', B' + \mb{M}')$ is a g-pair, $B'$ is effective, and $K_{X'} + B' + \mb{M}'_{X'}$
    is $\RR$-Cartier.  As $B$ is the Galois descent of $B'$, $B$ is also effective.
    Since $\pi\colon X'\to X$ is \'etale, we get $\pi^* K_X = K_{X'}$.
    Moreover, as $\mb{M}'_{X'} = \pi^* \mb{M}_X$, we see that $K_X + B + \mb{M}_X$
    is the Galois descent of $K_{X'} + B' + \mb{M}'_{X'}$ to $X$.
    Then by Lemma~\ref{Galois-descent-R-Cartier}, $K_X + B + \mb{M}_X$ is $\RR$-Cartier.
    Moreover, by \S\ref{ss-descent-b-nef}~(2) and (3), $\mb{M}$ is $\bb$-$\RR$-Cartier and
    $\bb$-nef over $S$, hence $(X, B + \mb{M})$ is a g-pair.

    Let $G \coloneqq \Gal(k'/k)$, then $\pi^*B = G_a(B')$.
    Let $\phi\colon W\to X$ be a log resolution of $(X, B + \mb{M})$.
    Then $\phi'\colon W'\to X'$, the base change of $\phi$ to $\spec k'$,
    is a $G$-equivariant log resolution of $(X', B' + \mb{M}')$ and $(X', \pi^*B + \mb{M}')$;
    in particular, the exceptional locus of $\phi'$ is $G$-invariant.  
    We can assume that $\mb{M}$ descends to $W$, hence $\mb{M}'$ descends to $W'$;  
    cf. \S\ref{ss-descent-b-nef}~(2).  Write
    \[ K_{W'} + E' + \mb{M}'_{W'} = \phi'^*(K_{X'} + B' + \mb{M}'_{X'}) \]
    and
    \[ K_{W'} + F' + \mb{M}'_{W'} = \phi'^*(K_{X'} + G_a(B') + \mb{M}'_{X'}). \]
    As $K_{W'}$ and $\mb{M}'_{W'}$ are $G$-invariant, we can deduce that $F' = G_a(E')$.
    If $(X', B' + \mb{M}')$ is g-$\epsilon$-lc, $E'\le 1-\epsilon$.
    Then by construction of $G_a(E')$, we see that $F'\le 1-\epsilon$.
    Thus, $(X', \pi^* B + \mb{M}')$ is also g-$\epsilon$-lc.
    By \S\ref{ss-descent-b-nef}~(6), $(X, B + \mb{M})$ is g-$\epsilon$-lc.
\end{proof}


\begin{remark}
    Keep the notation in Lemma~\ref{descent-g-pair}.
    The g-pair $(X', B' + \mb{M}')$ may not be g-$\epsilon$-lc 
    when $(X, B + \mb{M})$ is g-$\epsilon$-lc, even for $\epsilon = 0$.
    For instance, let $k = \QQ$ and $k' = \QQ[\sqrt{-1}]$. 
    Then $G\coloneqq \Gal(k'/k)\cong \ZZ_2$.
    Let $X = \PP_{\QQ}^1$, and let $B = [1\colon \sqrt{-1}]$, which is a $\QQ[\sqrt{-1}]$-point of $X$.
    Denote by $\pi$ the canonical morphism $X'\to X$, where $X' = \PP_{\QQ[\sqrt{-1}]}^1$.
    Then $\pi^*B = p + q$, where $p = [1\colon \sqrt{-1}]$ and $q=[-1\colon \sqrt{-1}]$.
    Let $B' = 2p$.  Then we get $G_a(B') = p + q = \pi^*B$, i.e. $B$ is the Galois descent of $B'$ to $X$.
    It is evident that $(X, B)$ is lc, while $(X', B')$ is not.
\end{remark}


\begin{lemma}\label{linear-0-to-closure}
    Let $k$ be a field of characteristic zero and $\bar{k}$ its algebraic closure.
    Let $f\colon X\to Z$ be a projective morphism of $k$-varieties and $D$ be an $\RR$-divisor on $X$.  Let
    $\bar{X}\coloneqq X\otimes_k\bar{k}$ and $\bar{Z}\coloneqq Z\otimes_k\bar{k}$,
    and let $\bar{f}\colon \bar{X}\to \bar{Z}$ be a induced morphism of $\bar{k}$-varieties.
    Denote by $\bar{D}$ the flat pullback of $D$ to $\bar{X}$.
    Then $\bar{D}\sim_{\RR}0/\bar{Z}$ if and only if $D\sim_{\RR}0/Z$.
\end{lemma}

\begin{proof}
    It is evident that $D\sim_{\RR}0/Z$ implies $\bar{D}\sim_{\RR}0/\bar{Z}$.
    Thus, we assume that $\bar{D}\sim_{\RR}0/\bar{Z}$; in other words,
    there exist an $\RR$-Cartier $\RR$-divisor $\bar{L}$ on $\bar{Z}$ and 
    rational functions $\alpha_1,\dots, \alpha_s\in K(\bar{X})$ such that
    \[ \bar{D} = \bar{f}^*\bar{L} + \sum_{i=1}^s a_i\Div_{\bar{X}}(\alpha_i), \]
    where $a_i\in \RR$.
    Pick a finite Galois extension $k'/k$ so that
    \begin{itemize}
        \item $\bar{L}$ is the flat pullback of an $\RR$-Cartier divisor $L'$ on $Z'\coloneqq Z\otimes_k k'$, and
        \item each $\alpha_i$ is the image of an $\alpha_i'\in K(X')$ via the canonical field extension $K(X')\to K(\bar{X})$ , where $X'\coloneqq X\otimes_k k'$.
    \end{itemize}
    By Lemma~\ref{descend-R-Cartier}, $D$ is $\RR$-Cartier.
    Let $f'\colon X'\to Z'$ be the base change of $f$ to $\spec k'$, and let $D'$
    be the pullback of $D$ to $X'$.  
    Then by \cite[Proposition~(21.10.6)]{EGA-IV-4}, the flat pullback of $f'^*L' + \sum_{i=1}^s \Div_{X'}(\alpha_i')$
    to $\bar{X}$ is equal to $\bar{D}$.  Hence we have
    \[ D' = f'^*L' + \sum_{i=1}^s a_i\Div_{X'}(\alpha_i'). \]
    Denote by $G\coloneqq \Gal(k'/k)$ the Galois group of $k'/k$.
    Let $\ell$ be the cardinality of $G$, and let $G=\Set{g_1,\dots, g_{\ell}}$.
    Let $L$ be the Galois descent of $L'$ to $Z$, which is an $\RR$-Cartier $\RR$-divisor on $Z$ 
    by Lemma~\ref{Galois-descent-R-Cartier}.
    Moreover, for each $\alpha_i'$, the Galois average of $\Div_{X'}(\alpha_i')$ with respect to $k'/k$ is equal to
    \[ \frac{1}{\ell}\sum_{j=1}^{\ell} g_j^*\Div_{X'}(\alpha_i') 
    = \frac{1}{\ell} \sum_{j=1}^{\ell}\Div_{X'}(g_j^*\alpha_i) 
    = \frac{1}{\ell}\Div_{X'}\big( \prod_{j=1}^{\ell} g_j^*\alpha_i' \big). \]
    Let
    \[ \beta_i \coloneqq \prod_{j=1}^{\ell} g_j^* \alpha_i', \]
    which is a $G$-invariant element of $K(X')$.  By \cite[Proposition~A7.1.3]{Katz-Mazur-1985-elliptic}, we have
    \[ K(X)\cong K(X')^G. \]
    Thus, $\beta_i$ is contained in $K(X)$ via the canonical field extension $K(X)\to K(X')$.
    Again by \cite[Proposition~(21.10.6)]{EGA-IV-4}, $\Div_{X'}(\beta_i)$
    is equal to the pullback of $\Div_X(\beta_i)$.  Hence
    the Galois descent of $\Div_{X'}(\alpha_i)$ to $X$
    is the divisor $(1/\ell)\Div_X(\beta_i)$.
    Then the Galois descent of $D'$ to $X$ is 
    \begin{equation}\label{descent-to-X}
        f^*L + \sum_{i=1}^s \frac{a_i}{\ell} \Div_X(\beta_i).
    \end{equation}
    Therefore, we can conclude that $D$ is equal to \eqref{descent-to-X}, so $D\sim_{\RR}0/Z$.
\end{proof}


\subsection{Fano type varieties}\label{fano-type-defn}

Let $k$ be a field of characteristic zero.
Assume that $X,Z$ are $k$-varieties and that $X\to Z$ is a contraction over $\spec k$.
We say that $X$ is \emph{of Fano type over $Z$} 
(or $X\to Z$ is \emph{of Fano type})
if there is a $\QQ$-divisor $B$ such that $(X, B)$ is klt and
that $-(K_X + B)$ is ample over $Z$; cf. \cite[\S 2.10]{B-Fano}.


\begin{lemma}\label{fano-type-to-closure}
    Let $k$ be a field of characteristic zero, and let $X\to Z$ be
    a contraction of $k$-varieties.  Let $\bar{X}\otimes_k\bar{k}$
    and $\bar{Z}\coloneqq Z\otimes_k \bar{k}$.  Then $X\to Z$ is of Fano type
    if and only if $\bar{X}\to \bar{Z}$ is of Fano type.
\end{lemma}

\begin{proof}
    By \S\ref{ss-descent-b-nef}~(6), it is evident that $\bar{X}\to \bar{Z}$
    is of Fano type if $X$ is of Fano type over $Z$.
    Now assume that $\bar{X}\to \bar{Z}$ is of Fano type.
    Then there exists a $\QQ$-divisor $\bar{\Gamma}$ on $\bar{X}$ such that
    $(\bar{X}, \bar{\Gamma})$ is klt, and $-(K_{\bar{X}} + \bar{\Gamma})$ is ample$/\bar{Z}$.
    There exist a finite Galois extension $k'/k$ and a $\QQ$-divisor $\Gamma'$
    on $X'\coloneqq X\otimes_k {k'}$ such that $\bar{\Gamma}$ is the flat pullback
    of $\Gamma'$.  Let $\Gamma$ be the Galois descent of $\Gamma'$ to $X$.
    By Lemma~\ref{descent-g-pair}, $(X, \Gamma)$ is a klt pair.
    By Lemma~\ref{fppf-ample}, $-(K_X+\Gamma)$ is ample$/Z$.
    Thus, $X$ is of Fano type over $Z$.
\end{proof}


\begin{lemma}\label{fano-type-conditions}
    Let $k$ be a field of characteristic zero.
    Let $X\to Z$ be a contraction of $k$-varieties.
    Then the following are equivalent:
    \begin{itemize}
        \item [\emph{(i)}] $X$ is of Fano type over $Z$,
        \item [\emph{(ii)}] there exists a big$/Z$ $\QQ$-divisor $\Gamma$ such that $(X, \Gamma)$ is klt and $K_X+\Gamma\sim_{\QQ}0/Z$, and
        \item [\emph{(iii)}] there exists a big$/Z$ $\RR$-divisor $B$ such that $(X/Z, B+\mb{M})$ is a g-klt g-pair and $K_X+B+\mb{M}_X\sim_{\RR}0/Z$.
    \end{itemize}
\end{lemma}

\begin{proof}
    Passing to the algebraic closure $\bar k$ of $k$, we see that (i) implies (ii).
    It is evident that (ii) gives (iii) by setting $\mb M = 0$.
    Now, assume (iii), and let $\bar{X}\coloneqq X\otimes_k\bar{k}$ and $\bar{Z}\coloneqq Z\otimes_k\bar{k}$.
    By \S\ref{ss-descent-b-nef}, $\bar{X}$ is of Fano type over $\bar{Z}$.
    Then by Lemma~\ref{fano-type-to-closure}, we can conclude that $X\to Z$ is of Fano type,
    i.e. (iii) implies (i).
\end{proof}


\begin{lemma}\label{Fano-type-pushdown}
    Let $k$ be a field of characteristic zero.
	Let $X\to Z$ be a contraction of $k$-varieties such that $X$ is of Fano type over $Z$,
	and let $\phi\colon X\dashrightarrow Y$ be a birational contraction of $k$-varieties.
	Then $Y$ is also of Fano type over $Z$.
\end{lemma}

\begin{proof}
    This follows from Lemma~\ref{fano-type-to-closure} and \cite[Lemma~2.19]{birkar2025steindegreelogcalabiyau}.
\end{proof}

\begin{lemma}\label{fano-type-at-dlt-model}
    Let $k$ be a field of characteristic zero.
    Let $f\colon X\to Z$ be a contraction of $k$-varieties such that $X$ is of Fano type over $Z$.
    Let $(X, B + \mb{M})$ be a g-lc g-pair such that $-(K_X + B + \mb{M}_X)$ is nef$/Z$.
    Let $\pi\colon X'\to X$ be a birational contraction of normal $k$-varieties.  Write
    \[ K_{X'} + B' + \mb{M}_{X'} = \pi^*(K_X + B + \mb{M}_X). \]
    If every exceptional$/X$ prime divisor has positive coefficient in $B'$, 
    then $X'$ is also of Fano type over $Z$.
\end{lemma}

\begin{proof}
    By Lemma~\ref{fano-type-to-closure} and \S\ref{ss-descent-b-nef},
    we can assume $k$ is algebraically closed, 
    then we can apply \cite[Lemma~2.20]{birkar2025steindegreelogcalabiyau} to conclude.
\end{proof}


\begin{lemma}\label{fano-type-upper}
    Let $k$ be a field of characteristic zero.
    Let $X\to Y$ and $Y\to Z$ be contractions between normal $k$-varieties.
    If $X$ is of Fano type over $Z$, then $X$ is of Fano type over $Y$.
\end{lemma}

\begin{proof}
    This follows from Lemma~\ref{fano-type-to-closure} and \cite[Lemma~2.21]{birkar2025steindegreelogcalabiyau}.
\end{proof}


\begin{lemma}\label{fano-type-on-base}
    Let $k$ be a field of characteristic zero.
    Let $f\colon X\to Y$ be a contraction$/Z$ of $k$-varieties $X, Y$
    that are projective over a $k$-variety $Z$,
    where $\dim (Y/Z)>0$.  Assume that $X$ is of Fano type over $Z$.
    Then $Y$ is of Fano type over $Z$.
\end{lemma}

\begin{proof}
    This follows from Lemma~\ref{fano-type-to-closure} and \cite[Lemma~2.22]{birkar2025steindegreelogcalabiyau}.
\end{proof}


\begin{lemma}\label{Fano-type-base-change}
    Let $K/k$ be an extension of fields of characteristic zero 
    (which may not be algebraic).
    Let $X\to Z$ be a contraction between $k$-varieties.
    Let $X_K\to Z_K$ be the base change of $X\to Z$, where
    $X_K\coloneqq X\otimes_k K$ and $Z_K\coloneqq Z\otimes_k K$.
    Assume $X\to Z$ is of Fano type.
    Then $X_K\to Z_K$ is also of Fano type.
\end{lemma}

\begin{proof}
    The proof is routine and left to the readers.
\end{proof}


\begin{theorem}\label{BPF-thm}
    Let $\mbb K$ be an algebraically closed field of characteristic zero.
    Fix $n\in \NN$.
    Then there is an $m\in \NN$ depending only on $n, \mbb K$ with the following property.
    
    Let $f\colon X\to Z$ be a projective morphism of normal $k$-varieties,
    where $k$ is a field with algebraic closure $\mbb K$.
    Let $D$ be a nef$/Z$ $\RR$-divisor on $X$ such that 
    $aD - (K_X + \Delta)$ is nef and big over $Z$
    for some $a\in \RR^{>0}$, where $(X, \Delta)$ is klt and $\dim X = n$.
    Then $D$ is semiample$/Z$, and if $aD$ is Cartier, then $maD$ is globally generated over $Z$.
\end{theorem}

\begin{proof}
    Let $\bar{X}, \bar{Z}$ be the base change of $X, Z$ to $\spec \mbb K$ respectively,
    and let $\bar{D}, \bar{\Delta}$ be the flat pullback of $D, \Delta$ to $\bar{X}$ respectively.
    Denote by $\bar{f}\colon \bar{X}\to \bar{Z}$ the induced projective morphism of 
    normal $\mbb K$-varieties.  Then $\bar{D}$ is nef$/\bar{Z}$, and
    $a \bar{D} - (K_{\bar{X}} + \bar{\Delta})$ is nef and big over $\bar{Z}$.
    By \cite[Theorem~3.9.1]{BCHM}, $\bar{D}$ is semiample$/\bar{Z}$.
    Thus, we can write 
    \[ \bar{D} = \sum_i a_i \bar{D}_i, \]
    where $a_i\in \RR^{>0}$, and each $\bar{D}_i$ is a semiample$/\bar{Z}$ Cartier divisor on $\bar{X}$.
    Pick a finite Galois extension $k'/k$ so that there are Cartier divisors $D_i'$ on $X'\coloneqq X\otimes_k k'$
    whose flat pullbacks to $\bar{X}$ are the $\bar{D}_i$'s.
    Then by Lemma~\ref{descent-semiample}, each $D_i'$ is semiample over $Z'\coloneqq Z\otimes_k k'$.
    Let $G_a(D_i')$ be Galois average of $D_i'$, and
    let $D_i$ be the Galois descent of $D_i'$ to $X$.
    By Lemma~\ref{Galois-descent-R-Cartier}, each $D_i$ is $\QQ$-Cartier.
    Moreover, as $G_a(D_i')$ is semiample$/Z'$, by Lemma~\ref{descent-semiample}, 
    $D_i$ is also semiample$/Z$.
    Let $D'$ be the pullback of $D$ to $X'$.  As $D'$ is $G$-invariant, we have
    \[ D' = G_a(D') = \sum_i a_i G_a(D_i'). \]
    Thus, we can conclude that $D = \sum_i a_i D_i$.
    Hence $D$ is a semiample$/Z$ $\RR$-divisor.
    
    Finally, if $aD$ is Cartier, then $a\bar{D}$ is also Cartier.
    Hence by \cite[Theorem~3.9.1]{BCHM}, there is an $m\in \NN$ depending only on $n,\mbb K$ 
    so that $ma\bar{D}$ is globally generated over $\bar{Z}$. 
    Then by faithfully flat descent (cf. Lemma~\ref{descent-semiample}),
    $maD$ is globally generated over $Z$.
\end{proof}


\begin{corollary}\label{fano-nef-semiample-cor}
    Let $k$ be a field of characteristic zero.
    Let $X\to Z$ be a contraction of $k$-varieties, which is of Fano type.
    Then every nef$/Z$ $\RR$-divisor on $X$ is semiample$/Z$.
\end{corollary}

\begin{proof}
    As $X\to Z$ is of Fano type, there is a $\QQ$-divisor $\Delta$ on $X$ such that
    $(X, \Delta)$ is klt, and $-(K_X + \Delta)$ is ample$/Z$.
    Then $D-(K_X + \Delta)$ is also ample$/Z$.  
    By Theorem~\ref{BPF-thm}, $D$ is semiample$/Z$.
\end{proof}


\subsection{Fano type fibrations}
Let $k$ be a field of characteristic zero.
A \emph{generalised log Calabi-Yau fibration $(X, B + \mb{M})\to Z$ over $k$}
consists of a $k$-scheme $S$, a g-lc g-pair $(X/S, B + \mb{M})$, 
and a contraction $X\to Z$ over $S$ between $k$-varieties such that
\[ K_X+B+\mb{M}_X\sim_{\RR} 0/Z. \]
We call a generalised log Calabi-Yau fibration $(X, B+\mb{M})\to Z$ over $k$
a \emph{Fano type generalised log Calabi-Yau fibration}
(or just a \emph{Fano type fibration} for short)
if $X$ is of Fano type over $Z$ in addition.


\begin{lemma}\label{fano-type-induction}
    Let $k$ be a field of characteristic zero.
    Let $(X, B + \mb{M})\to Z$ be a Fano type generalised log Calabi-Yau fibration over $k$.
    Assume $\dim Z\ge 2$.
    Let $L$ be a general Cartier hyperplane section of $Z$.  
    Denote by $X_L$ the fibre product $L\times_Z X$.  Then we can write
    \begin{equation}
        K_{X_L} + B_L + \mb{N}_{X_L} \sim_{\RR} (K_X + B + X_L + \mb{M}_X)|_{X_L}, \label{fano-type-adjunction}
    \end{equation}
    where $(X_L, B_L + \mb{N})\to L$ is also a Fano type generalised log Calabi-Yau fibration over $k$.
    Moreover, let $S$ be an irreducible component of $B$ whose image $T$ in $Z$ satisfies that
    \begin{itemize}
        \item $\dim T = \dim Z$ if $\dim Z \ge 2$, or
        \item $\dim T = \dim Z - 1$ if $\dim Z \ge 3$.
    \end{itemize}
    Assume that $\coeff_S B\ge t$ for some $t\in \RR^{>0}$,
    and set $S_L \coloneqq L\times_Z S$.  Then $S_L$ is also an irreducible component of $B_L$
    satisfying $\coeff_{S_L} B_L \ge t$.
\end{lemma}

\begin{proof}
    By Bertini's theorem for normality (see \cite[Corollary 3.4.9]{joins-intersections}),
    both $L$ and $X_L$ are also normal varieties.
    Moreover, for a general $L$, $(X, B + X_L + \mb{M})$ is also g-lc, hence by generalised adjunction
    (see \cite[Definition 4.7]{B-Zhang}), we can write \eqref{fano-type-adjunction},
    where $(X_L, B_L + \mb{N}) \to L$ is also a generalised log Calabi-Yau fibration.
    
    On the other hand, as $X$ is of Fano type over $Z$, there is a big$/Z$ $\QQ$-divisor $\Gamma$
    such that $(X, \Gamma)$ is klt and $K_X + \Gamma \sim_{\QQ}0/Z$.  For a general $L$, we can write
    \[ K_{X_L} + \Gamma_L \sim_{\QQ} (K_X + \Gamma + X_L)|_{X_L}, \]
    where $(X_L, \Gamma_L)$ is also a klt pair (by \cite[Theorem 5.50]{km98}) 
    satisfying that $K_{X_L} + \Gamma_L\sim_{\QQ}0/L$.
    It is clear that $\Gamma_L$ is big$/L$ by \cite[Corollary II.5.17]{nakayama_Zariski_decomposition}.
    Thus, $X_L$ is also of Fano type over $L$.
    Now pick an irreducible component $S$ of $B$ satisfying all the assumptions.
    Then $L\times_Z S$ is also reduced and irreducible by \cite[Th\'eor\`eme 6.10]{jou-bertini}
    and \cite[Corollary 3.4.9]{joins-intersections}.  Finally, the coefficient of $S_L$ 
    in $B_L$ is $\ge t$ by \cite[Definition 4.7]{B-Zhang}.
\end{proof}


\section{\bf MMP on g-pairs over non-closed fields}


\subsection{Good contraction of extremal rays}
Let $k$ be a field of characteristic zero.
Let $X\to Z$ be a projective morphism between $k$-varieties.
Let $f\colon X\to Y$ be a contraction of an extremal ray $R\subset \NE(X/Z)$.
As in \cite[Definition~11.6]{lyu-takumi-2026relativeminimalmodelprogram},
we say that the contraction $f$ is \emph{good}
if for each $\QQ$-Cartier $\QQ$-divisor $L$ on $X$, $L\cdot R = 0$
if and only if $L \sim_{\QQ} f^* L_Y$ for a $\QQ$-Cartier $\QQ$-divisor $L_Y$ on $Y$.

\begin{lemma}\label{good-cont-real}
    Let $k$ be a field of characteristic zero, 
    and let $X\to Z$ be a projective morphism between $k$-varieties.
    Let $f\colon X\to Y$ be a good contraction of an extremal ray $R\subset \NE(X/Z)$.
    Then for each $\RR$-Cartier $\RR$-divisor $M$ on $X$, $M\cdot R = 0$
    if and only if $M \sim_{\RR} f^*N$ for an $\RR$-Cartier $\RR$-divisor $N$ on $Y$.
\end{lemma}

\begin{proof}
    Assume that $M$ is an $\RR$-Cartier $\RR$-divisor on $X$ satisfying $M\cdot R = 0$.
    As in \cite[Definition~11.6]{lyu-takumi-2026relativeminimalmodelprogram}, we can write
    $ R = \RR^{\ge 0}\cdot [C] $
    for an integral scheme of dimension one $C\subset X$ that is contracted by $f$.  Write
    \[ M = \sum_{i=1}^s a_i M_i, \]
    where $a_i\in \RR$, and each $M_i$ is a Cartier divisor on $X$.
    Let $A$ be an ample$/Z$ Cartier divisor on $X$.
    As $C$ is contracted by $f$, it is also contracted by $X\to Z$.  Hence
    $A\cdot C$ is a positive integer.  For each $1\le i\le s$, set
    \[ L_i \coloneqq M_i - \frac{M_i\cdot C}{A\cdot C} A. \]
    It is evident that $L_i$ is a $\QQ$-Cartier $\QQ$-divisor, and $L_i\cdot C = 0$.
    As $f$ is a good contraction, there exists a $\QQ$-Cartier $\QQ$-divisor $N_i$ on $Y$ so that
    \[ L_i\sim_{\QQ} f^* N_i. \] 
    By construction, we get
    \[ \sum_{i=1}^s a_i L_i = \sum_{i=1}^s a_i \bigg( M_i - \frac{M_i\cdot C}{A\cdot C} A \bigg)
    = M -  \frac{M\cdot C}{A\cdot C} A = M. \]
    Let 
    \[ N\coloneqq \sum_{i=1}^s a_i N_i, \]
    which is an $\RR$-Cartier $\RR$-divisor on $Y$.
    Then we can conclude that $M \sim_{\RR} f^* N$.
\end{proof}


\subsection{MMP for klt pairs over non-closed fields}
The Minimal Model Program (MMP) with scaling as in \cite{BCHM}
is established in \cite{lyu-takumi-2026relativeminimalmodelprogram}
for several categories of spaces.  For simplicity, we only recall the following results 
for varieties over fields of characteristic zero.  
We refer the readers to \cite{lyu-takumi-2026relativeminimalmodelprogram} for more details.
Note that \cite{lyu-takumi-2026relativeminimalmodelprogram} only deals with 
pairs whose boundary divisors have coefficients in $\QQ$.


\begin{theorem}[cf. \protect{\cite{BCHM, lyu-takumi-2026relativeminimalmodelprogram}}]\label{bchm-type-mmp}
    Let $k$ be a field of characteristic zero.
    Let $X\to Z$ be a projective surjective morphism of $k$-varieties, 
    where $X$ is normal and $\QQ$-factorial.
    Let $\Delta\ge 0$ be a $\QQ$-divisor such that $(X, \Delta)$ is klt.
    Let $A$ be a $\QQ$-divisor on $X$ such that 
    $A$ is a good scaling divisor for $(X, \Delta)$, e.g.
    $A$ is ample$/Z$ and $K_X+\Delta +A$ is nef$/Z$.
    Assume that $c K_X + \Delta$ is big$/Z$ for some $c\in (-\infty, 1]\cap \QQ$.  Then
    \begin{itemize}
        \item if $K_X+\Delta$ is pseudo-effective$/Z$, then any process of MMP on $K_X+\Delta$ with scaling of $A$ over $Z$ terminates with a good minimal model over $Z$, and
        \item if $K_X+\Delta$ is not pseudo-effective$/Z$, then any process of MMP on $K_X+\Delta$ with scaling of $A$ over $Z$ terminates with a Mori fibre space over $Z$.
    \end{itemize}
\end{theorem}

\begin{proof}
    This follows from \cite[Corollary~21.9]{lyu-takumi-2026relativeminimalmodelprogram}
    except the semiampleness of minimal model if $K_X+\Delta$ is pseudo-effective$/Z$.
    To this end, let $\phi\colon X\bir Y$ be the sequence of MMP$/Z$ on $K_X+\Delta$,
    and let $\Delta_Y\coloneqq \phi_* \Delta$.
    Then by \cite[Corollary~21.9]{lyu-takumi-2026relativeminimalmodelprogram},
    $K_Y+\Delta_Y$ is nef$/Z$.  
    By \cite[Lemma~19.2, Lemma~20.3]{lyu-takumi-2026relativeminimalmodelprogram},
    $Y$ is $\QQ$-factorial; moreover,
    by \cite[Lemma~20.4]{lyu-takumi-2026relativeminimalmodelprogram},
    $(Y, \Delta_Y)$ is klt.  By Lemma~\ref{no-van-big-divisor},
    $cK_Y + \Delta_Y$ is big$/Z$.
    For the semiampleness of $K_Y + \Delta_Y$, we can assume that $Z$ is affine. 
    Then by \cite[Lemma~17.2]{lyu-takumi-2026relativeminimalmodelprogram}, we can write
    \[ K_Y+\Gamma + A \sim_{\QQ} e(K_Y + \Delta_Y) \]
    for some $e\in \QQ^{>0}$, where
    \begin{itemize}
        \item $A$ is an ample$/Z$ $\QQ$-divisor on $Y$, and
        \item $\Gamma$ is a $\QQ$-divisor such that $(Y, \Gamma)$ is klt.
    \end{itemize}
    Then by \cite[Corollary~17.8 ]{lyu-takumi-2026relativeminimalmodelprogram},
    $K_Y+\Gamma+A$ is semiample$/Z$.  Hence
    $(Y, \Delta_Y)$ is a good minimal model of $(X,\Delta)$.
\end{proof}


\subsection{MMP for klt pairs with real coefficients}\label{ss-big-real-boundary}
Let $k$ be a field of characteristic zero.
Let $f\colon X\to Z$ be a projective morphism between $k$-varieties,
where $X$ is normal and $\QQ$-factorial.
Let $\Delta, A\ge 0$ be $\RR$-divisors on $X$ such that
\begin{itemize}
    \item $(X, \Delta + A)$ is klt,
    \item $\Delta, A$ are big$/Z$, and
    \item $K_X+\Delta + A$ is nef$/Z$.
\end{itemize}
We show that a $(K_X+\Delta)$-MMP over $Z$ with scaling of $A$ exists, and every such MMP$/Z$ terminates. 
Moreover, if $K_X+\Delta$ is pseudo-effective$/Z$, the MMP$/Z$ terminates with a 
\emph{good minimal model} of $K_X+\Delta$ over $Z$;
if $K_X+\Delta$ is not pseudo-effective$/Z$, 
then the MMP$/Z$ terminates with a Mori fibre space of $K_X+\Delta$ over $Z$.

\medskip
(1)
We always fix a norm $\norm{\cdot}$ on $\N^1(X/Z)_{\RR}$, and we denote by $\norm{\cdot}^*$ the dual norm
on the space of linear functionals of $\N^1(X/Z)_{\RR}$.
Denote by $\Nef(X/Z)$ the \emph{relative nef cone} of $X\to Z$; see \cite[Definition~4.13]{lyu-takumi-2026relativeminimalmodelprogram}.
For \emph{supporting hyperplanes} of $\Nef(X/Z)$, see \cite[Definition~4.20]{lyu-takumi-2026relativeminimalmodelprogram}.
Let $H$ be a supporting hyperplane of $\Nef(X/Z)$;
for the \emph{extremal ray dual to $W$},
see \cite[Definition~4.20, Remark~4.22]{lyu-takumi-2026relativeminimalmodelprogram}.
Denote by $\Amp(X/Z)$ the \emph{relative ample cone} of $X\to Z$; see
\cite[Definition~4.13]{lyu-takumi-2026relativeminimalmodelprogram}.  
Then by \cite[Theorem~4.16]{lyu-takumi-2026relativeminimalmodelprogram}, we have
\[ \Amp(X/Z) = \interior \big(\Nef(X/Z)\big). \]
By \cite[Remark~4.17]{lyu-takumi-2026relativeminimalmodelprogram}, $\Amp(X/Z)$
is always open in $\N^1(X/Z)_{\RR}$, and $\Amp(X/Z)$ $\RR$-linearly spans $\N^1(X/Z)_{\RR}$.

\medskip
(2)
Let $B$ be an $\RR$-divisor on $X$ such that $B$ is big$/Z$ and $(X, B)$ is klt.
By \cite[Corollary~5.12]{lyu-takumi-2026relativeminimalmodelprogram},
we can write
\[ B\sim_{\RR} H + E/Z, \]
where $E\ge 0$, and $H$ is ample$/Z$.  Pick $0<\epsilon\ll 1$ sufficiently small, and let
\[ B_{\epsilon} \coloneqq (1-\epsilon)B + \epsilon E. \]
Then $(X, B_{\epsilon})$ is also klt.  Let $\Gamma$ be a $\QQ$-divisor so that
\[ 0< \Gamma \le B_{\epsilon}. \]
It is evident that $(X, \Gamma)$ is also klt.  
Taking $B_{\epsilon}-\Gamma$ sufficiently small, we can assume that
\[ \epsilon H + (B_{\epsilon}-\Gamma)~\mbox{is ample}/Z. \]
Thus, we can write
\begin{align*}
    K_X+B &\sim_{\RR} K_X+ \Gamma + \epsilon H + ((B_{\epsilon} - \Gamma))/Z \\
               &\sim_{\RR} K_X + \Gamma + H_{\Gamma}/Z,
\end{align*}
where $H_{\Gamma}\sim_{\RR} \epsilon H + (B_{\epsilon}-\Gamma)/Z$ is ample$/Z$.
It is evident that $(X, \Gamma)$ is klt.
Hence the class of $K_X+B$ in $\N^1(X/Z)$ belongs to the subset $\mcA(X/Z)$ 
in \cite[Theorem~18.2]{lyu-takumi-2026relativeminimalmodelprogram}.

\medskip
(3) [Existence of contraction]
Let
\begin{equation}\label{nef-lambda-defn}
    \lambda\coloneqq \inf\Set{ t\in \RR^{\ge 0}\mid K_X+\Delta + t A ~\mbox{is nef}/Z }.
\end{equation}
We assume that $\lambda>0$.
As $\Delta + \lambda A$ is big$/Z$, by (2), the class $[K_X+\Delta + \lambda A]$
in $\N^1(X/Z)_{\RR}$ belongs to $\mcA(X/Z)$.  Hence
\[ \mb{u}\coloneqq [K_X+\Delta + \lambda A]\in \mcA(X/Z)\cap \Nef(X/Z). \]
Then by \cite[Theorem~18.2]{lyu-takumi-2026relativeminimalmodelprogram},
there exists a closed convex \emph{rational} polytope $\mcP\subset \N^1(X/Z)_{\RR}$ such that
the following are satisfied.
\begin{itemize}
    \item [(i)] $[K_X+\Delta + \lambda A]$ is contained in the interior $\interior(\mcP)$ of $\mcP$.
    \item [(ii)] $\mcP\cap \Nef(X/Z)$ is a closed convex \emph{rational} polytope with nonempty interior.
    \item [(iii)] Let $F_1,\dots, F_m$ be the codimension one faces of $\mcP\cap \Nef(X/Z)$ that intersects $\interior(\mcP)$.  Then
    \begin{itemize}
        \item [$\bullet$] each $F_i$ spans a supporting hyperplane of $\Nef(X/Z)$, and
        \item [$\bullet$] $\interior(\mcP) \cap \partial \Nef(X/Z) = \interior(\mcP)\cap (F_1\cup F_2\cup \cdots \cup F_m)$.
    \end{itemize}
    \item [(iv)] Every compact subset of $\mcA(X/Z)\cap \partial\Nef(X/Z)$ is contained in a finite union of supporting hyperplanes.
    \item [(v)] Let $D$ be a $\QQ$-Cartier divisor on $X$ such that $[D]\in \mcA\cap \Nef(X/Z)$.  Then $D$ is semiample$/Z$.
\end{itemize}
By (ii), we have
\[ \mcP\cap \Nef(X/Z) = \bigcap_{j=1}^n H_j^+, \]
where each $H_j$ is a \emph{rational} hyperplane
\[ H_j\coloneqq \Set{v\in \N^1(X/Z)_{\RR}\mid a_j(v) = b_j} \]
for a nonzero \emph{rational} linear functional $a_j$ on $\N^1(X/Z)_{\RR}$ and $b_j\in \QQ$, and
\[ H_j^+ \coloneqq \Set{v\in \N^1(X/Z)_{\RR}\mid a_j(v) \ge b_j}. \]
As $\mb{u}\in \mcP\cap \Nef(X/Z)$, we have
\[ d_j \coloneqq a_j(\mb{u}) -b_j \ge 0. \]
Set
\[ J_{\mb{u}}\coloneqq \Set{1\le j\le n\mid d_j = 0}. \]
Then for each $j\in J_{\mb{u}}$, 
\[ H_j \coloneqq \Set{v\in \N^1(X/Z)_{\RR}\mid a_j(v) = b_j} \]
is a hyperplane containing $\mb{u}$.
Let $r(\mb{u}, \mcP)>0$ be the distance between $\mb{u}$ and $\partial \mcP$, and let
\[ r\coloneqq \frac{1}{2}\min\Set{r(\mb{u}, \mcP), \frac{d_j}{\norm{a_j}^*}~\mbox{with}~j\not\in J_{\mb{u}}}. \]
Let 
\[ U_0 \coloneqq B(\mb{u}, r) \]
be the open ball centred at $\mb{u}$ with radius $r$.  It is evident that $U_0\subset \mcP$.
Moreover, if $j\not\in J_{\mb{u}}$, and if $v\in U_0$, then
\[ a_j(v) - b_j = d_j + a_j(v-\mb{u}) \ge d_j - \norm{a_j}^*\cdot \norm{v-\mb{u}} > \frac{d_j}{2} >0. \]
Thus, for each $j\not\in J_{\mb{u}}$, we get
\[ U_0 \subset \interior(H_j^+). \]
On the other hand, for each $j\in J_{\mb{u}}$, $\mb{u}\in H_j$.
Hence $H_j$ is a codimension one face of $\mcP\cap \Nef(X/Z)$ that intersects $\interior(\mcP)$.
By (iii) as above, each such an $H_j$ is a supporting hyperplane of $\Nef(X/Z)$.
Thus, for such an $H_j$, we have $b_j = 0$.
Indeed, as $0\in \Nef(X/Z)$, we have $b_j\le 0$;
on the other hand, as $\mb{u}\in \Nef(X/Z)$, $2\mb{u}\in \Nef(X/Z)$,
which implies that $2b_j = a_j(2\mb{u}) \ge b_j$, i.e. $b_j\ge 0$.  Therefore, we can conclude that
\[ U_0\cap \Nef(X/Z) = U_0\cap \bigcap_{j\in J_{\mb{u}}} H_j^+,~\mbox{and}~\mb{u}\in H_j~
\mbox{for each}~j\in J_{\mb{u}}. \]
By \eqref{nef-lambda-defn}, for $0<\mu \ll 1$ sufficiently small, we have
\[ \mb{v}_{\mu} \coloneqq \mb{u}-\mu [A] \in U_0~~\mbox{and}~~\mb{v}_{\mu}\not\in \Nef(X/Z).  \]
Then there exists a $j\in J_{\mb{u}}$ so that $\mb{v}_{\mu}$ satisfies that
\[ a_j(\mb{v}_{\mu}) <0. \]
As $a_j(\mb{u}) = 0$, we get
\[ a_j([A]) > 0. \]
Let $R\subseteq \NE(X/Z)$ be the extremal ray dual to $H_j$.  Then
\[ (K_X+\Delta + \lambda A)\cdot (R\setminus \Set{0}) = 0~~\mbox{and}~~
A\cdot(R\setminus\Set{0}) > 0. \]
By \cite[Lemma~18.4]{lyu-takumi-2026relativeminimalmodelprogram}, $R$ has a \emph{good} contraction
\[ \phi\colon X\to Y~\mbox{over}~Z,\]
where $Y\to Z$ is projective.  By \cite[Lemma~19.1]{lyu-takumi-2026relativeminimalmodelprogram}, 
exactly one of the following holds.
\begin{itemize}
    \item $\dim X > \dim Y$.
    \item $\phi$ is birational, $\Ex(\phi)$ is a prime divisor, i.e. $\phi$ is divisorial.
    \item $\phi$ is birational, $\Ex(\phi)\subset X$ is of codimension $\ge 2$, i.e. $\phi$ is small.
\end{itemize}
Moreover, by \cite[Lemma~19.2]{lyu-takumi-2026relativeminimalmodelprogram}, if $\dim X>\dim Y$,
or if $\phi$ is divisorial, then $Y$ is $\QQ$-factorial.

\medskip
(4) [Existence of flips]
Assume that the contraction $\phi\colon X\to Y$ is small.
Pick a $\QQ$-divisor $\Delta'$ sufficiently close to $\Delta$ so that
\[ (K_X+\Delta')\cdot (R\setminus \Set{0}) < 0. \]
By \cite[Theorem~20.2]{lyu-takumi-2026relativeminimalmodelprogram}, there is a $(K_X+\Delta')$-flip
\[ \phi^+\colon X^+\to Y. \]
Moreover, by \cite[Lemma~20.3]{lyu-takumi-2026relativeminimalmodelprogram}, $X^+$ is $\QQ$-factorial.

We show that $\phi^+$ is also a $(K_X+\Delta)$-flip.
Let $C\subset X$ be an integral closed subscheme of dimension one such that $C$ is contracted by $\phi$, and
\[ R = \RR^{\ge 0}\cdot [C], \]
where $[C]$ is the class of $C$ in $\NE(X/Z)$.  Then we can write
\[ a\coloneqq \frac{(K_X+\Delta)\cdot C}{(K_X+\Delta')\cdot C} >0. \]
Hence
\[ \big((K_X+\Delta) - a(K_X+\Delta')\big)\cdot R = 0. \]
As $\phi\colon X\to Y$ is a good contraction, by Lemma~\ref{good-cont-real}, there exists 
an $\RR$-Cartier $\RR$-divisor $L_Y$ on $Y$ so that
\begin{equation}\label{flip-rational-real}
    K_X+\Delta \sim_{\RR} a(K_X+\Delta') + \phi^* L_Y.
\end{equation}
Let $\Delta^+, \Delta'^+$ be the pushdowns of $\Delta,\Delta'$ to $X^+$ respectively.  
By \eqref{flip-rational-real}, we get
\begin{equation}\label{flip-to-real}
    K_{X^+} + \Delta^+ \sim_{\RR} a(K_{X^+} + \Delta'^+) + (\phi^+)^* L_Y.
\end{equation}
By assumption, $K_{X^+} + \Delta'^+$ is ample$/Y$.  Then \eqref{flip-to-real} shows that
\[ K_{X^+} + \Delta^+~\mbox{is ample}/Y. \]
Therefore, we deduce that $\phi^+\colon X^+\to Y$ is the flip of $K_X+\Delta$.

\medskip
(5) [Descent of nefness]
Keep the notation in (3).  If $\phi\colon X\to Y$ is divisorial, let $\Delta_Y, A_Y$
be the pushdowns of $\Delta, A$ to $Y$ respectively.  Then Lemma~\ref{good-cont-real} shows that
\[ K_X+\Delta + \lambda A = \phi^*(K_Y + \Delta_Y + \lambda A_Y). \]
Hence $K_Y + \Delta_Y + \lambda A_Y$ is nef$/Z$.
By Lemma~\ref{no-van-big-divisor}, $\Delta_Y, A_Y$ are big$/Z$.

On the other hand, if $\phi\colon X\to Y$ is a flipping contraction, let $\phi^+\colon X^+ \to Y$
be the flip as in (4).  Denote by $\Delta^+, A^+$ the pushdowns of $\Delta, A$ to $X^+$.
By Lemma~\ref{good-cont-real}, there exists an $\RR$-Cartier $\RR$-divisor $N_Y$ on $Y$ so that
\[ K_X+\Delta+\lambda A = \phi^*N_Y. \]
Hence
\[ K_{X^+} + \Delta^+ + \lambda A^+ = (\phi^+)^*N_Y, \]
which implies that $K_{X^+} + \Delta^+ + \lambda A^+$ is nef$/Z$.
By Lemma~\ref{no-van-big-divisor}, $\Delta^+, A^+$ are big$/Z$.

\medskip
(6) [Termination of MMP]
By (3) (4) and (5) as above, the MMP$/Z$ on $K_X+\Delta$ with scaling of $A$ exists.
We show that every such MMP$/Z$ terminates.  

\medskip
(6.1)
We first assume that $Z$ is affine.  Let
\begin{equation}\label{MMP-sequence}
    X = X_0\bir X_1\bir \cdots \bir X_i \bir \cdots
\end{equation}
be the sequence of MMP$/Z$ with scaling of $A$.  
By (3) and (4), each $X_i$ is $\QQ$-factorial.
Let $\Delta_i, A_i$
be the pushdowns of $\Delta, A$ to $X_i$ respectively, and let
\[ \lambda_i \coloneqq \inf \Set{t\in \RR^{\ge 0} \mid K_{X_i} + \Delta_i + tA_i~\mbox{is nef}/Z }. \]
Then $\Set{\lambda_i}$ is a decreasing sequence of real numbers.
By \cite[Definition~11.6]{lyu-takumi-2026relativeminimalmodelprogram}, 
we can assume that \eqref{MMP-sequence} consists of only flips.

Consider the following classes in $\N^1(X_0/Z)$,
\[ \mb{u}\coloneqq [K_{X_0} + \Delta_0]~~\mbox{and}~~
\mb{v} \coloneqq [K_{X_0} + \Delta_0 + \lambda_0 A_0]. \]
Denote by $\ell_{\mb u,\mb v}$ the line segment in $\N^1(X_0/Z)_{\RR}$ connecting $\mb u$ and $\mb v$, i.e.
\[ \ell_{\mb u, \mb v} \coloneqq \Set{[K_{X_0}+\Delta_0+tA_0]\mid 0\le t\le \lambda_0}. \]
By (2), we can write
\begin{align*}
    &K_{X_0} + \Delta_0 \sim_{\RR} K_{X_0} + \Gamma_{\mb u} + H_{\mb u}/Z,~\mbox{and} \\
    &K_{X_0} + \Delta_0 + \lambda_0 A_0 \sim_{\RR} K_{X_0} + \Gamma_{\mb v} + H_{\mb v},
\end{align*}
where $(X_0, \Gamma_{\mb u}), (X_0, \Gamma_{\mb v})$ are klt, $\Gamma_{\mb u}, \Gamma_{\mb v}$
are $\QQ$-divisors, and $H_{\mb u}, H_{\mb v}$
are ample$/Z$.  As $\Amp(X_0/Z)$ is open in $\N^1(X_0/Z)_{\RR}$,
we can pick full-dimensional \emph{rational} polytopes $\mcQ_{\mb u}, \mcQ_{\mb v}$ such that
\begin{align*}
    \mb{u}\in \interior \mcQ_{\mb u}, &~\mcQ_{\mb u}\subset [K_{X_0} + \Gamma_{\mb u}] + \Amp(X/Z), ~\mbox{and} \\
    \mb{v}\in \interior \mcQ_{\mb v}, &~\mcQ_{\mb v}\subset [K_{X_0} + \Gamma_{\mb v}] + \Amp(X/Z).
\end{align*}
Denote by $\mcQ$ the convex hull of $\mcQ_{\mb u}$ and $\mcQ_{\mb v}$,
which is also a full-dimensional \emph{rational} polytope.  It is evident that
\[ \ell_{\mb u,\mb v}\subset \interior\mcQ. \]
Denote by 
\[ L_1, \dots, L_N \]
the $\QQ$-divisors so that $\Set{[L_1], \dots, [L_N]}$ 
is the set of vertices of $\mcQ$.  Then by \cite[Theorem~17.3]{lyu-takumi-2026relativeminimalmodelprogram}, 
the relative adjoint ring
\[ \mcR_0 \coloneqq R(X_0/Z; L_1, \dots, L_N) \]
is a finitely generated $\hg^0(Z, \mcO_Z)$-algebra.

Denote by $L_{j,i}$ the birational transform of $L_j$ to $X_i$, and by
\[ \mcR_i \coloneqq R(X_i/Z; L_{1,i}, \dots, L_{N,i}). \]
Since $X_0\bir X_i$ is small, for each sequence of natural numbers $(m_1, \dots, m_N)$, we have
\[ \hg^0(X_0, \mcO_{X_0}( \rddown{m_1L_1 + \cdots + m_N L_N} ))
= \hg^0(X_i, \mcO_{X_i}( \rddown{m_1L_{1,i}+\cdots + m_N L_{N,i}} )). \]
This gives a canonical isomorphism of $\hg^0(Z, \mcO_Z)$-algebras
\[ \mcR_i\cong \mcR_0. \]
In particular, $\mcR_i$ is a finitely generated $\hg^0(Z, \mcO_Z)$-algebra as $\mcR_0$ is so.

\medskip
(6.2)
Denote by $\Div_{\RR}(X_0)$ the group of $\RR$-Cartier $\RR$-divisors on $X_0$,
and by $\Div_{\RR}^{\eff}(X_0)$ the set of $\RR$-Cartier $\RR$-divisors on $X_0$
that are $\RR$-linearly equivalent to effective $\RR$-divisors;
see \cite[page~619]{CL13-new-outlook} and \cite[Definition~3.1]{lyu-takumi-2026relativeminimalmodelprogram}.
The 
\emph{support of $\mcR_0$} is defined to be
\[ \Supp \mcR_0 \coloneqq \bigg(\sum_{j=1}^N \RR^{\ge 0} L_j\bigg)\cap \Div_{\RR}^{\eff}(X_0) \subseteq \Div_{\RR}(X_0). \]
For more details, see \cite[Definition~8.2]{lyu-takumi-2026relativeminimalmodelprogram}.
Let
\[ \mathbb{D}\colon \RR^N\to \Div_{\RR}(X_0), ~(\alpha_1,\dots, \alpha_N)\mapsto \sum_{j=1}^N \alpha_j L_j \]
be the tautological map from \cite[page~620]{CL13-new-outlook}.
By \cite[Theorem~17.5]{lyu-takumi-2026relativeminimalmodelprogram}, 
there is a finite \emph{rational} polyhedral subdivision
\[ \Supp \mcR_0 = \bigsqcup_{l} \mcC_l \]
such that for each geometric valuation $v$ of $X_0$, 
$o_{v, X_0}$ is a linear function on each $\mcC_l$,
where $o_{v, X_0}$ denotes the \emph{asymptotic order of vanishing}; 
see \cite[Definition~9.2]{lyu-takumi-2026relativeminimalmodelprogram}.
Moreover, we can assume that the subdivision is a \emph{coarsest} subdivision
in the sense that if $\mcC_l, \mcC_{l'}$ are distinct, 
then there is at least one geometric valuation $v$ of $X_0$ such that
the linear extensions of $o_{v, X_0}|_{\mcC_l}, o_{v, X_0}|_{\mcC_{l'}}$ are different;
see \cite[Theorem~17.5(iii)]{lyu-takumi-2026relativeminimalmodelprogram}.

Denote by $\phi_i\colon X_0\bir X_i$ the birational map over $Z$
induced by flips in \eqref{MMP-sequence}.  Then $\phi_i$ is a small birational contraction.
Thus, the pushdown of divisors $(\phi_i)_*$ gives a \emph{rational} linear isomorphism
\[ \tau_i\colon \sum_{j=1}^N \RR^{\ge 0} L_j \xrightarrow{\sim} \sum_{j=1}^N \RR^{\ge 0} L_{j,i}. \]
Again since $\phi_i$ is small, $(\phi_i)_*$ preserves effectivity and $\RR$-linear equivalence, so
$(\phi_i)_*$ gives an isomorphism of cones
\[ (\phi_i)_*\colon \Div_{\RR}^{\eff}(X_0)\xrightarrow{\sim} \Div_{\RR}^{\eff}(X_i). \]
Therefore, $\tau_i$ restricts to an isomorphism
\[ \Supp \mcR_0 \xrightarrow{\sim} \Supp \mcR_i. \]
For each $\mcC_l$, let
\[ \mcC_{l,i}\coloneqq \tau_i(\mcC_l). \]
Then we have
\begin{equation}\label{subdivision-X-i}
    \Supp \mcR_i = \bigsqcup_l \mcC_{l, i}.
\end{equation}
As $\tau_i$ is a \emph{rational} linear isomorphism, each $\mcC_{l,i}$
is a \emph{rational} polyhedral cone.  This gives a finite \emph{rational} polyhedral subdivision of $\Supp \mcR_i$.

\medskip
(6.3)
We claim that for each geometric valuation $v$ of $X_i$, $o_{v, X_i}$ is linear on $\mcC_{l, i}$.
Let $p\colon W\to X_0, q\colon W\to X_i$ be a common resolution so that $v$ is the geometric valuation
corresponding to a prime divisor $E$ on $W$.
Pick an arbitrary $D\in \mcC_l$, and let $B\in \abs{D}_{\RR}$.  We can write
\[ B - D = \sum c_{\beta} \Div_{X_0}(g_{\beta}), \]
where each $g_{\beta}$ is a rational function on $X_0$.  Then
\begin{align*}
    q^*(\tau_i B) - p^*B &= q^*\bigg(\tau_i D + \sum c_{\beta} \Div_{X_i}(g_{\beta})\bigg) 
        - p^*\bigg(D + \sum c_{\beta}\Div_{X_0}(g_{\beta})\bigg) \\
     &= \bigg(q^*(\tau_i D) + \sum c_{\beta} \Div_W(g_{\beta})\bigg) - \bigg(p^*D + \sum c_{\beta}\Div_W(g_{\beta})\bigg) \\
     &= q^*(\tau_i D) - p^*D.
\end{align*}
Moreover, it is evident that 
\[ \ord_E(p^* B) = \ord_E(p^*B - q^*(\tau_i B)) + \ord_E(q^*(\tau_i B)). \]
Thus, we deduce that
\begin{equation}\label{order-B-D-resolution}
    \ord_E(p^*B) = \ord_E(q^*(\tau_i B)) + \ord_E(p^* D- q^*(\tau_i D)).
\end{equation}
As $\phi_i$ is a small birational map, the pushdown $(\phi_i)_*$ gives the canonical isomorphism
\[ \abs{D}_{\RR} \cong \abs{\tau_i D}_{\RR}. \]
Therefore, taking the infimum among all $B\in \abs{D}_{\RR}$, we get
\begin{equation}\label{order-compare}
    o_{v, X_0}(D) = o_{v, X_i}(\tau_i D) + \ord_E\big(p^*D - q^*(\tau_i D)\big).
\end{equation}
Hence $o_{v, X_i}$ is linear on every $\mcC_{l,i}$.

\medskip
(6.4)
We claim the subdivision \eqref{subdivision-X-i} is coarsest.
For any distinct $\tau_i(\mcC_l)$ and $\tau_i(\mcC_{l'})$, $\mcC_l, \mcC_{l'}$
are also distinct.  As the subdivision of $\Supp \mcR_0$ is coarsest,
there is a geometric valuation $v$ of $X_0$ so that the linear extensions 
of $o_{v, X_0}|_{C_l}, o_{v, X_0}|_{C_{l'}}$ are different.
Then by \eqref{order-compare}, the linear extensions 
of $o_{v, X_i}|_{\tau_i(\mcC_l)}$ and $o_{v, X_i}|_{\tau_i(\mcC_{l'})}$
are also distinct.  Hence the subdivision \eqref{subdivision-X-i} is coarsest.

\medskip
(6.5)
By \cite[Lemma~20.3]{lyu-takumi-2026relativeminimalmodelprogram}, there is a canonical isomorphism
\[ \N^1(X_0/Z)_{\RR}\cong \N^1(X_i/Z). \]
Let $\mcQ_i$ be the image of $\mcQ$ in $\N^1(X_i/Z)$. 
Then $\mcQ_i$ is the convex hull of $[L_{1,i}], \dots, [L_{N, i}]$.
In particular, $\mcQ_i$ is a full-dimensional \emph{rational} polytope.  Let
\[ \mb{v}_i \coloneqq [K_{X_i} + \Delta_i + \lambda_i A_i]\in \N^1(X_i/Z)_{\RR}. \]
Since $\ell_{\mb u, \mb v}\subset \interior \mcQ$, and since $\lambda_i\le \lambda_0$, we see that
\[ \mb{v}_i \in \interior \mcQ_i. \]
As $\mb{v}_i$ is nef$/Z$, adding a sufficiently small ample class to $\mb{v}_i$ gives an ample class in $\interior \mcQ_i$.  Thus,
\[ \interior\mcQ_i\cap \ADiv(X_i/Z) \neq \emptyset, \]
where $\ADiv(X_i/Z)\subset\Div_{\RR}(X_i)$ is the subset of ample$/Z$ $\RR$-divisors.
Consequently, we have
\begin{equation}\label{ample-nonempty}
    \Supp \mcR_i \cap \ADiv(X_i/Z)\neq \emptyset.
\end{equation}
Let
\[ \varphi_i\colon \sum_{j=1}^N \RR \cdot L_{j, i} \to \N^1(X_i/Z)_{\RR} \]
be the natural projection map, which sends an $\RR$-Cartier $\RR$-divisor 
to its class in $\N^1(X_i/Z)_{\RR}$.
By \cite[Lemma~18.1]{lyu-takumi-2026relativeminimalmodelprogram},
there exists a $\mcC_{l, i}$ in the subdivision \eqref{subdivision-X-i} so that
\begin{equation}\label{nef-chamber}
    \mcC_{l, i} = \Supp \mcR_i \cap \varphi_i^{-1}\big(\Nef(X_i/Z)\big).
\end{equation}
Here we have applied the property that \eqref{subdivision-X-i} is coarsest; see (6.4).

\medskip
(6.6)
Assume that \eqref{MMP-sequence} is an infinite sequence of flips.
Then by (6.2) and (6.5), there exist an $l$ and $i<j$ so that the small birational contraction
\[ \phi_{ij}\colon X_i\bir X_j \]
induces an isomorphism from $\mcC_{l, i}$ to $\mcC_{l, j}$, and that
\begin{equation}\label{nef-chambers-i-j}
    \mcC_{l, i} = \Supp \mcR_i \cap \varphi_i^{-1}\big(\Nef(X_i/Z)\big) ~~\mbox{and}~~
    \mcC_{l, j} = \Supp \mcR_j \cap \varphi_j^{-1}\big(\Nef(X_j/Z)\big).
\end{equation}
By \eqref{ample-nonempty} and \eqref{nef-chamber}, we can pick an effective \emph{rational}
ample$/Z$ $H_i$ on $X_i$ so that
\[ H_i\in \mcC_{l,i}. \]
Let $H_j$ be the pushdown of $H_i$ to $X_j$.
As $\mcC_{l, i}$ is isomorphic to $\mcC_{l,j}$, by \eqref{nef-chambers-i-j},
$H_j$ is nef$/Z$.
By \cite[Lemma~20.10]{lyu-takumi-2026relativeminimalmodelprogram}, $\phi_{ij}^{-1}\colon X_j\bir X_i$
is a morphism.  On the other hand, the same argument for the small birational contraction
shows that $\phi_{ij}$ is also a morphism.  
Hence we deduce that $\phi_{ij}$ is an isomorphism.
However, an infinite sequence of flips increases log discrepancies of some 
divisors strictly, which contradicts that $\phi_{ij}$ is an isomorphism.
Therefore, we can conclude that the sequence \eqref{MMP-sequence} terminates if $Z$ is affine.

\medskip
(6.7)
Assume that $Z$ is not affine, and that \eqref{MMP-sequence}
is an infinite sequence of flips over $Z$.
Pick a finite affine open covering of $Z$.
Let $U$ be one of the affine subset, and let 
\[ X_{i,U} \coloneqq X_i\times_Z U. \]
Let $L_U$ be an arbitrary Cartier divisor on $X_{i, U}$.
Taking the closure of irreducible components of $L_U$, we get a Weil divisor $L$ on $X_i$ 
whose restriction to $X_{i, U}$ is equal to $L_U$.  
As $X_i$ is $\QQ$-factorial, $L$ is $\QQ$-Cartier.  Thus, the canonical map
\[ \pic(X_i)_{\QQ}\to \pic(X_{i, U})_{\QQ} \]
is surjective.  Then by \cite[Lemma~20.8]{lyu-takumi-2026relativeminimalmodelprogram},
we can also assume that 
\[ X_{0, U}\bir X_{1, U}\bir \cdots \bir X_{i, U}\bir\cdots \]
is an infinite sequence of flips over $U$.
However, by (6.6), this sequence of flips over $U$ terminates.
Thus, we see that the sequence of flips over $Z$ also terminates.       

\medskip
(7) [Outcomes of MMP]
By (6), the MMP$/Z$ on $K_X+\Delta$ with scaling of $A$ terminates with a model $Y\to Z$.
Denote by $\Delta_Y$ the pushdown of $\Delta$ to $Y$.

\medskip
(7.1)
Assume that $K_X+\Delta$ is not pseudo-effective$/Z$.
By Lemma~\ref{neg-lemma}, $K_Y + \Delta_Y$ is also non-pseudo-effective$/Z$.
Thus, $Y\to Z$ admits a Mori fibre space structure $Y\to V$ over $Z$,
i.e. $\dim Y > \dim V$, and $-(K_Y + \Delta_Y)$ is ample$/V$.

\medskip
(7.2)
Assume that $K_X+\Delta_X$ is pseudo-effective$/Z$.
By \cite[Lemma~5.17]{lyu-takumi-2026relativeminimalmodelprogram},
$K_Y + \Delta_Y$ is pseudo-effective$/Z$.
As the MMP$/Z$ terminates with $Y$, $K_Y + \Delta_Y$ is nef$/Z$.
By Lemma~\ref{no-van-big-divisor}, $\Delta_Y$ is big$/Z$.
By \cite[Corollary~5.12]{lyu-takumi-2026relativeminimalmodelprogram}, we can write
\[ \Delta_Y\sim_{\RR} H+E/Z, \]
where $E\ge 0$, and $H$ is an ample$/Z$ $\RR$-divisor.
By Lemma~\ref{discrep-mmp}, $(Y, \Delta_Y)$ is klt.  For sufficiently small $0<\epsilon\ll 1$, set
\[ \Delta_Y^{\epsilon} \coloneqq (1-\epsilon)\Delta_Y + \epsilon E. \]
Then $(Y, \Delta_Y^{\epsilon})$ is also klt; see \cite[Lemma~6.10]{lyu-takumi-2026relativeminimalmodelprogram}.
Now we can write
\[ (K_Y + \Delta_Y) - (K_Y + \Delta_Y^{\epsilon}) = \epsilon \Delta_Y - \epsilon E \sim_{\RR} \epsilon H/Z. \]
By Theorem~\ref{BPF-thm}, $K_Y + \Delta_Y$ is semiample$/Z$.
Thus, $(Y, \Delta_Y)$ is a good minimal model of $(X, \Delta)$ over $Z$.


\subsection{MMP for g-pairs over non-closed fields}\label{existence-MMP-g-pair}
Let $k$ be a field of characteristic zero, and let
$X\to Z$ be a projective surjective morphism of $k$-varieties.
Let $(X/Z, B + \mb{M})$ be a $\QQ$-factorial g-lc g-pair. 
Assume that $K_X+B+\mb{M}_X + A$ is nef$/Z$ for some $\RR$-Cartier $\RR$-divisor $A\ge 0$
which is big$/Z$.  Let
\[ \lambda\coloneqq \inf\Set{t\ge 0 \mid K_X+B+\mb{M}_X + tA~\text{is nef}/Z}. \]
Assume $\lambda\neq 0$.  Consider the following condition (cf. \cite[\S 4]{B-Zhang}):
\begin{equation}\label{star-condition}\tag{$\ast$}
    \begin{aligned}
        &\mbox{for any}~s\in (0,\lambda),~\mbox{there is a boundary}~0\le \Delta\sim_{\RR}B + sA + \mb{M}_X/Z   \\
        &\mbox{such that}~(X, \Delta + (\lambda-s)A)~\mbox{is klt}.
    \end{aligned}
\end{equation}
Then for each $s\in (0, \lambda)$, we have
\[ \lambda - s = \inf\Set{u \ge 0 \mid K_X + \Delta + uA~\text{is nef}/Z}. \]
We show that we can run MMP$/Z$ on  $K_X+B+\mb{M}_X$ with scaling of $A$ under \eqref{star-condition}.


\medskip
(1) [Good contraction].
Pick an arbitrary $s\in (0,\lambda)$, 
and let $\Delta$ be a boundary as in \eqref{star-condition}.
As $A$ is big$/Z$, $\Delta$ is also big$/Z$.
Then by \S\ref{ss-big-real-boundary}(3), 
there exists an extremal ray $R\subseteq \NE(X/Z)$
with a \emph{good contraction} $X\to Y$ over $Z$ contracting $R$ such that
\[ (K_X + \Delta + (\lambda - s)A)\cdot R = 0~~\mbox{and}~~
   (K_X+\Delta)\cdot (R\setminus\Set{0}) < 0, \]
which is equivalent to 
\[ (K_X + B + \mb{M}_X + \lambda A)\cdot R = 0 ~~\mbox{and}~~
(K_X +B + \mb{M}_X + sA) \cdot (R\setminus\Set{0}) <0. \]
In particular, we get
\begin{equation}\label{A-positive}
    A\cdot (R\setminus \Set{0}) > 0.
\end{equation}
Thus, we deduce that
\begin{equation}\label{g-contraction}
    (K_X + B + \mb{M}_X + \lambda A)\cdot R = 0 ~~\mbox{and}~~
(K_X +B + \mb{M}_X) \cdot (R\setminus\Set{0}) <0.
\end{equation}
By \cite[Lemma~19.1]{lyu-takumi-2026relativeminimalmodelprogram}, exactly one of the following holds.
\begin{itemize}
    \item $\dim Y < \dim X$,
    \item $X\to Y$ is birational, $\Ex(X\to Y) \subset X$ is a prime divisor, and
    \item $X\to Y$ is birational, $\Ex(X\to Y)\subset X$ has codimension $\ge 2$, i.e. $X\to Y$ is small.
\end{itemize}


\medskip
(2) [Mori fibre space].
If $\dim Y<\dim X$, then we stop.  
In this case, by \cite[Lemma~19.2]{lyu-takumi-2026relativeminimalmodelprogram}, $Y$ is $\QQ$-factorial.
It is evident that $-(K_X+\Delta)$ is ample$/Y$.  Hence
$-(K_X+B+\mb{M}_X+sA)$ is ample$/Y$.  By \eqref{A-positive}, we can conclude that
\[ -(K_X+B+\mb{M}_X)~\mbox{is ample}/Y. \]
We call $X\to Y$ a \emph{Mori fibre space} of $K_X+B+\mb{M}_X$ over $Z$.


\medskip
(3) [Divisorial contraction].
If $X\to Y/Z$ is birational whose exceptional locus is a prime divisor, 
then by \cite[Lemma~19.2]{lyu-takumi-2026relativeminimalmodelprogram},
$Y$ is $\QQ$-factorial.  

\medskip
(3.1)
Denote by $A_Y, B_Y, \Delta_Y$ the pushdowns of $A,B, \Delta$ to $Y$.
Note that $A_Y$ is big$/Z$.  Moreover, as $K_X + B + \mb{M}_X + \lambda A$ is numerically trivial over $Y$, 
by Lemma~\ref{good-cont-real},
$K_Y + B_Y + \mb{M}_Y + \lambda A_Y$ is nef$/Z$.  Let
\[ \lambda_Y \coloneqq \inf\Set{t\ge 0 \mid K_Y+B_Y+\mb{M}_Y + tA_Y~\text{is nef}/Z}. \]
It is evident that $\lambda_Y \le \lambda$. 

\medskip
(3.2)
We claim that if $\lambda_Y\ne 0$, then the condition \eqref{star-condition} is preserved.
Pick an arbitrary $0<s'<\lambda_Y$.  By \eqref{star-condition}, 
there is a boundary $0\le \Delta'\sim_{\RR} B+\mb{M}_X + s'A/Z$
such that $(X, \Delta' + (\lambda-s')A)$ is klt.  
By Lemma~\ref{discrep-mmp}, as
\[ K_X + \Delta' + (\lambda-s')A\equiv 0/Y, \]
we see that $(Y, \Delta_Y' + (\lambda-s')A_Y)$ is klt,
where $\Delta_Y'$ is the pushdown of $\Delta'$ to $Y$.
As $\lambda_Y\le \lambda$, $(Y, \Delta_Y' + (\lambda_Y-s')A_Y)$ is also klt.
Hence \eqref{star-condition} is preserved.
Note that
\[ \lambda_Y-s' = \inf\Set{t\ge 0 \mid K_Y+\Delta_Y' + tA_Y~\text{is nef}/Z} >0. \]
Thus, $K_Y + \Delta_Y'$ is not nef$/Z$.  

\medskip
(3.3)
Now as \eqref{star-condition} is preserved, if $\lambda_Y\neq 0$, we go back to (1).


\medskip
(4) [Flipping contraction].
Assume that $X\to Y/Z$ is small.  
Then by \S\ref{ss-big-real-boundary}(4),  
there exists a $(K_X+\Delta)$-flip $X\dashrightarrow X^+$ over $Y$.  
By \cite[Lemma~20.3]{lyu-takumi-2026relativeminimalmodelprogram}, $X^+$ is $\QQ$-factorial.

\medskip
(4.1)
Let $A^+,B^+, \Delta^+$ be the pushdowns of $A,B, \Delta$ to $X^+$ respectively.
By \cite[Lemmas~5.13, 5.16]{lyu-takumi-2026relativeminimalmodelprogram}, $A^+$ is big$/Z$.
By \eqref{star-condition}, 
\[ K_{X^+} + B^+ + \mb{M}_{X^+} + s A^+\sim_{\RR} K_{X^+} + \Delta^+/Z~\mbox{is ample}/Y. \]
Moreover, as $X\to Y$ is a good contraction, by Lemma~\ref{good-cont-real}, we see that
\[ K_{X^+} + B^+ + \mb{M}_{X^+} + \lambda A^+\sim_{\RR}0/Y. \]
As $\lambda>s$, $-A^+$ is ample$/Y$.  Hence
\[ K_{X^+} + B^+ + \mb{M}_{X^+}~\mbox{is ample}/Y. \]
Thus, $X\dashrightarrow X^+/Y$ is the $(K_X+B+\mb{M}_X)$-flip over $Y$.  Let
\[ \lambda^+ \coloneqq \inf\Set{t\ge 0 \mid K_{X^+} + B^+ + \mb{M}_{X^+} + tA^+ ~\text{is nef}/Z }. \]
It is evident that $\lambda^+\le \lambda$.

\medskip
(4.2)
We show that if $\lambda^+\ne 0$, then \eqref{star-condition} is preserved.
Let $s'\in (0, \lambda^+)$.
Again by \eqref{star-condition}, there is a boundary $0\le \Delta'\sim_{\RR} B+\mb{M}_X + s' A/Z$
so that $(X, \Delta' + (\lambda-s')A)$ is klt.
Denote by $\Delta'^+$ the pushdown of $\Delta'$ to $X^+$.  
Then by Lemma~\ref{discrep-mmp},
$(X^+, \Delta'^+ + (\lambda-s')A^+)$ is also klt.  
As $\lambda^+\le \lambda$, $(X^+, \Delta'^+ + (\lambda^+-s')A^+)$ is klt.
Moreover, we have
\[ \Delta'^+\sim_{\RR} B^+ + \mb{M}_{X^+} + s' A^+/Z. \]
Hence the condition \eqref{star-condition} is preserved.

\medskip
(4.3)
Since \eqref{star-condition} is preserved, if $\lambda^+\neq 0$ we go back to (1).


\subsection{Termination of MMP for g-pairs}\label{ss-termination-g-pair}
The following result shows that \eqref{star-condition}
is satisfied if we start with a g-pair with mild singularities, and hence we can run the MMP.

\begin{theorem}[cf. \protect{\cite[Lemma~4.4]{B-Zhang}}]\label{mmp-g-pairs}
Let $k$ be a field of characteristic zero.
Let $(X, B + \mb{M})$ be a $\QQ$-factorial g-lc g-pair over a $k$-variety $Z$,
where $X\to Z$ is projective and surjective.
Let $A\ge 0$ be a general ample$/Z$ $\RR$-divisor on $X$ such that
$K_X+B+\mb{M}_X+A$ is nef$/Z$.  Assume either
\begin{itemize}
    \item [\emph{(i)}] $(X, B +\mb{M})$ is g-klt, or
    \item [\emph{(ii)}] $(X, B+\mb{M})$ is g-lc and $(X, 0)$ is klt.
\end{itemize}
Then 
\begin{itemize}
    \item [\emph{(1)}] the MMP$/Z$ on $K_X+B+\mb{M}_X$ with scaling of $A$ exists, 
    \item [\emph{(2)}] after finitely many steps, the MMP$/Z$ is a sequence of flips, and
    \item [\emph{(3)}] in each step of the MMP$/Z$, the resulting model is $\QQ$-factorial.
\end{itemize}
Moreover, 
\begin{itemize}
    \item [\emph{(a)}] if $K_X+B+\mb{M}_X$ is not pseudo-effective$/Z$, then any process of the MMP$/Z$ terminates with a Mori fibre space over $Z$, and
    \item [\emph{(b)}] if 
    \begin{itemize}
        \item [$\bullet$] $K_X+B+\mb{M}_X$ is pseudo-effective$/Z$,
        \item [$\bullet$] $(X, B+\mb{M})$ is g-klt, and 
        \item [$\bullet$] $K_X+(1+\alpha)B + (1+\beta)\mb{M}_X$ is big$/Z$ for some $\alpha,\beta\in \RR^{\ge 0}$, 
    \end{itemize}
    then any process of the MMP$/Z$ terminates with a good minimal model of $K_X+B+\mb{M}_X$.
\end{itemize}
\end{theorem}

\begin{proof}
    \emph{Step~1}.
    In this step, we show that the conditions \eqref{star-condition}
    in \S\ref{existence-MMP-g-pair} is satisfied 
    if (i) or (ii) holds, and consequently the results (1) (2) and (3) hold.
    Let
    \[ \lambda\coloneqq \inf\Set{t\ge 0\mid K_X+B+\mb{M}_X + tA}~\text{is nef}/Z. \]
    If $\lambda = 0$, $X\to Z$ is a minimal model of $K_X+B+\mb{M}_X$, so we can assume $\lambda > 0$.
    Note that as $K_X+B+\mb{M}_X+A$ is nef$/Z$, we have $\lambda\le 1$.
    
    Let $\phi\colon W\to X$ be a log resolution of $(X, B+\mb{M})$ so that $\mb{M}$ descends to 
    a nef$/Z$ $\RR$-divisor $\mb{M}_W$ on $W$.  Let $K_W+B_W+\mb{M}_W$
    be the pullback of $K_X+B+\mb{M}_X$.
    Let $A_W$ be the pullback of $A$ to $W$.  
    In case (i), for any $s\in (0,\lambda)$, $sA_W+\mb{M}_W$ is nef and big over $Z$.
    By \cite[Corollary~5.12]{lyu-takumi-2026relativeminimalmodelprogram}, we can write
    $sA_W+\mb{M}_W \sim_{\RR} H + E/Z$, where $H$ is ample$/Z$ and $E$ is effective, 
    and we can assume the coefficients of $E$ are arbitrarily small.
    By taking the coefficients of $E$ sufficiently small, and by 
    taking $H$ sufficiently general, we can assume that $(W, B_W+H+E)$ is sub-klt.
    Denote by $\Delta$ the pushdown of $B_W+H+E$ to $X$.
    Then $(X, \Delta)$ is klt, and $\Delta\sim_{\RR} B+sA+\mb{M}_X/Z$.
    In case (ii), as $(X, 0)$ is klt, for a sufficiently small $0<\epsilon\ll 1$, we can assume that
    \begin{itemize}
        \item $(X, (1-\epsilon)B + (1-\epsilon)\mb{M})$ is g-klt, and
        \item $sA +\epsilon (B + \mb{M}_X)$ is ample$/Z$.
    \end{itemize}
    Again by \cite[Corollary~5.12]{lyu-takumi-2026relativeminimalmodelprogram}, we can write
    \[ P_W\coloneqq (1-\epsilon)\mb{M}_W + \phi^*\big(sA + \epsilon(B+\mb{M}_X)\big) \sim_{\RR} H + E/Z, \]
    where the coefficients of $E\ge 0$ is sufficiently small, and $H$ is ample$/Z$.  We can assume that
    \[ (W, C_W + E)~\mbox{is sub-klt}, \]
    where $C_W$ is the $\RR$-divisor defined by 
    \[ K_W + C_W + P_W = \phi^*(K_X + B + \mb{M}_X + sA). \]
    Taking $H$ generally, we can assume 
    \[ (W, C_W + H + E)~\mbox{is sub-klt}. \]
    Let
    \[ \Delta \coloneqq \phi_*\big(C_W + H + E\big)\sim_{\RR} B + sA + \mb{M}_X/Z. \]
    Then $(X, \Delta)$ is klt.
    Therefore, in both (i) and (ii), for any $s\in (0,\lambda)$, 
    by taking a sufficiently general $A$, we can pick an $\RR$-divisor 
    $0\le \Delta\sim_{\RR} B+sA +\mb{M}_X/Z$ so that $(X, \Delta + (1-s)A)$ is klt.
    As $\lambda\le 1$, we see that 
    $(X, \Delta + (\lambda-s)A)$ is klt.
     Therefore, \eqref{star-condition} is satisfied.

    Since $0<s<\lambda$, $K_X+\Delta$ is not nef$/Z$.  As $A$ is ample$/Z$, we have
    \[ K_X + \Delta + A = (K_X+ \Delta + \lambda A) + (1-\lambda)A~\mbox{is nef}/Z. \]
    Then the construction and preservation argument in \S\ref{existence-MMP-g-pair} give a
    $(K_X+B+\mb{M}_X)$-MMP over $Z$ with scaling of $A$.
    Every resulting model is $\QQ$-factorial.
    Thus, the conclusions (1) and (3) hold.
    Finally, as a good contraction that is a divisorial contraction drops the relative Picard number
    (cf. \cite[Definition~11.6]{lyu-takumi-2026relativeminimalmodelprogram} and the proof of
    \cite[Theorem~21.7]{lyu-takumi-2026relativeminimalmodelprogram}),
    we see that after finitely many steps, the MMP$/Z$ is a sequence of flips, 
    which shows (2).

    \medskip
    \emph{Step~2}.
    In this step, we prove (a).  
    Assume that $K_X+B+\mb{M}_X$ is not pseudo-effective$/Z$.
    By Step~1, we can run the MMP$/Z$ on $K_X+B+\mb{M}_X$ with scaling of $A$.  Let
    \begin{equation}\label{mmp-sequence}
        X=X_1\xdashrightarrow{\phi_1} X_2\xdashrightarrow{\phi_2} \cdots \dashrightarrow X_i \xdashrightarrow{\phi_i}\cdots
    \end{equation}
    be the sequence of divisorial contractions or flips in the MMP$/Z$.  
    Denote by $A_i, B_i$ the pushdowns of $A, B$ to $X_i$ respectively.
    Let $\Set{\lambda_i}_{i\in \NN}$
    be the sequence of nef$/Z$ thresholds appearing in the MMP$/Z$, that is,
    \[ \lambda_i \coloneqq \inf \Set{ t\ge 0 \mid
    K_{X_i} + B_i + \mb{M}_{X_i} + t A_i~\mbox{is nef}/Z }. \]
    By \S\ref{existence-MMP-g-pair},
    if the MMP$/Z$ does not terminate, 
    then $\Set{\lambda_i}_{i\in \NN}$ is a decreasing sequence of positive real numbers.
    Let 
    \[ \lambda\coloneqq \lim_{i\to \infty}\lambda_i. \]
    We claim that $\lambda>0$.  Indeed, if $\lambda=0$, then by (2), for some $i$ large enough,
    each $\phi_j\colon X_j\dashrightarrow X_{j+1}$ with $j\ge i$ is a flip; we therefore see that
    $K_{X_i} + B_i + \mb{M}_{X_i}$ is the limit of $K_{X_i} + B_i + \mb{M}_{X_i} + \lambda_j A_i$
    as $j\to \infty$; hence by \cite[Lemma~5.17]{lyu-takumi-2026relativeminimalmodelprogram},
    $K_{X_i} + B_i + \mb{M}_{X_i}$ is pseudo-effective$/Z$, 
    contradicting the assumption that $K_X+B+\mb{M}_X$ is not pseudo-effective$/Z$.

    Pick a sufficiently small $0<\epsilon\ll \lambda$
    such that $K_X+B+\epsilon A + \mb{M}_X$ is not pseudo-effective$/Z$.
    By Step~1, we can pick an $\RR$-divisor
    \[ 0\le \Delta\sim_{\RR} B+\epsilon A + \mb{M}_X/Z \]
    such that $(X, \Delta + (1-\epsilon)A)$ is klt.  
    We claim that the sequence \eqref{mmp-sequence}
    is also an MMP$/Z$ on $K_X+\Delta$ with scaling of $(1-\epsilon)A$.
    Let $\Delta_i$ be the pushdown of $\Delta$ to $X_i$.
    Then it is evident that the nef$/Z$ threshold  of $(X_i, \Delta_i)$
    with respect to $(1-\epsilon)A_i$ is equal to 
    $(\lambda_i - \epsilon)/(1-\epsilon)$.
    Let $f_i\colon X_i\to Y_i$ be the good contraction so that either $\phi_i$ is equal to $f_i$
    or $\phi_i$ is the flip of $f_i$.  Then $f_i$
    is the contraction of an $A_i$-positive extremal ray $R_i\subset \NE(X_i/Z)$.  As
    \[ K_{X_i} + B_i + \mb{M}_{X_i} + \lambda_i A_i\sim_{\RR}
    K_{X_i} + \Delta_i + \frac{\lambda_i - \epsilon}{1-\epsilon} \cdot 
    (1-\epsilon)A_i/Z, \]
    we see that $R_i$ is also a $(K_{X_i}+\Delta_i)$-negative extremal ray.
    For the same reason, if $\phi_i$ is a flip, then $X_{i+1}+\Delta_{i+1}$
    is ample$/Y_i$.  This proves the claim.

    As $\Delta$ is big$/Z$, by \S\ref{ss-big-real-boundary},
    any process of MMP on $K_X+\Delta$ with scaling of $(1-\epsilon)A$
    over $Z$ terminates with a Mori fibre space over $Z$.
    Thus, \eqref{mmp-sequence} terminates at a good contraction $f_i\colon X_i\to Y_i/Z$ with
    $\dim Y_i<\dim X_i$.  As $f_i$ contracts a $(K_{X_i} + B_i + \mb{M}_{X_i})$-negative
    extremal ray, $-(K_{X_i} + B_i + \mb{M}_{X_i})$ is ample$/Y_i$.
    Hence the outcome $X_i\to Y_i/Z$ is a Mori fibre space of $K_X+B+\mb{M}_X$.
    This proves (a).

    \medskip

    \emph{Step~3}.
    In this step, we prove (b).
    As $K_X+(1+\alpha)B+(1+\beta)\mb{M}_X$ is big$/Z$, we can write
    \[ K_X+(1+\alpha)B+(1+\beta)\mb{M}_X\sim_{\RR} P+G/Z, \]
    where $G\ge 0$, and $P$ is ample$/Z$.  
    Let $\epsilon\in \RR^{>0}$.  Then
    \[ (1+\epsilon)(K_X+B+\mb{M}_X) \sim_{\RR} K_X+(1-\epsilon \alpha)B + (1-\epsilon\beta)\mb{M}_X + \epsilon P + \epsilon G/Z. \]
    As $(X, B + \mb{M})$ is g-klt, $(X, B + \epsilon G + \mb{M})$ is also g-klt 
    for sufficiently small $\epsilon\in \RR^{>0}$.
    Then by the argument in Step~1, we can pick
    \[ 0\le \Delta\sim_{\RR} (1-\epsilon \alpha)B + (1-\epsilon\beta)\mb{M}_X + \epsilon P + \epsilon G/Z \]
    such that $(X, \Delta + (1+\epsilon)A)$ klt.  By construction, $\Delta$ is big$/Z$.
    By \S\ref{ss-big-real-boundary}, any process of MMP$/Z$ on $K_X+\Delta$ with scaling 
    of $(1+\epsilon)A$ terminates with a \emph{good} minimal model over $Z$.
    As 
    \begin{equation}\label{mmp-scaling-epsilon}
        K_X+\Delta \sim_{\RR}(1+\epsilon) (K_X + B + \mb{M}_X)/Z,
    \end{equation}
    any process of MMP$/Z$ on $K_X+B+\mb{M}_X$ with scaling of $A$
    is also an MMP$/Z$ on $K_X+\Delta$ with scaling of $(1+\epsilon)A$.
    Hence any process of MMP$/Z$ on $K_X+B+\mb{M}_X$ with scaling of $A$
    terminates with a \emph{good} minimal model of $K_X+\Delta$ over $Z$,
    which is therefore a \emph{good} minimal model of $K_X+B+\mb{M}_X$ by \eqref{mmp-scaling-epsilon}. 
\end{proof}


\begin{lemma}\label{fano-type-MMP}
    Let $k$ be a field of characteristic zero, and
    let $X\to Z$ be a contraction of $k$-varieties.
    Assume $X$ is $\QQ$-factorial, and that $X$ is of Fano type over $Z$.
    Then we can run MMP$/Z$ on any $\RR$-divisor $D$ on $X$, which terminates with
    either a good minimal model of $D$ over $Z$ or a Mori fibre space of $D$ over $Z$.
\end{lemma}

\begin{proof}
    As $X\to Z$ is of Fano type, we can take a
    $\QQ$-divisor $\Omega$ such that $(X,\Omega)$ is klt and that
    $-(K_X+\Omega)$ is ample$/Z$.
    Let $C$ be the reduced divisor supported on $\Supp D$, which is $\QQ$-Cartier as $X$ is $\QQ$-factorial.
    Pick a rational $0<\mu\ll 1$ sufficiently small so that
    $(X, \Omega+\mu C)$ is klt, and that
    \[ L\coloneqq -(K_X + \Omega + \mu C) \]
    is ample$/Z$.  Take a sufficiently general $G\sim_{\QQ}L/Z$.  Set
    \[ \Gamma\coloneqq\Omega + \mu C + G. \]
    It is evident that $\Gamma\ge \mu C$ and $K_X+\Gamma\sim_{\QQ}0/Z$.
    We can assume $(X, \Gamma)$ is klt.
    As $G$ is ample$/Z$, $\Gamma$ is big$/Z$.
    Pick a $0<\epsilon\ll 1$ sufficiently small so that 
    \[ E \coloneqq \Omega+\mu C + G + \epsilon D \]
    is effective, and that $(X, E)$ is klt.  Then we get
    \[ K_X + E \sim_{\RR} \epsilon D/Z. \]
    As $\Gamma$ is big$/Z$, $E$ is also big$/Z$.
    By Theorem~\ref{mmp-g-pairs}, we can run a $(K_X+E)$-MMP over $Z$ with scaling
    with scaling of a sufficiently general ample$/Z$ divisor on $X$,
    which is also a $D$-MMP over $Z$.
    If $D$ is pseudo-effective$/Z$, then $K_X+E$ is
    pseudo-effective$/Z$.  Hence the MMP$/Z$ terminates with a good minimal model of $D$ over $Z$.
    If $D$ is not pseudo-effective$/Z$, then the MMP$/Z$ terminates with 
    a Mori fibre space of $D$ over $Z$.
\end{proof}


\subsection{Generalised dlt models}
We adopt the notion of \emph{generalised dlt generalised pairs} (or \emph{g-dlt g-pairs} for short)
from \cite[\S 2.13~(2)]{B-Fano}.
Let $k$ be a field of characteristic zero,
and let $(X, B + \mb{M})$ be a g-pair over $\spec k$.
We say $(X, B + \mb{M})$ is \emph{generalised dlt} (or \emph{g-dlt} for short)
if it is g-lc, and if $\eta$ is the generic point of any
generalised non-klt centre of $(X, B + \mb{M})$, then $(X, B)$ is log smooth near $\eta$
and $\mb{M}$ descends to a neighbourhood of $\eta$.


\begin{lemma}[cf. \protect{\cite[Corollary~3.44]{km98}}]\label{dlt-preserve}
    Let $(X, \Delta+\mb{M})$ be a g-dlt g-pair over $Z$, where $X, Z$ are normal varieties over a field 
    of characteristic zero.  Let $\phi\colon X\bir X'$ be either
    \begin{itemize}
        \item a divisorial contraction of a $(K_X + \Delta+\mb{M}_X)$-negative extremal ray over $Z$, or 
        \item a $(K_X + \Delta+\mb{M}_X)$-flip over $Z$.
    \end{itemize}
    Then $(X', \Delta'+\mb{M})$ is also g-dlt, where $\Delta'\coloneq \phi_* \Delta$.
\end{lemma}

\begin{proof}
    By definition of g-dlt g-pairs,
    there exists a closed subset $\Gamma\subset X$ 
    such that $(X\setminus \Gamma, \Delta|_{X\setminus \Gamma})$ is snc, 
    $\mb{M}|_{X\setminus \Gamma}$ descends to $X\setminus \Gamma$, and that
    for each prime divisor $E$ over $X$ with $\centre_X E\subseteq \Gamma$, 
    $a(E, X, \Delta+\mb{M}) > 0$.  
    Taking a resolution $W$ of $\phi$, taking the inverse image $\Gamma_W$ of $\Gamma$
    in $W$, and then taking the image of $\Gamma_W$ in $X'$, we denote by $\phi(\Gamma)$
    this subset in $X'$.
    Let 
    \[ \Gamma'\coloneqq \phi(\Gamma)\cup \Ex(\phi^{-1}). \]
    Then $X'\setminus \Gamma'$ is isomorphic to an open subset of $X\setminus\Gamma$.
    Thus, $(X'\setminus \Gamma', \Delta'|_{X'\setminus \Gamma'})$ is also snc, 
    and $\mb{M}|_{X'\setminus \Gamma'}$ descends to $X'\setminus \Gamma'$.

    Let $E$ be a prime divisor over $X'$ such that $\centre_{X'} E\subseteq \Gamma'$.
    Then $\centre_X E\subseteq \Gamma\cup \Ex(\phi)$.
    By Lemma~\ref{discrep-mmp}, we get
    \begin{equation}\label{log-disc-dlt-inequ}
        a(E, X, \Delta+\mb{M}) \le a(E, X', \Delta'+\mb{M}).
    \end{equation}
    If $\centre_X E\subseteq \Gamma$, then $a(E, X, \Delta+\mb{M}) > 0$, 
    so $a(E, X', \Delta' + \mb{M})>0$.
    If $\centre_X E\subseteq \Ex(\phi)$, then $\centre_X E$ is contained
    in the exceptional locus of the birational contraction 
    induced by the $(K_X + \Delta + \mb{M}_X)$-negative extremal ray.
    Again by Lemma~\ref{discrep-mmp},
    \eqref{log-disc-dlt-inequ} is a strict inequality, hence $a(E, X', \Delta'+\mb{M}) > 0$.
    Thus, $(X', \Delta'+\mb{M})$ is g-dlt.
\end{proof}


\begin{corollary}\label{run-g-dlt-mmp}
    Let $k$ be a field of characteristic zero.
    Let $(X, B + \mb{M})$ be a $\QQ$-factorial g-dlt g-pair over a $k$-variety $Z$.
    Let $A\ge 0$ be a general ample$/Z$ $\RR$-divisor on $X$ such that
    $K_X+B+\mb{M}_X+A$ is nef$/Z$.
    Then the MMP$/Z$ on $K_X+B+\mb{M}_X$ with scaling of $A$ exists,
    and each step of the MMP$/Z$ gives a $\QQ$-factorial g-dlt g-pair.
\end{corollary}

\begin{proof}
    This follows from Theorem~\ref{mmp-g-pairs} and Lemma~\ref{dlt-preserve}.
\end{proof}


\begin{lemma}[cf. \protect{\cite[Theorem~1.9~(ii)]{birkar12-special-LMMP}}]\label{termination-lambda}
    Let $k$ be a field of characteristic zero.
    Let $(X, B + \mb{M})$ be a $\QQ$-factorial g-dlt g-pair over a $k$-variety $Z$.
    Let $A\ge 0$ be a general ample$/Z$ $\RR$-divisor on $X$ such that
    $K_X+B+\mb{M}_X+A$ is nef$/Z$.
    Run an MMP$/Z$ on $K_X+B+\mb{M}_X$ with scaling of $A$.
    Let $\Set{\lambda_i}$ be the sequence of nef$/Z$ thresholds appearing in the MMP$/Z$,
    and $\lambda\coloneqq \lim_{i\to \infty} \lambda_i$.  Assume that either
    \begin{itemize}
        \item [\emph{(i)}] $B\ge H\ge 0$ for some ample$/Z$ $\RR$-divisor $H$, or
        \item [\emph{(ii)}] $\lambda > 0$.
    \end{itemize}
    Then the MMP$/Z$ terminates.
\end{lemma}

\begin{proof}
    By Theorem~\ref{mmp-g-pairs}, the MMP$/Z$ on $K_X+B+\mb{M}_X$ with scaling of $A$ exists.

    \medskip
    \emph{Step~1}.
    In this step, we prove (i).  By assumption, for sufficiently small $\epsilon\in \RR^{>0}$, we get
    \[ B = (1-\epsilon)(B-H) + \epsilon(B-H) + H \sim_{\RR} (1-\epsilon)(B-H) + H'/Z, \]
    where $0\le H'\sim_{\RR} \epsilon(B-H) + H/Z$ is a sufficiently general ample$/Z$ $\RR$-divisor.
    Let 
    \[ \Delta\coloneqq (1-\epsilon)(B-H) + H', \] 
    which is big$/Z$.
    We can assume that $(X, \Delta + \mb{M})$ is g-klt.
    The MMP$/Z$ is also an MMP$/Z$ on $K_X + \Delta + \mb{M}_X$ with scaling of $A$.
    If $K_X + \Delta + \mb{M}_X$ is not pseudo-effective$/Z$, the MMP$/Z$ terminates
    with a Mori fibre space over $Z$ by Theorem~\ref{mmp-g-pairs}~(a).
    If $K_X+\Delta+\mb{M}_X$ is pseudo-effective$/Z$, then for every $\alpha\in \RR^{>0}$,
    $K_X+(1+\alpha)\Delta + \mb{M}_X$ is big$/Z$.
    Thus, by Theorem~\ref{mmp-g-pairs}~(b), the MMP$/Z$ terminates with a good minimal model over $Z$.

    \medskip
    \emph{Step~2}.
    In this step, we prove (ii).
    As $A$ is a general ample$/Z$ $\RR$-divisor, we can assume that 
    $(X, B + A +\mb{M})$ is g-dlt.  The MMP$/Z$ is also an MMP$/Z$ on $K_X + B + \frac{\lambda}{2}A + \mb{M}_X$
    with scaling of $(1-\frac{\lambda}{2})A$;
    cf. the proof of Theorem~\ref{mmp-g-pairs}~(a).
    As $(X, B + \frac{\lambda}{2}A + \mb{M})$ is also g-dlt, we can conclude by (i).
\end{proof}


\begin{lemma}[cf. \protect{\cite[Lemma~4.5]{B-Zhang}}]\label{exist-g-dlt-model}
    Let $k$ be a field of characteristic zero.
    Let $(X, B + \mb{M})$ be a g-lc g-pair over $\spec k$.
    Let $S_1, \dots, S_r$ be prime divisors over $X$ with $r\in \ZZ^{\ge 0}$, which are exceptional$/X$
    and $a(S_i, X, B + \mb{M})\le 1$.  
    (If $r=0$, then the set $\Set{S_i}$ is empty by convention.)
    Then there exist a $\QQ$-factorial 
    g-dlt g-pair $(Y, B_Y + \mb{M})$ with a projective birational contraction 
    $\phi\colon Y\to X$ such that
    \begin{itemize}
        \item $S_1, \dots, S_r$ are divisors on $Y$,
        \item each exceptional prime divisor of $\phi$ is one of the $S_i$'s or an irreducible component of $\rddown{B_Y}$, and
        \item $K_Y + B_Y + \mb{M}_Y = \phi^* (K_X + B + \mb{M}_X)$.
    \end{itemize}
    In particular, the exceptional divisors of $\phi$ are exactly the $S_i$'s
    if $(X, B + \mb{M})$ is g-klt.
\end{lemma}

\begin{proof}
    \emph{Step~1}.
    Let $f\colon W\to X$ be a log resolution of $(X, B)$ such that $S_1, \dots, S_r$
    are divisors on $W$, and $\mb{M}$ descends to $W$.
    Let $\Set{E_j}_j$ be the set of $f$-exceptional divisors.
    We can assume $S_i = E_i$ for $1\le i\le r$.  Write
    \[ K_W + B_W + \mb{M}_W = f^*(K_X + B + \mb{M}_X). \]
    Let $\Delta_W \coloneqq B_W + E$, where $E\coloneqq \sum_{j>r} a_j E_j$
    and $a_j \coloneqq a(E_j, X, B+\mb{M})$.  Then by assumption,
    \[ \coeff_{S_i} \Delta_W = \coeff_{S_i} B_W = 1-a(S_i, X, B+\mb{M}) \ge 0. \]
    Moreover, as $(X, B+\mb{M})$ is g-lc, each $a_j$ is $\ge 0$.  Then
    \[ K_W + \Delta_W + \mb{M}_W = f^*(K_X + B + \mb{M}_X) + E\equiv E/X, \]
    where $E\ge 0$ is exceptional$/X$.  
    By construction, none of the $S_i$'s is an irreducible component of $E$. 

    \medskip
    \emph{Step~2}.
    Let $A\ge 0$ be a general ample$/Z$ $\RR$-divisor on $W$.
    By Theorem~\ref{mmp-g-pairs}, we can run an MMP on $K_W + \Delta_W+\mb{M}_W$
    with scaling of $A$ over $Z$.  Note that the termination of this MMP is not known
    as $(W, \Delta_W + \mb{M})$ is g-dlt, but may not be g-klt; cf. Theorem~\ref{mmp-g-pairs}~(b).
    If the MMP$/X$ terminates, then the resulting model is a minimal model $Y\to X$
    of $K_W + \Delta_W + \mb{M}_W$, i.e. $K_Y + \Delta_Y + \mb{M}_Y$ is nef$/X$,
    where $\Delta_Y$ is the pushdown of $\Delta_W$ to $Y$.
    Assume that the MMP$/X$ does not terminate.
    Let $\Set{\lambda_i}$ be the sequence of nef$/X$ thresholds appearing in the MMP$/X$,
    and $\lambda\coloneqq \lim_{i\to \infty} \lambda_i$.
    By Lemma~\ref{termination-lambda}, $\lambda = 0$.  Moreover,
    by Theorem~\ref{mmp-g-pairs}, we can assume that the MMP$/X$ reaches at a 
    model $Y\to X$, whereupon the MMP$/X$ is an infinite sequence of flips $X_i\bir X_{i+1}$ over $X$.
    Denote by $A_i, \Delta_i$ the pushdowns of $A, \Delta_W$ to $X_i$ respectively.
    Then each $K_{X_i} + \Delta_i + \mb{M}_{X_i} + \lambda_i A_i$ is nef$/X$.
    As $\lambda = 0$, $K_Y + \Delta_Y + \mb{M}_Y$ is the limit of $K_Y+\Delta_Y+\mb{M}_Y + \lambda_i A_Y$,
    where $A_Y$ is the pushdown of $A$ to $Y$.

    \medskip
    \emph{Step~3}.
    Let $\bar{k}$ be the algebraic closure of $k$.
    Enlarging $\bar{k}$ if necessary, we assume that $\bar{k}$ is \emph{uncountable}.
    Denote by $\bar{Y}\to \bar{X}$ the base change of $Y\to X$ to $\spec \bar{k}$,
    where $\bar{X}\coloneqq X\otimes_k \bar{k}$, and $\bar{Y}\coloneqq Y\otimes_k\bar{k}$.
    Let $\ol{\mb{M}}$ be the $\bb$-divisor over $\bar{X}$ induced by $\mb{M}$; see \S\ref{ss-pullback-R-Cartier-pairs}. 
    Let $\bar{\Delta}_Y, \bar{A}_Y$ be the flat pullbacks of $\Delta_Y, A_Y$ to $\bar{Y}$ respectively.
    Then $K_{\bar{Y}} + \bar{\Delta}_Y + \ol{\mb{M}}_{\bar{Y}}$
    is the flat pullback of $K_Y + \Delta_Y + \mb{M}_Y$, and it is 
    the limit of $K_{\bar{Y}} + \bar{\Delta}_Y + \ol{\mb{M}}_{\bar{Y}} + \lambda_i \bar{A}_Y$ as $i\to \infty$.
    Let $\bar{X}_i\coloneqq X_i\otimes_k \bar{k}$, and 
    let $\bar{\Delta}_i, \bar{A}_i$ be the flat pullbacks of $\Delta_i, A_i$ to $\bar{X}_i$ respectively.
    Then $K_{\bar{X}_i} + \bar{\Delta}_i + \ol{\mb{M}}_{\bar{X}_i} + \lambda_i \bar{A}_i$
    is nef$/\bar{X}$, hence movable$/\bar{X}$.
    Thus, $K_{\bar{Y}} + \bar{\Delta}_Y + \ol{\mb{M}}_{\bar{Y}}$ is a limit
    of movable$/\bar{X}$ divisors, hence it is nef on \emph{very general} curves$/\bar{X}$
    of any exceptional divisors of $\bar{Y}\to \bar{X}$.
    Let $E_Y$ be the pushdown of $E$ to $Y$, and let
    $\bar{E}_{\bar{Y}}$ be its flat pullback to $\bar{Y}$.
    By Lemma~\ref{fppf-peff}, we have
    \[ K_{\bar{Y}} + \bar{\Delta}_Y + \ol{\mb{M}}_{\bar{Y}} \equiv \bar{E}_{\bar{Y}}/\bar{X}. \]
    By \cite[Lemma~3.3]{birkar12-special-LMMP}, we see that $\bar{E}_{\bar{Y}} = 0$, 
    hence $E_Y = 0$.  Denote by $\phi$ the birational contraction $Y\to X$.  Then
    \[ K_Y + B_Y + \mb{M}_Y = \phi^*(K_X+B+\mb{M}_X). \]
    By construction, none of the $S_i$'s is contracted by $W\bir Y$.
    Moreover, if $E_j$ is an exceptional divisor of $f\colon W\to X$ with $j>r$,
    then $E_j$ is contracted by $W\bir Y$ if $a(E_j, X, B+\mb{M}) >0$.
    Thus, each exceptional prime divisor of $\phi$ is one of the $S_i$'s
    or an irreducible component of $\rddown{B_Y}$.
    Moreover, by Corollary~\ref{run-g-dlt-mmp}, $(Y, B_Y + \mb{M})$
    is a $\QQ$-factorial g-dlt g-pair.
\end{proof}

Let $k$ be a field of characteristic zero.
Let $(X, B + \mb{M})$ be a g-lc g-pair over $\spec k$.
Then Lemma~\ref{exist-g-dlt-model} shows the existence of 
a birational contraction $\phi\colon Y\to X$ from a normal variety $Y$ such that
\begin{itemize}
    \item $Y$ is $\QQ$-factorial,
    \item $K_Y + B_Y + \mb{M}_Y = \phi^*(K_X + B + \mb{M}_X)$,
    \item each exceptional prime divisor of $\phi$ is a component of $\rddown{B_Y}$, and
    \item $(Y, B_Y + \mb{M})$ is a g-dlt g-pair.
\end{itemize}
We call $(Y, B_Y + \mb{M})$ a \emph{$\QQ$-factorial g-dlt model} of $(X, B + \mb{M})$.
This is consistent with the notion of $\QQ$-factorial dlt model in \cite{KK10} when $\mb{M} = 0$.



\subsection{Divisors with log discrepancy close to $0$}
The following result is a version of \cite[Lemma~2.48]{B-Fano} over non-closed fields. 
The key point is that the constant $\epsilon>0$ is independent of the ground field.

\begin{lemma}\label{l-ep-to-0}
    Let $d\in\NN$ and $\Phi\subset [0,1]$ be a DCC set. 
    Then there is an $\epsilon>0$ depending only on $d, \Phi$ 
    so that the following holds.
    
    Let $k$ be a field of characteristic zero.
    Assume that $(X,B)$ is a projective pair over $k$,
    and $D$ is a prime divisor on birational models of $X$ such that
\begin{itemize}
    \item $(X,B)$ is lc of dimension $d$ and $(X,0)$ is klt,
    \item $K_X+B\sim_{\RR} 0/\spec k$ and $B\in \Phi$, and
    \item $a(D,X,B)<\epsilon$.
\end{itemize}
Then $a(D,X,B)=0$.
\end{lemma}

\begin{proof}
    Let $Y\to X$ be a birational contraction between normal $k$-varieties
    so that $D$ is a divisor on $Y$.  Let $\alpha_1, \dots, \alpha_r$ be rational functions on $X$ so that
    \[ K_X+ B = \sum_{i=1}^r a_i \Div_X(\alpha_i), \]
    where $a_i\in \RR$.    
    There exists a subfield $k'\subset k$
    that is finitely generated over $\QQ$, a birational contraction between normal $k'$-varieties,
    rational functions $\alpha_1', \dots, \alpha_r'$ on $X'$,
    a prime divisor $D'$ on $Y'$, and an $\RR$-divisor $B'$ on $X'$ such that
    \begin{itemize}
        \item $X = X'\otimes_{k'}k$, $Y = Y'\otimes_{k'}k$, $D = D'\otimes_{k'} k$,
        \item $Y\to X$ is the base change of $Y'\to X'$ to $\spec k$, 
        \item $B$ is the flat pullback of $B'$ to $X$, and 
        \item each $\alpha_i$ is the image of $\alpha_i'$ via the canonical homomorphism $K(X')\to K(X)$ so that
              \[ K_{X'} + B' = \sum_{i=1}^r a_i \Div_{X'}(\alpha_i'). \]
    \end{itemize}
    By \S\ref{ss-pullback-R-Cartier-pairs}, $(X', B')$ is lc of dimension $d$, $(X', 0)$ is klt, and
    $a(D', X', B') < \epsilon$.  Let
    \[ \iota\colon k'\to \CC \]
    be an embedding of fields.  Let
    \[ Y_{\CC}\coloneqq Y'\otimes_{k'}\CC~~\mbox{and}~~X_{\CC}\coloneqq X'\otimes_{k'}\CC, \]
    where the fibre products are formed via $\iota$.
    Then the induced $Y_{\CC}\to X_{\CC}$ is also a birational contraction.
    Let $B_{\CC}$ be the flat pullback of $B'$ to $X_{\CC}$.
    Again by \S\ref{ss-pullback-R-Cartier-pairs}, $(X_{\CC}, B_{\CC})$
    is lc of dimension $d$ and $(X_{\CC}, 0)$ is klt.
    Set $D_{\CC}\coloneqq D'\otimes_{k'} \CC$,
    which is a reduced Weil divisor on $Y_{\CC}$.
    Note that $D$ may not be geometrically irreducible over $k$.
    Hence $D_{\CC}$ may admit several irreducible components.
    Pick an irreducible component $E$ of $D_{\CC}$.  Then
    \[ a(E, X_{\CC}, B_{\CC}) = a(D', X', B') = a(D, X, B) < \epsilon. \]
    Then the result follows from \cite[Lemma~2.48]{B-Fano}.
\end{proof}


\subsection{Boundedness of Fano varieties}

\begin{lemma}\label{fano-bnd-very-ample}
    Let $d\in \NN$ and $\epsilon\in \RR^{>0}$.
    Then there exist $n, r\in \NN$ depending only on $d,\epsilon$ so that the following holds.

    Let $k$ be a field of characteristic zero.
    Let $X$ be a $k$-variety such that $(X,0)$ is $\epsilon$-lc and $-K_X$ is ample over $\spec k$,
    i.e. $X$ is an $\epsilon$-lc Fano variety over $k$.
    Then $-r K_X$ is very ample over $\spec k$ and $(-r K_X)^d_k < n$.
\end{lemma}

\begin{proof}
    Let $k'$ be a subfield of $k$ that is finitely generated over $\QQ$
    so that there is a $k'$-variety $X'$ satisfying $X= X'\otimes_{k'} k$.
    By \S\ref{ss-descent-b-nef}(5), we see that $X'$ is also $\epsilon$-lc.
    Moreover, by Lemma~\ref{fppf-ample}, $-K_{X'}$ is ample over $\spec k'$.
    Hence $X'$ is an $\epsilon$-lc Fano variety over $k'$.
    Let $\iota\colon k'\hookrightarrow \CC$ be an embedding of fields, and let
    $X_{\CC}\coloneqq X'\otimes_{k'}\CC$.  For the same reason,
    $X_{\CC}$ is an $\epsilon$-lc Fano variety over $\CC$.
    By \cite[Theorem~1.1]{B-BAB}, $\epsilon$-lc Fano varieties of dimension $d$ form a bounded family.
    Thus, we can take $n, r\in \NN$ depending only on $d, \epsilon$ such that
    $-r K_{X_{\CC}}$ is very ample and that $(-rK_{X_{\CC}})^d_{\CC} < n$.
    Then by Lemma~\ref{fppf-ample}, $-r K_X$ is very ample over $\spec k$;
    moreover, by \cite[Corollary~VI.2.14]{Kol-rc} and \cite[(7.9.5)]{EGA-III-2}, 
    we have $(-rK_X)^d_k<n$.
\end{proof}


\section{\bf Descent to finite type datum over number fields}\label{s-descent}

\subsection{Descent of contractions}\label{ss-descent-contractions}
Let $d\in \NN$, and let $\mf{S}$ be a set of real numbers.
Let $k$ be a field of characteristic zero.
Let $f\colon X\to Z$ be a contraction of $k$-varieties such that
\[ \dim (X/Z) = d~~\mbox{and}~~B\in \mf{S}. \]
Since the schemes $X,Z$ and irreducible components of $B$ are of finite type over $\spec k$, 
by \cite[Corollaire~(8.7.3), Th\'eor\`eme~(8.10.5)]{EGA-IV-3}, there exist 
    \begin{itemize}
        \item a finitely generated (over $\ZZ$) algebra $A$ contained in $k$,
        \item a morphism of \emph{integral} schemes $f_A\colon X_A\to Z_A$ over $\spec A$,
        \item an $\RR$-divisor $B_A$ on $X_A$,
    \end{itemize}
     such that 
     \begin{itemize}
         \item $f_A$ is projective,
         \item $Z_A\to \spec A$ is quasi-projective,
         \item $f\colon X\to Z$ is the base change of $f_A\colon X_A\to Z_A$ from $\spec A$ to $\spec k$, 
         \item $B$ is the flat pullback of $B_A$ from $X_A$ to $X$, and
         \item each irreducible component of $B$ is the flat pullback of an irreducible component of $B_A$ (in particular, the set of coefficients of $B_A$ is equal to that of $B$).
     \end{itemize}
We show that $f_A$ is indeed a contraction between normal varieties (up to shrinking $\spec A$).

\medskip
(1)
Replacing $A$ by $A\otimes_{\ZZ}\QQ$, we can assume that $A$ is of finite type over $\QQ$.
As $A$ is excellent, the normalisation $A^{\nor}$ of $A$ is finite over $A$.
Hence replacing $A$ by $A^{\nor}$, we can also assume that $A$ is normal.
Moreover, by \cite[Proposition~(9.9.4)]{EGA-IV-3},
shrinking $\spec A$ near its generic point, we assume $\spec A$
is regular.
Then $\spec A$ is a regular irreducible scheme of finite type over $\QQ$.
\[\xymatrix{
X\ar[d]^f\ar[r] & X_A\ar[d]^{f_A} \\
Z\ar[d]\ar[r] & Z_A\ar[d] \\
\spec k\ar[r] & \spec A
}\]
Let $k_A$ be the algebraic closure of $Q$ in the fraction field of $A$.
Then $k_A/\QQ$ is a finite extension, i.e.
$k_A$ is a number field, and $\spec A$ is geometrically connected over $k_A$.
As $\chara k_A =0$, $\spec A$ is also geometrically regular 
over $k_A$ \cite[Proposition~(6.7.4)]{EGA-IV-2}.
Therefore, $\spec A$ is a geometrically regular variety over $\spec k_A$.

\medskip
(2)
As $X, Z$ are $k$-varieties, they are geometrically integral by definition,
so the generic fibres of $X_A\to \spec A, Z_A\to \spec A$ are geometrically integral.
By \cite[Th\'eor\`eme~(9.7.7)]{EGA-IV-3}, up to shrinking $\spec A$, we can assume that
each fibre of $X_A\to \spec A$ and $Z_A\to \spec A$ is geometrically integral.
Moreover, by generic flatness \cite[Th\'eor\`eme~(6.9.1)]{EGA-IV-2}, we can assume that
$X_A\to \spec A, Z_A\to \spec A$ are flat.  
Since $\spec A$ is geometrically irreducible over $k_A$, by \cite[Proposition~(4.5.7)]{EGA-IV-2},
$X_A, Z_A$ are also geometrically irreducible over $k_A$.
In particular, we can conclude that $X_A, Z_A$ are $k_A$-varieties.

\medskip
(3)
Since $X, Z$ are geometrically normal over $k$ \cite[Corollaire~(6.14.2)]{EGA-IV-2},
we can assume that each fibre of $X_A\to \spec A$ and $Z_A\to \spec A$ is geometrically normal;
see \cite[Corollaire~(9.9.5)]{EGA-IV-3}.
Note that $X_A\to \spec A, Z_A\to \spec A$ are flat.  
Then by \cite[Proposition~(6.14.1)]{EGA-IV-2}, $X_A, Z_A$ are normal schemes;
in particular, $X_A, Z_A$ are geometrically normal over $k_A$ \cite[Corollaire~(6.14.2)]{EGA-IV-2}.

\medskip
(4)
As $f$ is a contraction, we have $f_*\OO_X = \OO_Z$.  By flat base change of cohomology
\cite[Proposition~(1.4.15)]{EGA-III-1}, we have
\[ \big((f_A)_*\OO_{X_A} \big)\otimes_A k = f_* \OO_X. \]
Thus, $(f_A)\mcO_{X_A}$ and $\mcO_{Z_A}$ are canonically isomorphic at the generic point of $\spec A$.
Then by \cite[Chap.~0, (5.2.7)]{EGA-I},
up to shrinking $\spec A$ to an open subset, we have
\[ (f_A)_* \mcO_{X_A} = \mcO_{Z_A}. \]
Therefore, $f_A$ is a contraction between normal $k_A$-varieties.

\medskip
(5)
As $\dim (X/Z) = d$, the generic fibre of $X\to Z$ is of dimension $d$.
Hence the generic fibre of $X_A\to Z_A$ also has dimension $d$.
By \cite[Th\'eor\`eme~(6.9.1)]{EGA-IV-2} and \cite[Corollary~III.9.6]{Hart}, 
a general fibre of $f_A$ is of dimension $d$, i.e. $\dim (X_A/Z_A) = d$.


\subsection{Descent of singularities}\label{ss-descent-singularities}
Keep the notation and assumptions in \S\ref{ss-descent-contractions}.  
Pick $\epsilon\in \RR^{\ge 0}$, and assume that
\[ (X, B)~\mbox{is an}~\epsilon\mbox{-lc pair}. \]
We show that $(X_A, B_A)$ is also $\epsilon$-lc (up to shrinking $\spec A$).

\medskip
(1)
As $X_A$ is normal, the canonical divisor $K_{X_A}$ is well-defined.

\medskip
(2)
Now we show that $(X_A, B_A)$ is a pair up to shrinking $\spec A$.
As $(X, B)$ is a pair, there are Cartier divisors $D_i$ and real numbers 
$a_i\in \RR$ such that
\[ K_X+B = \sum_i a_i D_i. \]
We can assume there are integral Weil divisors $D_{A,i}$ on $X_A$
whose flat pullback to $X$ are equal to $D_i$.
By \cite[Proposition~(8.5.5)]{EGA-IV-3}, we can assume each $D_{A,i}$
is a Cartier divisor on $X_A$.
Let
\[ E_A \coloneqq K_{X_A} + B_A - \sum_i a_i D_{A,i}. \]
Then $E_A$ is zero over the generic point of $\spec A$.
As the image of $\Supp E_A$ in $\spec A$ is constructible
(see \cite[Corollaire~(1.8.5)]{EGA-IV-1}) and avoids the generic point of $\spec A$,
we see that $E_A$ is zero over an open subset of $\spec A$.
Thus, shrinking $\spec A$ to an open subset,
we can assume that $E_A=0$.  Hence $K_{X_A} + B_A$ is $\RR$-Cartier, i.e.
$(X_A, B_A)$ is a pair.

\medskip
(3)
Next we show that a log resolution can be spread out over $\spec A$.
Let $\pi\colon W\to X$ be a log resolution of $(X, B)$ and write
\[ K_W + \Gamma = \pi^*(K_X+B). \]
Let $A\subset k$ be the $k_A$-algebra as in \S\ref{ss-descent-contractions}.
Denote by $\Lambda$ the reduced exceptional divisor of $\pi$, 
which is an snc divisor on $W$.
We can assume that there exist
\begin{itemize}
    \item a projective birational morphism $\pi_A\colon W_A\to X_A$, 
    \item a reduced divisor $\Lambda_A$ on $W_A$, and
    \item an $\RR$-divisor $\Gamma_A$ on $W_A$
\end{itemize}
such that
\begin{itemize}
    \item the base change of $\pi_A$ to $\spec k$ is equal to $\pi$,
    \item the flat pullback of $\Gamma_A$ to $W$ is equal to $\Gamma$,
    \item each irreducible component of $\Lambda$ is the flat pullback of an irreducible component of $\Lambda_A$, and
    \item each irreducible component of $\Gamma$ is the flat pullback of an irreducible component of $\Gamma_A$ (in particular, the set of coefficients of $\Gamma_A$ is equal to that of $\Gamma$).
\end{itemize}
By the same argument in \S\ref{ss-descent-contractions}~(2),
we can assume $W_A$ is a $k_A$-variety.
Moreover, we assume $W_A\to\spec A$ is flat.
By \cite[Corollaire~(9.9.5)]{EGA-IV-3}, we can assume further
each fibre of $W_A\to \spec A$ is geometrically regular.
Then by \cite[Corollaire~(6.5.2)]{EGA-IV-2}, $W_A$ is a nonsingular $k_A$-variety.
Applying the same argument to each stratum of $\Lambda_A$,
we can assume that up to shrinking $\spec A$,
$(W_A, \Gamma_A)$ is a log resolution of $(X_A, B_A)$.

\medskip
(4)
Write
\[ F_A \coloneqq K_{W_A} + \Gamma_A - \pi_A^*(K_{X_A} + B_A). \]
Then $F_A$ is zero over the generic point of $\spec A$.
Again by \cite[Corollaire~(1.8.5)]{EGA-IV-1}, the image of $\Supp F_A$
in $\spec A$ is constructible.
Since the image of $\Supp F_A$ in $\spec A$ does not contain the generic point of $\spec A$,
we deduce that $F_A$ is zero over an open subset of $\spec A$.
Thus, up to shrinking $\spec A$, we can assume $F_A = 0$.
By construction in (3), the coefficients of $\Gamma_A$ are $\le 1-\epsilon$.
Hence $(X_A, B_A)$ is also $\epsilon$-lc.


\subsection{Descent of log Calabi-Yau fibrations}\label{ss-descent-log-CY}
Keep the notation and assumptions in \S\ref{ss-descent-contractions}.  Assume that
$(X, B)\to Z$ is a log Calabi-Yau fibration, that is, $(X, B)$ is lc and 
\begin{equation}\label{log-CY-condition}
    K_X+B\sim_{\RR}0/Z.
\end{equation}
We show that $(X_A, B_A)\to Z_A$ is also a log Calabi-Yau fibration (up to shrinking $\spec A$).

By \S\ref{ss-descent-singularities}, we can assume $(X_A, B_A)$ is lc.  Moreover,
\eqref{log-CY-condition} says that
\[ K_X+B = f^*N + \sum_j b_j \Div_X(\alpha_j) \]
for an $\RR$-Cartier $\RR$-divisor $N$ on $Z$ and finitely many rational functions $\alpha_j\in K(X)$.
We can choose the finitely generated algebra $A\subset k$ so that 
\begin{itemize}
    \item there is an $\RR$-divisor $N_A$ on $Z_A$ whose flat pullback to $Z$ is $N$, and
    \item each $\alpha_j$ is a rational function via the injection $K(X_A)\to K(X)$.
\end{itemize}
By \cite[Proposition~(8.5.5)]{EGA-IV-3}, we can assume $N_A$ is also $\RR$-Cartier.
Write
\[ G_A\coloneqq K_{X_A} + B_A  - f_A^*N_A - \sum_j b_j \Div_{X_A}(\alpha_j). \]
As $G_A$ is zero over the generic point of $\spec A$, 
again by \cite[Corollaire~(1.8.5)]{EGA-IV-1}, we can assume $G_A=0$ over the whole $\spec A$.
Thus, we have
\[ K_{X_A} + B_A \sim_{\RR}0/Z_A. \]
Hence $(X_A, B_A)\to Z_A$ is a log Calabi-Yau fibration over $k_A$.


\subsection{Descent of Fano type fibrations}\label{ss-descent-FT}
Keep the notation and assumptions in \S\ref{ss-descent-contractions}.
Assume the contraction
\[ X\to Z~\mbox{is of Fano type}. \]
We show that $f_A\colon X_A\to Z_A$ is also of Fano type (up to shrinking $\spec A$).

By assumption, there is a big$/Z$ $\QQ$-divisor $\Lambda$ on $X$ such that $(X, \Lambda)$
is klt, and $K_X+\Lambda\sim_{\QQ}0/Z$.
We can take the algebra $A\subset k$ so that there is a $\QQ$-divisor $\Lambda_A$ on $X_A$
whose flat pullback to $X$ is equal to $\Lambda$.
By \S\ref{ss-descent-singularities}, we can assume $(X_A, \Lambda_A)$ is klt.
Moreover, by \S\ref{ss-descent-log-CY}, we can assume 
\[ K_{X_A} + \Lambda_A \sim_{\QQ} 0/Z_A. \]
By \cite[Lemma~5.21]{lyu-takumi-2026relativeminimalmodelprogram}, $\Lambda_A$ is big$/Z_A$.
Hence $f_A\colon X_A\to Z_A$ is of Fano type.


\subsection{Descent of relatively ample divisors}\label{ss-descent-ample}
Keep the notation and assumptions in \S\ref{ss-descent-contractions}.
Assume that $H$ is an ample$/Z$ $\RR$-divisor on $X$.
We show that $H$ can descend to an ample $\RR$-divisor on $X_A$.

Since $H$ is ample$/Z$, we can write
$H =\sum_i a_i H_i$,
where $a_i\in \RR^{>0}$, and each $H_i$ is an ample$/Z$ Cartier divisor on $X$.
We can choose the algebra $A\subset k$ and an $\RR$-divisor $H_A$ on $X_A$
such that every irreducible component of $H$ is the flat pullback of an irreducible component of $H_A$.
In particular, each $H_i$ is the flat pullback of an integral divisor $H_{A,i}$ on $X_A$.
By \cite[Lemme~(8.10.5.2)]{EGA-IV-3}, we can assume every $H_{A,i}$ is ample$/Z_A$.
Thus, we can conclude that $H_A$ is ample$/Z_A$.


\subsection{Descent of relatively nef divisors}\label{ss-descent-nef}
Keep the notation and assumptions in \S\ref{ss-descent-contractions}.
Assume that $N$ is a nef$/Z$ $\RR$-divisor on $X$.
As relative nefness is not an open condition on the base \cite{nef-not-open},
we can not spread out the divisor $N$ to a nef$/Z_A$ divisor $N_A$ on $X_A$.
Instead, we assume that
\[ f\colon X\to Z~\mbox{is of Fano type, and}~N~\mbox{is a nef}/Z~\RR\mbox{-divisor on}~X. \]
By Corollary~\ref{fano-nef-semiample-cor}, $N$ is semiample$/Z$.
Thus, there exist contractions $g\colon X\to Y$, $h\colon Y\to Z$, and an ample$/Z$ $\RR$-divisor $H$ on $Y$
such that $f = h\circ g$ and $N = g^* H$.

By \S\ref{ss-descent-contractions},
we can choose the algebra $A\subset k$ and contractions $g_A\colon X_A\to Y_A$, $h_A\colon Y_A\to Z_A$
of normal $k_A$-varieties such that $H$ is the flat pullback of an $\RR$-Cartier $\RR$-divisor $H_A$ on $Y_A$ and that
$N_A = g_A^* H_A$.  
Moreover, by \S\ref{ss-descent-ample}, we can assume $H_A$ is ample$/Z_A$.
Therefore, we can conclude that $N_A$ is semiample$/Z_A$;
in particular, $N_A$ is nef$/Z_A$.


\subsection{Descent of \texorpdfstring{$\bb$-$\RR$-Cartier $\bb$-divisors}{TEXT}}\label{ss-b-Cartier}
Keep the notation and assumptions in \S\ref{ss-descent-contractions}.

\medskip
(1)
Let $\mb{M}$ be a $\bb$-$\RR$-Cartier $\bb$-divisor over $X$.
Assume that $\mb{M}$ descends to a projective birational model $Y\to X$.
We can choose an algebra $A\subset k$, a projective birational morphism
between $k_A$-varieties $Y_A\to X_A$, and an $\RR$-Cartier divisor $\mb{N}_{Y_A}$ on $Y_A$ such that
\begin{itemize}
    \item the base change of $Y_A\to X_A$ to $\spec k$ is $Y\to X$, and
    \item the pullback of $\mb{N}_{Y_A}$ to $Y_A$ is $\mb{M}_Y$.
\end{itemize}
Then $\mb{N}_{Y_A}$ induces a $\bb$-$\RR$-Cartier $\bb$-divisor $\mb{N}$ over $X_A$.
It is evident that the base change of $\mb{N}_{X_A}$ to $X$ is equal to $\mb{M}_X$.

\medskip
(2)
Now assume that $X\to Z$ is a contraction over another $k$-variety $S$, and
that $\mb{M}$ is $\bb$-nef over $S$, which descends to a 
nef$/Z$ $\RR$-divisor $\mb{M}_Y$ on a projective birational model $Y\to X$.  
We can assume $X_A\to Y_A$ is a contraction over a $k_A$-variety $S_A$
whose base change to $\spec k$ is equal to $S$.
Let $\mb{N}$ be the $\bb$-$\RR$-Cartier $\bb$-divisor over $X_A$ as in (1).
As relative nefness is not open on the base \cite{nef-not-open},
$\mb{N}$ is not necessarily $\bb$-nef over $S_A$.
However, let $K(A)\to K$ be an arbitrary field extension, where $K(A)$ is the fraction field of $A$, and let
\begin{itemize}
    \item $f_K\colon X_K\to Z_K$ be the base change of $f_A$ to $\spec K$, 
    \item $Y_K\to X_K$ be the base change of $Y_A\to X_A$, and
    \item $B_K$ be the flat pullback of $B_A$ to $X_K$.
\end{itemize}
It is evident that $Y_K\to X_K$ is a projective birational morphism of $K$-varieties, and that
the flat pullback of $\mb{N}_{Y_A}$ is an $\RR$-Cartier $\RR$-divisor on $Y_K$.
Hence $\mb{N}_{Y_A}$ induces a $\bb$-$\RR$-Cartier $\bb$-divisor $\mb{P}$ over $X_K$.
Let $S_K \coloneqq S_A\otimes_A K$.
As $\mb{M}_Y$ is nef$/S$, 
we can deduce that $\mb{P}_{Y_K}$ is nef$/S_K$.
Hence $(X_K, B_K + \mb{P})$ is a g-pair.


\subsection{Descent of closed points}\label{ss-descent-pts}
Keep the notation and assumptions in \S\ref{ss-descent-contractions}.
Let $z\in Z$ be a closed point and $k(z)$ its residue field.
Since $k(z)/k$ is a finite extension, we can assume there is 
an integral scheme $z_A$ over $\spec A$ whose base change to $\spec k$
is equal to $\spec k(z)$.
Moreover, by \cite[Th\'eor\`eme~(8.10.5)]{EGA-IV-3}, we can assume that
$z_A$ is a closed subscheme of $Z_A$.


\section{\bf Boundedness of complements}

\subsection{General remarks}\label{ss-compl-remarks}
Let $k$ be a field of characteristic zero.
Let $(X,B+\mb{M})$ be a projective g-pair over $\spec k$.

\medskip

(1)
Assume that there is a $\QQ$-divisor $B^+\ge B$ 
such that $(X, B^+ + \mb{M})$ is g-lc, $n\mb{M}$ is $\bb$-Cartier, and that
\[ n(K_X+B^+ + \mb{M}_X)\sim 0/\spec k \]
for some $n\in\NN$. 
We show $K_X+B^+ + \mb{M}_X$ is an $n$-complement 
of $K_X+B + \mb{M}_X$. 
Write $B =T + \Delta$, where $T = \rddown{B}$, then it suffices to show that
\[ nB^+\ge nT + \rddown{(n+1)\Delta}. \]
The proof of this fact is identical to \cite[\S 6.1(1)]{B-Fano}.  
Moreover, for any $m\in \NN$, $mn\mb{M}$ is also $\bb$-Cartier, and we have
\[ mn(K_X+B^++\mb{M}_X)\sim 0/\spec k. \]
Hence $K_X+B^+ + \mb{M}_X$ is also an $mn$-complement of $K_X+B+\mb{M}_X$
for any $m\in \NN$.

\medskip

(2)
Assume $X\bir X'$ is a birational map to a normal projective $k$-variety $X'$.
(Note that $X\bir X'$ may not be a birational contraction.)
Let $\phi\colon W\to X$ be a resolution of $X$ such that the induced map $\psi\colon W\bir X'$
is a morphism.  Assume that
\begin{itemize}
    \item $\phi^*(K_X + B + \mb{M}_X) + P = \psi^*(K_{X'} + B' + \mb{M}_{X'})$ for some $P\ge 0$ and $B'\ge 0$, and
    \item $K_{X'}+B' + \mb{M}_{X'}$ has an $n$-complement $K_{X'}+{B'}^+ + \mb{M}_{X'}$ with ${B'}^+\ge B'$.
\end{itemize}
Then $K_X + B + \mb{M}_X$ also has an $n$-complement $K_X + B^+ + \mb{M}_X$ with $B^+\ge B$. 

Let $C'={B'}^+ - B'$ and let $B^+ \coloneqq B + \phi_*(P+\psi^*C')$. Then 
\begin{align*}
    K_X + B^+ + \mb{M}_X &= K_X + B + \mb{M}_X + \phi_*P + \phi_*\psi^*C' \\
      &= \phi_*\psi^*(K_{X'} + B' + \mb{M}_{X'}) + \phi_*\psi^*C' \\
      &= \phi_*\psi^*(K_{X'} + {B'}^+ + \mb{M}_{X'}),
\end{align*}
which implies that 
\[ n(K_X + B^+ + \mb{M}_X)\sim 0 ~~\text{and}~~
\phi^*(K_X + B^+ + \mb{M}_X) = \psi^*(K_{X'} + B'^+ + \mb{M}_{X'}). \]
Thus, $(X,B^+ + \mb{M})$ is gl-c, then we can apply (1).
In other words, the birational pullback of a monotonic $n$-complement of $K_{X'} + B' + \mb{M}_{X'}$
gives a monotonic $n$-complement of $K_X + B + \mb{M}_X$.

\medskip

(3) Assume $X\bir X'$ is a partial MMP on $-(K_X + B + \mb{M}_X)$ and $B'$ is the pushdown of $B$. 
Then there is $P\ge 0$ as in (2). Thus if $K_{X'} + B' + \mb{M}_{X'}$ has an $n$-complement 
$K_{X'} + {B'}^+ + \mb{M}_{X'}$ with ${B'}^+\ge B'$, then  
$K_X + B + \mb{M}_X$ has an $n$-complement $K_X + B^+ + \mb{M}_X$ with $B^+\ge B$. 


\subsection{Linear systems}
Let $X$ be a normal scheme, and let $M$ be an $\RR$-divisor on $X$.
We set
\[ \hg^0(X, M) \coloneqq \hg^0(X, \mc{O}_X(\rddown{M}))
= \Set{0\neq \alpha\in K(X)\mid \Div_X(\alpha) + M \ge 0}\cup \Set{0}, \]
where $\Div_X(\alpha)$ is the principal divisor associated to $\alpha$ on $X$.
Assume $\hg^0(X, M)\neq \Set{0}$.  The \emph{linear system} $\abs{M}$ is defined as
\[ \abs{M}\coloneqq \Set{0\le N\sim M} = \Set{\Div_X(\alpha) + M \mid 0\neq \alpha\in \hg^0(X, M)}. \]
Note that $\abs{M}$ is not equal to $\abs{\rddown{M}}$ unless $M$ is integral.

Let $X$ be a variety over a field $k$, and let $M$ be an $\RR$-divisor on $X$.
Then $\hg^0(X, M)$ is a $k$-linear space, which may be infinite dimensional.
Moreover, there is a canonical map
\[ \hg^0(X, M)\to \abs{M} \]
sending each $\alpha\in \hg^0(X, M)$ to $M+\Div_X(\alpha)$.
If $V\subseteq \hg^0(X, M)$ is a $k$-linear subspace,
we denote by $\abs{V}$ its image in $\abs{M}$, 
and we call $\abs{V}$ the \emph{linear system of $V$}.
Note that $\abs{V}$ is merely a set 
unless $X$ is projective over $\spec k$,
in which case $V$ is finite dimensional over $k$, 
and $\abs{V}$ is equal to the projective space $\PP(V)$.


\begin{lemma}\label{g-lc-Bertini}
    Let $(X/S, B + \mb{M})$ be a g-lc g-pair over an algebraically closed field $\mbb{K}$
    of characteristic zero.  Let $N$ be an $\RR$-Cartier divisor on $X$,
    and let $V\subseteq \hg^0(X, N)$ be a finite dimensional $\mbb{K}$-linear subspace.
    Assume that there exists an $\alpha_0\in V$ such that 
    \[ (X, B + N + \Div_X(\alpha_0) + \mb{M}) \]
    is a g-lc g-pair.  Then for a general $\alpha\in V$,
    \[ (X, B + N + \Div_X(\alpha) +\mb{M}) \]
    is also a g-lc g-pair.
\end{lemma}

\begin{proof}
    Let $\pi\colon X'\to X$ be a log resolution of $(X,B+\mb{M}_X)$
    such that $\mb{M}$ descends to a nef$/S$ divisor on $X'$.
    Set $N'\coloneqq \pi^*N$.  By \cite[Lemma~3.2]{vol-hilb-R-divisors}, 
    $\hg^0(X, N) = \hg^0(X', N')$ via the identification $K(X)=K(X')$.  
    Then $V$ is identified with a finite dimensional $\mbb{K}$-linear subspace $V'\subseteq \hg^0(X', N')$. 

    Let $A'\coloneqq \ceil{N'}$, which is an integral Cartier divisor on $X'$.
    As $A'\ge N'$, it is evident that $V'\subseteq \hg^0(X', A')$.
    Let $\psi\colon X''\to X'$ be a resolution.
    Denote by $N'' \coloneqq \psi^*N'$ and $A''\coloneqq \psi^*A'$.
    Let $V''\subseteq \hg^0(X'', A'')$ be the $\mbb K$-linear subspace 
    corresponding to $V'$; cf. \cite[Lemma~3.2]{vol-hilb-R-divisors}.
    Then by \cite[\S9.1.B]{Rob-II}, we can choose $\psi$ and 
    an effective integral divisor $E''$ on $X''$ such that 
    \begin{itemize}
        \item $E'' + \Ex(\psi)$ has snc support, and
        \item $\abs{V''} = \abs{W''} + E''$, where the linear system $\abs{W''}$ of
        \[ W''\subseteq \hg^0(X'', A'' - E'') \]
        is base point free.
    \end{itemize}
    (Note that \cite[\S9.1.B]{Rob-II} works over $\CC$, 
    but the construction is valid over any
    algebraically closed field of characteristic zero.)
    Denote by $\phi\coloneqq \psi\circ \pi\colon X''\to X$, and by
    \[ F'' \coloneqq N'' - A'' + E''. \]
    Let $B''$ be the $\RR$-divisor defined by
    \[ K_{X''} + B'' + \mb{M}_{X''} = \phi^*(K_X+B+\mb{M}_X). \]
    Replacing $X''$ by higher resolutions of $X'$, we can assume that $B'' + F''$ has snc support.
    
    For every $\alpha\in V\setminus\set{0}$, let
    \[ P_{\alpha}\coloneqq N + \Div_X(\alpha). \]
    Then we can write
    \[ \phi^*P_{\alpha} = N'' + \Div_{X''}(\alpha)
    = F'' + (A'' - E'' + \Div_{X''}(\alpha)), \]
    where
    \[ G_{\alpha}'' \coloneqq A'' - E'' + \Div_{X''}(\alpha) \]
    belongs to the base point free linear system $\abs{W''}$.
    
    Since $(X, B+P_{\alpha_0}+\mb{M})$ is g-lc, the coefficients of $B''+\phi^*P_{\alpha_0}$ are all $\le 1$.
    We view $V$ as the affine $\mbb K$-variety $\AAA_{\mbb K}^{\dim V}$; cf. \cite[\S 3.4.3]{joins-intersections}.
    Then by \cite[Corollary~3.4.9]{joins-intersections}, 
    there exists a Zariski open subset $U$ of $V$ 
    such that for each $\mbb K$-point $\alpha\in U$, 
    \begin{itemize}
        \item $G_{\alpha}''$ is a reduced divisor on $X''$,
        \item $G_{\alpha}''$ has no irreducible components in common with $B'' + F''$, and
        \item $B'' + F'' + G_{\alpha}''$ has snc support.
    \end{itemize}
    In particular, for each $\mbb K$-point $\alpha\in U$, the coefficients 
    of $B'' + F'' + G_{\alpha}''$ are all $\le 1$.  As
    \[ K_{X''} + B'' + F'' + G_{\alpha}'' + \mb{M}_{X''}
    = \phi^*(K_X + B + P_{\alpha} + \mb{M}_X), \]
    we can conclude that $(X, B + N + \Div_X(\alpha) + \mb{M})$ is g-lc for a general $\alpha\in V$.
\end{proof}


\subsection{Universal boundedness of complements}


\begin{theorem}\label{universal-local-complement}
    Let $d\in \NN$ and $\mf{R}$ be a finite set of rational numbers.
    Then there exists an $n\in \NN$ depending only on $d, \mf{R}$ satisfying the following.

    Let $\mbb{K}$ be an algebraically closed field of characteristic zero.
    Assume $(X, B)$ is a pair over $\mbb K$ and $f\colon X\to Z$ is a contraction of $\mbb K$-varieties such that
    \begin{itemize}
        \item $(X, B)$ is lc,
        \item $\dim X = d$ and $\dim Z > 0$,
        \item $B\in \Phi(\mf{R})$,
        \item $X$ is of Fano type over $Z$, and
        \item $-(K_X + B)$ is nef over $Z$.
    \end{itemize}
    Then for any closed point $z\in Z$, there is an $n$-complement $K_X + B^+$ of $K_X + B$
    over $z$ such that $B^+\ge B$.
    Moreover, the complement is also an $mn$-complement for any $m\in \NN$.
\end{theorem}

\begin{proof}
    \emph{Step~1}.
    In this step, we descend $(X, B)\to Z$ to the universal domain of countable transcendence degree
    of characteristic zero , i.e.
    \[ \Omega\coloneqq \ol{\QQ (t_1, t_2, \dots) }, \]
    where $t_1, t_2, \dots$ is an infinite sequence of independent variables. 
    
    By \S\ref{s-descent}, there is a subfield $L\subset \mbb K$, 
    which is finitely generated over $\QQ$, such that $X, Z, \Supp B$ are all defined over $L$;
    in other words, there exist a contraction between $L$-varieties $f_L \colon X_L\to Z_L$
    and a $\QQ$-divisor $B_L$ on $X_L$ such that 
    \begin{itemize}
        \item $(X_L, B_L)$ is lc,
        \item $\dim X_L = d$ and $\dim Z_L > 0$,
        \item $B_L\in \Phi(\mf{R})$,
        \item $X_L\to Z_L$ is of Fano type,
        \item $-(K_{X_L} + B_L)$ is nef over $Z_L$,
        \item $f$ is the base change of $f_L$,
        \item $X = X_L\otimes_L \mbb K$ and $Z = Z_L\otimes_L \mbb K$, and
        \item $B$ is the flat pullback of $B_L$ to $X$.
    \end{itemize}
    As $z\in Z$ is a $\mbb K$-point, we can also assume that there exists an $L$-point $z_L\in Z_L$
    whose base change to $\spec \mbb K$ is equal to $z$.  Since $L$ is finitely generated over $\QQ$,
    we can pick an inclusion of fields
    \[ \iota\colon L\to \Omega. \]
    Let
    \[ X_{\Omega}\coloneqq X_L\otimes_L \Omega,~
    Z_{\Omega} \coloneqq Z_L\otimes_L \Omega,~\mbox{and}~z_{\Omega}\coloneqq z_L\otimes_L \Omega, \]
    where the fibre products are formed via the inclusion $\iota$.
    Denote by $B_{\Omega}$ the flat pullback of $B_L$ to $X_{\Omega}$.  Then 
    \begin{itemize}
        \item $(X_{\Omega}, B_{\Omega})$ is lc (see \S\ref{ss-pullback-R-Cartier-pairs}(5))
        \item $\dim X_{\Omega} = d$ and $\dim Z_{\Omega} > 0$,
        \item $B_{\Omega}\in \Phi(\mf{R})$,
        \item $f_{\Omega}\colon X_{\Omega}\to Z_{\Omega}$ is of Fano type (see Lemma~\ref{Fano-type-base-change}), and
        \item $-(K_{X_{\Omega}} + B_{\Omega})$ is nef over $Z_{\Omega}$ (see \cite[Lemma~2.18]{Keeler2003}).
    \end{itemize}
    Then by \cite[Theorem~1.8]{B-Fano}, there is an open subset $U_{\Omega}\subseteq Z_{\Omega}$ of $z_{\Omega}$
    such that $K_{X_{\Omega}} + B_{\Omega}$ admits a monotonic $n$-complement $Y_{\Omega}+ B_{\Omega}^+$ over $U_{\Omega}$,
    where
    \[ Y_{\Omega} \coloneqq f_{\Omega}^{-1}(U_{\Omega})~~\mbox{and}~~
    B_{\Omega}^+ \ge B_{\Omega}|_{Y_{\Omega}}. \]
    Note that $n\in \NN$ depends on $d, \mf{R}, \Omega$, hence depends only on $d,\mf{R}$.

    \medskip

    \emph{Step~2}.
    In this step, we descend the complement $K_{Y_{\Omega}} + B_{\Omega}^+$ to a finitely generated $L$-algebra.
    Indeed, by \S\ref{s-descent}, there exist
    \begin{itemize}
        \item a finitely generated $L$-algebra $R\subset \Omega$,
        \item an open subset $U_R\subset Z_R\coloneqq Z_L\otimes L R$,
        \item an open subset $Y_R\subset X_R\coloneqq X_L\otimes_L R$, and 
        \item a $\QQ$-divisor $B_R^+$ on $Y_R$,
    \end{itemize}
    such that
    \begin{itemize}
        \item $Y_{\Omega}\to U_{\Omega}$ is the base change of $Y_R\to U_R$ to $\spec \Omega$,
        \item $B_{\Omega}^+$ is the flat pullback of $B_R^+$ to $Y_{\Omega}$, and
        \item $(Y_R, B_R^+)$ is lc.
    \end{itemize}
    Let $B_R$ be the flat pullback of $B_L$ to $X_R$.
    As $B_{\Omega}^+\ge B_{\Omega}$, up to shrinking $\spec R$, we can assume that 
    \[ B_R^+ \ge B_R|_{Y_R}. \]
    For convenience, we include the relevant commutative diagram as following.
    \[\xymatrix{
    X\ar[d]\ar[r] & X_L\ar[d] & X_R\ar[d]\ar[l] & Y_R\ar[d]\ar[l] & Y_{\Omega}\ar[d]\ar[l] \\
    Z\ar[r] & Z_L & Z_R\ar[l] & U_R\ar[l] & U_{\Omega}\ar[l]
    }\]
    Let
    \[ z_R \coloneqq z_L\otimes_L R. \]
    Then $z_R$ is a closed subscheme of $Z_R$ that is mapped isomorphically onto $\spec R$
    via $Z_R\to \spec R$.  In particular, $z_R$ is a section of $Z_R\to \spec R$.  

    \medskip

    \emph{Step~3}.
    In this step, we construct an $n$-complement of $K_X+B$ over $z\in Z$
    by taking a general fibre of $(Y_R, B_R^+)$ over $\spec R$.
    By generic flatness \cite[Th\'eor\`eme~(6.9.1)]{EGA-IV-2}, 
    we can assume that each irreducible component of $B_R,B_R^+$ is flat over $\spec R$.
    Moreover, by construction, the generic point of $z_R$ is contained in $U_R$.
    Pick a general closed point $s\in \spec R$.
    As $z_R\subset Z_R$ is a section of $Z_R\to \spec R$,
    $s$ corresponds uniquely to a closed point $z'\in z_R$ that is contained in $U_R$.
    Denote by $L'$ the residue field of $Z_R$ at $z'$, which is finite over $L$.
    It is evident that the two composite homomorphisms are equal:
    \[ L\to k(s) \to L'~~\mbox{and}~~ L \to R\to L'. \]
    Let 
    \[ X_{L'}\coloneqq X_R\otimes_R k(s)\cong X_L\otimes_L L' ~~\mbox{and}~~
    Z_{L'}\coloneqq Z_R\otimes_R k(s)\cong Z_L\otimes_L L'. \]
    In addition, let
    \[ Y_{L'} \coloneqq Y_R\otimes_R k(s) ~~\mbox{and}~~ U_{L'} \coloneqq U_R\otimes_R k(s). \]
    Then $U_{L'}$ is an open subset of $Z_{L'}$ containing $z'$.
    Since irreducible components of $B_R, B_R^+$ are flat over $\spec R$, the restriction 
    of $B_R$ to $X_{L'}$ and the restriction of $B_R^+$ to $Y_{L'}$ 
    are well-defined $\QQ$-divisors on $X_{L'}, Y_{L'}$ respectively, which we denote by
    \[ B_{L'}~~\mbox{and}~~B_{L'}^+~~\mbox{respectively}. \]
    As $s\in \spec R$ is a general closed point, we can assume that
    \begin{itemize}
        \item $(Y_{L'}, B_{L'}^+)$ is an lc pair,
        \item $B_{L'}^+ \ge B_{L'}|_{Y_{L'}}$, and
        \item $n(K_{Y_{L'}} + B_{L'}^+)\sim 0/U_{L'}$.
    \end{itemize}
    Note that the image of $z'\in \Z_R$ in $Z_L$ is $z_L$.
    Since $L'/L$ is a finite extension, and since $\mbb K$ is algebraically closed, 
    we can assume that 
    \begin{itemize}
        \item $L'$ embeds into $\mbb K$, and
        \item $z\in Z$ is the base change of $z'\in Z_{L'}$ to $\spec \mbb K$.
    \end{itemize}
    Let $U\coloneqq U_{L'}\otimes_{L'}\mbb K$, which is an open subset of $Z$ containing $z$, 
    and let $B^+$ be the flat pullback of $B_{L'}^+$ to $Y\coloneqq Y_{L'}\otimes_{L'}\mbb K$.
    It is evident that $Y$ is an open subset of $X$ satisfying $Y = X\times_Z U$.
    Then $K_Y + B^+$ is an $n$-complement of $K_X+B$ over $U$ satisfying $B^+ \ge B$ over $U$.
    Moreover, by \S\ref{ss-compl-remarks}, this is also an $mn$-complement for any $m\in \NN$.
\end{proof}


\begin{theorem}\label{universal-global-complement}
    Let $d,p\in\NN$, and let $\mf{R}$ be a finite set of rational numbers.
    Then there exists an $n\in\NN$ depending only on
    $d,p,\mf{R}$ so that the following holds.
    
    Let $\mbb{K}$ be an algebraically closed field of
    characteristic zero, and let $(X,B+\mb{M})$ be a
    projective g-pair over $\mbb{K}$ such that
    \begin{itemize}
        \item $X$ is of Fano type and $\dim X=d$,
        \item $B\in\Phi(\mf{R})$ and $p\mb{M}$ is $\bb$-Cartier,
        \item $(X,B+\mb{M})$ is g-lc, and
        \item $-(K_X+B+\mb{M}_X)$ is nef.
    \end{itemize}
    Then there is an $n$-complement
    $K_X+B^++\mb{M}_X$ of $K_X+B+\mb{M}_X$
    with $B^+\ge B$.
    Moreover, this is also an $mn$-complement for every $m\in\NN$.
\end{theorem}

\begin{proof}
    The proof follows the same strategy as that of Theorem~\ref{universal-local-complement},
    using \S\ref{s-descent} and applying \cite[Theorem~1.10]{B-Fano} in place of \cite[Theorem~1.8]{B-Fano}.
\end{proof}


\subsection{Boundedness of global complements}

The following boundedness result for global complements is not used
elsewhere in this paper. We include it for completeness in our treatment
of complements over non-closed fields.

\begin{theorem}\label{thm:global-complements-nonclosed-hyperstandard}
    Let $d,p\in\NN$, and let $\mf{R}$ be a finite set of rational numbers.
    Then there exists an $n\in\NN$ depending only on
    $d,p,\mf{R}$ satisfying the following.

    Let $k$ be a field of characteristic zero.
    Suppose that $(X, B + \mb{M})$ is a g-pair over $\spec k$ such that
    \begin{itemize}
        \item $X\to \spec k$ is a contraction of Fano type and $\dim X = d$,
        \item $B\in \Phi(\mf{R})$ and $p \mb{M}$ is $\bb$-Cartier,
        \item $(X, B + \mb{M})$ is g-lc, and
        \item $-(K_X + B + \mb{M}_X)$ is nef over $\spec k$.
    \end{itemize}
    Then there is an $n$-complement $K_X + B^+ + \mb{M}_X$ of $K_X +B + \mb{M}_X$
    with $B^+ \ge B$.  Moreover, the complement is also an $mn$-complement for any $m\in \NN$.
\end{theorem}

\begin{proof}
    \emph{Step~1}.
    Let $\mbb{K}$ be an algebraic closure of $k$.
    As $X\to \spec k$ is a contraction, $X$ is projective and geometrically integral over $\spec k$.
    Let $\bar{X} \coloneqq X\otimes_k \mbb K$, and let $\bar{B}$ be the flat pullback of $B$ to $\bar{X}$.
    Denote by $\ol{\mb M}$ the $\bb$-divisor over $\bar{X}$ that is induced by $\mb{M}$.
    By \S\ref{ss-descent-b-nef}~(6), $(\bar{X}, \bar{B} + \ol{\mb M})$ is a g-lc pair over 
    $\spec \mbb{K}$.  By Lemma~\ref{fano-type-to-closure}, 
    $\bar{X}$ is of Fano type over $\spec \mbb{K}$.
    Moreover, the flat pullback of $-(K_X+B+\mb{M}_X)$ to $\bar{X}$
    is $-(K_{\bar{X}} + \bar{B} + \ol{\mb{M}}_{\bar{X}})$, which is nef over $\spec \mbb{K}$.

    By Theorem~\ref{universal-global-complement}, there is an $n$-complement 
    $K_{\bar{X}} + \bar{B}^+ + \ol{\mb{M}}_{\bar{X}}$ of
    $K_{\bar{X}} + \bar{B} + \ol{\mb{M}}_{\bar{X}}$ with $\bar{B}^+ \ge \bar{B}$, 
    where $n\in \NN$ depends only on $d, p, \mf{R}$.  
    Since the complement is also an $mn$-complement for every $m\in \NN$,
    and since $p\mb{M}$ is $\bb$-Cartier, 
    we can assume that $n$ is a multiple of $p$ and $n\mb{M}_X$ is an integral divisor.  Set
    \[ N\coloneqq -n(K_X+B+\mb{M}_X)~~\mbox{and}~~ 
    \bar{N}\coloneqq -n(K_{\bar{X}} + \bar{B} + \ol{\mb{M}}_{\bar{X}}). \]
    Then $N$ is a $\QQ$-Cartier $\QQ$-divisor on $X$, and $\bar{N}$ is the flat pullback of $N$.
    As
    \[ n(K_{\bar{X}} + \bar{B}^+ + \ol{\mb{M}}_{\bar{X}})\sim 0/\spec \mbb{K}, \]
    we see that
    \[ \bar{P}_0 \coloneqq n(\bar{B}^+ - \bar{B}) \sim \bar{N} /\spec \mbb{K}. \]
    In other words, $\bar{P}_0\in \abs{\bar{N}}$.
    By Lemma~\ref{g-lc-Bertini}, for a general $\bar{P}\in \abs{\bar{N}}$,
    $(\bar{X}, \bar{B} + \frac{1}{n}\bar{P} + \ol{\mb{M}})$ is a g-lc g-pair over $\spec \mbb{K}$.

    \medskip

    \emph{Step~2}.
    Denote by $\bar{U}\subseteq \PP(\hg^0(\bar{X}, \bar{N}))$ a Zariski open subset
    such that every closed point $\bar{u}\in \bar{U}$ gives a member $\bar{P}_{\bar{u}}\in \abs{\bar{N}}$ 
    with $(\bar{X}, \bar{B}+\frac{1}{n}\bar{P}_{\bar{u}}+\ol{\mb M})$ a g-lc g-pair.
    By \cite[Proposition~III.9.3]{Hart}, we have
    \[ \PP(\hg^0(\bar{X}, \bar{N})) = \PP(\hg^0(X, N))\otimes_k \mbb{K}. \]
    Let $k'/k$ be a finite field extension such that there is an open subset $U'\subset \PP(\hg^0(X', N'))$
    satisfying $\bar{U} = U'\otimes_{k'} \mbb{K}$,
    where $X'\coloneqq X\otimes_k k'$, and $N'$ is the flat pullback of $N$ to $X'$.  Let 
    \[ \phi\colon \PP(\hg^0(X', N')) \to \PP(\hg^0(X, N)) \]
    be the induced morphism of projective spaces, which is finite and \'etale.
    As the generic point of $\PP(\hg^0(X, N))$ is not contained in the image $T$
    of $T'\coloneqq\PP(\hg^0(X', N')) \setminus U'$ in $X$,
    \[ U\coloneqq\PP(\hg^0(X, N)) \setminus T \]
    is an open dense subset of $\PP(\hg^0(X, N))$.
    By construction, the inverse image $\phi^{-1}(U)$ is contained in $U'$,
    hence $U_{\mbb K}\coloneqq U\otimes_k \mbb{K}$ is contained in $\bar{U}$, where
    $U_{\mbb K}$ is viewed as the inverse image of $U$ under 
    the canonical projection $\PP(\hg^0(\bar{X}, \bar{N}))\to \PP(\hg^0(X, N))$.

    As $k$ is an infinite field, the set of $k$-points is dense in $\PP(\hg^0(X, N))$.
    Thus, the set of $k$-points $U(k)$ of $U$ is also dense in $U$.
    Pick an arbitrary $k$-point $u\in U$, which corresponds to an element $P_u\in \abs{N}$.
    The base change of $u$ to $\spec \mbb{K}$ gives a $\mbb{K}$-point $\bar{u}$ of $U_{\mbb{K}}\subseteq \bar{U}$,
    which in turn gives a $\bar{P}_{\bar{u}}\in \abs{\bar{N}}$.
    It is evident that $\bar{P}_{\bar{u}}$ is the flat pullback of $P_u$ to $\PP(\hg^0(\bar{X}, \bar{N}))$.
    By construction of $\bar{U}$, $K_{\bar{X}} + \bar{B} + \frac{1}{n}\bar{P}_{\bar{u}} + \ol{\mb{M}}_{\bar{X}}$
    is an $n$-complement of $K_{\bar{X}} + \bar{B} + \ol{\mb{M}}_{\bar{X}}$.
    Then by \S\ref{ss-descent-b-nef}~(6), we can conclude that 
    $(X, B + \frac{1}{n}P_u + \mb{M})$ is also g-lc.
    Moreover, as $P_u\in \abs{N}$, we see that
    \[ n\bigg(K_X + B + \frac{1}{n}P_u + \mb{M}_X\bigg) \sim 0/\spec k. \]
    Thus, $B^+ \coloneqq B + \frac{1}{n}P_u$ is an $n$-complement of $K_X + B + \mb{M}_X$ with $B^+\ge B$.
    Furthermore, by \S\ref{ss-compl-remarks}~(1), for any $m\in \NN$,
    $K_X+B^+ + \mb{M}_X$ is also an $mn$-complement of $K_X+B+\mb{M}_X$.
\end{proof}


\subsection{Boundedness of relative complements}


\begin{theorem}\label{t-bnd-compl-usual-local}
    Let $d\in \NN$ and $\mf{R}$ be a finite set of rational numbers.
    Then there exists an $n\in \NN$ depending only on $d, \mf{R}$ satisfying the following.

    Let $k$ be a field of characteristic zero.
    Assume $(X, B)$ is a pair over $k$ and $f\colon X\to Z$ is a contraction of $k$-varieties such that
    \begin{itemize}
        \item $(X, B)$ is lc of dimension $d$ and $\dim Z > 0$,
        \item $B\in \Phi(\mf{R})$,
        \item $X$ is of Fano type over $Z$, and
        \item $-(K_X + B)$ is nef over $Z$.
    \end{itemize}
    Then for any closed point $z\in Z$, there is an $n$-complement $K_X + B^+$ of $K_X + B$
    over $z$ such that $B^+\ge B$.
    Moreover, the complement is also an $mn$-complement for any $m\in \NN$.
\end{theorem}

\begin{proof}
    \emph{Step~1}.
    We fix some notation in this step.
    Shrinking $Z$ around $z$, we can assume $Z$ is affine, say, $Z = \spec A$
    for a finitely generated $k$-algebra $A$.
    Denote by $\bar{f}\colon \bar{X}\to \bar{Z}$ the base change of $f\colon X\to Z$ to $\mbb K$.
    Let $\bar{B}$ be the flat pullback of $B$ to $\bar{X}$.
    Let $\bar{z}_1,\dots, \bar{z}_r$ be the inverse images of $z$ in $\bar{Z}$.
    By Theorem~\ref{universal-local-complement}, 
    for each $\bar{z}_i$, there is an $n$-complement $K_{\bar{X}} + \bar{B}_i^+$
    of $K_{\bar{X}} + \bar{B}$ over a Zariski open neighbourhood $\bar{U}_i\subseteq \bar{Z}$ of $\bar{z}_i$,
    where $n\in \NN$ depends only on $d, \mf{R}$.  
    We can assume that each $\bar{U}_i$ is a principal open subset of $\bar{Z}$,
    say, $\bar{U}_i = D(\bar{g}_i)\subseteq \bar{Z}$ for some $\bar{g}_i\in \bar{A}\coloneqq A\otimes_k \mbb{K}$.  Let
    \[ \bar{X}_i\coloneqq \bar{f}^{-1}(\bar{U}_i). \]
    Denote by $\bar{B}_i$ the restriction of $\bar{B}$ to $\bar{X}_i$.  Let
    \[ N \coloneqq -n(K_X+B)~~\mbox{and}~~ \bar{N}\coloneqq -n(K_{\bar{X}} + \bar{B}), \]
    which are $\QQ$-Cartier $\QQ$-divisors on $X$ and $\bar{X}$, respectively.
    Denote by $\bar{N}_i$ the restriction of $\bar{N}$ to $\bar{X}_i$.
    Then
    \begin{equation}\label{local-compl-each-i}
        n(\bar{B}_i^+ - \bar{B}_i)\sim \bar{N}_i/\bar{U}_i.
    \end{equation}
    Note that \eqref{local-compl-each-i} is equivalent to that
    \begin{equation}\label{local-compl-explicit}
        n(\bar{B}_i^+ - \bar{B}_i) = \bar{N}_i + \bar{f}_i^* (\bar{M}_i) + \Div(\alpha_i),
    \end{equation}
    where $\alpha_i\in K(\bar{X}_i)$, $\bar{M}_i$ is a Cartier divisor on $\bar{U}_i$, 
    and $\bar{f}_i$ denotes the morphism $\bar{X}_i\to \bar{U}_i$.
    Shrinking $\bar{U}_i$ around $\bar{z}_i$, we can assume $\bar{M}_i = \Div(\beta_i)$
    for some $\beta_i\in K(\bar{U}_i)$.  Then \eqref{local-compl-explicit} gives
    \begin{equation}\label{local-compl-Div}
        n(\bar{B}_i^+ - \bar{B}_i) = \bar{N}_i + \Div(\gamma_i),
    \end{equation}
    where $\gamma_i \coloneqq \alpha_i \cdot \bar{f}_i^*(\beta_i)$.
    Hence $n(\bar{B}_i^+ - \bar{B}_i) \in \abs{\bar{N}_i}$.
    
    \medskip

    \emph{Step~2}.
    As $\bar{X}_i$ is an open subset of $\bar{X}$, it is evident that
    \[ \rddown{\bar{N}_i} = \rddown{\bar{N}}\big|_{\bar{X}_i}. \]
    Then by \cite[Proposition~III.9.3]{Hart}, we have the canonical isomorphism
    \begin{equation}\label{global-local-sections}
        \begin{aligned}
            \hg^0(\bar{X}_i, \bar{N}_i) &= \hg^0(\bar{X}_i, \mc{O}_{\bar{X}_i}(\rddown{\bar{N}_i}))
        =\hg^0(\bar{X}_i, \mc{O}_{\bar{X}_i}(\rddown{\bar{N}\big|_{\bar{X}_i} })) \\
        &=\hg^0(\bar{X}_i, \big(\mc{O}_{\bar{X}}(\rddown{\bar{N}})\big)\big|_{\bar{X}_i})
        \cong \hg^0(\bar{X}, \mc{O}_{\bar{X}}(\bar{N}))\otimes_{\bar{A}} \bar{A}_{\bar{g}_i} \\
        &= \hg^0(\bar{X}, \bar{N})\otimes_{\bar{A}} \bar{A}_{\bar{g}_i}.
        \end{aligned}
    \end{equation}
    Let 
    \[ \rho_i\colon \hg^0(\bar{X}, \bar{N}) \to \hg^0(\bar{X}_i, \bar{N}_i) \]
    be the canonical $\mbb{K}$-linear map of restriction of global sections.
    
    For each $1\le i\le r$, let $\bar{V}_i\subseteq \hg^0(\bar{X}_i, \bar{N}_i)$ be a finite dimensional
    $\mbb{K}$-linear subspace containing the $\gamma_i$ in \eqref{local-compl-Div},
    which corresponds to $n(\bar{B}_i^+ - \bar{B}_i) \in \abs{\bar{N}_i}$.
    We claim that there exist an integer $a_i\in \NN$ and a $\mbb{K}$-linear map
    \[ \tau_i\colon \bar{V}_i \to \hg^0(\bar{X}, \bar{N}) \]
    such that 
    \begin{equation}\label{proj-relation}
        (\rho_i\circ \tau_i)(v) = \bar{g}_i^{a_i} v ~~\mbox{for every}~v\in \bar{V}_i.
    \end{equation}
    To this end, let $\Set{v_{i1}, \dots, v_{it}}$ be a set of basis of $\bar{V}_i$ over $\mbb{K}$.
    By \eqref{global-local-sections}, there exist rational functions $m_{ij}\in \hg^0(\bar{X}, \bar{N})$
    and integers $a_{ij}\in \NN$ such that
    \[ v_{ij} = \frac{m_{ij}}{\bar{g}_i^{a_{ij}}}~~\mbox{for}~1\le j\le t. \]
    Take $a_i\coloneqq \max_{1\le j\le t} a_{ij}$.
    Let $\tau_i\colon \bar{V}_i\to \hg^0(\bar{X}, \bar{N})$ be the $\mbb{K}$-linear map
    sending each $v_{ij}$ to $\bar{g}_{i}^{a_i-a_{ij}}\cdot m_{ij}$ with $1\le j\le t$.
    Then it is evident that $\tau_i$ satisfies \eqref{proj-relation} as desired.

    \medskip

    \emph{Step~3}.
    In this step, we construct a $\mbb{K}$-linear subspace of $\hg^0(\bar{X}, \bar{N})$
    whose general $\mbb{K}$-points give an $n$-complement of $K_{\bar{X}} + \bar{B}$
    over every $\bar{z}_1,\dots, \bar{z}_r$.
    Set
    \[ \bar{W}_i \coloneqq \tau_i(\bar{V}_i)\subseteq \hg^0(\bar{X},\bar{N})~~\mbox{and}~~ 
    \bar{W}\coloneqq \bar{W}_1 +\cdots + \bar{W}_r\subseteq \hg^0(\bar{X}, \bar{N}), \]
    which are finite dimensional $\mbb{K}$-linear subspaces of $\hg^0(\bar{X}, \bar{N})$.  
    Since
    \[ \hg^0(\bar{X}, \bar{N}) = \hg^0(X, N)\otimes_k \mbb{K}, \]
    there exists a finite dimensional $\mbb{K}$-linear subspace $W\subseteq \hg^0(X, N)$
    such that 
    \[ \bar{W}\subseteq W_{\mbb K}\coloneqq W\otimes_k \mbb{K}. \]
    For every $1\le i\le r$, set
    \[ \bar{V}_i' \coloneqq \rho_i(W_{\mbb K}) \subseteq \hg^0(\bar{X}_i, \bar{N}_i). \]
    Note that $\bar{V}_i'$ may not be equal to $\bar{V}_i$.
    However, as $\gamma_i \in \bar{V}_i$, by \eqref{proj-relation}, we have 
    \[ \bar{g}_i^{a_i} \gamma_i \in \rho_i(\bar{W})\subseteq \rho_i(W_{\mbb K}) = \bar{V}_i'. \]
    Since $\bar{g}_i^{a_i}$ is invertible on $\bar{U}_i$, we see that
    \[ n(\bar{B}_i^+ - \bar{B}_i) = \bar{N}_i + \Div_{\bar{X}_i}(\gamma_i)
    = \bar{N}_i + \Div_{\bar{X}_i}(\bar{g}_i^{a_i}\gamma_i) \in \abs{\bar{V}_i'}. \]
    Then by Lemma~\ref{g-lc-Bertini}, there exists a Zariski open subset $\bar{\Omega}_i\subseteq \bar{V}_i'$
    such that for each $\mbb K$-point $\delta_i\in \bar{\Omega}_i$, 
    \[ \big(\bar{X}_i, \bar{B}_i + \frac{1}{n}(\bar{N}_i + \Div_{\bar{X}_i}(\delta_i))\big) \]
    is an lc pair.  Moreover, by construction, we have
    \[ K_{\bar{X}_i} + \bar{B}_i + \frac{1}{n}(\bar{N}_i + \Div_{\bar{X}_i}(\delta_i)) \sim 0/\spec \mbb{K}. \]
    Hence $K_{\bar{X}_i} + \bar{B}_i + \frac{1}{n}(\bar{N}_i + \Div_{\bar{X}_i}(\delta_i))$
    is an $n$-complement of $K_{\bar{X}_i} + \bar{B}_i$ over $\bar{U}_i$.
    Since
    \[ \rho_i\big|_{W_{\mbb K}}\colon W_{\mbb K}\to \bar{V}_i' \]
    is a $\mbb{K}$-linear map of finite dimensional $\mbb K$-linear spaces,
    it is also a morphism of affine $\mbb{K}$-varieties.
    Then as $\rho_i\big|_{W_{\mbb K}}$ is surjective by construction, 
    the inverse image $(\rho_i\big|_{W_{\mbb K}})^{-1}(\bar{\Omega}_i)$
    is a Zariski dense open subset of $W_{\mbb K}$.  Let
    \[ \bar{\Omega} \coloneqq \bigcap_{1\le i\le r} (\rho_i\big|_{W_{\mbb K}})^{-1}(\bar{\Omega}_i). \]
    Then $\bar{\Omega}$ is a Zariski dense open subset of $W_{\mbb K}$.
    Let $\gamma\in \bar{\Omega}$ be an arbitrary $\mbb K$-point, and let
    \[ \bar{P}_{\gamma} \coloneq \bar{N} + \Div_{\bar{X}}(\gamma). \]
    By construction, $K_{\bar{X}} + \bar{B} + \frac{1}{n}\bar{P}_{\gamma}$
    is an $n$-complement of $K_{\bar{X}} + \bar{B}$ over each $\bar{U}_i$ for $1\le i\le r$.

    \medskip

    \emph{Step~4}.
    We claim that there exists a nonempty open subset $\Omega\subseteq W$ such that its base change
    $\Omega_{\mbb K} \coloneqq \Omega\otimes_k \mbb{K}$ is contained in $\bar{\Omega}$.
    To this end, let $L/k$ be a finite field extension so that $\bar{\Omega}$ is the base change
    of an open subset $\Omega'\subseteq W_L \coloneqq W\otimes_k L$, i.e.
    $\bar{\Omega} = \Omega' \otimes_L \mbb{K}$.
    Let $T$ be the closed subset $W_L\setminus \Omega'$, 
    equipped with the reduced scheme structure.
    As $W_L\to W$ is finite, the image $T_W$ of $T$ in $W$ is a proper closed subset of $W$.  Let
    \[ \Omega\coloneq W\setminus T_W. \]
    It is evident that
    \[ \Omega\otimes_k L \subseteq W_L\setminus T \subseteq \Omega'. \]
    Hence $\Omega_{\mbb K}$ is contained in $\Omega'\otimes_L \mbb{K} = \bar{\Omega}$ as desired.

    \medskip

    \emph{Step~5}.
    As $k$ is an infinite field, the set of $k$-points of $W$ is dense in $W$.
    Hence the set of $k$-points of $\Omega$ is also nonempty and dense in $\Omega$.
    Pick an arbitrary $k$-point $s\in \Omega$.  Then by Step~4, the inverse image $\bar{s} \coloneqq s\otimes 1$
    of $s$ in $W_{\mbb K}$ is contained in $\bar{\Omega}$.  Thus,
    \[ K_{\bar{X}} + \bar{B} + \frac{1}{n}\Div_{\bar{X}}(\bar{s}) \]
    is an $n$-complement of $K_{\bar{X}} + \bar{B}$ over each $\bar{U}_i$ for $1\le i\le r$.
    
    Let $\bar{U}$ be the union of $\bar{U}_1,\dots, \bar{U}_r$.
    Let $L'/k$ be a finite field extension such that there exists an open subset 
    $U'\subseteq Z\otimes_k L'$ whose base change to $\mbb{K}$ is equal to $\bar{U}$, i.e.
    $\bar{U} = U'\otimes_{L'} \mbb{K}$.  In particular, the inverse images of $z$
    in $Z\otimes_k L'$ is contained in $U'$.
    Denote by $S'$ the complement $(Z\otimes_k L')\setminus U'$, 
    and by $S$ its image in $Z$.  As $\Set{z}\otimes_k L'$ is contained in $U'$,
    the closed point $z$ is not contained in $S$.
    Let $U\coloneqq Z\setminus S$.  Then $U$ is an open neighbourhood of $z$ such that 
    $U\otimes_k L'$ is contained in $U'$.  Hence $U\otimes_k \mbb{K}$ is contained in $\bar{U}$.
    Denote by
    \[ B^+ \coloneqq B + \frac{1}{n}(N + \Div_X(s)). \]
    By \S\ref{ss-descent-b-nef} (6), we can conclude that $K_X + B^+$
    is an $n$-complement of $K_X+B$ over $z\in U$.
    Moreover, by \S\ref{ss-compl-remarks}, this is also an $mn$-complement for any $m\in \NN$.
\end{proof}


\section{\bf Stein degree on fibrations of relative dimension one}\label{s-rel-dim-one}

We prove $\mc{V}_{\univ}(1, t)$ in this section, which is the first step
towards the inductive proof of $\mc{V}_{\univ}(d, t)$ on the relative dimension $d$.

\subsection{Proof of $\mc{H}_{\univ}(1,t)$}\label{pf-S-1-t}
In this subsection, we prove $\mcH_{\univ}(1,t)$,
which will be applied in the proof of $\mcV_{\univ}(1,t)$.
	
\begin{lemma}[$\mc{H}_{\univ}(1, t)$]\label{base-case-r-dim=1}
    Let $t\in (0,1]$.  There exists an $N_h(1,t)\in \NN$ depending only on $t$ so that
    the following holds.

    Let $k$ be a field of characteristic zero.
    Let $(X, B + \mb{M})\to Z$ be a generalised log Calabi-Yau fibration over $k$ of relative dimension one.
    Let $S$ be a horizontal$/Z$ irreducible component of $B$ such that $\coeff_S B \ge t$.
    Then $\sdeg (S^{\nor}/Z) \le N_h(1,t)$.
\end{lemma}
	
\begin{proof}
    Let $F$ be the generic fibre of $X\to Z$.
	As $\dim (X/Z) = 1$, $F$ is a geometrically irreducible nonsingular curve.
    It is evident that $S^{\nor}\to Z$ and $S\to Z$ are generically finite.  Thus, we have
    \[ \sdeg (S^{\nor}/Z) = \sdeg (S/Z) = \deg (S/Z). \]
    Write
    \[ K_F + B_F + M_F = (K_X + B + \mb{M}_X)|_F \sim_{\RR} 0, \]
    where $B_F \coloneqq B|_F$, and $M_F \coloneqq \mb{M}_X |_F$ is an $\RR$-divisor on $F$ with $\deg M_F \ge 0$.
    As $S|_F$ is effective and nonzero, $\deg K_F < 0$.  Hence
    \[ \bar{F}\coloneqq F\otimes_{K(Z)}\ol{K(Z)} \cong \PP^1_{\ol{K(Z)}}. \]
    Thus, we get
   \[ \deg (S/Z) = \deg (S|_{\bar{F}}) \le \frac{1}{t} \deg (-K_{\bar{F}}) = \frac{2}{t}. \]
   Then we can take $N_h(1,t) = \rddown{2/t}$.
\end{proof}


\subsection{Proof of $\mc{V}_{\univ}(1,t)$}\label{pf-V-1-t}

Recall that a \emph{Fano type (generalised log Calabi-Yau) fibration} over a field $k$ means 
a generalised log Calabi-Yau fibration $(X, B + \mb{M})\to Z$ over $k$ whose underlying $k$-variety 
$X$ is of Fano type over $Z$.

\begin{theorem}[$\mc{V}_{\univ}(1, t)$]\label{vertical-case-fib-dim-one}
    Let $t\in \RR^{>0}$.  There exists an $N_v(1,t)\in \NN$ depending only on $t$
    such that the following holds.

    Let $k$ be a field of characteristic zero.  Let 
    \[ f\colon (X, B + \mb{M}) \to Z \] 
    be a Fano type fibration over $k$ of relative dimension one.
    Let $S$ be a vertical$/Z$ irreducible component of $B$ such that
    \begin{itemize}
        \item the coefficient of $S$ in $B$ is $\ge t$, and
        \item the image $T$ of $S$ is a prime divisor on $Z$.
    \end{itemize}
    Then $\sdeg (S^{\nor}/T) \le N_v(1,t)$, where $S^{\nor}$
    is the normalisation of $S$.
\end{theorem}

\begin{proof}
    Shrinking $Z$ near the generic point of $T$, we can assume that $Z$ is nonsingular.
    By Lemma~\ref{fano-type-induction}, we can then assume that $\dim X = 3$, $\dim Z = 2$ and $\dim T = 1$.
    As $X\to Z$ is of Fano type, by \cite[Corollary~22.3]{lyu-takumi-2026relativeminimalmodelprogram}
    and Lemma~\ref{fano-type-at-dlt-model}, we can assume $X$ is $\QQ$-factorial.

    \medskip
    \emph{Step~1}.
    As $X$ is of Fano type over $Z$, by Lemma~\ref{fano-type-MMP}, 
    we can run $\phi\colon X\dashrightarrow X'$
    a $(-S)$-MMP over $Z$ that ends with a $\QQ$-factorial good minimal model $X'$ of $-S$.  
    By negativity lemma \cite[Lemma~5.15]{lyu-takumi-2026relativeminimalmodelprogram}, 
    $S$ is not contracted by $\phi$.  Set $S' \coloneqq \phi_* S$.
    Let $g\colon X'\to Z'$ be the contraction induced by $-S'$.
    As $S'$ is vertical$/Z$, 
    the contraction $Z'\to Z$ is birational. 
    Let $T'$ be the birational transform of $T$ to $Z'$.
    As $-S'$ is nef$/Z'$, $S'$ is the unique irreducible component of $\Supp (g^* T')$.
    Thus, we can write
    \begin{equation}
        S' = b\cdot g^* T' \label{S-whole-fibre-3}
    \end{equation}
    for some $b\in \QQ^{>0}$.   
    Shrinking $Z$ near the generic point of $T$, we can identify $(Z', T')$ with $(Z, T)$.
    Set $B' \coloneqq \phi_* B$.
    By Lemma~\ref{CY-log-discre-no-change}, we have 
    \[ \coeff_{S'} B' = \coeff_S B \ge t. \]
    Moreover, by Lemma~\ref{Fano-type-pushdown}, $(X', B' + \mb{M})\to Z$ is also a Fano type fibration.  
    Decreasing $t$ if necessary, we can assume that $t\in \QQ^{>0}$.
    Run $\varphi\colon X'\dashrightarrow X''$ a $(- K_{X'} - t S')$-MMP 
    over $Z$ that ends with a good minimal model $X''\to Z$ of $-K_{X'} - t S'$, then $-(K_{X''}+ t S'')$ is nef$/Z$,
    where $S'' \coloneqq \varphi_* S'$.
    By \eqref{S-whole-fibre-3}, $S'$ is not contracted by $\varphi$. 

    By Lemma~\ref{Fano-type-pushdown}, $X''$ is also of Fano type over $Z$.
    Then by Theorem~\ref{t-bnd-compl-usual-local}, 
    up to shrinking $Z$ near the generic point of $T$, 
    $K_{X''} + t S''$ has a monotonic $n$-complement $K_{X''} + B^+$
    for some bounded $n\in \NN$ depending only on $t$, that is,
    \begin{itemize}
        \item $(X'', B^+)$ is an lc $k$-pair,
        \item $n(K_{X''} + B^+)\sim 0/Z$, and
        \item $t S''\le B^+$.
    \end{itemize}
    Hence from now on, instead of generalised pairs, we can assume that 
    \begin{itemize}
        \item $(X, B)$ is an lc $k$-pair,
        \item $f\colon X\to Z$ is of Fano type,
        \item $S$ is an irreducible component of $B$ with nonzero coefficient,
        \item the image $T$ of $S$ in $Z$ is a prime divisor and $S = \Supp f^*T$, and
        \item $n(K_X+B)\sim 0/Z$, where $n\in \NN$ depends only on $t$.
    \end{itemize}
    As $-K_X$ is big$/Z$, $B$ admits a horizontal$/Z$ irreducible component $H$ 
    whose coefficient in $B$ is $\ge 1/n$.  

    \medskip
	\emph{Step~2}.  
    By Lemma~\ref{base-case-r-dim=1}, there exists an $N_h(1,t)\in \NN$
    depending only on $t$ so that
    \[ \sdeg (H^{\nor}/Z) \le N_h(1,t). \]
    By Lemma~\ref{fano-type-MMP}, we can
	run $\phi\colon X\dashrightarrow X'$ a $(-H)$-MMP over $Z$ that ends with a Mori fibre space
	$f'\colon X'\to Z'$ over $Z$.  Then the pushdown $H'$ of $H$ to $X'$ is ample$/Z'$.
    As $\dim (X/Z) = 1$, $Z'\to Z$ is a birational contraction of
	normal surfaces over $k$.
    Shrinking $Z$ near the generic point of $T$, we can assume that $Z'\to Z$
    is an isomorphism of nonsingular surfaces over $k$.
	Denote by $B'$ the pushdown of $B$ to $X'$.
    Set $S'\coloneqq \phi_* S$.
	Then by Lemma~\ref{CY-log-discre-no-change}, 
    we also have $K_{X'} + B'\sim_{\QQ}0/Z'$ and $\coeff_{S'}B' \ge 1/n$.
    Note that since $S = \Supp f^* T$ by assumption, $S$ is not contracted by $\phi$.
    Thus, we can assume $H$ is ample$/Z$ in the rest of the proof.

    Denote by $\bar{k}$ the algebraic closure of $k$.  Let
    \[ \bar{X} \coloneqq X\otimes_k\bar{k}~~
    \mbox{and}~~ \bar{Z} \coloneqq Z\otimes_k\bar{k}, \]
    which are varieties over $\bar{k}$, and the induced $\bar{f}\colon \bar{X}\to \bar{Z}$
    is a $k$-contraction of Fano type; see Lemma~\ref{fano-type-to-closure}.  
    Denote by $\bar{B}$ the flat pullback of $B$ to $\bar{X}$.
    By \S\ref{ss-descent-b-nef}, $(\bar{X}, \bar{B})$ is an lc pair over $\bar{k}$.
    Let 
    \[ \bar{H}\coloneqq H\otimes_k\bar{k}~~\mbox{and}~~\bar{S}\coloneqq S\otimes_k\bar{k}, \]
    which may not be irreducible schemes.  However, every irreducible component of $\bar{S}$
    has coefficient $\ge 1/n$ in $\bar{B}$.  By Lemma~\ref{fppf-ample}, $\bar{H}$ is ample$/\bar{Z}$.
    Moreover, let 
    \[ \bar{T}\coloneqq T\otimes_k \bar{k}, \]
    which may not be irreducible neither.  By Lemma~\ref{finite-fibre},
    there is an open subset $U\subset Z$ such that $\codim(Z\setminus U, Z) = 2$,
    and for each scheme-theoretic point $z\in U$, we have
    \[ \mf{n}(H_z\otimes_{k(z)}\ol{k(z)}) \le \sdeg(H^{\nor}/Z), \]
    where $k(z)$ denotes the residue field at $z$.  Let
    \[ \bar{U}\coloneqq U\otimes_k\bar{k}. \]
    Then we also have $\codim (\bar{Z}\setminus \bar{U}, \bar{Z}) = 2$.
    Let $\bar{L}\subset \bar{Z}$ be a general hyperplane section.
    By Bertini's theorem \cite[3.4.9]{joins-intersections}, $\bar{L}$ is a nonsingular curve.
    We can assume that $\bar{L}$ intersects with $\bar{T}$ at general points of $\bar{T}\cap \bar{U}$.
    Pick $\bar{u}\in \bar{L}\cap \bar{T}$, and let $u$ be its image in $U$.  Then
    \[ \mf{n}(\bar{H}_{\bar{u}}) = \mf{n}(H_u\otimes_{k(u)}\ol{k(u)})
    \le \sdeg(H^{\nor}/Z) \le N_h(1,t). \]
    On the other hand, it is evident that
    \[ \sdeg(S^{\nor}/T) = \mf{n}(S_{u}\otimes_{k(u)}\ol{k(u)}) = \mf{n}(\bar{S}_{\bar{u}}). \]
    Let 
    \[ \bar{X}_{\bar{L}}\coloneqq \bar{X}\times_{\bar{Z}}\bar{L},~~\mbox{and}~~
    \bar{H}_{\bar{L}}\coloneqq \bar{H}\times_{\bar{Z}}\bar{L}. \]
    By Bertini's theorem \cite[3.4.9]{joins-intersections}, $\bar{X}_{\bar{L}}$
    is a normal surface over $\bar{k}$.  By adjunction, we can write
    \[ K_{\bar{X}_{\bar{L}}} + \bar{B}_{\bar{L}} = (K_{\bar{X}} + \bar{B} + \bar{X}_{\bar{L}})|_{\bar{X}_{\bar{L}}} \]
    for some $\QQ$-divisor $\bar{B}_{\bar{L}}$.
    By \cite[Remark~1.15]{Filipazzi-bnd-g-pair-surface},
    $(\bar{X}_{\bar{L}}, \bar{B}_{\bar{L}})$ is a $\QQ$-factorial lc pair over $\bar{k}$.
    By construction, every irreducible component of $\bar{S}_{\bar{u}}$
    is an irreducible component of $\bar{B}_{\bar{L}}$ with coefficient $\ge 1/n$.
    Write
    \[ \red(\bar{H}_{\bar{u}}) = \Set{\bar{x}_1, \dots, \bar{x}_r} \]
    for some natural number $r\le N_h(1,t)$.  As $\bar{H}$ is ample$/\bar{Z}$,
    every irreducible component of $\bar{S}_{\bar{u}}$
    intersects at least one of the $\bar{x}_i$'s.
    However, for each fixed $\bar{x}_i$, by Corollary~\ref{surface-comp-bnd},
    the number of irreducible components of $\bar{B}_{\bar{L}}$
    passing through $\bar{x}_i$ is bounded from above depending only on $n$
    (whence depending only on $t$).
    Therefore, we can conclude that there exists an $N_v(1,t)\in \NN$ so that
    \[ \sdeg(S^{\nor}/T) = \mf{n}(\bar{S}_{\bar{u}}) \le N_v(1,t). \qedhere \]
\end{proof}


\section{\bf Stein degree of vertical divisors}

\begin{lemma}[cf. \protect{\cite[Lemma 3.3]{B-moduli}}]\label{sdeg-subvariety}
    Let $X\to Z$ be a surjective projective morphism of integral schemes.
    Let $S$ be an integral scheme admitting a morphism $S\to X$ 
    so that the induced morphism $S\to Z$ is projective and surjective.
    Then $\sdeg (X/Z)\le \sdeg (S/Z)$.
\end{lemma}

\begin{proof}
    Let $X\to V\to Z$ be the Stein factorisation of $X\to Z$.  
    As $S\to Z$ is surjective, the composite morphism $S\to X\to V$ is also surjective.
    Let $S\to U\to V$ be the Stein factorisation of $S\to V$.
    Then we see that $\sdeg (S/Z) = \deg(U/V)\cdot \deg (V/Z) \ge \deg (V/Z) = \sdeg(X/Z)$,
    whence the result.
\end{proof}


\subsection{Normal multiplicity and analytic branches}
In this subsection, we show that if the coefficient of an irreducible component of an lc pair 
is bounded from below, then the degree of normalisation of this irreducible component is bounded from above.

\medskip
(1)
Let $X$ be an integral scheme, and let $\pi\colon X'\to X$ be the normalisation.
Let $x\in X$ be a closed point.  The \emph{normal multiplicity} (respectively, 
\emph{geometric normal multiplicity}) of $X$ at $x$ 
is the number of points of 
\[ X'\times_X\spec k(x)~ \mbox{(respectively,}~ X'\times_X\spec \ol{k(x)}),\] 
where $k(x)$ denotes the residue field of $X$ at $x$.


\medskip
(2)
Let $X$ be an integral scheme and $x\in X$ a closed point.
Denote by $\wh{\OO}_{X,x}$ the completion of $\mcO_{X,x}$
with respect to its maximal ideal.
We call the irreducible components of $\spec \wh{\mcO}_{X,x}$
the \emph{analytic branches} of $X$ at $x$.

\medskip
(3)
We first treat lc pair over $\CC$.

\begin{lemma}\label{mult-branches}
    Let $X$ be a variety over $\CC$ and $x\in X$ a closed point.
    Then the number of analytic branches of $X$ at $x$ 
    is equal to the normal multiplicity.
\end{lemma}

\begin{proof}
    Let $X^{\nu}\to X$ be the normalisation.
    Denote by $X^{\an}$ the analytic space associated to $X$.
    By \cite[Satz~4, page~122]{algebraic-complex-normalisation}, the analytic space associated to $X^{\nu}$
    is the normalisation of $X^{\an}$ is the category of analytic spaces.
    Then the result follows from \cite[Expos\'e~XII, Th\'eor\`eme et d\'efinition~1.1]{SGA1} and
    \cite[Theorem~2.1~(i)]{normalisation-complex-variety}.
\end{proof}


\begin{lemma}\label{surface-analytic-comp-bnd}
    Let $t\in \RR^{>0}$.
    Let $(X, \Delta)$ be a $\QQ$-Gorenstein surface lc pair over $\CC$.
    Let $x\in X$ be an arbitrary closed point.
    Then the number of analytic branches of $\Delta$ at $x$, which have coefficients $\ge t$,
    is bounded from above depending only on $t$.
\end{lemma}

\begin{proof}
    Since analytic branches are exactly the local irreducible components of $\Delta$ at $x$ 
    in the complex topology, the proof is the same as \cite[Lemma~4.4]{birkar2025steindegreelogcalabiyau}.
\end{proof}


\begin{corollary}\label{surface-comp-bnd}
    Let $t\in \RR^{>0}$.
    Let $(X, \Delta)$ be a $\QQ$-Gorenstein surface lc pair 
    over an algebraically closed field $k$ of characteristic zero.
    Let $x\in X$ be an arbitrary closed point.
    Then the number of irreducible components of $\Delta$ 
    passing through $x\in X$ and having coefficients $\ge t$
    is bounded from above depending only on $t$.
\end{corollary}

\begin{proof}
    Let $k'\subset k$ be a subfield that is finitely generated over $\QQ$
    so that there exist a normal $k'$-surface $X'$ and an $\RR$-divisor $\Delta'$ on $X'$,
    where $X = X'\otimes_{k'}k$, and $\Delta$ is the flat pullback of $\Delta'$ to $X$.
    We can assume that $x$ is the unique inverse image of a $k'$-point $x'$ of $X'$,
    and that every irreducible component of $\Delta'$ is geometrically irreducible.
    As $k'$ is finitely generated over $\QQ$, we can take an embedding of fields
    $\iota\colon k'\to \CC$.  Let
    \[ X_{\CC}\coloneqq X'\otimes_{k'}\CC, \]
    where the fibre product is formed via $\iota$, and let $\Delta_{\CC}$ be the flat pullback of $\Delta'$ to $X_{\CC}$.
    By Lemma~\ref{descend-R-Cartier} and \S\ref{ss-pullback-R-Cartier-pairs}(5),
    $(X_{\CC}, \Delta_{\CC})$ is a $\QQ$-Gorenstein surface lc pair.
    As $x'\in X'$ is a $k'$-point, the inverse image of $x'$ in $X_{\CC}$ is a single 
    closed point $\bar{x}\in X_{\CC}$.  As each irreducible component of $\Delta'$
    is geometrically irreducible, the number of irreducible components of $\Delta$ passing through $x\in X$
    is equal to that of $\Delta_{\CC}$ passing through $\bar{x}\in X_{\CC}$.
    Thus, it suffices to consider $\QQ$-Gorenstein surface lc pairs over $\CC$,
    and this follows from Lemma~\ref{surface-analytic-comp-bnd}.
\end{proof}


\begin{lemma}\label{normal-mult-bnd}
    Let $t\in \RR^{>0}$.
    Let $(X, B)$ be a $\QQ$-Gorenstein lc pair over $\CC$
    and $S$ an irreducible component of $B$ with $\coeff_S B \ge t$.
    Then there exists a closed subset $Z\subset S$ of codimension $\ge 2$ in $S$ such that 
    for every closed point $s\in S\setminus Z$, the normal multiplicity of $S$ at $s$ 
    is bounded from above depending only on $t$.
\end{lemma}

\begin{proof}
    Taking general hyperplane sections, we can assume that $\dim X = 2$.
    Then the result follows from Lemmas~\ref{mult-branches} and \ref{surface-analytic-comp-bnd}.
\end{proof}


\medskip
(4)
Now we show the same result as Lemma~\ref{normal-mult-bnd} for lc pair over non-closed fields.

\begin{proposition}\label{normal-mult-bnd-general-fields}
    Let $t\in \RR^{>0}$.
    Let $(X, B)$ be a $\QQ$-Gorenstein lc pair over a field $k$ of characteristic zero.
    Let $S$ be an irreducible component of $B$ with $\coeff_S B \ge t$.
    Then there exists a closed subset $Z\subset S$ of codimension $\ge 2$ such that 
    for every closed point $s\in S\setminus Z$, the geometric normal multiplicity of $S$ at $s$ 
    is bounded from above depending only on $t$.
\end{proposition}

\begin{proof}
    In Steps~1 to 3, we assume $k$ is algebraically closed.
    In Step~4, we treat the situation when $k$ is an arbitrary field of characteristic zero.

    \medskip
    \emph{Step~1}.
    In this step, we spread out all the data to a finitely generated subalgebra of $\CC$.
    Let $\pi\colon S^{\nu}\to S$ be the normalisation of $S$.
    By \S\ref{s-descent}, there exist
    \begin{itemize}
        \item a subfield $L\subset k$ that is finitely generated over $\QQ$,
        \item a $\QQ$-Gorenstein lc pair $(X_L ,B_L)$ over $L$,
        \item an irreducible component $S_L$ of $B_L$ with $\coeff_{S_L}B_L \ge t$, and 
        \item a normal scheme $S_L^{\nu}$ admitting a finite birational morphism $\pi_L\colon S_L^{\nu}\to S_L$,
    \end{itemize}
    such that 
    \begin{itemize}
        \item $X = X_L\otimes_L k$ and $S^{\nu} = S_L^{\nu}\otimes_L k$, 
        \item $B, S$ are the flat pullbacks of $B_L, S_L$ to $X$ respectively, and
        \item $\pi\colon S^{\nu}\to S$ is the base change of $\pi_L\colon S_L^{\nu}\to S_L$.
    \end{itemize}
    As $L$ is finitely generated over $\QQ$, there exists an embedding of fields
    \[ \iota\colon L\hookrightarrow \CC. \]
    Let 
    \[ X_{\CC}\coloneqq X_L\otimes_L \CC~~\mbox{and}~~S_{\CC}^{\nu} \coloneqq S_L^{\nu}\otimes_L \CC, \]
    where the fibre product is formed via the embedding $\iota$,
    and let $B_{\CC}, S_{\CC}$ be the flat pullbacks of $B_L, S_L$ to $X_{\CC}$ respectively.
    By \cite[Proposition~(4.5.1)]{EGA-IV-2}, $S_{\CC}$ is also irreducible.
    Let $Z_{\CC}\subset S_{\CC}$ be a closed subset of codimension $\ge 2$ as in Lemma~\ref{normal-mult-bnd}.
    Denote by 
    \[ \pi_{\CC}\colon S_{\CC}^{\nu}\to S_{\CC} \]
    the base change of $\pi_L$ to $\spec \CC$.
    Then for each closed point $w\in S_{\CC}\setminus Z_{\CC}$, the cardinality of $\pi_{\CC}^{-1}(w)$
    is bounded from above depending only on $t$.
    \[\xymatrix{
    X\ar[d]\ar[r] & X_L\ar[d] & X_R\ar[d]\ar[l] & X_{\CC}\ar[d]\ar[l] \\
    \spec k\ar[r] & \spec L & \spec R\ar[l] & \spec \CC\ar[l]
    }\]
    By \S\ref{s-descent}, there exist
    \begin{itemize}
        \item a subalgebra $R\subset\CC$ that is finitely generated over $L$,
        \item a $\QQ$-Gorenstein lc pair $(X_R, B_R)$ over $\spec R$,
        \item an irreducible component $S_R$ of $B_R$, 
        \item a normal scheme $S_R^{\nu}$, 
        \item a finite birational morphism $\pi_R\colon S_R^{\nu}\to S_R$, and
        \item a closed subset $Z_R\subset S_R$ of codimension $\ge 2$,
    \end{itemize}
    such that 
    \begin{itemize}
        \item $\spec R$ is a regular scheme,
        \item $X_{\CC}=X_R\otimes_R\CC$, $Z_{\CC}=Z_R\otimes_R \CC$, $S_{\CC}= S_R\otimes_R \CC$ and $S_{\CC}^{\nu} = S_R^{\nu}\otimes_R \CC$, 
        \item $X_R, S_R, S_R^{\nu}$ and irreducible components of $Z_R, B_R$ are flat over $\spec R$, and
        \item $\pi_{\CC}$ is the base change of $\pi_R$ to $\spec \CC$.
    \end{itemize}

    \medskip
    \emph{Step~2}.
    Let $\eta_R$ be the generic point of $\spec R$.  Denote by
    \[ S_{\eta_R}^{\nu}\coloneqq S_R^{\nu}\times_{\spec R}\eta_R,~ S_{\eta_R} \coloneqq S_R\times_{\spec R} \eta_R,~\mbox{and}~
    Z_{\eta_R} \coloneqq Z_R\times_{\spec R} \eta_R. \]
    By construction in Step~1, we have
    \[ S_{\CC}^{\nu}=S_{\eta_R}^{\nu}\times_{\eta_R}\spec \CC, ~S_{\CC}=S_{\eta_R}\times_{\eta_R}\spec \CC, ~
    Z_{\CC} = Z_{\eta_R}\times_{\eta_R}\CC. \]
    Moreover, let 
    \[ U_R\coloneqq S_R\setminus Z_R,~
    U_R^{\nu} \coloneqq U_R\times_{S_R}S_R^{\nu},~
    U_{\eta_R}\coloneqq U_R\times_{\spec R}\eta_R,~\mbox{and}~
    U_{\eta_R}^{\nu}\coloneqq U_R^{\nu}\times_{\spec R}\eta_R. \]
    By Step~1, there exists a $\delta\in \NN$ \emph{depending only on $t$} so that 
    for each scheme-theoretic point $q\in U_{\eta_R}$,
    the cardinality of geometric irreducible components of the fibre of 
    $U^{\nu}_{\eta_R} \to U_{\eta_R}$ over $q$ is $\le \delta$.
    By \cite[Corollaire~(9.7.9)]{EGA-IV-3}, 
    there is an open dense subset $W_R\subset U_R$ containing $U_{\eta_R}$
    such that the cardinality of geometric irreducible 
    components of the fibre of 
    \[ W_R\times_{S_R}{S_R^{\nu}}\to W_R \]
    over each closed point of $W_R$ is $\le \delta$.  Set
    \[ T_R\coloneqq U_R\setminus W_R. \]
    As $T_R$ is quasi-compact, $T_R$ admits a finite affine open cover $\spec A_i$ with
    $i=1,\dots, l$.  Since $T_R\cap U_{\eta_R} = \emptyset$, we have
    $A_i\otimes_R K = \emptyset$ for each $1\le i\le l$, where $K$ is the fraction field of $R$.
    Then there exists a nonzero $a_i\in R$ so that the image of $a_i$ in $A_i$ is zero.  Let
    \[ \alpha \coloneqq a_1\cdot a_2\cdots a_l, \]
    which is a nonzero element in $R$.  Denote by $R_{\alpha}$
    the localisation of $R$ by $\alpha$, which is also a regular ring of finite type over $L$.
    Then by assumption,
    \[ T_R\otimes_R R_{\alpha} = \emptyset. \]
    In other words, the inverse image of $\spec R_{\alpha}$ in $U_R$ is contained in $W_R$.
    Replacing $R$ by $R_{\alpha}$, we can assume that for each 
    scheme-theoretic point $u\in U_R$, the cardinality of geometric irreducible components of the fibre of 
    $U^{\nu}_R \to U_R$ over $u$ is $\le \delta$.
    Moreover, as $S_R$ and irreducible components of $Z_R$ are flat over $\spec R$,
    we can assume for each scheme-theoretic point $b\in \spec R$,
    the closed subset $Z_{R,b}$ has codimension $\ge 2$ in $S_{R,b}$,
    where $Z_{R,b}, S_{R,b}$ are the fibres of $Z_R, S_R$ over $b$ respectively; 
    see \cite[Corollaire~(14.2.2)]{EGA-IV-3}.
    
    \medskip
    \emph{Step~3}.
    Pick a general closed point $p\in \spec R$.  Denote by $L'$ the residue field of $\spec R$ at $p$,
    which is finite over $L$.
    Denote by $X_{L'}, Z_{L'}, S_{L'}, S_{L'}^{\nu}$ the fibres of $X_R, Z_R, S_R, S_R^{\nu}$ over $p$ respectively.
    Then
    \[ X_{L'} = X_L\otimes_L L',~S_{L'} = S_L\otimes_L L',~\mbox{and}~S_{L'}^{\nu} = S_L^{\nu}\otimes_L L'. \]
    Again by \cite[Proposition~(4.5.1)]{EGA-IV-2}, $S_{L'}$ is irreducible.
    As irreducible components of $B_R$ are flat over $\spec R$, 
    the restriction of $B_R$ to $X_{L'}$ is well-defined,
    which we denote by $B_{L'}$.  
    It is evident that $B_{L'}$ is the flat pullback of $B_L$ to $X_{L'}$.
    
    As $k$ is algebraically closed, there is an embedding between fields
    \[ \tau\colon L'\to k \]
    so that the composite homomorphism $L \to L'\xrightarrow{\tau}k$
    is equal to $L\to k$.  Then we have
    \[ X = X_{L'}\otimes_{L'}k,~S=S_{L'}\otimes_{L'}k, ~\mbox{and}~S^{\nu} = S_{L'}^{\nu}\otimes_{L'}k, \]
    where the fibre products are form via $\tau\colon L'\to k$.  Let
    \[ Z \coloneqq Z_{L'}\otimes_{L'}k~~\mbox{and}~~ U\coloneqq S\setminus Z. \]
    It is evident that $Z$ has codimension $\ge 2$ in $S$.
    Moreover, by Step~2, for each closed point $s\in U$, the fibre of $S^{\nu}\to S$ over $s$
    has cardinality $\le \delta$.  
    Note that $\delta$ \emph{depends only on $t$}, but not on the field $k$.
    This proves the desired result
    when $k$ is an algebraically closed field.
    
    \medskip
    \emph{Step~4}.
    Now we show that the result holds for arbitrary field $k$ of characteristic zero.
    Denote by $\bar{k}$ the algebraic closure of $k$.
    Let $\bar{X}\coloneqq X\otimes_k \bar{k}$, and let $\bar{B}$
    be the flat pullback of $B$ to $\bar{X}$.
    By \S\ref{ss-pullback-R-Cartier-pairs}, $(\bar{X}, \bar{B})$ is a 
    $\QQ$-Gorenstein lc pair over $\bar{k}$.
    Let $\bar{S}$ be the base change of $S$ to $\spec \bar{k}$.
    Denote by $\bar{S}_1,\dots, \bar{S}_r$ the irreducible components of $\bar{S}$.
    Then each $\bar{S}_i$ has coefficient $\ge t$ in $\bar{B}$.
    For each $\bar{S}_i$, let $\bar{S}_i^{\nu}\to \bar{S}_i$ be the normalisation of $\bar{S}_i$.
    Let $S^{\nu}\to S$ be the normalisation of $S$.
    Then by \cite[Corollaire~(6.14.5)]{EGA-IV-2}, we see that $S^{\nu}\otimes_k \bar{k}$
    is equal to the disjoint union $\bar{S}_1^{\nu}\cup \cdots \cup \bar{S}_r^{\nu}$.
    Let $\bar{Z}_i\subset \bar{S}_i$ be a closed subset of codimension $\ge 2$ in $\bar{S}_i$
    as in Step~3.  Let $k'/k$ be a finite extension so that there are closed subschemes $Z_i'\subset X'\coloneqq X\otimes_k k'$
    satisfying $\bar{Z}_i = Z_i'\otimes_{k'}\bar{k}$.
    As $X'\to X$ is finite, the image $T_i$ of $Z_i'$ in $X$ is contained in $S$, and $T_i$
    is of codimension $\ge 2$ in $S$.  Let
    \[ T \coloneqq T_1\cup \cdots \cup T_r. \]
    Then $T$ is a closed subset of $S$ of codimension $\ge 2$.
    Let $s\in S\setminus T$ be an arbitrary closed point.
    Pick one of the inverse images $\bar{u}\in S\otimes_k \bar{k}$ of $s$ in $\bar{X}$. 
    By construction, if $\bar{u}\in \bar{S}_i$ for some $1\le i\le r$,
    then $\bar{u}\in \bar{S}_i\setminus \bar{Z}_i$.
    Note that $\bar{u}$ can be contained in several $\bar{S}_i$'s.

    If $\dim X = 2$, by \cite[Lemma~4.4]{birkar2025steindegreelogcalabiyau}, there exists
    an $N(t)\in \NN$ \emph{depending only on $t$} so that the number of $\bar{S}_i$'s
    passing through $\bar u$ is $\le N(t)$.  Now assume that $\dim X \ge 3$.  Set
    \[ \bar{P}\coloneqq \bigcup_{\substack{J\subseteq \Set{1,\dots, r} \\ \abs{J}=N(t)+1}}
    \bigg(\bigcap_{j\in J} \bar{S}_j\bigg). \]
    As $\dim (\bar{S}_i\cap \bar{S}_{i'}) \le \dim X -2$ for any distinct $i, i'$,
    it is evident that $\dim \bar{P} \le \dim X - 2$.  We claim that 
    \begin{equation}\label{dim-wt-T-3}
        \dim \bar{P} \le \dim X - 3.
    \end{equation}
    To this end, assume that $\bar{P}$ admits an irreducible component $\Gamma$ with $\dim \Gamma = \dim X - 2$.
    By applying Bertini's theorem and adjunction inductively, we can pick a general 
    complete intersection surface $V\subset \bar{X}$ so that
    \begin{itemize}
        \item $V$ is normal and $\QQ$-Gorenstein,
        \item $(V, \bar{B}|_V)$ is lc, and
        \item $V\cap \Gamma$ is nonempty.
    \end{itemize}
    Taking $V$ sufficiently general, we can assume that each $V\cap \bar{S}_j$ is an irreducible curve.
    It is evident that each $V\cap \bar{S}_i$ has coefficient $\ge t$ in $\bar{B}|_V$.
    Pick $v\in V\cap \Gamma$.  Then by construction, there are $N(t)+1$ irreducible components of $\bar{B}|_V$
    with coefficient $\ge t$ passing through $v$, contradicting \cite[Lemma~4.4]{birkar2025steindegreelogcalabiyau}.
    Thus, we can conclude that \eqref{dim-wt-T-3} holds.

    Let $k'/k$ be a finite extension such that there exists a closed subscheme $P'\subset X'$
    satisfying $\bar{P} = P'\otimes_{k'}\bar{k}$.
    As $X'\to X$ is finite, the image $P$ of $P'$ in $X$ is contained in $S$, and $P$
    is of codimension $\ge 3$ in $S$.  Let
    \[ Z \coloneqq T \cup P. \]
    By construction, $Z$ is a closed subset of $S$ of codimension $\ge 2$.
    For each closed point $s\in S\setminus Z$, and for each inverse image $\bar u\in S\otimes_k \bar{k}$,
    we see that the number of $\bar{S}_i$'s passing through $\bar{u}$ is $\le N(t)$.
    If $\bar{u}\not\in \bar{S}_i$, we set the fibre product 
    $\bar{S}_i^{\nu}\times_{\bar{S}_i} \spec k(\bar{u})$ as empty by convention.  Then
    it is evident that
    \[ \mf{n}(S^{\nu}\times_S \spec \ol{k(s)})
    = \sum_{i=1}^r \mf{n}(\bar{S}_i^{\nu}\times_{\bar{S}_i} \spec k(\bar{u})). \]
    By Step~3, each $\mf{n}(\bar{S}_i^{\nu}\times_{\bar{S}_i} \spec k(\bar{u}))$ 
    is bounded from above \emph{depending only on $t$}.
    Hence $\mf{n}(S^{\nu}\times_S \spec \ol{k(s)})$, which is equal to the geometric normal multiplicity 
    of $S$ at $s$, is bounded from above \emph{depending only on $t$}.
\end{proof}


\subsection{Vertical Stein degree}

\begin{definition}\label{stable-CY-defn}
    Let $k$ be a field of characteristic zero.
    A \emph{stable log Calabi-Yau fibration} 
    \[ \Set{(X, B), A}\to Z~\mbox{over}~ k \]
    consists of $(X, B)\to Z$ and an $\RR$-divisor $A\ge 0$, where
    \begin{itemize}
        \item $(X, B)$ is an lc pair over $k$,
        \item $X\to Z$ is a contraction over $k$ such that $K_X + B \sim_{\RR}0/Z$,
        \item $A$ is ample$/Z$, and
        \item $(X, B + tA)$ is lc for some $t>0$.  
    \end{itemize}
    If $\dim Z = 0$, we call $\Set{(X, B), A}$ a \emph{stable log Calabi-Yau pair over $k$}.
\end{definition}


\begin{lemma}\label{complement-coeffs}
    Let $n\in \NN$.  Let $k$ be a field of characteristic zero.
    Let $X$ be a normal $k$-scheme, which may not be geometrically irreducible, 
    and let $D$ be a $\QQ$-divisor on $X$.  Assume $(X, D)$ is lc and that
    \[ n(K_X + D)~\mbox{is Cartier}. \]
    Let $W$ be a codimension two irreducible closed subscheme of $X$ contained in $\rddown{D}$.  
    If $W$ is not an lc centre of $(X, D)$, then any Weil divisor is $\QQ$-Cartier 
    along $W$, and the Cartier index along $W$ is bounded from above depending only on $n$.
\end{lemma}

\begin{proof}
    \emph{Step~1}.
    In this step, we reduce to the case of pairs over algebraically closed fields.
    Taking the Stein factorisation of $X\to \spec k$, we can assume
    that $X$ is geometrically irreducible.
    Let $\bar{k}$ be the algebraic closure of $k$ and $\bar{X}\coloneqq X\otimes_k\bar{k}$.
    Denote by $\bar{D}$ the flat pullback of $D$ to $\bar{X}$.
    By \S\ref{ss-descent-b-nef} (5), $(\bar{X}, \bar{D})$ is an lc pair over $\bar{k}$.
    Let $\bar{W}_1,\dots, \bar{W}_r$ be the irreducible components of $W\otimes_k\bar{k}$.
    It is evident that each $\bar{W}_i$ is a closed subscheme of $\rddown{\bar{D}}$
    of codimension $\ge 2$.  Again by \S\ref{ss-descent-b-nef} (5),
    every $\bar{W}_i$ is not an lc centre of $(\bar{X}, \bar{D})$.
    Let $H$ be an arbitrary Weil divisor containing $W$,
    and let $\bar{H}$ be the flat pullback of $H$ to $\bar{X}$.
    Then every $\bar{W}_i$ is contained in $\Supp \bar{H}$.
    Assume that $\bar{H}$ is $\QQ$-Cartier along each $\bar{W}_i$
    and that the Cartier index is $\ell\in \NN$, i.e. $\ell\bar{H}$ is Cartier along each $\bar{W}_i$.
    Let $\eta_W$ be the generic point of $W$, and let $\bar{\eta}_i$
    be the generic point of $\bar{W}_i$.  Then the local homomorphism
    \[ \mcO_{X, \eta_W} \to \mcO_{\bar{X}, \bar{\eta}_i} \]
    is faithfully flat.  Let 
    $\mcO_X(\ell H)_{\eta_W}, \mcO_{\bar{X}}(\ell \bar{H})_{\bar{\eta}_i}$
    be the localisations of $\mcO_X(\ell H), \mcO_{\bar{X}}(\ell\bar{H})$ at $\eta_W, \bar{\eta}_i$ respectively.
    By \cite[Propositions~1.6, 1,8]{Hart-stable-refl}, we have
    \[ \mcO_X(\ell H)_{\eta_W} \otimes_{\mcO_{X,\eta_W}} \mcO_{\bar X,\bar\eta_i}
       \cong \mcO_{\bar X}(\ell\bar{H})_{\bar{\eta}_i}. \]
    As $\ell \bar{H}$ is Cartier, $\mcO_{\bar X}(\ell\bar{H})_{\bar{\eta}_i}$ is free of rank one.
    Then by \cite[Proposition~14.48]{GW-agbook}, $\mcO_X(\ell H)_{\eta_W}$
    is also free of rank one.  Thus, $\ell H$ is Cartier along $W$.
    Hence in the rest of proof, we can assume that $k$ is algebraically closed.

    \medskip
    \emph{Step~2}.
    In this step, we reduce to the case when $X$ is a surface.
    To this end, we assume the lemma has been established for lc surface pairs
    over an algebraically closed field of characteristic zero;
    in particular, let $N$ be the upper bound of Cartier indices in this case,
    which depends only on $n$.
    Let $H$ be an arbitrary Weil divisor on $X$ such that $\Supp H$ contains $W$.
    For each $1\le j\le N$, let
    \[ \mcM_j\coloneqq \mcO_X(jH), \]
    which is a reflexive sheaf of rank one.
    As $\mcO_{X, \eta_W}$ is a normal local ring of dimension 2, it is Cohen-Macaulay;
    see \cite[\S2.2]{Cohen-Macaulay-rings-book}.
    Denote by $(\mcM_j)_{\eta_W}$ the localisation of $\mcM_j$ at $\eta_W$.
    Since $\mcM_j$ is reflexive, by \cite[Proposition~1.4.1]{Cohen-Macaulay-rings-book},
    $(\mcM_j)_{\eta_W}$ is also Cohen-Macaulay.
    Thus, by \cite[Corollaire~(6.11.3)]{EGA-IV-2}, and by shrinking $X$ near $\eta_W$, we can assume that 
    $X$ and all the sheaves $\mcM_1,\dots, \mcM_N$ are Cohen-Macaulay.

    Let $d\coloneqq \dim X$, and let $L_1, \dots, L_{d-2}$ be general hyperplane sections of $X$.  Set
    \[ V\coloneqq X\cap L_1\cap \cdots \cap L_{d-2}. \]
    By \cite[Th\'eor\`eme~6.10]{jou-bertini} and \cite[3.4.9]{joins-intersections}, $V$ is a normal $k$-variety.
    Moreover, by adjunction \cite[\S4.1]{Kol_singularities_of_MMP}, we can write
    \[ K_V + D_V = (K_X+D+L_1+\cdots + L_{d-2})|_V, \]
    where $(V, D_V)$ is an lc pair.  As $L_1,\dots, L_{d-2}$ are Cartier, we see that
    \[ n(K_V+D_V)~\mbox{is Cartier}. \]
    Note that $D_V$ is equal to the restriction of $D$ to $V$.
    As $\codim (W, X) = 2$, $W\cap V$ consists of finitely many closed points.
    Pick $p\in W\cap V$.
    It is evident that $p\in \rddown{D_V}$, and that $p$ is not an lc centre of $(V, D_V)$.  Let
    \[ A\coloneqq \mcO_{X, p}~~\mbox{and}~~\bar{A}\coloneqq \mcO_{V,p}. \]
    Moreover, let $g_1, \dots, g_{d-2}$ be local defining equations of $V$ near $p$,
    where each $L_i$ is defined by $g_i$.
    For each $g_i$, let $g_{i,p}$ be its image in $A$, and let
    \[ I \coloneqq (g_{1,p}, \dots, g_{d-2,p})\subset A. \]
    Then we have 
    \[ \bar{A} = A/I. \]
    For each $1\le j\le N$, let $M_j$ be the localisation of $\mcM_j$ at $p$.
    Then by \cite[Theorem~2.1.2]{Cohen-Macaulay-rings-book}, 
    $\Set{g_{1,p}, \dots, g_{d-2,p}}$ is a regular sequence for $A$ and all the $M_j$.
    Thus, by \cite[Chap.~3, Exercise~19]{AM-comm-algebra} and \cite[Theorem~2.1.3]{Cohen-Macaulay-rings-book}, each
    \[ \bar{M}_j\coloneqq M_j/I M_j \]
    is a Cohen-Macaulay $\bar{A}$-module of dimension two.
    Hence $\bar{M}_j$ is a reflexive $\bar{A}$-module by \cite[Proposition~1.4.1]{Cohen-Macaulay-rings-book}.
    Let $q\in V$ be a codimension one point such that $p$ is contained 
    in the closure of $q$.
    As $L_1,\dots, L_{d-2}$ are general, we can assume that
    $q$ is contained in the nonsingular locus of $X$.  Then
    \[ (\bar{M}_j)_q = (\mcM_j|_V)_q = (\mcO_V(jH|_V))_q, \]
    which is a free $\bar{A}_q$-module of rank one.
    Thus, $\bar{M}_j$ is a reflexive $\bar{A}$-module of rank one.
    By \cite[Proposition~1.6]{Hart-stable-refl}, we get
    \[ \bar{M}_j \cong \mcO_V(j H|_V)_p. \]
    We claim that $\bar{M}_j$ is free over $\bar{A}$ if and only if $M_j$ is free over $A$.
    It is evident that $\bar{M}_j$ is free if $M_j$ is so.
    Conversely, assume that $\bar{M}_j$ is a free $\bar{A}$-module.
    In this case, as $\bar{M}_j$ has rank one, $\bar{M}_j$ is generated by a single element $\bar{e}\in \bar{M}_j$
    as an $\bar{A}$-module.  Let $e\in M_j$ so that the image of $e$ in $\bar{M}_j$ is $\bar{e}$.  Then
    \[ M_j = A e+ I M_j. \]
    By Nakayama's lemma \cite[page~8]{CA-Mat}, we have $M_j = Ae$.
    As $M_j$ is torsion free, we can conclude that $M_j$ is an free $A$-module generated by $e$,
    whence the claim.

    By taking $L_1,\dots, L_{d-2}$ sufficiently generally, 
    we can assume that for each $1\le j\le N$, if $\mcM_j$ is locally free at $\eta_W$,
    then it is locally free at $p\in W\cap V$.  Hence by the discussion as above, for each $1\le j\le N$,
    we have that 
    \[ jH~\mbox{is Cartier at}~\eta_W~\Longleftrightarrow~jH|_V~\mbox{is Cartier at}~p. \]
    By assumption, there exists a $1\le j\le N$ so that $j H|_V$ is Cartier at $p$.
    Then we deduce that $H$ is $\QQ$-Cartier along $W$, and its Cartier index is also $\le N$.

    \medskip
    \emph{Step~3}.
    Now we treat the case when $(X, D)$ is an lc surface pair over an algebraically closed field.
    In this case, $W$ is a closed point of $X$ contained in $\rddown{D}$.
    Let $P$ be an irreducible component of $\rddown{D}$ containing $W$,
    and let $\nu\colon P^{\nu}\to P$ be the normalisation of $P$.  By adjunction, we can write
    \[ K_{P^{\nu}} + D_{P^{\nu}} = \nu^*(K_X + D). \]
    Since $n(K_X + D)\sim 0$, every coefficient of $D_{P^{\nu}}$ is of the form $l/n$ for some natural number
    $1\le l\le n$; cf. \cite[\S 4.1]{Kol_singularities_of_MMP}.
    As $W$ is not an lc centre, $(X, D)$ is log terminal near $W$.
    Thus, $X$ is $\QQ$-factorial near $W$ by \cite[Proposition~4.11]{km98}.
    Let $m\in \NN$ be the Cartier index of $K_X + P$ at $W$.
    By \cite[Proposition~3.9]{Sh-3-fold-log-flips}, $m=1$ if and only if $X$ is nonsingular along $P$,
    so we can assume that $m\ge 2$.  Let $p\in P^{\nu}$ be a closed point mapping to $W$.
    Then by \cite[Corollary~3.10]{Sh-3-fold-log-flips}, 
    \[ \coeff_p D_{P^{\nu}} = \frac{m-1}{m} + \sum \frac{k_i}{m}d_i, \]
    where $k_i$ are natural numbers and $d_i$ are coefficients of some irreducible components of $D$.
    As $W$ is not an lc centre of $(X, D)$, $p$ is not an lc centre of $(P^{\nu}, D_{P^{\nu}})$.
    Hence we get
    \[ \frac{m-1}{m}\le \coeff_p D_{P^{\nu}} \le \frac{n-1}{n}. \]
    Then we see that $m\le n$.  Since the Cartier index of any integral divisor along $W$ divides $m$ 
    \cite[Proposition~3.9]{Sh-3-fold-log-flips}, we get the desired result.
\end{proof}


\begin{lemma}[cf. \protect{\cite[Theorem~1.1]{birkar-geo-polar-variety}}]\label{vol-bir}
    Let $d\in \NN$ and $\epsilon\in \RR^{>0}$.
    Then there exists a natural number $m\in \NN$ 
    depending only on $d,\epsilon$ such that the following properties hold.

    Let $k$ be a field of characteristic zero.  Assume that
    \begin{itemize}
        \item $X$ is a projective $\epsilon$-lc $k$-variety of dimension $d$,
        \item $N$ is an integral divisor on $X$ that is nef and big over $\spec k$, and
        \item $N-K_X$ is pseudo-effective over $\spec k$.
    \end{itemize}
    Then $\abs{m' N + L}$ and $\abs{K_X + m'N + L}$ define birational maps for any natural number $m'\ge m$
    and any integral pseudo-effective $\QQ$-Cartier divisor $L$.
\end{lemma}

\begin{proof}
    Let $N, L$ be integral divisors on $X$ as in the lemma.
    Let $k'\subset k$ be a subfield that is finitely generated over $\QQ$
    such that $X = X'\otimes_{k'}k$ for a $k'$-variety $X'$, and $N, L$
    are the flat pullbacks of integral divisors $N', L'$ on $X'$.
    Then $X'$ is a projective $\epsilon$-lc $k'$-variety of dimension $d$,
    $N'$ is nef and big over $\spec k'$, and $N' - K_{X'}$ is pseudo-effective over $\spec k'$.
    As $k'$ is finitely generated over $\QQ$, we can take an embedding of fields
    \[ \iota\colon k'\hookrightarrow \CC. \]
    Let 
    \[ X_{\CC}\coloneqq X'\otimes_{k'}\CC, \]
    where the fibre product is formed via $\iota$.
    Let $N_{\CC}, L_{\CC}$ be the flat pullbacks of $N', L'$ to $X_{\CC}$ respectively.
    Then it is evident that
    \begin{itemize}
        \item $X_{\CC}$ is a projective $\epsilon$-lc complex variety,
        \item $N_{\CC}$ is a nef and big integral divisor on $X_{\CC}$, and
        \item $N_{\CC} - K_{X_{\CC}}$ and $L_{\CC}$ are pseudo-effective.
    \end{itemize}
    By \cite[Theorem~1.1]{birkar-geo-polar-variety}, there exists an $m\in \NN$
    depending only on $d,\epsilon$ such that
    $\abs{m' N_{\CC} + L_{\CC}}$ and $\abs{K_{X_{\CC}} + m'N_{\CC} + L_{\CC}}$ 
    define birational maps for any natural number $m'\ge m$.
    By \cite[Proposition~III.9.3]{Hart}, $\abs{m' N + L}$ and $\abs{K_X + m' N + L}$
    also define birational maps.
\end{proof}


\begin{theorem}\label{vertical-thm}
    Let $\epsilon, t\in \RR^{>0}$ and $n,d\in \NN$.
    Let $k$ be a field of characteristic zero.
    Let 
    \[ f\colon (X, B)\to Z \]
    be a Fano type fibration over $k$ of relative dimension $d$ satisfying that
    \begin{itemize}
        \item $-K_X$ is ample$/Z$,
        \item $n(K_X + B)\sim 0/Z$, and
        \item $X$ is $\epsilon$-lc over the generic point of $Z$.        
    \end{itemize}
    Let $S$ be a prime divisor on $X$ such that
    \begin{itemize}
        \item the coefficient of $S$ in $B$ is $\ge t$, and
        \item the image $T$ of $S$ in $Z$ is a prime divisor.
    \end{itemize}
    Assume $\mc{H}_{\univ}(d,t)$ holds.
    Then $\sdeg(S^{\nor}/T)$ is bounded from above depending only on $n,d,\epsilon,t$,
    where $S^{\nor}$ is the normalisation of $S$.
\end{theorem}

\begin{proof}
    Denote by $\bar{k}$ the algebraic closure of $k$.
    If $\dim Z\ge 3$, applying Bertini's theorem and shrinking $Z$ near the generic point of $T$, 
    we can assume that $Z$ is a nonsingular surface and that $T$ is a nonsingular curve.  
    If $\dim Z = 1$, then $Z$ is a nonsingular curve over $k$, and $T$ is a closed point on $Z$.
    Denote by $k(T)$ the residue field of $Z$ at $T$.  As 
    \[ \sdeg\big((\AAA^1_k\times_{\spec k}S)^{\nor}/(\AAA_k^1\otimes_k k(T))\big) = \sdeg (S^{\nor}/\spec k(T)), \]
    we can also assume that $Z$ is a nonsingular surface by replacing $Z$ by $Z\times_{\spec k}\AAA_k^1$.

    \medskip

    \emph{Step~1}.
    In this step, we pick an appropriate log resolution of $X$.
    Let $\pi\colon W\to X$ be a log resolution of $(X, B + \Supp f^*T)$, 
    and let $h\colon W\to Z$ be the induced contraction.
    Denote by $\eta$ the generic point of $Z$, and 
    by $W_{\eta}, X_{\eta}$ the fibres of $W\to Z, X\to Z$ over $\eta$ respectively.
    Let $\pi_{\eta}\colon W_{\eta}\to X_{\eta}$ be the induced birational contraction,
    which is a log resolution of $(X_{\eta}, B|_{X_{\eta}})$.
    
    By Lemma~\ref{fano-bnd-very-ample}, we can take an $l\in \NN$
    and a finite set of natural numbers $\mf{R}$ depending only on $d, \epsilon$ 
    such that $-l K_{X_{\eta}}$ is very ample over $\spec K(Z)$ and that 
    \[ \vol_{/K(Z)}(-l K_{X_{\eta}}) \in \mf{R}. \]
    Let 
    \[ 0\le A_{\eta}\sim -l K_{X_{\eta}} \]
    be a general element of $\abs{-l K_{X_{\eta}}}$.
    Denote by $A_{\eta}^{\sim}$ the birational transform of $A_{\eta}$ to $W_{\eta}$.
    Then by taking $A_{\eta}$ sufficiently general, we can assume that
    $A_{\eta}$ is a prime divisor on $X_{\eta}$, and that
    \[ \pi_{\eta}^* A_{\eta} = A_{\eta}^{\sim}. \]
    In particular, we can assume that $\pi_{\eta}$ is a log resolution of $(X_{\eta}, B|_{X_{\eta}} + A_{\eta})$.
    Let $A$ be the closure of $A_{\eta}$ in $X$, which is a prime divisor on $X$.
    (Note that $A$ may not be $\QQ$-Cartier.)
    Let $A^{\sim}$ be the birational transform fo $A$ to $W$.
    By taking a higher log resolution without modifying $W_{\eta}$,
    we can assume that $\pi\colon W\to X$ is a log resolution of $(X, B + \Supp f^* T + A)$.    

    \medskip
    
    \emph{Step~2}.
    In this step, we run a sequence of MMPs over $Z$ to produce a boundary divisor 
    whose irreducible components over $T$ all have coefficient one.
    Let $B_h$ be the horizontal$/Z$ part of $B$ and
    $B_h^{\sim}$ the birational transform of $B_h$ to $W$.
    Denote by $\Ex(\pi)$ the reduced exceptional divisor of $\pi$.  
    Shrinking $Z$ around the generic point of $T$, 
    we can assume that each vertical$/Z$ irreducible component of $\Ex(\pi)$ dominates $T$.
    Write
    \[ K_W+B_W = \pi^*(K_X+B). \]
    Let $\Delta_W$ be the $\QQ$-divisor on $W$ that satisfies the following properties:
    assume $E$ is a prime divisor on $W$,
    \begin{itemize}
        \item if $E\subseteq \Supp B_h^{\sim}$, $\coeff_E \Delta_W = \coeff_E B_h^{\sim}$,
        \item if $E\subseteq \Supp h^*T$, $\coeff_E \Delta_W = 1$,
        \item if $E$ is horizontal$/Z$ and $E\subseteq \Ex(\pi)$, $\coeff_E \Delta_W = \max\Set{\coeff_E B_W, 1/2n}$, and
        \item $\coeff_E \Delta_W=0$ otherwise.
    \end{itemize}
    It is evident that 
    \[ B_W \le \Delta_W~~\mbox{and}~~\Delta_W\in \mb{T}_{2n}. \]
    By assumption, $B$ is $\QQ$-Cartier and ample$/Z$, hence $\pi^* B$ is nef and big over $Z$.
    It is evident that $\Delta_W\ge a\pi^*B$ for some sufficiently small $a\in \QQ^{>0}$.
    By \cite[Lemma~5.21]{lyu-takumi-2026relativeminimalmodelprogram}, we can take a sufficiently general
    \[ 0\le G\sim_{\QQ} a\pi^*B/Z \]
    so that it intersects $\Supp \Delta_W$ transversally.
    Moreover, pick $0<c\ll 1$ sufficiently small, and let
    \[ \wt{\Delta}_W \coloneqq (\Delta_W - a\pi^*B - c h^*T) + G. \]
    Then we get
    \[ 0\le \wt{\Delta}_W \sim_{\QQ}\Delta_W/Z. \]
    It is evident that $(W, \wt{\Delta}_W + (1/2)A^{\sim})$ is a klt pair over $k$.
    
    By Theorem~\ref{bchm-type-mmp}, we can run $\phi\colon W\dashrightarrow Y'$
    an MMP$/Z$ on the divisor $K_W + \wt{\Delta}_W + (1/2)A^{\sim}$ that ends with a good minimal model
    $Y'\to Z$ of $K_W + \wt{\Delta}_W + (1/2)A^{\sim}$.  As $\Delta_W\sim_{\QQ}\wt{\Delta}_W/Z$,
    $\phi\colon W\dashrightarrow Y'$ is also a $(K_W + \Delta_W + (1/2)A^{\sim})$-MMP over $Z$.  Let
    \[ \wt{\Delta}_{Y'} \coloneqq \phi_* \wt{\Delta}_W,~ \Delta_{Y'} \coloneqq \phi_*\Delta_W,~
    \mbox{and}~ A' \coloneqq \phi_*A^{\sim}. \]
    By construction in Step~1, $A^{\sim}$ is big$/Z$, hence by Lemma~\ref{no-van-big-divisor}, 
    $A'$ is also big$/Z$.
    Again by Theorem~\ref{bchm-type-mmp}, we can run $\varphi\colon Y'\bir Y''$ an MMP$/Z$
    on $K_{Y'} + \wt{\Delta}_{Y'}$ with scaling of $(1/2)A'$
    that ends with a good minimal model of $(Y', \wt{\Delta}_{Y'})$.  
    Then $\varphi\colon Y'\bir Y''$ is also an MMP$/Z$ on $K_{Y'}+\Delta_{Y'}$
    with scaling of $(1/2)A'$.
    Let
    \[ \wt{\Delta}_{Y''} \coloneqq \varphi_* \wt{\Delta}_{Y'},~ 
    \Delta_{Y''} \coloneqq \varphi_*\Delta_{Y'},~
    \mbox{and}~ A'' \coloneqq \varphi_*A'. \]
    Thus, we have that
    \[ K_{Y''} + \Delta_{Y''}\sim_{\QQ} K_{Y''} + \wt{\Delta}_{Y''}/Z \]
    is semiample$/Z$.  
    
    As $\varphi$ is an MMP$/Z$ with scaling of $(1/2)A'$,
    there exists a $\lambda>0$ such that
    \[ K_{Y''} + \wt{\Delta}_{Y''} + \lambda A'' ~\mbox{is nef}/Z. \]
    Then for any $0<\lambda_0<1$, 
    \[ (1-\lambda_0)(K_{Y''} + \wt{\Delta}_{Y''}) + \lambda_0
    (K_{Y''} + \wt{\Delta}_{Y''} + \lambda A'') = K_{Y''} + \wt{\Delta}_{Y''} + \lambda_0\lambda A''~
    \mbox{is nef}/Z. \]
    We can pick a sufficiently small $0<\lambda_0\ll 1$ so that 
    $(Y'', \wt{\Delta}_{Y''} + \lambda_0\lambda A'')$ is klt.
    As $\wt{\Delta}_{Y''}, A''$ are big$/Z$, the boundary $\wt{\Delta}_{Y''} + \lambda_0\lambda A''$
    is big$/Z$.  Then by \cite[Lemma~17.2, Corollary~17.8]{lyu-takumi-2026relativeminimalmodelprogram},
    we can conclude that $K_{Y''} + \wt{\Delta}_{Y''} + \lambda_0\lambda A''$ is semiample$/Z$.
    
    Let $\psi\colon Y''\to Y$ be the ample model of $K_{Y''} + \wt{\Delta}_{Y''} + \lambda_0\lambda A''$,
    that is, $\psi$ is a contraction so that $K_{Y''} + \wt{\Delta}_{Y''} + \lambda_0\lambda A''$
    is the pullback of an ample$/Z$ divisor on $Y$.
    \[\xymatrix{
    W\ar[d]_{\pi}\ar@{-->}[r]^{\phi} & Y'\ar@{-->}[r]^-{\varphi} & Y''\ar[r]^-{\psi} & Y \\
    X & & 
    }\]  
    Let
    \[ \wt{\Delta}_{Y} \coloneqq \psi_* \wt{\Delta}_{Y''},~ 
    \Delta_{Y} \coloneqq \psi_*\Delta_{Y''},~
    \mbox{and}~ A_Y \coloneqq \psi_*A''.  \]
    Since $K_{Y''}+\Delta_{Y''} + \lambda_0\lambda A''$ is big$/Z$, 
    $\psi$ is a birational contraction.  

    \medskip

    \emph{Step~3}.
    Let $Z^{\circ}\subset Z$ be an open subset over which $K_X+B+(1/2)A$ is ample,
    and let $X^{\circ}\subset X$ (respectively, $Y^{\circ}\subset Y$, 
    respectively, $Y''^{\circ}\subset Y''$, respectively, $W^{\circ}\subset W$)
    be the inverse image of $Z^{\circ}$ in $X$
    (respectively, respectively, in $Y$, in $Y''$, respectively, in $W$).
    Then $W^{\circ}\to X^{\circ}$ is the ample model of 
    $(K_W+\Delta_W+\lambda_0\lambda A^{\sim})|_{W^{\circ}}$.
    Since ample model is unique up to unique isomorphism 
    \cite[Lemma~26.3]{lyu-takumi-2026relativeminimalmodelprogram}, the pairs
    \[ (X^{\circ}, (B+\lambda_0\lambda A)|_{X^{\circ}})~~
    \mbox{and}~~ (Y^{\circ}, (\Delta_Y+\lambda_0\lambda A_Y)|_{Y^{\circ}}) \]
    are isomorphic.  Thus,
    \[ (K_Y+\Delta_Y)|_{Y^{\circ}}\sim_{\QQ}0/Z^{\circ}. \]
    In particular, the divisors $K_Y+\Delta_Y$ and $A_Y$ are $\QQ$-Cartier on $Y^{\circ}$.
    Write
    \[ E''^{\circ} \coloneqq (K_{Y''} + \Delta_{Y''})|_{Y''^{\circ}} - \psi^*((K_Y+\Delta_Y)|_{Y^{\circ}})~~
    \mbox{and}~~ F''^{\circ}\coloneqq \psi^*(A_Y|_{Y^{\circ}}) - A''|_{Y''^{\circ}}. \]
    Since $A''$ is the strict transform of $A_Y$ on $Y''$, the divisors $E''^{\circ}, F''^{\circ}$
    are exceptional over $Y^{\circ}$.  Moreover, by construction, we have
    \[ E''^{\circ} - \lambda_0\lambda F''^{\circ} \sim_{\QQ}0/Y^{\circ}. \]
    Then by negativity lemma (see Lemma~\ref{neg-lemma}), we have
    \[ E''^{\circ} = \lambda_0\lambda F''^{\circ} \ge 0. \]
    On the other hand, as $K_{Y''} + \Delta_{Y''}$ is semiample over $Z$, $E''^{\circ}$ is nef over $Y^{\circ}$.
    Again by Lemma~\ref{neg-lemma}, we get $E''^{\circ} \le 0$.  
    Therefore, we get 
    \[ E''^{\circ} = F''^{\circ} = 0. \]  
    Hence
    \[ (K_{Y''}+\Delta_{Y''})|_{Y''^{\circ}} = \psi^*((K_Y+\Delta_Y)|_{Y^{\circ}})\sim_{\QQ}0/Z^{\circ}. \]
    Since $K_{Y''}+\Delta_{Y''}$ is semiample$/Z$, by Lemma~\ref{semi-ample-codim-1},
    and by shrinking $Z$ near the generic point of $T$, we can assume that
    \[ K_{Y''}+\Delta_{Y''}\sim_{\QQ}0/Z. \]
    Pushing down to $Y$, we deduce that
    \[ K_Y+\Delta_Y\sim_{\QQ}0/Z. \]
    Hence $A_Y$ is an ample$/Z$ $\QQ$-Cartier prime divisor on $Y$.
    Moreover, since $(Y'', \wt{\Delta}_{Y''})$ is klt, 
    and since $\wt{\Delta}_{Y''}$ is big$/Z$, we see that $Y''$ is of Fano type over $Z$.
    Hence $Y$ is also of Fano type over $Z$.
    By construction, the coefficients of $\Delta_Y$ belong to the finite set $\mb T_{2n}$.
    By Theorem~\ref{t-bnd-compl-usual-local}, there exist an $\ell\in \NN$
    depending only on $d, n$ and a monotonic $\ell$-complement $K_Y + \Delta_Y^+$
    of $K_Y + \Delta_Y$ over the generic point of $T$.
    Set $E=\Delta_Y^+-\Delta_Y\geq 0$.
    Since both $K_Y+\Delta_Y$ and $K_Y+\Delta_Y^+$ are
    $\QQ$-linearly trivial over $Z$, we have
    $E\sim_{\QQ}0/Z$.  Effectivity implies that $E$ is vertical over $Z$. 
    Denote by $g$ the contraction $Y\to Z$.  
    After shrinking $Z$ around the generic point of $T$, we get
    \[ \Supp E\subseteq\Supp(g^*T)\subseteq\lfloor\Delta_Y\rfloor. \]
    As $(Y, \Delta_Y^+)$ is lc, $E=0$.  Hence $\ell(K_Y+\Delta_Y)\sim 0/Z$.
    
    Let $\gamma\colon V\to W, \beta\colon V\to Y$ be a common resolution, and let 
    $\alpha\colon V\to X$ the the composite birational contraction $\pi\circ\gamma$.  Write
    \[ K_V + B_V = \alpha^*(K_X+B)~~\mbox{and}~~K_V+\Delta_V = \beta^*(K_Y + \Delta_Y). \]
    Set
    \[ M \coloneqq \Delta_V - B_V = \beta^*(K_Y + \Delta_Y) - \alpha^*(K_X+B). \]
    Then $\beta_* M$ is equal to the pushdown of $\Delta_W-B_W$ to $Y$.  Hence $\beta_* M \ge 0$.
    Moreover, as $K_X+B\sim_{\QQ}0/Z$, we get $M\equiv 0/Y$.
    Thus, by Lemma~\ref{neg-lemma}, we have that $M\ge 0$.
    Let $S_V$ be the birational transform of $S$ to $V$.  Then we get
    \[ \coeff_{S_V} \Delta_V \ge \coeff_{S_V} B_V = \coeff_S B \ge t. \]
    Therefore, we can conclude that $a(S, Y, \Delta_Y)\le 1-t$.
    
    \medskip

    \emph{Step~4}.
    By construction, $\Supp (g^* T) \subseteq \rddown{\Delta_Y}$.
    In this step, we show that for every horizontal$/T$ 
    irreducible component $P$ of $g^*T$, $\sdeg(P^{\nor}/T)$
    is bounded from above depending only on $n,d, \epsilon$, 
    where $P^{\nor}$ is the normalisation of $P$.
    
    Shrinking $Z$ near the generic point of $T$, we can assume that $g\colon Y\to Z$ is flat. 
    Then by Lemma~\ref{volume-fibre-constant},
    for every scheme-theoretic point $z\in Z$, we have
    \[ \vol_{/k(z)}(A_Y|_{Y_z}) = \vol_{/k(u)}(A|_{X_u}), \]
    where $u\in Z$ is a general closed point, and $Y_z, X_u$ are the fibres of $Y\to Z, X\to Z$
    over $z, u$ respectively.  Thus, 
    \begin{equation}\label{vol-constant-in-pf}
        \vol_{/k(z)}(A_Y|_{Y_z}) \in \mf{R},
    \end{equation}
    where $\mf{R}$ is the finite set of natural numbers depending only on $d,\epsilon$ in Step~1.
    
    Let $H\subset Z$ be a general hyperplane section and $G \coloneqq g^* H$.
    By Bertini's Theorem \cite[3.4.9, 3.4.10]{joins-intersections}, $G$ is a normal $k$-variety.  Let
    \[ \Delta_G \coloneqq \Delta_Y|_G ~~\mbox{and}~~ A_G \coloneqq A_Y|_G. \]
    By construction, $A'$ contains no lc centre of $(Y', \Delta_{Y'})$.
    Then by the description of MMP$/Z$ steps in \S\ref{existence-MMP-g-pair},
    and by \cite[Lemma~20.4]{lyu-takumi-2026relativeminimalmodelprogram}, $A''$ does not 
    contain lc centre of $(Y'', \Delta_{Y''})$ neither.
    Moreover, by Steps~1 and 2, $(Y'', \Delta_{Y''} + \lambda_0\lambda A'')$
    is the crepant pullback of $(Y, \Delta_Y + \lambda_0\lambda A_Y)$.
    Hence $A_Y$ does not contain lc centre of $(Y, \Delta_Y)$.
    By construction, 
    \[ \ell (K_G + \Delta_G)\sim 0/H. \]
    Thus, by Definition~\ref{stable-CY-defn},
    $\Set{(G, \Delta_G), A_G}\to H$ is a stable Calabi-Yau fibration over $k$.
    
    Let $z$ be one of the closed points $T\cap H$.  Taking $H$ generally,
    we can assume that $z$ is a general closed point of $T$.
    Denote by $R$ the reduction of the fibre $G_z = Y_z$.
    Then $R$ is contained in $\rddown{\Delta_G}$.  Denote by $\nu\colon R^{\nu}\to R$ the normalisation of $R$.
    By adjunction \cite[\S 4.1]{Kol_singularities_of_MMP}, we can write
    \[ K_{R^{\nu}} + \Delta_{R^{\nu}} = \nu^*(K_G + \Delta_G), \]
    where the coefficients of $\Delta_{R^{\nu}}$ belong to a finite set of rational numbers depending only on $\ell$.
    Let
    \[ A_{R^{\nu}} \coloneqq \nu^* A_G. \]
    By \cite[Lemma~4.8]{Kol_singularities_of_MMP}, 
    there is no lc centre of $(R^{\nu}, \Delta_{R^{\nu}})$ contained in $A_{R^{\nu}}$.
    By Lemma~\ref{complement-coeffs}, $q(\ell) A_{R^{\nu}}$ is an integral divisor 
    for some $q(\ell)\in \NN$ depending only on $\ell$.

    Let 
    \[ \bar{R}\coloneqq R\otimes_{k(z)} \ol{k(z)}, \]
    where $k(z)$ denotes the residue field of $Z$ at $z$.
    Let $\bar{R}_1,\dots, \bar{R}_r$ be irreducible components of $\bar{R}$,
    and let $\bar{R}_i^{\nu}$ be the normalisation of $\bar{R}_i$.  
    By \cite[Corollaire~(6.14.5)]{EGA-IV-2}, we have
    \[ R^{\nu}\otimes_{k(z)}\ol{k(z)} = \bigsqcup_{i=1}^r \bar{R}_i^{\nu}. \]
    Let $\Delta_{\bar{R}_i^{\nu}}$ (respectively, $A_{\bar{R}_i^{\nu}}$) be the 
    flat pullback of $\Delta_{R^{\nu}}$ (respectively, of $A_{R^{\nu}}$) to $\bar{R}_i^{\nu}$.
    \[\xymatrix{
    \bar{R}\ar[d]\ar[r] & R\ar[d]\ar[r] & G\ar[d]\ar[r] & Y\ar[d] \\
    \spec \ol{k(z)}\ar[r] & z\ar[r] & H\ar[r] & Z
    }\]
    As $A_Y$ is ample$/Z$, every $A_{\bar{R}_i^{\nu}}$ is nonzero.
    Moreover, $q(\ell) A_{\bar{R}_i^{\nu}}$ is an integral $\QQ$-Cartier divisor
    that is ample over $\spec \ol{k(z)}$.
    By construction, we have
    \[ \ell (K_{R^{\nu}} + \Delta_{R^{\nu}}) \sim 0/\spec k(z). \]
    Hence
    \begin{equation}\label{ell-trivial}
        \ell(K_{\bar{R}_i^{\nu}} + \Delta_{\bar{R}_i^{\nu}}) \sim 0/\spec \ol{k(z)}.
    \end{equation}
    By Lemma~\ref{exist-g-dlt-model}, let 
    \[ \mu_i\colon (\bar{R}_i^{\dagger}, \Delta_{\bar{R}_i^{\dagger}})\to (\bar{R}_i^{\nu}, \Delta_{\bar{R}_i^{\nu}}) \]
    be a $\QQ$-factorial dlt model of $(\bar{R}_i^{\nu}, \Delta_{\bar{R}_i^{\nu}})$.  
    Then by \eqref{ell-trivial}, $\bar{R}_i^{\dagger}$ is $(1/\ell)$-lc.
    Denote by
    \[ q(\ell) A_{\bar{R}_i^{\dagger}} = \mu_i^* (q(\ell) A_{\bar{R}_i^{\nu}}). \]
    As $A_{\bar{R}_i^{\nu}}$ contains no lc centre of $(\bar{R}_i^{\nu}, \Delta_{\bar{R}_i^{\nu}})$,
    $q(\ell)A_{\bar{R}_i^{\dagger}}$ is an integral $\QQ$-Cartier divisor on $\bar{R}_i^{\dagger}$
    that is nef and big over $\spec \ol{k(z)}$.  Note that
    \[ q(\ell)A_{\bar{R}_i^{\dagger}} - K_{\bar{R}_i^{\dagger}} \sim_{\QQ} 
    q(\ell)A_{\bar{R}_i^{\dagger}} + \Delta_{\bar{R}_i^{\dagger}}, \]
    which is pseudo-effective over $\spec \ol{k(z)}$.
    By Lemma~\ref{vol-bir}, there exists an $m\in \NN$ depending only on $d, \ell$ such that
    \[ \abs{m q(\ell)A_{\bar{R}_i^{\dagger}}}~\mbox{defines a birational map}. \]
    Thus by \cite[Lemma~2.2]{Hacon-McKernan-06-bnd}, we get
    \begin{equation}\label{vol-lower-bnd}
        \vol_{/\ol{k(z)}}(A_{\bar{R}_i^{\nu}}) = \vol_{/\ol{k(z)}}(A_{\bar{R}_i^{\dagger}}) 
        \ge \delta\coloneqq \frac{1}{m^dq(\ell)^d},
    \end{equation}
    where the right hand side depends only on $d, \ell$.  Write 
    \[ [G\times_H \spec \ol{k(z)}] = \sum_{i=1}^r m_i [\bar{R}_i], \]
    where $m_i\in \NN$ is the multiplicity of $\bar{R}_i$ in geometric fibre of $G\to H$ over $z$,
    and the brackets denote the fundamental cycle of schemes; see \cite[\S3.1]{Kol-rc}.
    Then by \cite[Proposition~VI.2.7]{Kol-rc} and \cite[Proposition~III.9.3]{Hart}, we have
    \[ \vol_{/k(z)}(A_Y|_{Y_z}) = \sum_{i=1}^r m_i \vol_{/\ol{k(z)}}(A_Y|_{\bar{R}_i})
    = \sum_{i=1}^r m_i \vol_{/\ol{k(z)}}(A_Y|_{\bar{R}_i^{\nu}}). \]
    By \eqref{vol-constant-in-pf} and \eqref{vol-lower-bnd}, we deduce that
    \[ \delta r\le \sum_{i=1}^r m_i \vol_{/\ol{k(z)}}(A_Y|_{\bar{R}_i^{\nu}}) = \vol_{/k(z)}(A_Y|_{Y_z}) \in \mf{R}. \]
    Hence $r$ is bounded from above depending only on $d, \ell, \epsilon$.
    Since every geometric irreducible component of the fibre of $P^{\nor}\to T$ over $z$ is one of $\bar{R}_1,\dots, \bar{R}_r$,
    we can conclude that $\sdeg (P^{\nor}/T)$ is bounded from above depending only on $n,d, \epsilon$.

    \medskip
    \emph{Step~5}.
    We can assume that $\centre_Y S$ is not a divisor on $Y$.
    As $Y\to Z$ is of Fano type,
    by taking a $\QQ$-factorialisation \cite[Corollary~22.3]{lyu-takumi-2026relativeminimalmodelprogram},
    which is a small birational contraction,
    we can assume that $Y$ is $\QQ$-factorial.
    Let $P$ be a horizontal$/T$ irreducible component of $g^*T$ that contains $\centre_Y S$.
    Let $\hat{Y}\to Y$ be a birational contraction that extracts only the geometric valuation of $S$,
    where $\hat{Y}$ is $\QQ$-factorial; see Lemma~\ref{exist-g-dlt-model}.
    Write $K_{\hat{Y}} + \Delta_{\hat{Y}}$ as the pullback of $K_Y + \Delta_Y$.
    Denote by $\hat{P}, \hat{S}$ the centres of $P, S$ on $\hat{Y}$.
    We claim that $\hat{P}\cap \hat{S}$ has a horizontal$/T$ irreducible component $J$ such that 
    $J$ is of codimension one in $\hat{P}$ and $\hat{S}$.  
    Indeed, $\hat{P}\sim_{\QQ}-b\hat{S}/Y$ for some $b\in \QQ^{>0}$,
    so $\hat{P}$ is ample$/Y$; see \cite[\S 1.12.1]{Kol_singularities_of_MMP}.  
    Then $\hat{P}$ intersects every integral curve in the fibre of $\hat{S}\to \centre_Y S$, 
    which implies the existence of an irreducible component $J$ as desired.

    Let $\hat{P}^{\nor} \to \hat{P}$ be the normalisation.  Then by adjunction, we can write
    \[ K_{\hat{P}^{\nor}} + \Delta_{\hat{P}^{\nor}} = (K_{\hat Y} + \Delta_{\hat Y})|_{\hat{P}^{\nor}}. \]
    Let $\hat{J}$ be a prime divisor on $\hat{P}^{\nor}$ mapping onto $J\subset \hat{P}$.  Then
    by \cite[Theorem~4.9(3)]{Kol_singularities_of_MMP},
    \[ \coeff_{\hat{J}} \Delta_{\hat{P}^{\nor}} \ge t. \]
    Note that 
    \[ K_{\hat{P}^{\nor}} + \Delta_{\hat{P}^{\nor}} \sim_{\QQ}0/T. \]
    Let 
    \[ \hat{P}^{\nor}\to \hat{T} \to T \]
    be the Stein factorisation of $\hat{P}^{\nor} \to T$.  By Step~4, 
    \[ \deg(\hat{T}/T) \le \sdeg(\hat{P}^{\nor}/T) = \sdeg (P^{\nor}/T) \]
    is bounded from above depending only on $n,d,\epsilon$.  
    Denote by $\hat k$ the algebraic closure of $k$ in the function field $K(\hat T)$ of $\hat T$.  Then 
    \[ (\hat{P}^{\nor}, \Delta_{\hat{P}^{\nor}}) \to \hat{T} \]
    is a log Calabi-Yau fibration over $\hat k$.
    By assumption $\mc{H}_{\univ}(d, t)$, we see that $\sdeg(\hat{J}^{\nor}/\hat{T})$ is bounded from above depending
    only on $d, t$.  Since
    \[ \sdeg(\hat{J}^{\nor}/T) \le \sdeg(\hat{J}^{\nor}/\hat{T})\cdot \deg(\hat{T}/T), \]
    we can conclude that $\sdeg (\hat{J}^{\nor}/T)$ is bounded from above depending only on $n,d,t,\epsilon$.
    Since $\hat{J}\to J$ is a surjective morphism, and since $\dim \hat{J} = \dim J$, we see that
    \[ \sdeg (J^{\nor}/T) \le \sdeg (\hat{J}^{\nor}/T) \]
    is bounded from above depending only on $n, d, t, \epsilon$, 
    where $J^{\nor}$ is the normalisation of $J$.

    \medskip
    \emph{Step~6}.
    Let $\hat{S}^{\nor}\to \hat{S}$ be the normalisation of $\hat{S}$.  As $\coeff_{\hat{S}}\Delta_{\hat{Y}}\ge t$,
    by Proposition~\ref{normal-mult-bnd-general-fields}, 
    there exist a closed subset $\hat{U}\subset \hat{S}$ satisfying $\codim (\hat{U}, \hat{S})\ge 2$
    such that for every closed point $s\in \hat{S}\setminus \hat{U}$, 
    the geometric normal multiplicity of $\hat{S}$ at $s$ is bounded 
    from above depending only on $t$.  
    Let $I\subset \hat{S}^{\nor}$ be a prime divisor mapped onto $J$.
    As the geometric normal multiplicity of $\hat{S}$ outside $\hat{U}$ is bounded, we see that $\deg (I/J)$
    is bounded from above depending only on $t$.  
    Denote by $I^{\nor}\to I$ the normalisation of $I$.  
    Then by Lemma~\ref{stein-multiplicity},
    \[ \sdeg (I^{\nor}/T)\le \deg(I/J)\cdot \sdeg(J^{\nor}/T), \]
    which is bounded from above depending only on $n,d,t, \epsilon$.
    By Lemma~\ref{sdeg-subvariety}, we have 
    \[ \sdeg (\hat{S}^{\nor}/T)\le \sdeg(I^{\nor}/T). \]
    Thus, we can conclude that $\sdeg (S^{\nor}/T) = \sdeg (\hat{S}^{\nor}/T)$ 
    is bounded from above depending only on $n, d, \epsilon, t$.
\end{proof}


\section{\bf Proof of main results}\label{pf-main-results}


\subsection{Stein degree of horizontal divisors}

\begin{theorem}\label{pf-dim-d-t}
    Let $d\in \NN$ and let $t\in \RR^{>0}$.  
    Assume $\mcV_{\univ}(\ell,s)$ holds for each $1\le \ell\le d-1$ and each $s\in (0,1]$.
    Then $\mc{H}_{\univ}(d, t)$ holds.
\end{theorem}

\begin{proof}
    We argue by induction on $d$.  If $d = 1$, then $\mc{H}_{\univ}(1, t)$ holds by Lemma~\ref{base-case-r-dim=1}.
    Thus, we can assume that $\mc{H}_{\univ}(l, t)$ holds for every $1\le l\le d-1$.
    
    Let $(X, B + \mb{M})\to Z$ be a generalised log Calabi-Yau fibration of relative dimension $d$
    over a field $k$ of characteristic zero, and let $S$ be a horizontal$/Z$ 
    irreducible component of $B$ such that $\coeff_S B\ge t$.
    By the same proof of Lemma~\ref{fano-type-induction},
    we can assume $\dim Z = 1$.

    \medskip
    \emph{Step~1}.
    In this step, we reduce to the case of $\QQ$-factorial g-klt g-pairs.
    Set $\alpha \coloneqq \coeff_S B$.
    Taking a $\QQ$-factorial g-dlt model of $(X, B + \mb{M})$ (see Lemma~\ref{exist-g-dlt-model}), 
    we can assume that $(X, B + \mb{M})$ is a $\QQ$-factorial g-dlt g-pair.  
    In particular, $X$ has klt singularities.
    By Theorem~\ref{mmp-g-pairs}, as $-S$ is non-pseudo-effective$/Z$, 
    we can run $\phi\colon X\dashrightarrow X'/Z$ an MMP over $Z$
    on the divisor
    \[ K_X + (B - \alpha S) + \mb{M}_X \sim_{\RR} -\alpha S/Z, \]
    which ends with a Mori fibre space $g\colon X' \to Z'$ of $S$, so $S'\coloneqq \phi_*S$ is ample$/Z'$.
    By Lemma~\ref{CY-log-discre-no-change}, 
    the coefficient of $S'$ in $B' \coloneqq \phi_* B$ is also equal to $\alpha$.

    If $\dim (X'/Z')\le d-1$, then $\sdeg \big((S')^{\nor}/Z'\big)$ is bounded from above depending only on $d, t$
    by assumption, hence by Lemmas~\ref{stein-multiplicity}, 
    \[ \sdeg (S^{\nor}/Z) = \sdeg \big( (S')^{\nor}/Z \big) \le \sdeg \big( (S')^{\nor}/Z' \big). \]
    Thus, we can assume that $Z'\to Z$ is an isomorphism of normal curves. 
    Then $S'$ is ample$/Z$. 

    Since $(X, (B - \alpha S) + \mb{M})$ is a $\QQ$-factorial g-dlt g-pair, 
    \[ (X', (B' - \alpha S') + \mb{M}) \]
    is also a $\QQ$-factorial g-dlt g-pair; see Corollary~\ref{run-g-dlt-mmp}.  
    Pick $0< \epsilon < 1$ and write
    \[ B' = (1-\epsilon)(B' - \alpha S') + \bigg(\epsilon (B' - \alpha S') + \frac{\alpha}{2} S'\bigg) + \frac{\alpha}{2} S'.  \]
    For sufficiently small $0< \epsilon \ll 1$, the divisor 
    $\epsilon (B' - \alpha S') + (\alpha/2)S'$
    is also ample$/Z$.  Let $H'$ be a general effective ample$/Z$ divisor such that
    \begin{equation}
        \epsilon (B' - \alpha S') + \frac{\alpha}{2}S' \sim_{\RR} H'/Z. \label{linear-H}
    \end{equation}
    Set
    \[ B_{\epsilon}' \coloneqq (1-\epsilon) (B' - \alpha S'). \]
    As $(X', (B' - \alpha S') + \mb{M})$ is $\QQ$-factorial and g-dlt, $(X', B_{\epsilon}' + \mb{M})$
    is a $\QQ$-factorial g-klt g-pair.  On the other hand, since
    $(X', B_{\epsilon}' + \epsilon (B'-\alpha S') + \alpha S' + \mb{M})$ is g-lc 
    by Lemma~\ref{CY-log-discre-no-change}, the g-pair $(X', B_{\epsilon}' + \alpha S' + \mb{M})$
    is also g-lc, hence
    \[ (X', B_{\epsilon}' + \frac{\alpha}{2} S' + \mb{M}) \]
    is a $\QQ$-factorial g-klt g-pair; cf. \cite[Corollary 2.35(5)]{km98}.
    Thus, as $H'$ is general, 
    \[ (X', C' + \mb{M}) \coloneqq (X', B_{\epsilon}' + H' + \frac{\alpha}{2} S' + \mb{M}) \]
    is also a $\QQ$-factorial g-klt g-pair.  By~\eqref{linear-H}, it is evident that
    $(X', C' + \mb{M}) \to Z$ is also a generalised log Calabi-Yau fibration.
    Moreover, as $S'$ is ample$/Z$, by Lemma~\ref{fano-type-conditions},
    $X'\to Z$ is of Fano type.  By replacing $t$ with $t/2$, it suffices to consider the 
    Fano type generalised log Calabi-Yau fibration $(X', C' +\mb{M})\to Z$.

    \medskip
    \emph{Step~2}.  
    In this step, we reduce to the case of usual pairs by taking bounded complements.
    Decreasing $t$ if necessary, we can assume that $t\in \QQ^{>0}$.
    Since $X'$ is of Fano type over $Z$, by Lemma~\ref{fano-type-MMP}, 
    we can run an MMP$/Z$ on the divisor 
    \[ -K_{X'} - t S' \sim_{\RR} (C'- t S') + \mb{M}_X, \]
    which is pseudo-effective$/Z$.  Denote by $\varphi\colon X'\dashrightarrow X''$
    the $(-K_{X'} - tS')$-MMP over $Z$ that ends with a good minimal model $X''\to Z$ 
    of the divisor $-K_{X'} - tS'$.  Set $S'' \coloneqq \varphi_* S'$, then
    $- (K_{X''} + t S'')$ is nef$/Z$.  Note that $S''$ is a nonzero prime divisor on $X''$
    as $S'$ is ample$/Z$; see Lemma~\ref{no-van-big-divisor}.  Moreover,
    by Lemma~\ref{Fano-type-pushdown}, $X''$ is also of Fano type over $Z$.

    By Theorem~\ref{t-bnd-compl-usual-local}, up to shrinking
    $Z$ near its generic point, there is a monotonic $n$-complement $K_{X''} + B^+$ of $K_{X''} + t S''$
    for some $n\in \NN$ depending on $d,t$, that is,
    \begin{itemize}
        \item $(X'', B^+)$ is lc,
        \item $n(K_{X''} + B^+) \sim 0/Z$, and
        \item $t S'' \le B^+$.
    \end{itemize}
    In particular, the coefficients of $B^+$ belong to $\mathbf{T}_n$.  
    Set $\beta \coloneqq \coeff_{S''} B^+ \in \mathbf{T}_n$. 

    \medskip
    \emph{Step~3}.  
    Let $(X''', B''')$ be a $\QQ$-factorial dlt model of $(X'', B^+)$; see Lemma~\ref{exist-g-dlt-model}.
    The coefficients of $B'''$ also belong to the set $\mathbf{T}_n$.
    Denote by $S'''$ the birational transform of $S''$ to $X'''$.
    Note that $S'''$ may not be big$/Z$.  
    By Lemma~\ref{l-ep-to-0}, shrinking $Z$ near its generic point if necessary,
    we can assume that $(X''', 0)$ is $\epsilon$-lc for some $\epsilon>0$ depending only on $d, n$, whence only on $d,t$.
    By Lemma~\ref{fano-type-at-dlt-model}, $X'''\to Z$ is also of Fano type.
    Hence by Lemma~\ref{fano-type-MMP}, we can
    run $\psi\colon X''' \dashrightarrow \wt{X}$ a $K_{X'''}$-MMP over $Z$
    that ends with a Mori fibre space $\wt{f}\colon \wt{X}\to \wt{Z}$ of $K_{X'''}$. 
    \[\xymatrix{
     & & X'''\ar[d]\ar@{-->}[r]^{\psi} & \wt{X}\ar[dd]^{\wt{f}} \\
    X\ar@{-->}[r]^{\phi}\ar[d]_f & X'\ar@{-->}[r]^{\varphi} & X'' &  \\
    Z & & &  \wt{Z}\ar[lll]
    }\]
    Then $-K_{\wt{X}}$ is ample$/\wt{Z}$.
    By Lemma~\ref{discrep-mmp}, $(\wt{X}, 0)$ is also $\epsilon$-lc. 

    \medskip
    \emph{Step~4}.
    Assume that $\dim (\wt{Z}/Z) \ge 1$.  By Lemma~\ref{exist-g-dlt-model},
    let $\pi\colon X^{\circ} \to \wt{X}$ be a birational contraction extracting only the valuation of $S$,
    where $X^{\circ}$ is $\QQ$-factorial;
    if the centre of the valuation of $S$ on $\wt{X}$ is a divisor, then $\pi$ is the identity map.
    Let $\wt{B}$ be the pushdown of $B'''$ to $\wt{X}$.  Write
    \[ K_{X^{\circ}} + B^{\circ} = \pi^*(K_{\wt{X}} + \wt{B}). \]
    By Lemma~\ref{fano-type-at-dlt-model}, 
    $(X^{\circ}, B^{\circ}) \to Z$ is a Fano type log Calabi-Yau fibration.
    Hence by Lemma~\ref{fano-type-upper}, $(X^{\circ}, B^{\circ}) \to \wt{Z}$ 
    is also a Fano type log Calabi-Yau fibration.
    
    Let $S^{\circ}$ be the centre of $S$ on $X^{\circ}$.
    If $S^{\circ}$ is horizontal$/\wt{Z}$, then
    \[ \sdeg \big( (S^{\circ})^{\nor}/\wt{Z}\big) \]
    is bounded from above depending only on $d,n$ (hence depending only on $d,t$)
    by applying induction hypothesis 
    to the Fano type fibration $(X^{\circ}, B^{\circ})\to \wt{Z}$.
    Then by Lemma~\ref{stein-multiplicity}, we have
    \[ \sdeg (S^{\nor}/Z) = \sdeg \big( (S^{\circ})^{\nor}/Z \big) \le \sdeg \big( (S^{\circ})^{\nor}/\wt{Z}\big), \]
    which shows that $\sdeg (S^{\nor}/Z)$ is bounded from above depending only on $d,t$.
    Thus, we can assume that $S^{\circ}$ is vertical$/\wt{Z}$.

    Recall that $X^{\circ}$ is also of Fano type over $\wt{Z}$.  By Lemma~\ref{fano-type-MMP},
    we can run $\upsilon\colon X^{\circ}\dashrightarrow X^{\triangle}$ a $(-S^{\circ})$-MMP over $\wt{Z}$
    that ends with a good minimal model $X^{\triangle}$ of $-S^{\circ}$.
    By negativity lemma, $S^{\circ}$ is not contracted by $\upsilon$.
    Denote by $f^{\triangle}\colon X^{\triangle}\to Z^{\triangle}$ the contraction$/\wt{Z}$ induced by $-S^{\triangle}$,
    where $S^{\triangle} \coloneqq \upsilon_* S^{\circ}$.
    As $S^{\triangle}$ is vertical$/\wt{Z}$, the contraction $Z^{\triangle}\to \wt{Z}$ is birational.
    \[\xymatrix{
    X^{\circ}\ar@{-->}[r]^{\upsilon}\ar[d] & X^{\triangle}\ar[d]^{f^{\triangle}} \\
    \wt{Z}\ar[d] & Z^{\triangle}\ar[l]\ar[ld] \\
    Z &
    }\]
    Denote by $T^{\triangle}$ the image of $S^{\triangle}$ in $Z^{\triangle}$,
    which is a prime divisor on $Z^{\triangle}$.  
    As $-S^{\triangle}$ is nef$/Z^{\triangle}$, $S^{\triangle}$ is 
    the unique irreducible component of $\Supp ((f^{\triangle})^* T^{\triangle})$.
    Then we can write
    \[ S^{\triangle} = b\cdot (f^{\triangle})^* T^{\triangle} \]
    for some $b\in \QQ^{>0}$.  Let $B^{\triangle}$ be the pushdown of $B^{\circ}$ to $X^{\triangle}$.
    We can write
    \[ K_{X^{\triangle}} + B^{\triangle} \sim_{\RR} (f^{\triangle})^*(K_{Z^{\triangle}} + B_{Z^{\triangle}} + \mb{N}_{Z^{\triangle}}), \]
    where $(Z^{\triangle}, B_{Z^{\triangle}} + \mb{N})$ is a g-lc g-pair.
    By Lemma~\ref{fano-type-on-base}, $Z^{\triangle}$ is of Fano type over $Z$.  Then 
    \[ (Z^{\triangle}, B_{Z^{\triangle}} + \mb{N}) \to Z \]
    is also a Fano type fibration of relative dimension 
    \[ \dim (Z^{\triangle}/Z) = \dim (\wt{Z}/Z) \le d-1. \]
    By Lemma~\ref{l-adj-fib-eps-lc}, the coefficient of $T^{\triangle}$ in $B_{Z^{\triangle}}$
    is $\ge \beta$, where $\beta = \coeff_{S''} B^+$ is $\ge t$ as in Step~2.
    By induction hypothesis $\mcH_{\univ}(l, t)$ for $1\le l\le d-1$, 
    \[ \sdeg \big( (T^{\triangle})^{\nor}/Z\big) \]
    is bounded from above depending only on $d,t$.  
    Moreover, as we assume that $\mcV_{\univ}(\ell, t)$ 
    holds for each $1\le \ell \le d-1$,
    \[ \sdeg \big( (S^{\triangle})^{\nor}/T^{\triangle} \big)  \]
    is also bounded from above depending only on $d,t$.  As $S^{\triangle}\to Z$ has the factorisation
    \[ S^{\triangle} \to T^{\triangle} \to Z, \]
    we can conclude that
    \[ \sdeg (S^{\nor}/Z) = \sdeg \big((S^{\triangle})^{\nor}/Z\big) 
    \le \sdeg \big( (S^{\triangle})^{\nor}/T^{\triangle} \big) \cdot \sdeg \big( (T^{\triangle})^{\nor}/Z\big), \]
    which is bounded from above depending only on $d,t$ by Lemma~\ref{stein-multiplicity}.
    Therefore, we can assume that $\wt{Z}\to Z$ is an isomorphism of nonsingular curves. 

    \medskip
    \emph{Step~5}.  
    By construction in Step~3, we have that
    \begin{itemize}
        \item $(\wt{X}, 0)$ is $\epsilon$-lc,
        \item $-K_{\wt{X}}$ is ample$/Z$, and
        \item $n(K_{\wt{X}} + \wt{B}) \sim 0/Z$.
    \end{itemize}
    Note that $S'''$ may be contracted by $\psi$.  
    Denote by $F$ the generic fibre of $\wt{X}\to Z$, i.e.
    \[ F\coloneqq \wt{X}\times_Z \spec K(Z), \]
    which is an $\epsilon$-lc Fano variety of dimension $d$ over $K(Z)$.
    By Lemma~\ref{fano-bnd-very-ample},
    there exist $m, r\in \NN$ depending only on $d,\epsilon$ (whence only on $d, t$)
    such that $-mK_F$ is very ample, and $(-mK_F)^d < r$.
    Denote by $\Gamma$ the restriction of $\wt{B}$ to $F$.  Then
    \[ G\coloneqq mn \Gamma \sim -mn K_F/\spec K(Z) \]
    is a very ample Cartier divisor on $F$.  Let
    \[ N+1\coloneqq \dim_{K(Z)}\hg^0(F, \mcO_F(-mK_F)). \]
    By \cite[Chap.~VI, (2.15.8.5)]{Kol-rc}, we have
    \[ N+1 \le (-mK_F)^d + d< r+d, \]
    where the right most quantity depends only on $d, t$.
    Then the very ample divisor $-mK_F$ gives a closed immersion
    \[ \iota_F\colon F\hookrightarrow \PP_{K(Z)}^N, \]
    whose image in $\PP_{K(Z)}^N$ is nondegenerate.  
    By \cite[Lemma~1.4]{Salberger-r-pts}, there are only finitely many Hilbert polynomials
    of such $F$ with respect to $\iota_F$.
    These Hilbert polynomials depend only on $N, r$, whence only on $d,t$.
    
    Let $P_F(u)\in \QQ[u]$ be the Hilbert polynomial of $F$ with respect to $\iota_F$, 
    and let $\mcO_F(1)\coloneqq \mcO_F(-mK_F)$.  Then for each $u\in \NN$, we have the exact sequence
    \[ 0\to \mcO_F(u-n) \to \mcO_F(u)\to \mcO_G(u)\to 0. \]
    Thus, we deduce that
    \[ P_G(u) = P_F(u) - P_F(u-n), \]
    where $P_G(u)\in \QQ[u]$ denotes the Hilbert polynomial of $G$ with respect to 
    the closed immersion $G\hookrightarrow \PP_{K(Z)}^N$ induced by $\iota_F$.
    Thus, there are only finitely many Hilbert polynomials for such $G$
    depending only on $d,t$.  We can assume that $P_F, P_G$ are the only polynomials.  Let
    \[ \mcH\coloneqq \hilb_{P_F}(\PP_{\QQ}^N)~~\mbox{and}~~ \mcD\coloneqq \hilb_{P_G}(\PP_{\QQ}^N). \]
    Moreover, let
    \[ \mcF\to \mcH~~\mbox{and}~~\mcG\to \mcD \]
    be the universal families over $\mcH, \mcD$ respectively.
    Note that $\mcH, \mcD$ are $\QQ$-schemes, so they depends only on $d,t$, 
    but not on the ground field $K(Z)$ of $F, G$.

    Set
    \[ \mcW \coloneqq \mcH\times_{\spec \QQ} \mcD. \]
    Moreover, let
    \[ \mcX\coloneqq \mcF\times_{\mcH} \mcW~~\mbox{and}~~ 
    \mcC\coloneqq \mcG\times_{\mcD}\mcW. \]
    Then we have closed immersions
    \[ \mcX \hookrightarrow \PP_{\mcW}^N~~\mbox{and}~~\mcC\hookrightarrow \PP_{\mcW}^N, \]
    which are morphisms over $\mcW$.

    Denote by $p$ the canonical projection $\PP_{\mcW}^N\to \mcW$, and by $\mcO(1)$
    the very ample invertible sheaf on $\PP_{\mcW}^N$ corresponding to hyperplane sections.
    Consider the composite homomorphism between coherent sheaves
    \[ \mcI_{\mcX} \to \mcO_{\PP_{\mcW}^N} \to \mcO_\mcC, \]
    where $\mcI_{\mcX}$ is the ideal sheaf on $\PP_{\mcW}^N$ defining $\mcX$.
    Pick $b\in \NN$ sufficiently large so that 
    \[ p^*(p_*\mcI_{\mcX}(b)) \to \mcI_{\mcX}(b)~~\mbox{and}~~
    p^*(p_*\mcO_{\mcC}(b)) \to \mcO_{\mcC}(b) \]
    are surjective; see \cite[Theorem~III.8.8]{Hart}.  For each scheme-theoretic point $w\in W$,
    let $\mcX_w, \mcC_w$ be the fibres of $\mcX\to \mcW, \mcC\to \mcW$ over $w$ respectively.
    Then it is evident that $\mcC_w\subset \mcX_w$ if and only if
    the image of the induced homomorphism 
    \[ (p_*\mcI_{\mcX}(b))\otimes_{\mcO_\mcW} k(w) \to (p_*\mcO_{\mcC}(b))\otimes_{\mcO_\mcW} k(w) \]
    is zero, where $k(w)$ denotes the residue field of $\mcW$ at $w$.  
    Denote by $\mcW^{\circ}\subseteq \mcW$ the subset consisting points $w$
    such that $\mcC_w\subset \mcX_w$.
    By assumption, $G$ is a closed subscheme of $F$.  Hence there exists a $\QQ$-morphism
    \[ i_{(F, G)}\colon \spec K(Z)\to \mcW \]
    whose image is contained in $\mcW^{\circ}$.
    Thus, $\mcW^{\circ}$ is a nonempty closed subset of $\mcW$.
    Replacing $\mcW$ with $\mcW^{\circ}$, we can assume that
    \[ \mcC\subset\mcX \subset \PP_{\mcW}^N. \]
    Taking an irreducible component containing the image of $i_{(F, G)}$, 
    and taking its reduction, we can assume that $\mcW$ is irreducible.
    By \cite[Th\'eor\`eme~(9.7.7)]{EGA-IV-3} and \cite[Theorem~9.3.7]{FGA-explained},
    replacing $\mcW$ with an open subset, we can assume that
    $\mcX\to \mcW, \mcC\to \mcW$ are flat, each fibre of $\mcX\to \mcW$
    is geometrically integral, and that $\mcC$ is a relative effective Cartier divisor on $\mcX\to \mcW$.
    By \cite[Corollaires~(2.3.5) and~(3.3.5)]{EGA-IV-2}, $\mcX$ is integral.

    By \S\ref{log-resolution}, there exists a log resolution
    \[ \mu\colon \mcY\to \mcX \]
    of $(\mcX, \Supp \mcC)$.  Let
    \[ \mcE\coloneqq \Ex(\mu)\cup \mu^{-1}(\Supp \mcC). \]
    Then $\mcY$ is a regular $\QQ$-scheme, and $\mcE$ is an snc divisor on $\mcY$.  Write
    \[ \mcE = \mcE_1 + \cdots + \mcE_s, \]
    where each $\mcE_j$ is an irreducible component of $\mcE$.
    By generic smoothness \cite[Proposition~(17.7.8), Corollaire~(17.15.2)]{EGA-IV-4},
    there exits an open subset $\mcU\subset\mcW$ such that
    \[ \mcY_{\mcU} \to \mcU~~\mbox{and each}~~\mcE_{j, \mcU} \to \mcU \]
    are smooth morphisms, where
    \[ \mcY_{\mcU}\coloneqq \mcU\times_{\mcW}\mcY~~\mbox{and}~~
    \mcE_{j, \mcU}\coloneqq \mcU\times_{\mcW} \mcE_j. \]
    Again by generic smoothness, we can assume each stratum of $\mcE$ is smooth over $\mcU$.
    We can also assume that for each $w\in \mcU$, the induced morphism between
    fibres over $w$, $\mcY_w\to \mcX_w$, is a birational contraction.
    
    As $\mcW$ is integral, every irreducible component of $\mcW\setminus \mcU$
    has dimension strictly less than that of $\mcW$.
    Thus, doing induction on $\dim \mcW$, we can assume that
    the image of $i_{(F, G)}$ is contained in $\mcU$.  Denote by
    \[ Y\coloneqq \spec K(Z)\times_{\mcW} \mcY~~\mbox{and}~~
    E \coloneqq \spec K(Z)\times_{\mcW} \mcE, \]
    where the fibre products are formed via $i_{(F,G)}$.  Then by construction, the induced
    \[ \mu_F\colon Y\to F \]
    is a log resolution of $(F, \Supp G)$; in particular, 
    \[ E = \Ex(\mu_F)\cup \mu_F^{-1}(\Supp G) \]
    is an snc divisor on $Y$.  Write
    \[ K_Y + \Gamma_Y = \mu_F^*(K_F + \Gamma). \]
    Then $\Gamma_Y \le E$ as $(F, \Gamma)$ is lc.  Let $S_F$ be the prime divisor over $F$
    induced by $S'''$.  It is evident that
    \[ a(S_F, F, \Gamma) = a(S, \wt{X}, \wt{B}) < 1- t < 1. \]
    Then we have
    \[ a(S_F, Y, E) \le a(S_F, Y, \Gamma_Y) = a(S_F, F, \Gamma) <1. \]
    As $(Y, E)$ is log smooth, we can conclude that
    \[ a(S_F, Y, E) = 0. \]
    Then $R\coloneqq \centre_Y (S_F)$ is an lc centre of $(Y, E)$.  Hence $R$ is a strata of $E$.
    In particular, there exists a strata $\mcR$ of $\mcE$ such that $R$ is an irreducible component of
    \[ \spec K(Z)\times_{\mcW}\mcR. \]
    By \cite[Proposition~(9.7.8)]{EGA-IV-3}, the number of geometric irreducible components
    \[ \mf{n}\big(R\otimes_{K(Z)} \ol{K(Z)}\big) \]
    is bounded from above depending only on $\mcR\to \mcW$, whence only on $d,t$.
    Moreover, by \cite[Corollaire~(19.4.3)]{EGA-IV-4} and \cite[Lemma~2.45]{km98}, we can write
    \[ K(S_F) \cong K(R)(u_1, \dots, u_{c-1}), \]
    where $c \coloneqq \codim(R, Y)$, and $u_1,\dots, u_{c-1}$ are independent variables.  
    Then the generic fibre of $S_F\to R$
    is geometrically irreducible and rational.  
    By Lemma~\ref{stein-multiplicity}, we deduce that
    \[ \sdeg(S^{\nor}/Z) = \sdeg (S_F/\spec K(Z)) \le \mf{n}(S_F/R)\cdot \mf{n}(R/\spec K(Z)) 
    = \mf{n}\big(R\otimes_{K(Z)} \ol{K(Z)}\big). \]
    Therefore, $\sdeg (S^{\nor}/Z)$
    is bounded from above depending only on $d,t$.
\end{proof}


\subsection{Stein degree of vertical divisors}

\begin{theorem}[=Theorem~\ref{vertical-bnd-intro}]\label{pf-vertical-case}
    Let $d\in \NN$ and let $t\in \RR^{>0}$.  Then $\mc{V}_{\univ}(d, t)$ holds.
\end{theorem}

\begin{proof}
    Keep the notation in the statement of Theorem~\ref{vertical-bnd-intro}.
    Decreasing $t$ slightly, we can assume that $t\in (0,1]\cap \QQ$.
    By Theorem~\ref{vertical-case-fib-dim-one}, $\mc{V}_{\univ}(1, t)$ holds.
    We prove the result by induction on the relative dimension $d$.  We assume that 
    $\mc{V}_{\univ}(l, s)$ holds for every $1\le l \le d-1$ and for every $s\in (0,1]$. 
    By Lemma~\ref{exist-g-dlt-model}, taking a $\QQ$-factorial g-dlt model, we can assume that $(X, B + \mb{M})$
    is a $\QQ$-factorial g-dlt g-pair.  By Lemma~\ref{fano-type-at-dlt-model}, the Fano type property 
    of $X\to Z$ is preserved. 
    By taking general Cartier hyperplane sections of $Z$ inductively, 
    we can assume that $\dim Z = 2$ and $\dim T = 1$;
    see Lemma~\ref{fano-type-induction}. 

    \medskip
    \emph{Step~1}.
    In this step, we reduce the proof to an lc pair with standard coefficients.
    Since $X$ is of Fano type over $Z$, and since
    $-S$ is pseudo-effective over $Z$, by Lemma~\ref{fano-type-MMP},
    we can run $\phi\colon X\dashrightarrow X'$ a $(-S)$-MMP  
    over $Z$ that ends with a good minimal model $X'\to Z$ of $-S'$, where $S' \coloneqq \phi_* S$.    
    Let $g\colon X'\to Z'$ be the contraction induced by $-S'$, and let $T'$ be the birational transform
    of $T$ to $Z'$.  As $-S'$ is nef$/Z'$, $S'$ is the unique irreducible component of $\Supp (g^* T')$.
    Then we can write
    \begin{equation}
        S' = a\cdot g^* T' \label{S-whole-fibre}
    \end{equation}
    for some $a\in \QQ^{>0}$.  
    Again by Lemma~\ref{fano-type-MMP},
    running $\varphi\colon X'\dashrightarrow X''$ a $(- K_{X'} - t S')$-MMP 
    over $Z'$, we have that $-(K_{X''} + t S'')$
    is nef$/Z'$, where $S'' \coloneqq \varphi_* S'$.
    By \eqref{S-whole-fibre}, $S'$ is not contracted by $\varphi$.
    Denote by $h\colon X'' \to Z'$ the induced contraction over $Z$. 
    Then by \eqref{S-whole-fibre}, we also have
    \begin{equation}
        S'' = a\cdot h^* T'. \label{S-whole-fibre-1}
    \end{equation}
    By Lemma~\ref{Fano-type-pushdown}, $X''$ is also of Fano type over $Z'$.
    Thus, by Theorem~\ref{t-bnd-compl-usual-local}, 
    over the generic point of $T'$, there exists a monotonic $n$-complement
    $K_{X''} + B^+$ of $K_{X''} + t S''$ for some $n\in \NN$ depending only on $d, t$ such that
    \begin{itemize}
        \item $(X'', B^+)$ is lc,
        \item $n(K_{X''} + B^+) \sim 0/Z'$, and
        \item $t S'' \le B^+$.
    \end{itemize}
    In particular, the coefficients of $B^+$ belong to $\mb{T}_n$. 

    \medskip
    \emph{Step~2}.
    In this step, we reduce the proof to Mori fibre spaces.
    Set $\alpha \coloneqq \coeff_{S''} B^+\in \mb{T}_n$, and set
    \[ C'' \coloneqq B^+ - \frac{\alpha}{2} S''.  \]
    Then the coefficients of $C''$ belong to the set of rational numbers $\mb{T}_{2n}$.  Moreover, 
    over the generic point of $T'$, we also have
    \begin{itemize}
        \item $(X'', C'')$ is lc,
        \item $K_{X''} + C''\sim_{\QQ} 0/Z'$, and
        \item $\frac{t}{2} S'' \le C''$.
    \end{itemize}
    As $(X'', B^+)$ is lc, there is no lc centre of $(X'', C'')$ contained in $S''$ by \eqref{S-whole-fibre-1}
    and \cite[Lemma~2.5]{Kol_singularities_of_MMP}.  
    By Lemma~\ref{exist-g-dlt-model}, let 
    \[ \pi\colon (X''', C''') \to (X'', C'') \] 
    be a $\QQ$-factorial dlt model of $(X'', C'')$.
    Then shrinking $Z'$ near the generic point of $T'$ if necessary, 
    we can assume that every stratum of $\rddown{C'''}$ is horizontal$/Z'$.
    Let $F$ be a general fibre of $X'''\to Z'$.  Since
    \[ K_F + C'''|_F \sim_{\QQ}0, \]
    and since the coefficients of horizontal$/Z'$ irreducible components of $C'''$ belong to $\mb{T}_n$, 
    there exists an $\epsilon >0$
    such that $(F, 0)$ is $\epsilon$-lc by Lemma~\ref{l-ep-to-0}.
    Thus, up to shrinking $Z'$ near the generic point of $T'$, we can assume that
    $(X''', 0)$ is $\epsilon$-lc over $Z'\setminus T'$.  
    
    As $X'''$ is of Fano type over $Z'$ (by Lemma~\ref{fano-type-at-dlt-model}), 
    $K_{X'''}$ is non-pseudo-effective$/Z'$.  By Lemma~\ref{fano-type-MMP}, we can
    run $\psi\colon X'''\dashrightarrow \wt{X}$ a $K_{X'''}$-MMP over $Z'$
    that ends with a Mori fibre space $\wt{f}\colon \wt{X} \to \wt{Z}$ over $Z'$;
    in particular, $-K_{\wt{X}}$ is ample$/\wt{Z}$.  
    \[\xymatrix{
     & & X'''\ar[d]^{\pi}\ar@{-->}[r]^-{\psi} & \wt{X}\ar[dd]^{\wt{f}} \\
    X\ar@{-->}[r]^-{\phi}\ar[d]_f & X'\ar@{-->}[r]^-{\varphi}\ar[d]_g & X''\ar[ld]^-h & \\
    Z & Z'\ar[l] & & \wt{Z}\ar[ll]
    }\]
    Let $S'''$ be the birational transform of $S''$ to $X'''$.  
    Since every exceptional divisor of $\pi$ is horizontal$/Z'$, we also have that
    (up to shrinking $Z'$ near the generic point of $T'$)
    \begin{equation}
        S''' = a\cdot (h\circ \pi)^* T' \label{S-whole-fibre-2}
    \end{equation}
    for the $a\in \QQ^{>0}$ as in \eqref{S-whole-fibre-1}.  
    Thus, $S'''$ is not contracted by $\psi$.  Set $\wt{S} \coloneqq \psi_* S'''$. 

    \medskip
    \emph{Step~3}.
    In this step, we show that the diagram obtained in Step~2 gives a tower of two Fano type fibrations, and then
    apply induction hypothesis to these Fano type fibrations.
    Let $\wt{T}$ be the image of $\wt{S}$ in $\wt{Z}$.  
    By \eqref{S-whole-fibre-2}, $\wt{T}$ is also of codimension one in $\wt{Z}$.
    Let $\wt{C}$ be the pushdown of $C'''$ to $\wt{X}$, then
    \[ K_{\wt{X}} + \wt{C} \sim_{\QQ} 0/Z'. \]
    Assume that $\dim (\wt{Z}/Z') \ge 1$.
    By Theorem~\ref{t-adj-fib-spaces}, we can write
    \[ K_{\wt{X}} + \wt{C} \sim_{\QQ} \wt{f}^* (K_{\wt{Z}} + C_{\wt{Z}} + \mb{N}_{\wt{Z}}), \]
    where $(\wt{Z}, C_{\wt{Z}} + \mb{N}) \to Z'$ is a Fano type generalised log Calabi-Yau fibration
    by Lemma~\ref{fano-type-on-base}.
    Moreover, by Lemma~\ref{l-adj-fib-eps-lc}, we also have that
    \[ \coeff_{\wt{T}} C_{\wt{Z}} \ge \frac{t}{2}. \]
    By Lemma~\ref{fano-type-upper}, 
    $(\wt{X}, \wt{C}) \to \wt{Z}$ is also a Fano type (generalised) log Calabi-Yau fibration
    with $\coeff_{\wt{S}} \wt{C} \ge \frac{t}{2}$.
    Thus, by induction hypothesis $\mc{V}_{\univ}(l, t/2)$ for $1\le l\le d-1$, 
    \[ \sdeg \big( (\wt{S})^{\nor}/\wt{T} \big)~~\mbox{and}~~\sdeg \big((\wt{T})^{\nor}/T'\big) \]
    are bounded from above depending only on $d, t$.  Hence by Lemma~\ref{stein-multiplicity},
    \[ \sdeg (S^{\nor}/T) = \sdeg \big( (\wt{S})^{\nor}/T' \big)
    \le \sdeg \big( (\wt{S})^{\nor}/\wt{T} \big)\cdot \sdeg \big( (\wt{T})^{\nor}/T' \big) \]
    is also bounded from above depending only on $d,t$.
    Therefore, we can assume that $\dim (\wt{Z}/Z') = 0$, 
    that is, the contraction $\wt{Z}\to Z'$ is birational.
    Shrinking $Z'$ near the generic point of $T'$ if necessary, we can assume further
    that $\wt{Z} \to Z'$ is an isomorphism of nonsingular surfaces.

    As $X'''$ is $\epsilon$-lc over $Z\setminus T$,
    by Lemma~\ref{CY-log-discre-no-change}, $\wt{X}$ is $\epsilon$-lc over the generic point of $Z$.
    By construction, $-K_{\wt{X}}$ is ample$/Z'$.  
    By Theorem~\ref{t-bnd-compl-usual-local}, there exists an $m\in \NN$ depending 
    only on $d,n$, whence only on $d,t$, such that there is a monotonic $m$-complement 
    $K_{\wt{X}} + \wt{C}^+$ of $K_{\wt{X}} + \wt{C}$ over the generic point of $T$.
    Shrinking $Z$ around the generic point of $T$, we can assume that
    \[ m(K_{\wt{X}} + \wt{C}^+)\sim 0/Z. \]
    Moreover, it is evident that 
    \[ \coeff_{\wt{S}} \wt{C}^+ \ge \coeff_{\wt{S}} \wt{C} \ge \frac{t}{2}. \]
    As we assume $\mcV_{\univ}(l, s)$ holds for each $1\le l\le d-1$ 
    and each $s\in (0,1]$, by Theorem~\ref{pf-dim-d-t}, $\mcH_{\univ}(d, t/2)$ holds.
    Then by Theorem~\ref{vertical-thm} and $\mcH_{\univ}(d,t/2)$,
    we can conclude that $\sdeg(S^{\nor}/T)$ is bounded from above depending only on $d,t$.    
\end{proof}


\begin{theorem}[{=Theorem~\ref{dim-d-t}}]\label{pf-main-d-t}
    Let $d\in \NN$ and let $t\in \RR^{>0}$.  
    Then $\mc{H}_{\univ}(d, t)$ holds.
\end{theorem}

\begin{proof}
    This follows from Theorem~\ref{pf-dim-d-t} and Theorem~\ref{pf-vertical-case}.
\end{proof}


\subsection{Uniform bound of residue degree}

\begin{proof}[Proof of Theorem~\ref{arithmetic-thm}]
    By Lemma~\ref{mld-attainment}, there exists a prime divisor $S$ over $X$,
    whose centre is equal to $x$, such that
    \[ a(S, X, B) = \mld_x(X, B) \le t. \]
    By Lemma~\ref{exist-g-dlt-model}, there exists a birational contraction
    $\phi\colon Y\to X$ such that
    \[ K_Y + B_Y = \phi^*(K_X + B), \]
    where $(Y, B_Y)$ is a $\QQ$-factorial dlt pair, and $S$ is a prime divisor on $Y$.  Then
    \[ \coeff_S B_Y = 1- a(S, X, B)\ge 1-t > 0. \]
    By Theorem~\ref{dim-d-t}, $\sdeg (S^{\nor}/\spec k)$ is bounded from above 
    depending only on $d,t$.  It is evident that $S^{\nor}\to \spec k$ has a factorisation
    \[ S^{\nor}\to \spec k(x) \to \spec k. \]
    By Lemma~\ref{stein-multiplicity}, we get
    \[ \deg (k(x)/k) \le \sdeg (S^{\nor}/\spec k). \]
    Hence $\deg(k(x)/k)$ is bounded from above depending only on $d,t$.
\end{proof}


\medskip

\printbibliography[title={\bf References}]

@misc{OAI:Arithmetic-Stein-degree-bounds-for-log-Calabi-Yau-pairs-September-25-2026,
  author = {{OpenAI}},
  title = {{Arithmetic Stein-degree bounds for log Calabi--Yau pairs}},
  howpublished = {OpenAI Math Release preprint
                  \href{https://github.com/openai/math/blob/main/preprints/Arithmetic-Stein-degree-bounds-for-log-Calabi-Yau-pairs-September-25-2026/paper.pdf}{OAI:Arithmetic-Stein-degree-bounds-for-log-Calabi-Yau-pairs-September-25-2026}},
  year = {2026},
  shorthand = {OAI26}
}

@article {Salberger-r-pts,
    AUTHOR = {Salberger, Per},
     TITLE = {On the density of rational and integral points on algebraic
              varieties},
   JOURNAL = {J. Reine Angew. Math.},
  FJOURNAL = {Journal f\"ur die Reine und Angewandte Mathematik. [Crelle's
              Journal]},
    VOLUME = {606},
      YEAR = {2007},
     PAGES = {123--147},
      ISSN = {0075-4102,1435-5345},
   MRCLASS = {11G35 (11G50 14G05)},
  MRNUMBER = {2337644},
MRREVIEWER = {Ulrich\ Derenthal},
       DOI = {10.1515/CRELLE.2007.037},
       URL = {https://doi.org/10.1515/CRELLE.2007.037},
}

@article {CL13-new-outlook,
    AUTHOR = {Corti, Alessio and Lazi\'c, Vladimir},
     TITLE = {New outlook on the minimal model program, {II}},
   JOURNAL = {Math. Ann.},
  FJOURNAL = {Mathematische Annalen},
    VOLUME = {356},
      YEAR = {2013},
    NUMBER = {2},
     PAGES = {617--633},
      ISSN = {0025-5831,1432-1807},
   MRCLASS = {14E30},
  MRNUMBER = {3048609},
MRREVIEWER = {Thomas\ Eckl},
       DOI = {10.1007/s00208-012-0858-1},
       URL = {https://doi.org/10.1007/s00208-012-0858-1},
}

@article {birkar-geo-polar-variety,
    AUTHOR = {Birkar, Caucher},
     TITLE = {Geometry of polarised varieties},
   JOURNAL = {Publ. Math. Inst. Hautes \'Etudes Sci.},
  FJOURNAL = {Publications Math\'ematiques. Institut de Hautes \'Etudes
              Scientifiques},
    VOLUME = {137},
      YEAR = {2023},
     PAGES = {47--105},
      ISSN = {0073-8301,1618-1913},
   MRCLASS = {14E30 (14D20)},
  MRNUMBER = {4588595},
MRREVIEWER = {Haidong\ Liu},
       DOI = {10.1007/s10240-022-00136-w},
       URL = {https://doi.org/10.1007/s10240-022-00136-w},
}

@article {Hacon-McKernan-06-bnd,
    AUTHOR = {Hacon, Christopher D. and McKernan, James},
     TITLE = {Boundedness of pluricanonical maps of varieties of general
              type},
   JOURNAL = {Invent. Math.},
  FJOURNAL = {Inventiones Mathematicae},
    VOLUME = {166},
      YEAR = {2006},
    NUMBER = {1},
     PAGES = {1--25},
      ISSN = {0020-9910,1432-1297},
   MRCLASS = {14E05 (14J40)},
  MRNUMBER = {2242631},
MRREVIEWER = {Meng\ Chen},
       DOI = {10.1007/s00222-006-0504-1},
       URL = {https://doi.org/10.1007/s00222-006-0504-1},
}

@book {AM-comm-algebra,
    AUTHOR = {Atiyah, M. F. and Macdonald, I. G.},
     TITLE = {Introduction to commutative algebra},
 PUBLISHER = {Addison-Wesley Publishing Co., Reading, Mass.-London-Don
              Mills, Ont.},
      YEAR = {1969},
     PAGES = {ix+128},
   MRCLASS = {13.00},
  MRNUMBER = {242802},
MRREVIEWER = {Johnny\ A.\ Johnson},
}

@book {Cohen-Macaulay-rings-book,
    AUTHOR = {Bruns, Winfried and Herzog, J\"urgen},
     TITLE = {Cohen-{M}acaulay rings},
    SERIES = {Cambridge Studies in Advanced Mathematics},
    VOLUME = {39},
 PUBLISHER = {Cambridge University Press, Cambridge},
      YEAR = {1993},
     PAGES = {xii+403},
      ISBN = {0-521-41068-1},
   MRCLASS = {13H10 (13-02)},
  MRNUMBER = {1251956},
MRREVIEWER = {Matthew\ Miller},
}

@article {vol-non-closed-fields,
    AUTHOR = {Cutkosky, Steven Dale and N\'u\~nez, Roberto},
     TITLE = {Volumes of line bundles on schemes},
   JOURNAL = {Proc. Amer. Math. Soc.},
  FJOURNAL = {Proceedings of the American Mathematical Society},
    VOLUME = {149},
      YEAR = {2021},
    NUMBER = {10},
     PAGES = {4099--4108},
      ISSN = {0002-9939,1088-6826},
   MRCLASS = {14C40 (14C17)},
  MRNUMBER = {4305966},
MRREVIEWER = {Xiping\ Zhang},
       DOI = {10.1090/proc/15526},
       URL = {https://doi.org/10.1090/proc/15526},
}

@book {Katz-Mazur-1985-elliptic,
    AUTHOR = {Katz, Nicholas M. and Mazur, Barry},
     TITLE = {Arithmetic moduli of elliptic curves},
    SERIES = {Annals of Mathematics Studies},
    VOLUME = {108},
 PUBLISHER = {Princeton University Press, Princeton, NJ},
      YEAR = {1985},
     PAGES = {xiv+514},
   MRCLASS = {11G05 (11F11 14G25 14K15)},
  MRNUMBER = {772569},
MRREVIEWER = {Kenneth\ A.\ Ribet},
       DOI = {10.1515/9781400881710},
       URL = {https://doi.org/10.1515/9781400881710},
}

@article {algebraic-complex-normalisation,
    AUTHOR = {Kuhlmann, Norbert},
     TITLE = {Die {N}ormalisierung komplexer {R}\"aume},
   JOURNAL = {Math. Ann.},
  FJOURNAL = {Mathematische Annalen},
    VOLUME = {144},
      YEAR = {1961},
     PAGES = {110--125},
      ISSN = {0025-5831,1432-1807},
   MRCLASS = {32.47 (32.50)},
  MRNUMBER = {137846},
MRREVIEWER = {H.\ R\"ohrl},
       DOI = {10.1007/BF01451331},
       URL = {https://doi.org/10.1007/BF01451331},
}

@article {normalisation-complex-variety,
    AUTHOR = {Acquistapace, Francesca and Broglia, Fabrizio and Fernando,
              Jos\'e{} F.},
     TITLE = {Normalization of complex analytic spaces from a global
              viewpoint},
   JOURNAL = {J. Geom. Anal.},
  FJOURNAL = {Journal of Geometric Analysis},
    VOLUME = {29},
      YEAR = {2019},
    NUMBER = {3},
     PAGES = {2888--2930},
      ISSN = {1050-6926,1559-002X},
   MRCLASS = {32C20 (14P15 32C05 32Q35 58A07)},
  MRNUMBER = {3969448},
MRREVIEWER = {Truong\ Le\ Hoang},
       DOI = {10.1007/s12220-018-00098-8},
       URL = {https://doi.org/10.1007/s12220-018-00098-8},
}

@misc{birkar2025steindegreelogcalabiyau,
      title={Stein degree on log Calabi-Yau fibrations}, 
      author={Caucher Birkar and Santai Qu},
      year={2025},
      eprint={2509.20948},
      archivePrefix={arXiv},
      primaryClass={math.AG},
      url={https://arxiv.org/abs/2509.20948}, 
}

@article {Sh-3-fold-log-flips,
    AUTHOR = {Shokurov, V. V.},
     TITLE = {Three-dimensional log perestroikas},
      NOTE = {With an appendix in English by Yujiro Kawamata},
   JOURNAL = {Izv. Ross. Akad. Nauk Ser. Mat.},
  FJOURNAL = {Izvestiya Rossiiskoi Akademii Nauk. Seriya Matematicheskaya},
    VOLUME = {56},
      YEAR = {1992},
    NUMBER = {1},
     PAGES = {105--203},
      ISSN = {1607-0046,2587-5906},
   MRCLASS = {14E05 (14E35)},
  MRNUMBER = {1162635},
MRREVIEWER = {A.\ S.\ Tikhomirov},
       DOI = {10.1070/IM1993v040n01ABEH001862},
       URL = {https://doi.org/10.1070/IM1993v040n01ABEH001862},
}

@article {nef-not-open,
    AUTHOR = {Lesieutre, John},
     TITLE = {The diminished base locus is not always closed},
   JOURNAL = {Compos. Math.},
  FJOURNAL = {Compositio Mathematica},
    VOLUME = {150},
      YEAR = {2014},
    NUMBER = {10},
     PAGES = {1729--1741},
      ISSN = {0010-437X,1570-5846},
   MRCLASS = {14E30 (14E07)},
  MRNUMBER = {3269465},
MRREVIEWER = {Anne-Sophie\ Kaloghiros},
       DOI = {10.1112/S0010437X14007544},
       URL = {https://doi.org/10.1112/S0010437X14007544},
}

@book {Serre-local-algebra,
    AUTHOR = {Serre, Jean-Pierre},
     TITLE = {Local algebra},
    SERIES = {Springer Monographs in Mathematics},
      NOTE = {Translated from the French by CheeWhye Chin and revised by the
              author},
 PUBLISHER = {Springer-Verlag, Berlin},
      YEAR = {2000},
     PAGES = {xiv+128},
      ISBN = {3-540-66641-9},
   MRCLASS = {13-01 (13D05 13H05 13H15)},
  MRNUMBER = {1771925},
       DOI = {10.1007/978-3-662-04203-8},
       URL = {https://doi.org/10.1007/978-3-662-04203-8},
}

@book{BLR,
 author = {Bosch, Siegfried and L{\"u}tkebohmert, Werner and Raynaud, Michel},
 title = {N{\'e}ron models},
 fseries = {Ergebnisse der Mathematik und ihrer Grenzgebiete. 3. Folge},
 series = {Ergeb. Math. Grenzgeb., 3. Folge},
 issn = {0071-1136},
 volume = {21},
 isbn = {3-540-50587-3},
 year = {1990},
 publisher = {Berlin etc.: Springer-Verlag},
 doi = {10.1007/978-3-642-51438-8},
 zbMATH = {46134},
 Zbl = {0705.14001}
}

@article{Raynaud-Gruson-plat-1971,
 author = {Raynaud, Michel and Gruson, Laurent},
 title = {Crit{\`e}res de platitude et de projectivit{\'e}. {Techniques} de ''platification'' d'un module. ({Criterial} of flatness and projectivity. {Technics} of ''flatification of a module.)},
 fjournal = {Inventiones Mathematicae},
 journal = {Invent. Math.},
 issn = {0020-9910},
 volume = {13},
 pages = {1--89},
 year = {1971},
 doi = {10.1007/BF01390094},
 url = {https://eudml.org/doc/142084},
 zbMATH = {3360302},
 Zbl = {0227.14010}
}

@article {vol-hilb-R-divisors,
    AUTHOR = {Fulger, Mihai and Koll\'ar, J\'anos and Lehmann, Brian},
     TITLE = {Volume and {H}ilbert function of {$\mathbb R$}-divisors},
   JOURNAL = {Michigan Math. J.},
  FJOURNAL = {Michigan Mathematical Journal},
    VOLUME = {65},
      YEAR = {2016},
    NUMBER = {2},
     PAGES = {371--387},
   MRCLASS = {14C20},
  MRNUMBER = {3510912},
MRREVIEWER = {Montserrat\ Teixidor i Bigas},
       DOI = {10.1307/mmj/1465329018},
       URL = {https://doi.org/10.1307/mmj/1465329018},
}

@book {Rob-II,
    AUTHOR = {Lazarsfeld, Robert},
     TITLE = {Positivity in algebraic geometry. {II}},
    SERIES = {Ergebnisse der Mathematik und ihrer Grenzgebiete. 3. Folge. A
              Series of Modern Surveys in Mathematics [Results in
              Mathematics and Related Areas. 3rd Series. A Series of Modern
              Surveys in Mathematics]},
    VOLUME = {49},
      NOTE = {Positivity for vector bundles, and multiplier ideals},
 PUBLISHER = {Springer-Verlag, Berlin},
      YEAR = {2004},
     PAGES = {xviii+385},
      ISBN = {3-540-22534-X},
   MRCLASS = {14-02 (14C20 14F05 14F17)},
  MRNUMBER = {2095472},
MRREVIEWER = {Mihnea\ Popa},
       DOI = {10.1007/978-3-642-18808-4},
       URL = {https://doi.org/10.1007/978-3-642-18808-4},
}

@article {Hart-stable-refl,
    AUTHOR = {Hartshorne, Robin},
     TITLE = {Stable reflexive sheaves},
   JOURNAL = {Math. Ann.},
  FJOURNAL = {Mathematische Annalen},
    VOLUME = {254},
      YEAR = {1980},
    NUMBER = {2},
     PAGES = {121--176},
   MRCLASS = {14F05 (14D22)},
  MRNUMBER = {597077},
MRREVIEWER = {Klaus\ Hulek},
       DOI = {10.1007/BF01467074},
       URL = {https://doi.org/10.1007/BF01467074},
}

@article {birkar12-special-LMMP,
    AUTHOR = {Birkar, Caucher},
     TITLE = {Existence of log canonical flips and a special {LMMP}},
   JOURNAL = {Publ. Math. Inst. Hautes \'Etudes Sci.},
  FJOURNAL = {Publications Math\'ematiques. Institut de Hautes \'Etudes
              Scientifiques},
    VOLUME = {115},
      YEAR = {2012},
     PAGES = {325--368},
   MRCLASS = {14E30},
  MRNUMBER = {2929730},
MRREVIEWER = {Hsian-Hua\ Tseng},
       DOI = {10.1007/s10240-012-0039-5},
       URL = {https://doi.org/10.1007/s10240-012-0039-5},
}

@article {Keeler2003,
    AUTHOR = {Keeler, Dennis S.},
     TITLE = {Ample filters of invertible sheaves},
   JOURNAL = {J. Algebra},
  FJOURNAL = {Journal of Algebra},
    VOLUME = {259},
      YEAR = {2003},
    NUMBER = {1},
     PAGES = {243--283},
   MRCLASS = {14F05 (16S38)},
  MRNUMBER = {1953719},
MRREVIEWER = {Michel\ Van den Bergh},
       DOI = {10.1016/S0021-8693(02)00557-4},
       URL = {https://doi.org/10.1016/S0021-8693(02)00557-4},
}

@misc{chen-han-liu-xie-2023,
      title={Minimal model program for algebraically integrable foliations and generalized pairs}, 
      author={Guodu Chen and Jingjun Han and Jihao Liu and Lingyao Xie},
      year={2023},
      eprint={2309.15823},
      archivePrefix={arXiv},
      primaryClass={math.AG},
      url={https://arxiv.org/abs/2309.15823}, 
}

@book {Fulton,
    AUTHOR = {Fulton, William},
     TITLE = {Intersection theory},
    SERIES = {Ergebnisse der Mathematik und ihrer Grenzgebiete. 3. Folge. A
              Series of Modern Surveys in Mathematics [Results in
              Mathematics and Related Areas. 3rd Series. A Series of Modern
              Surveys in Mathematics]},
    VOLUME = {2},
   EDITION = {Second},
 PUBLISHER = {Springer-Verlag, Berlin},
      YEAR = {1998},
     PAGES = {xiv+470},
   MRCLASS = {14C17 (14-02)},
  MRNUMBER = {1644323},
       DOI = {10.1007/978-1-4612-1700-8},
       URL = {https://doi.org/10.1007/978-1-4612-1700-8},
}

@book {Kol-rc,
    AUTHOR = {Koll\'ar, J\'anos},
     TITLE = {Rational curves on algebraic varieties},
    SERIES = {Ergebnisse der Mathematik und ihrer Grenzgebiete. 3. Folge. A
              Series of Modern Surveys in Mathematics [Results in
              Mathematics and Related Areas. 3rd Series. A Series of Modern
              Surveys in Mathematics]},
    VOLUME = {32},
 PUBLISHER = {Springer-Verlag, Berlin},
      YEAR = {1996},
     PAGES = {viii+320},
      ISBN = {3-540-60168-6},
   MRCLASS = {14-02 (14C05 14E05 14F17 14J45)},
  MRNUMBER = {1440180},
MRREVIEWER = {Yuri\ G.\ Prokhorov},
       DOI = {10.1007/978-3-662-03276-3},
       URL = {https://doi.org/10.1007/978-3-662-03276-3},
}

@incollection {Corti-3-folds,
    AUTHOR = {Corti, Alessio},
     TITLE = {3-fold flips after {S}hokurov},
 BOOKTITLE = {Flips for 3-folds and 4-folds},
    SERIES = {Oxford Lecture Ser. Math. Appl.},
    VOLUME = {35},
     PAGES = {18--48},
 PUBLISHER = {Oxford Univ. Press, Oxford},
      YEAR = {2007},
      ISBN = {978-0-19-857061-5},
   MRCLASS = {14E30},
  MRNUMBER = {2359340},
       DOI = {10.1093/acprof:oso/9780198570615.003.0002},
       URL = {https://doi.org/10.1093/acprof:oso/9780198570615.003.0002},
}

@book {GW-agbook,
    AUTHOR = {G\"ortz, Ulrich and Wedhorn, Torsten},
     TITLE = {Algebraic geometry {I}. {S}chemes---with examples and
              exercises},
    SERIES = {Springer Studium Mathematik---Master},
   EDITION = {Second},
 PUBLISHER = {Springer Spektrum, Wiesbaden},
      YEAR = {2020},
     PAGES = {vii+625},
   MRCLASS = {14-01},
  MRNUMBER = {4225278},
       DOI = {10.1007/978-3-658-30733-2},
       URL = {https://doi.org/10.1007/978-3-658-30733-2},
}

@book{FGA-explained,
 author = {Fantechi, Barbara and G{\"o}ttsche, Lothar and Illusie, Luc and Kleiman, Steven L. and Nitsure, Nitin and Vistoli, Angelo},
 title = {Fundamental algebraic geometry: {Grothendieck}'s {FGA} explained},
 fseries = {Mathematical Surveys and Monographs},
 series = {Math. Surv. Monogr.},
 issn = {0076-5376},
 volume = {123},
 isbn = {0-8218-3541-6},
 year = {2005},
 publisher = {American Mathematical Society (AMS), Providence, RI},
 zbMATH = {2229020},
 Zbl = {1085.14001},
 shorthand = {FGA\textsubscript{exp}}
}

@article {EGA-III-1,
    AUTHOR = {Grothendieck, A.},
     TITLE = {\'El\'ements de g\'eom\'etrie alg\'ebrique. {III}. \'Etude
              cohomologique des faisceaux coh\'erents. {I}.},
   JOURNAL = {Inst. Hautes \'Etudes Sci. Publ. Math.},
  FJOURNAL = {Institut des Hautes \'Etudes Scientifiques. Publications
              Math\'ematiques},
    NUMBER = {11},
      YEAR = {1961},
     PAGES = {167},
      ISSN = {0073-8301,1618-1913},
   MRCLASS = {14.55},
  MRNUMBER = {217085},
       URL = {http://www.numdam.org/item?id=PMIHES_1961__11__167_0},
 shorthand = {EGAIII\textsubscript{1}},
}

@article {EGA-III-2,
    AUTHOR = {Grothendieck, A.},
     TITLE = {\'El\'ements de g\'eom\'etrie alg\'ebrique. {III}. \'Etude
              cohomologique des faisceaux coh\'erents. {II}},
   JOURNAL = {Inst. Hautes \'Etudes Sci. Publ. Math.},
  FJOURNAL = {Institut des Hautes \'Etudes Scientifiques. Publications
              Math\'ematiques},
    NUMBER = {17},
      YEAR = {1963},
     PAGES = {91},
      ISSN = {0073-8301,1618-1913},
   MRCLASS = {14.05},
  MRNUMBER = {163911},
MRREVIEWER = {H.\ Hironaka},
       URL = {http://www.numdam.org/item?id=PMIHES_1963__17__91_0},
 shorthand = {EGAIII\textsubscript{2}},
}

@article {EGA-IV-1,
    AUTHOR = {Grothendieck, A.},
     TITLE = {\'El\'ements de g\'eom\'etrie alg\'ebrique. {IV}. \'Etude
              locale des sch\'emas et des morphismes de sch\'emas. {I}},
   JOURNAL = {Inst. Hautes \'Etudes Sci. Publ. Math.},
  FJOURNAL = {Institut des Hautes \'Etudes Scientifiques. Publications
              Math\'ematiques},
    NUMBER = {20},
      YEAR = {1964},
     PAGES = {259},
      ISSN = {0073-8301,1618-1913},
   MRCLASS = {14.05 (13.10)},
  MRNUMBER = {173675},
MRREVIEWER = {H.\ Hironaka},
       URL = {http://www.numdam.org/item/PMIHES_1964__20__5_0},
 shorthand = {EGAIV\textsubscript{1}},
}

@book {SGA1,
     TITLE = {Rev\^etements \'etales et groupe fondamental},
    SERIES = {Lecture Notes in Mathematics},
    VOLUME = {Vol. 224},
      NOTE = {S\'eminaire de G\'eom\'etrie Alg\'ebrique du Bois Marie
              1960--1961 (SGA 1),
              Dirig\'e{} par Alexandre Grothendieck. Augment\'e{} de deux
              expos\'es de M. Raynaud},
 PUBLISHER = {Springer-Verlag, Berlin-New York},
      YEAR = {1971},
     PAGES = {xxii+447},
   MRCLASS = {14-06 (14E20)},
  MRNUMBER = {354651},
  shorthand = {SGA1},
}

@misc{lyu-takumi-2026relativeminimalmodelprogram,
      title={The relative minimal model program for excellent algebraic spaces and analytic spaces in equal characteristic zero}, 
      author={Shiji Lyu and Takumi Murayama},
      year={2026},
      eprint={2209.08732},
      archivePrefix={arXiv},
      primaryClass={math.AG},
      url={https://arxiv.org/abs/2209.08732}, 
}

@book{Liu06,
 author = {Liu, Qing},
 title = {Algebraic geometry and arithmetic curves. {Transl}. by {Reinie} {Ern{\'e}}},
 fseries = {Oxford Graduate Texts in Mathematics},
 series = {Oxf. Grad. Texts Math.},
 volume = {6},
 isbn = {0-19-920249-4},
 year = {2006},
 publisher = {Oxford: Oxford University Press},
 zbMATH = {5048200},
 Zbl = {1103.14001}
}

@article {temkin-desing-chara-0,
    AUTHOR = {Temkin, Michael},
     TITLE = {Desingularization of quasi-excellent schemes in characteristic
              zero},
   JOURNAL = {Adv. Math.},
  FJOURNAL = {Advances in Mathematics},
    VOLUME = {219},
      YEAR = {2008},
    NUMBER = {2},
     PAGES = {488--522},
      ISSN = {0001-8708,1090-2082},
   MRCLASS = {14E15},
  MRNUMBER = {2435647},
MRREVIEWER = {Ana\ Bravo},
       DOI = {10.1016/j.aim.2008.05.006},
       URL = {https://doi.org/10.1016/j.aim.2008.05.006},
}

@book {km98,
    AUTHOR = {Koll\'ar, J\'anos and Mori, Shigefumi},
     TITLE = {Birational geometry of algebraic varieties},
    SERIES = {Cambridge Tracts in Mathematics},
    VOLUME = {134},
      NOTE = {With the collaboration of C. H. Clemens and A. Corti,
              Translated from the 1998 Japanese original},
 PUBLISHER = {Cambridge University Press, Cambridge},
      YEAR = {1998},
     PAGES = {viii+254},
      ISBN = {0-521-63277-3},
   MRCLASS = {14E30},
  MRNUMBER = {1658959},
MRREVIEWER = {Mark\ Gross},
       DOI = {10.1017/CBO9780511662560},
       URL = {https://doi.org/10.1017/CBO9780511662560},
}

@article {BCHM,
    AUTHOR = {Birkar, Caucher and Cascini, Paolo and Hacon, Christopher D.
              and McKernan, James},
     TITLE = {Existence of minimal models for varieties of log general type},
   JOURNAL = {J. Amer. Math. Soc.},
  FJOURNAL = {Journal of the American Mathematical Society},
    VOLUME = {23},
      YEAR = {2010},
    NUMBER = {2},
     PAGES = {405--468},
      ISSN = {0894-0347,1088-6834},
   MRCLASS = {14E30 (14E05)},
  MRNUMBER = {2601039},
MRREVIEWER = {Mark\ Gross},
       DOI = {10.1090/S0894-0347-09-00649-3},
       URL = {https://doi.org/10.1090/S0894-0347-09-00649-3},
}

@book {Kol_singularities_of_MMP,
    AUTHOR = {Koll\'ar, J\'anos},
     TITLE = {Singularities of the minimal model program},
    SERIES = {Cambridge Tracts in Mathematics},
    VOLUME = {200},
      NOTE = {With a collaboration of S\'andor Kov\'acs},
 PUBLISHER = {Cambridge University Press, Cambridge},
      YEAR = {2013},
     PAGES = {x+370},
      ISBN = {978-1-107-03534-8},
   MRCLASS = {14E30 (14B05)},
  MRNUMBER = {3057950},
MRREVIEWER = {Tommaso\ De Fernex},
       DOI = {10.1017/CBO9781139547895},
       URL = {https://doi.org/10.1017/CBO9781139547895},
}

@Article{Ambro04-Sh-boundary,
 Author = {Ambro, Florin},
 Title = {Shokurov's boundary property},
 FJournal = {Journal of Differential Geometry},
 Journal = {J. Differ. Geom.},
 ISSN = {0022-040X},
 Volume = {67},
 Number = {2},
 Pages = {229--255},
 Year = {2004},
 DOI = {10.4310/jdg/1102536201},
 zbMATH = {5033770},
 Zbl = {1097.14029}
}

@article{B-Fano,
 author = {Birkar, Caucher},
 title = {Anti-pluricanonical systems on {Fano} varieties},
 fjournal = {Annals of Mathematics. Second Series},
 journal = {Ann. Math. (2)},
 issn = {0003-486X},
 volume = {190},
 number = {2},
 pages = {345--463},
 year = {2019},
 doi = {10.4007/annals.2019.190.2.1},
 zbMATH = {7107180},
 Zbl = {1470.14078}
}

@article{B-BAB,
 author = {Birkar, Caucher},
 title = {Singularities of linear systems and boundedness of {Fano} varieties},
 fjournal = {Annals of Mathematics. Second Series},
 journal = {Ann. Math. (2)},
 issn = {0003-486X},
 volume = {193},
 number = {2},
 pages = {347--405},
 year = {2021},
 doi = {10.4007/annals.2021.193.2.1},
 url = {www.repository.cam.ac.uk/handle/1810/314696},
 zbMATH = {7331713},
 Zbl = {1469.14085}
}

@misc{B-moduli,
      title={Moduli of algebraic varieties}, 
      author={Caucher Birkar},
      year={2022},
      eprint={2211.11237},
      archivePrefix={arXiv},
      primaryClass={math.AG},
      url={https://arxiv.org/abs/2211.11237}, 
}

@article {B-Zhang,
    AUTHOR = {Birkar, Caucher and Zhang, De-Qi},
     TITLE = {Effectivity of {I}itaka fibrations and pluricanonical systems
              of polarized pairs},
   JOURNAL = {Publ. Math. Inst. Hautes \'Etudes Sci.},
  FJOURNAL = {Publications Math\'ematiques. Institut de Hautes \'Etudes
              Scientifiques},
    VOLUME = {123},
      YEAR = {2016},
     PAGES = {283--331},
   MRCLASS = {14E05 (14E30)},
  MRNUMBER = {3502099},
MRREVIEWER = {Alex\ Massarenti},
       DOI = {10.1007/s10240-016-0080-x},
       URL = {https://doi.org/10.1007/s10240-016-0080-x},
}

@Article{Filipazzi-bnd-g-pair-surface,
 Author = {Filipazzi, Stefano},
 Title = {Boundedness of log canonical surface generalized polarized pairs},
 FJournal = {Taiwanese Journal of Mathematics},
 Journal = {Taiwanese J. Math.},
 ISSN = {1027-5487},
 Volume = {22},
 Number = {4},
 Pages = {813--850},
 Year = {2018},
 DOI = {10.11650/tjm/171204},
 zbMATH = {6965398},
 Zbl = {1422.14025}
}

@book{joins-intersections,
 author = {Flenner, Hubert and O'Carroll, Liam and Vogel, Wolfgang},
 title = {Joins and intersections},
 fseries = {Springer Monographs in Mathematics},
 series = {Springer Monogr. Math.},
 issn = {1439-7382},
 isbn = {3-540-66319-3},
 year = {1999},
 publisher = {Berlin: Springer},
 doi = {10.1007/978-3-662-03817-8},
 zbMATH = {1357758},
 Zbl = {0939.14003}
}

@article {EGA-I,
    AUTHOR = {Grothendieck, A.},
     TITLE = {\'El\'ements de g\'eom\'etrie alg\'ebrique. {I}. {L}e langage
              des sch\'emas.},
   JOURNAL = {Inst. Hautes \'Etudes Sci. Publ. Math.},
  FJOURNAL = {Institut des Hautes \'Etudes Scientifiques. Publications
              Math\'ematiques},
    NUMBER = {4},
      YEAR = {1960},
     PAGES = {228},
      ISSN = {0073-8301,1618-1913},
   MRCLASS = {14.55},
  MRNUMBER = {217083},
       URL = {http://www.numdam.org/item?id=PMIHES_1960__4__228_0},
  shorthand = {EGAI},
}

@article {EGA-IV-2,
    AUTHOR = {Grothendieck, A.},
     TITLE = {\'{E}l\'{e}ments de g\'{e}om\'{e}trie alg\'{e}brique. {IV}. \'{E}tude locale des
              sch\'{e}mas et des morphismes de sch\'{e}mas. {II}},
   JOURNAL = {Inst. Hautes \'{E}tudes Sci. Publ. Math.},
  FJOURNAL = {Institut des Hautes \'{E}tudes Scientifiques. Publications
              Math\'{e}matiques},
    NUMBER = {24},
      YEAR = {1965},
     PAGES = {231},
      ISSN = {0073-8301},
   MRCLASS = {14.00},
  MRNUMBER = {199181},
MRREVIEWER = {H. Hironaka},
       URL = {http://www.numdam.org/item?id=PMIHES_1965__24__231_0},
shorthand = {EGAIV\textsubscript{2}},
}

@article {EGA-IV-3,
    AUTHOR = {Grothendieck, A.},
     TITLE = {\'El\'ements de g\'eom\'etrie alg\'ebrique. {IV}. \'Etude
              locale des sch\'emas et des morphismes de sch\'emas. {III}},
   JOURNAL = {Inst. Hautes \'Etudes Sci. Publ. Math.},
  FJOURNAL = {Institut des Hautes \'Etudes Scientifiques. Publications
              Math\'ematiques},
    NUMBER = {28},
      YEAR = {1966},
     PAGES = {255},
   MRCLASS = {14.55},
  MRNUMBER = {217086},
MRREVIEWER = {J.\ P.\ Murre},
       URL = {http://www.numdam.org/item?id=PMIHES_1966__28__255_0},
shorthand = {EGAIV\textsubscript{3}},
}

@article {EGA-IV-4,
    AUTHOR = {Grothendieck, A.},
     TITLE = {\'El\'ements de g\'eom\'etrie alg\'ebrique. {IV}. \'Etude
              locale des sch\'emas et des morphismes de sch\'emas {IV}},
   JOURNAL = {Inst. Hautes \'Etudes Sci. Publ. Math.},
  FJOURNAL = {Institut des Hautes \'Etudes Scientifiques. Publications
              Math\'ematiques},
    NUMBER = {32},
      YEAR = {1967},
     PAGES = {361},
      ISSN = {0073-8301,1618-1913},
   MRCLASS = {14.55},
  MRNUMBER = {238860},
MRREVIEWER = {J.\ P.\ Murre},
       URL = {http://www.numdam.org/item?id=PMIHES_1967__32__361_0},
       shorthand = {EGAIV\textsubscript{4}},
}

@book {Hart,
    AUTHOR = {Hartshorne, Robin},
     TITLE = {Algebraic geometry},
    SERIES = {Graduate Texts in Mathematics},
    VOLUME = {No. 52},
 PUBLISHER = {Springer-Verlag, New York-Heidelberg},
      YEAR = {1977},
     PAGES = {xvi+496},
      ISBN = {0-387-90244-9},
   MRCLASS = {14-01},
  MRNUMBER = {463157},
MRREVIEWER = {Robert\ Speiser},
}

@book{jou-bertini,
 author = {Jouanolou, Jean-Pierre},
 title = {Th{\'e}oremes de {Bertini} et applications},
 fseries = {Progress in Mathematics},
 series = {Prog. Math.},
 issn = {0743-1643},
 volume = {42},
 year = {1983},
 publisher = {Birkh{\"a}user, Cham},
 zbMATH = {3821916},
 Zbl = {0519.14002}
}

@article{KK10,
 author = {Koll{\'a}r, J{\'a}nos and Kov{\'a}cs, S{\'a}ndor J.},
 title = {Log canonical singularities are {Du} {Bois}},
 fjournal = {Journal of the American Mathematical Society},
 journal = {J. Am. Math. Soc.},
 issn = {0894-0347},
 volume = {23},
 number = {3},
 pages = {791--813},
 year = {2010},
 doi = {10.1090/S0894-0347-10-00663-6},
 zbMATH = {5775685},
 Zbl = {1202.14003}
}

@book{CA-Mat,
 author = {Matsumura, Hideyuki},
 title = {Commutative ring theory. {Transl}. from the {Japanese} by {M}. {Reid}.},
 edition = {Paperback ed.},
 fseries = {Cambridge Studies in Advanced Mathematics},
 series = {Camb. Stud. Adv. Math.},
 volume = {8},
 isbn = {0-521-36764-6},
 year = {1989},
 publisher = {Cambridge etc.: Cambridge University Press},
 zbMATH = {43569},
 Zbl = {0666.13002}
}

@book{nakayama_Zariski_decomposition,
 author = {Nakayama, Noboru},
 title = {Zariski-decomposition and abundance},
 fseries = {MSJ Memoirs},
 series = {MSJ Mem.},
 issn = {2189-1494},
 volume = {14},
 isbn = {4-931469-31-0},
 year = {2004},
 publisher = {Tokyo: Mathematical Society of Japan},
 doi = {10.2969/msjmemoirs/014010000},
 url = {projecteuclid.org/euclid.msjm/1389986103},
 zbMATH = {2121550},
 Zbl = {1061.14018}
}

@Article{Sh-3-fold-log-models,
 Author = {Shokurov, V. V.},
 Title = {3-fold log models},
 FJournal = {Journal of Mathematical Sciences (New York)},
 Journal = {J. Math. Sci., New York},
 ISSN = {1072-3374},
 Volume = {81},
 Number = {3},
 Pages = {2667--2699},
 Year = {1996},
 DOI = {10.1007/BF02362335},
 zbMATH = {953756},
 Zbl = {0873.14014}
}
 
\vspace{1em}
 
\noindent\small{Caucher Birkar} 

\noindent\small{\textsc{Yau Mathematical Sciences Center, Tsinghua University, Beijing, China} }

\noindent\small{Email: \texttt{birkar@mail.tsinghua.edu.cn}}

\vspace{1em}
 
\noindent\small{Santai Qu} 

\noindent\small{\textsc{Beijing Institute of Mathematical Sciences and Applications, Beijing, China} }

\noindent\small{Email: \texttt{qusantai@bimsa.cn}}

\end{document}